\documentclass[12pt]{amsart}
\usepackage[margin=1in]{geometry}
\usepackage{amssymb,amsmath,exscale,latexsym,amsthm,graphics,graphics,mathtools,overpic,subcaption}%,accents}
\usepackage{epsfig}    % eps-Bilder einlesen
\usepackage{epic}      % irgendsone Graphikerweiterung
\usepackage{eepic}     % nochne erweiterung
\usepackage[matrix,arrow,curve,cmtip]{xy} % xy-pic, package for commutative
\usepackage{amsfonts,amsbsy}
\usepackage{enumerate} % more flexibility for numbered lists
\usepackage{paralist}
\usepackage{xcolor}
\usepackage{mathrsfs}
\usepackage{enumitem}
\usepackage{verbatim} %Added by Annina (wanna use "comment")
\usepackage{tikz-cd}
\usepackage{MnSymbol}  % additional symbols need mnsymbol.sty
\usepackage{hyperref}
\usepackage[normalem]{ulem}
\hypersetup{%
  colorlinks=false,
  urlbordercolor=cyan, linkbordercolor=cyan}

\theoremstyle{plain}
        \newtheorem{theorem}{Theorem}[section]
        \newtheorem*{theorem*}{Theorem}
        \newtheorem*{maintheorem*}{Main Theorem}
        \newtheorem*{conj*}{Conjecture}
        \newtheorem{lemma}[theorem]{Lemma}
        \newtheorem{cor}[theorem]{Corollary}
        \newtheorem{prop}[theorem]{Proposition}

\theoremstyle{definition}
        \newtheorem{definition}[theorem]{Definition}
        \newtheorem*{definition*}{Definition}

\theoremstyle{remark}

          \newtheorem{rem}[theorem]{Remark}

        \newtheorem*{claim}{Claim}

\numberwithin{equation}{section}

\def\reminder #1 {{\sf #1}}
\def\hide #1 {}
\long\def\longhide #1 {}

\newcounter{mylistnum}

\makeatletter
\newcommand{\mylabel}[2]{#2\def\@currentlabel{#2}\label{#1}}
\makeatother

\definecolor{anisred}{rgb}{1, 0, 0.58} %255,0,150
\definecolor{anisblue}{rgb}{0,0,0.82} %0,0,210
\definecolor{anisgray}{rgb}{0.7, 0.7, 0.7} %0,0,210
\definecolor{anisgreen}{rgb}{0.42, 0.56, 0.15} %107,142.37
\newcommand{\halpha}{\widehat{\alpha}}
\newcommand{\talpha}{\widetilde{\alpha}}

\newcommand{\tgam}{\widetilde{\gamma}}
\newcommand{\tU}{\widetilde{U}}
\newcommand{\hU}{\widehat{U}}
\newcommand{\fhat}{\widehat{f}}   %they are needed in the figures!

\newcommand{\inte}  {\operatorname{int}}
\newcommand{\cl}  {\operatorname{cl}}

\newcommand{\id} {\operatorname{id}}
\newcommand{\R}{\mathbb{R}}      % reelle Zahlen
\newcommand{\C}{\mathbb{C}}      % komplexe Zahlen
\newcommand{\N}{\mathbb{N}}      % natuerliche Zahlen
\newcommand{\Z}{\mathbb{Z}}      % ganze Zahlen
\newcommand{\Q}{\mathbb{Q}}      % ganze Zahlen
\newcommand{\CDach}{\widehat{\mathbb{C}}}% Riemannsche Zahlenkugel
\newcommand{\D}{\mathbb{D}}      % Einheitskreis
\newcommand{\AC}{\mathcal{A}}

\newcommand{\HC}{\mathcal{H}}

\newcommand{\TC}{\mathcal{T}}

\newcommand{\GC}{\mathcal{G}}

\newcommand{\spH}{\textcolor{anisred}{$\mathcal{H}$}}     % used in figure for curve attractor!

\renewcommand{\:}{\colon}   
\newcommand{\ra}{\rightarrow} 
\newcommand{\sub}{\subset}

\newcommand{\Sp}{S^2}

\newcommand{\CC}{\mathcal{C}}
\newcommand{\LL}{\mathcal{L}}

\newcommand{\PP}{\mathbb{P}}
\newcommand{\PPh}{\widehat{\mathbb{P}}} %command used in figures
\newcommand{\Phat}{\widehat{\mathbb{P}}}

\newcommand{\La}{\mathcal{L}}

\newcommand{\I}{\mathbb{I}}

  \newcommand{\postf}{P_f}
   \newcommand{\critf}{C_f}

    \newcommand{\al}{\alpha}
     \newcommand{\be}{\beta}
    \newcommand{\eps}{\epsilon}
  \newcommand{\gam}{\gamma}
   \newcommand{\ga}{\gamma}
    \newcommand{\om}{\omega}
  \newcommand{\ins}{\mathrm{i}} %intersection number

\newcommand{\inter}{\operatorname{int}}

\begin{document}
\title[Thurston obstructions and dynamics on curves]{Eliminating  Thurston obstructions and controlling dynamics on curves}

\emph{ }

 \author{Mario Bonk}
\address {Department of Mathematics, University of California,   Los Angeles, CA 90095, USA}
\email{mbonk@math.ucla.edu}

\author{Mikhail Hlushchanka}
\address {Aix--Marseille Universit\'{e}, Institut de Math\'{e}atiques de Marseille, %163 Avenue de Luminy, 
13009 Marseille, France}

\address{Mathematisch Instituut, Universiteit Utrecht,
 %Postbus 80.010, 
 3508 TA Utrecht, The Netherlands}
\email{m.hlushchanka@uu.nl}

\author{Annina Iseli}
\address{Department of Mathematics, University of Fribourg,  
%Chemin du Muséee 23,
 1700 Fribourg, Switzerland}
\address {Department of Mathematics, University of California,   Los Angeles, CA 90095, USA}
\email{annina.iseli@math.ucla.edu}

\date{May 14, 2021}


\thanks{M.B.\ was partially supported by NSF grant DMS-1808856. M.H.\ was partially supported by the ERC advanced grant ``HOLOGRAM''. A.I.\ was partially supported by the Swiss National Science Foundation (project no.\ 181898).}

\begin{abstract} Every Thurston map $f\: S^2\ra S^2$ on a $2$-sphere $S^2$ induces a pull-back operation on Jordan curves $\alpha\sub S^2\setminus \postf$, where $\postf$ is the postcritical set of $f$.   Here the isotopy class $[f^{-1}(\alpha)]$ (relative to $\postf$) only depends on the isotopy class $[\alpha]$. We study this operation for Thurston maps with four postcritical points.  In this case a Thurston obstruction for the map $f$ 
can be seen as a fixed point of the pull-back operation. 

We show that if a Thurston map $f$ with a hyperbolic orbifold and four postcritical points has 
 a Thurston obstruction, then one can ``blow up" suitable arcs in the  underlying $2$-sphere and construct a  new Thurston map $\widehat f$ for which 
this obstruction  is eliminated. We prove that  no other obstruction arises 
and so $\widehat f$ is realized by a rational map.  In particular, this allows for the combinatorial construction of a large class of rational Thurston maps with 
four postcritical points. 

We also study the dynamics of the pull-back operation under iteration. We exhibit a subclass of our rational Thurston maps with four postcritical  points  for which we can give  positive answer to the global curve attractor problem. 
\end{abstract}
\maketitle
\setcounter{equation}{0}

\tableofcontents

\section{Introduction}
\label{sec:introduction}

In this paper, we consider Thurston maps and the dynamics 
of the induced pull-back operation on Jordan curves on the underlying $2$-sphere.  By definition, a  {\em Thurston map}  is a branched covering map $f\: S^2\ra S^2$ on a topological $2$-sphere $S^2$ such  that $f$ is not a homeo\-morphism and  every critical point of $f$ (points where $f$ is not a local homeomorphism)  
has a finite orbit under iteration of $f$. These maps  
are named after William Thurston who introduced them  in his quest for a better 
 understanding of   the dynamics of postcritically-finite rational maps  on the Riemann sphere.  We refer to \cite[Chapter 2]{THEBook} for general background on Thurston maps and related concepts.

 For a branched covering map $f\: S^2\ra S^2$ we denote by $\critf$ the set of critical points of $f$ and by 
$f^n$ the $n$-th iterate of $f$ for $n\in \N$. Then the postcritical set of $f$ is defined as 
$$ \postf=\bigcup_{n\in \N} \{ f^n(c): c\in \critf\}. $$
For a Thurston map $f$ this set has finite cardinality $2\le \#\postf<\infty$ (for the first inequality see \cite[Corollary 2.13]{THEBook}). 

A Thurston map $f$ often admits a description in purely combinatorial-topological terms. In this context,
 it is an interesting question whether $f$ can be {\em realized} (in a suitable sense) by a rational map with the same combinatorics.  Roughly speaking, this means that $f$ is conjugate to  a rational map ``up to isotopy" (see Section \ref{sec:thurston-maps} for the precise definition).

It is not hard to see  that each Thurston map with two or three postcritical  points is realized.  The situation  is much more complicated for Thurston maps $f$ with $\#\postf\ge 4$.   
William Thurston  found a necessary and sufficient condition when a Thurston map can be realized by a rational map \cite{DH_Th_char}. Namely, if $f$ has an associated hyperbolic orbifold 
 (this is always true apart from some well-understood exceptional  maps), then $f$ is realized if and only if $f$ has no 
 \emph{(Thurston) obstruction}.
 Such an obstruction is given by a finite collection of disjoint Jordan curves in $\Sp\setminus \postf$ (up to isotopy) with certain invariance properties  (see 
 Section~\ref{subsec:thurston-char} for more discussion).   
 
  The ``if part'' of this statement gives a positive criterion for $f$ to be realized, but it is  very hard to apply in practice, because, at least  in principle,   it involves the  verification of infinitely many conditions for the 
 map $f$.  For this reason, in each
individual case a successful verification  for a map, or a class of maps, is difficult and usually constitutes an interesting result in its own right. 

We mention two  results  in this direction.  The first one  is   the ``arcs intersecting obstructions'' theorem by   Kevin Pilgrim and Tan Lei \cite[Theorem 3.2]{PT} that gives  control on the position of an obstruction and has many applications in holomorphic dynamics (see, for instance, \cite{PT,Newton}).
The other one is the ``mating criterion'' by Tan Lei, Mary Rees, and Mitsuhiro Shishikura that addresses the question when two postcritically-finite quadratic polynomials can be topologically glued together  to form a rational map (see \cite{TanLeiMatings,Rees_Degree_2,ShishikuraMatings}).

 The investigation of obstructions of a Thurston map $f\: S^2\ra S^2$ is closely related to 
 the study of the \emph{pull-back operation} on Jordan curves. It is easy to show that if  $\alpha  \sub S^2 \setminus \postf$ is a Jordan curve, then the isotopy class $[f^{-1}(\alpha)]$ (rel.\ $\postf$) only depends on the isotopy class $[\alpha]$ (see Lemma~\ref{lem:pull}).  Intuitively, the number of postcritical points of a Thurston map can be seen as a measure of its combinatorial complexity.  In this paper, we focus on the simplest non-trivial case, namely Thurston maps $f$ with $\#\postf=4$. In this case the pull-back operation gives rise to a well-defined map, the {\em slope map}, on these isotopy classes $[\alpha]$ (we will discuss this in more detail below). The search for obstructions of $f$ amounts to understanding the fixed points of the slope map.
 
There exist various natural constructions that allow one to \emph{combine} or \emph{modify}  given (rational) Thurston maps to obtain  a new dynamical system. The most studied constructions are \emph{mating} (see  \cite{TanLeiMatings,Mating}), \emph{tuning} (see  \cite{Rees_Degree_2}), and  \emph{capture} (see \cite{Captures, LeiNewton}). 
In this paper, we study the operation of \emph{blowing up arcs}, originally introduced by Kevin Pilgrim and Tan Lei in \cite{PT}. This operation can  be applied to an arbitrary Thurston map $f$ and results in a new Thurston map $\widehat f$ that is of higher degree, but combinatorially closely related to the original map $f$.   In particular, $f$ and $\widehat f$  have the same set of postcritical points and the same dynamics on them. Nevertheless, the dynamical behavior of Jordan curves under the pull-back operation for the original map $f$ and the new  map $\widehat f$ may differ drastically.

We show that, if a Thurston map $f\:S^2\ra S^2$ with $\#\postf=4$ has  an obstruction $\alpha $, then one can naturally modify $f$ by blowing up certain arcs to produce a new Thurston map $\widehat f$ for which this  obstruction $\alpha $ is eliminated.   The main result of this paper is the fact that then no new obstructions arise for 
 $\widehat f$ and so it is realized by a rational map.

\begin{theorem}\label{thm:blow-up-obstr}
Let $f\:\Sp\to\Sp$ be a Thurston map with $\#\postf=4$ and a hyperbolic orbifold. Suppose that $f$ has an obstruction represented by a Jordan curve $\alpha\sub S^2\setminus \postf$, and $E\ne \emptyset$ is a finite set of arcs in $(\Sp,f^{-1}(\postf))$ that satisfy the $\alpha$-restricted blow-up conditions. 

Let $\widehat{f}$ be a Thurston map obtained from $f$ by blowing up arcs in $E$ (with some multiplicities) so that $\lambda_{\widehat{f}}(\alpha)<1$. Then $\widehat{f}$ is realized by a rational map.
\end{theorem}

The technical verbiage and the notation in this formulation will be explained in subsequent sections (see in particular  \eqref{eq:Thurst_coeff} for the definition of the ``eigenvalue" 
$\lambda_{\widehat{f}}(\alpha)$ and Definition~\ref{def:blowup_conditions} for  $\alpha$-restricted blow-up conditions).

  Recently, Dylan Thurston provided a positive characterization when a Thurston map is realized, at least in the case when each critical point eventually lands in a critical cycle under iteration.  He proved that such a Thurston map $f$ is realized by a rational map if and only if there is an ``elastic spine'' (that is, a planar embedded graph in $\Sp\setminus\postf$ with a suitable metric on it) that gets ``looser'' under backwards iteration (see \cite{DylanReport, Dylan_Positive} for more details).  In concrete cases, especially for   Thurston maps  that should be realized by rational maps with Julia sets homeomorphic to the Sierpi\'{n}ski carpet, the application of Dylan Thurston's criterion is not so straightforward.  Moreover,  his criterion is only valid for Thurston maps with periodic critical points.   In contrast, 
 for some  maps for which Dylan Thurston's 
 criterion  is  not applicable or hard to apply,  Theorem~\ref{thm:blow-up-obstr}  can  be used 
 to verify that the maps are  realized. In particular, many maps obtained  by blowing up Latt\`es maps (see below) are of this type.

 \subsection{Blowing up Latt\`es maps} \label{sec:blowupL} 
 We will now discuss a special case of Theorem~\ref{thm:blow-up-obstr} in   detail to give the reader 
 some intuition for  the  geometric ideas behind this statement and its proof.  
 
 Let $\PP$ be  a {\em pillow} obtained from two copies 
 of the unit square $[0,1]^2\sub \R^2\cong \C$ glued together along their boundaries. 
 We consider the two copies of $[0,1]^2$ in $\PP$ as the front and  back side  of $\PP$ and call them the {\em tiles of level $0$} or simply 
 $0$-{\em tiles}.  We  denote by $A\coloneq(0,0)\in \PP$ the lower left corner of $\PP$ (see the right part of Figure \ref{fig_Lattes}).
The pillow $\PP$ is a topological $2$-sphere. Actually, if we consider $\PP$ as an abstract  polyhedral surface, then $\PP$ carries a conformal structure making $\PP$  conformally equivalent to the Riemann sphere $\CDach$. See Section~\ref{subsec:pillow} for more discussion.

% \begin{center}
%\vspace{12pt}
\begin{figure}[t]
\def\svgwidth{0.70\columnwidth}
%% Creator: Inkscape 1.0.1 (c497b03c, 2020-09-10), www.inkscape.org
%% Accompanies image file 'pillow_checker.pdf' (pdf, eps, ps)
\begingroup%
  \makeatletter%
  \providecommand\color[2][]{%
    \errmessage{(Inkscape) Color is used for the text in Inkscape, but the package 'color.sty' is not loaded}%
    \renewcommand\color[2][]{}%
  }%
  \providecommand\transparent[1]{%
    \errmessage{(Inkscape) Transparency is used (non-zero) for the text in Inkscape, but the package 'transparent.sty' is not loaded}%
    \renewcommand\transparent[1]{}%
  }%
  \providecommand\rotatebox[2]{#2}%
  \newcommand*\fsize{\dimexpr\f@size pt\relax}%
  \newcommand*\lineheight[1]{\fontsize{\fsize}{#1\fsize}\selectfont}%
  \ifx\svgwidth\undefined%
    \setlength{\unitlength}{501.07694893bp}%
    \ifx\svgscale\undefined%
      \relax%
    \else%
      \setlength{\unitlength}{\unitlength * \real{\svgscale}}%
    \fi%
  \else%
    \setlength{\unitlength}{\svgwidth}%
  \fi%
  \global\let\svgwidth\undefined%
  \global\let\svgscale\undefined%
  \makeatother%
  \begin{picture}(1,0.35101893)%
    \lineheight{1}%
    \setlength\tabcolsep{0pt}%
    \put(0,0){\includegraphics[width=\unitlength,page=1]{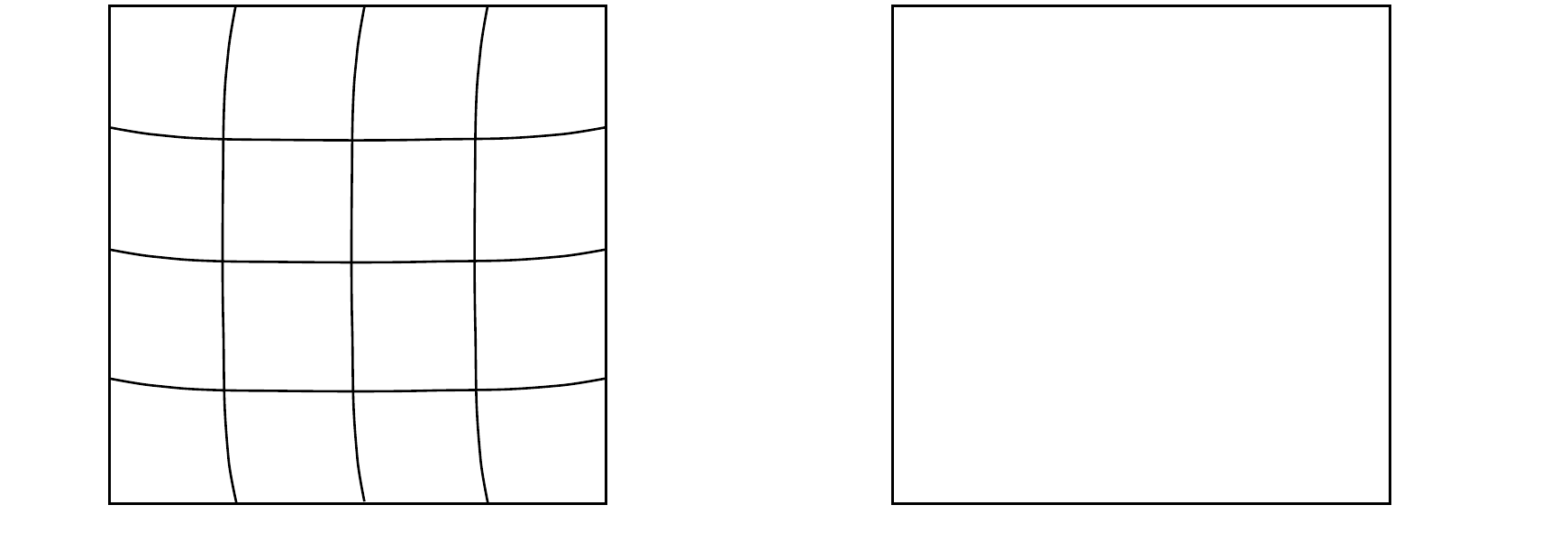}}%
    \put(0.45911869,0.2219748){\color[rgb]{0,0,0}\makebox(0,0)[lt]{\lineheight{1.25}\smash{\begin{tabular}[t]{l}$\LL_4$\end{tabular}}}}%
    \put(0,0){\includegraphics[width=\unitlength,page=2]{pillow_checker.pdf}}%
    \put(0.52600946,0.00547333){\color[rgb]{0,0,0}\makebox(0,0)[lt]{\lineheight{1.25}\smash{\begin{tabular}[t]{l}$A$\end{tabular}}}}%
    \put(0,0){\includegraphics[width=\unitlength,page=3]{pillow_checker.pdf}}%
    \put(0.02596719,0.00547333){\color[rgb]{0,0,0}\makebox(0,0)[lt]{\lineheight{1.25}\smash{\begin{tabular}[t]{l}$A$\end{tabular}}}}%
    \put(-0.00188748,0.18305839){\color[rgb]{0,0,0}\makebox(0,0)[lt]{\lineheight{1.25}\smash{\begin{tabular}[t]{l}$\PP$\end{tabular}}}}%
    \put(0.92854364,0.18305839){\color[rgb]{0,0,0}\makebox(0,0)[lt]{\lineheight{1.25}\smash{\begin{tabular}[t]{l}$\PP$\end{tabular}}}}%
  \end{picture}%
\endgroup%

\caption{The $(4\times 4)$-Latt\`es map.}\label{fig_Lattes}
\end{figure}
%\end{center}

We now fix $n\in \N$ with $n\ge 2$.  
We subdivide each of the two $0$-tiles of $\PP$ into $n^2$ small squares of sidelength $1/n$, called the $1$-{\em tiles}.
We color these $1$-tiles in a checkerboard fashion black and white so that the $1$-tile  in the front  $0$-tile  that contains the  vertex $A$ on its boundary is colored white (see the left part of Figure \ref{fig_Lattes}). We map this white $1$-tile to the front $0$-tile  of the right-hand pillow by an orientation-preserving Euclidean similarity  that fixes the vertex 
$A$. This similarity scales distances by the factor $n$.  We can uniquely extend the similarity by a successive Schwarz  reflection process to the whole pillow $\PP$ to obtain a continuous map $\La_n\: \PP\ra \PP$. Then 
on each $1$-tile $S$   the map $\La_n$ is a Euclidean similarity 
that sends $S$ to the   front or back $0$-tile  of $\PP$ depending on whether $S$ is  white  or black. 
We call  $\La_n$ the $(n\times n)$-{\em Latt\`{e}s map}, because 
under a suitable conformal equivalence $\PP\cong \CDach$, the map $\La_n$ is conjugate to a rational map 
obtained from $n$-multiplication of a  Weierstrass $\wp$-function.
See  Figure~\ref{fig_Lattes} for an illustration of the map $\La_4$.  Here, the marked points on the left  pillow $\PP$ (the domain of the map) correspond to the preimage points $\La_4^{-1}(A)$.  Note that there is exactly one preimage of $A$ in the interior of the back side of the pillow.

It is easy to see that the  $(n\times n)$-Latt\`{e}s map $\La_n\: \PP\ra \PP$ is  a Thurston map
with four postcritical points, namely, the four corners of the pillow $\PP$. The map $\La_n$ is realized by a  rational map, because it is even conjugate to such a map. 

We now modify the map $\La_n$ by gluing in vertical or horizontal {\em flaps} to $\PP$.
This is a special case of  the  more general construction of  blowing up arcs mentioned above. We will describe this in detail 
  in Section~\ref{sec:blow-up}, but will  illustrate 
  the procedure in  Figure~\ref{fig_flap}, where we  show how to glue in one horizontal flap.

   % \begin{center} \label{fig:1}
%\begin{figure}[h]
%\includegraphics[scale=0.10]{skizze6.jpg}
%\caption{Gluing in a flap.}\label{fig_flap}
%\end{figure}
%\end{center}

We cut the pillow $\PP$ open along a horizontal side $e$ of one of the $1$-tiles. Note that in this process $e$ is ``doubled" into two arcs $ e'$ and $ e''$ with common endpoints.   We then take two disjoint copies of the Euclidean square $[0,1/{n}]^2$ and  identify them along three corresponding sides to obtain a {\em flap} $F$. It has two ``free" 
sides on its boundary. We glue each free side to one of the arcs  $ e'$ and $ e''$ of the cut in the obvious way.  

 %\begin{center}
%\vspace{12pt}
\begin{figure}[t]
\def\svgwidth{0.90\columnwidth}
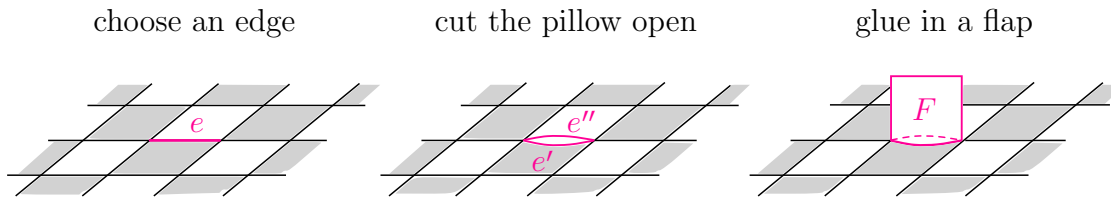
\caption{Gluing in a flap.}\label{fig_flap}
\end{figure}
%\end{center}

In general, one can repeat this construction and glue several flaps at the location given by the  arc $e$. We assume that  this has been done  simultaneously for $n_h\ge 0$ flaps along horizontal  edges and  $n_v\ge 0$ flaps along vertical 
edges. By this procedure  we obtain a ``flapped" pillow~$\Phat$, which is still a topological $2$-sphere (see the left part of Figure \ref{fig_Lattesblowup}). 
By construction it is  tiled by $2n^2+2(n_h+n_v)$ squares 
of sidelength $1/n$, which we consider as the $1$-tiles of  $\Phat$.  The  checkerboard coloring  of the base surface $\PP$ extends in a unique way to  the new surface $\widehat{\PP}$. The original  $(n\times n)$-Latt\`{e}s map 
$\La_n\: \PP\ra \PP$ can  naturally be ``extended" to a continuous map $\widehat{\La}\:\widehat{\PP}\to \PP$ so that each $1$-tile  $S$ 
of  $\widehat{\PP}$ is mapped to  the front or back $0$-tile  of $\PP$ (depending on the color of $S$)  by a Euclidean similarity scaling  distances by the factor $n$. 
See Figure~\ref{fig_Lattesblowup} for an illustration of a map 
$\widehat{\La}$ obtained from the Latt\`es map $\La_4$ by gluing in flaps at a vertical and a horizontal edge.   Similarly  as in Figure \ref{fig_Lattes}, on the left we marked  the preimages of $A$ under $\widehat \La$.

 %\begin{center}
%\vspace{12pt}
\begin{figure}[b]
\def\svgwidth{0.7\columnwidth}
%% Creator: Inkscape 1.0.1 (c497b03c, 2020-09-10), www.inkscape.org
%% Accompanies image file 'blowupwice_intro.pdf' (pdf, eps, ps)
\begingroup%
  \makeatletter%
  \providecommand\color[2][]{%
    \errmessage{(Inkscape) Color is used for the text in Inkscape, but the package 'color.sty' is not loaded}%
    \renewcommand\color[2][]{}%
  }%
  \providecommand\transparent[1]{%
    \errmessage{(Inkscape) Transparency is used (non-zero) for the text in Inkscape, but the package 'transparent.sty' is not loaded}%
    \renewcommand\transparent[1]{}%
  }%
  \providecommand\rotatebox[2]{#2}%
  \newcommand*\fsize{\dimexpr\f@size pt\relax}%
  \newcommand*\lineheight[1]{\fontsize{\fsize}{#1\fsize}\selectfont}%
  \ifx\svgwidth\undefined%
    \setlength{\unitlength}{497.53122298bp}%
    \ifx\svgscale\undefined%
      \relax%
    \else%
      \setlength{\unitlength}{\unitlength * \real{\svgscale}}%
    \fi%
  \else%
    \setlength{\unitlength}{\svgwidth}%
  \fi%
  \global\let\svgwidth\undefined%
  \global\let\svgscale\undefined%
  \makeatother%
  \begin{picture}(1,0.35321195)%
    \lineheight{1}%
    \setlength\tabcolsep{0pt}%
    \put(0,0){\includegraphics[width=\unitlength,page=1]{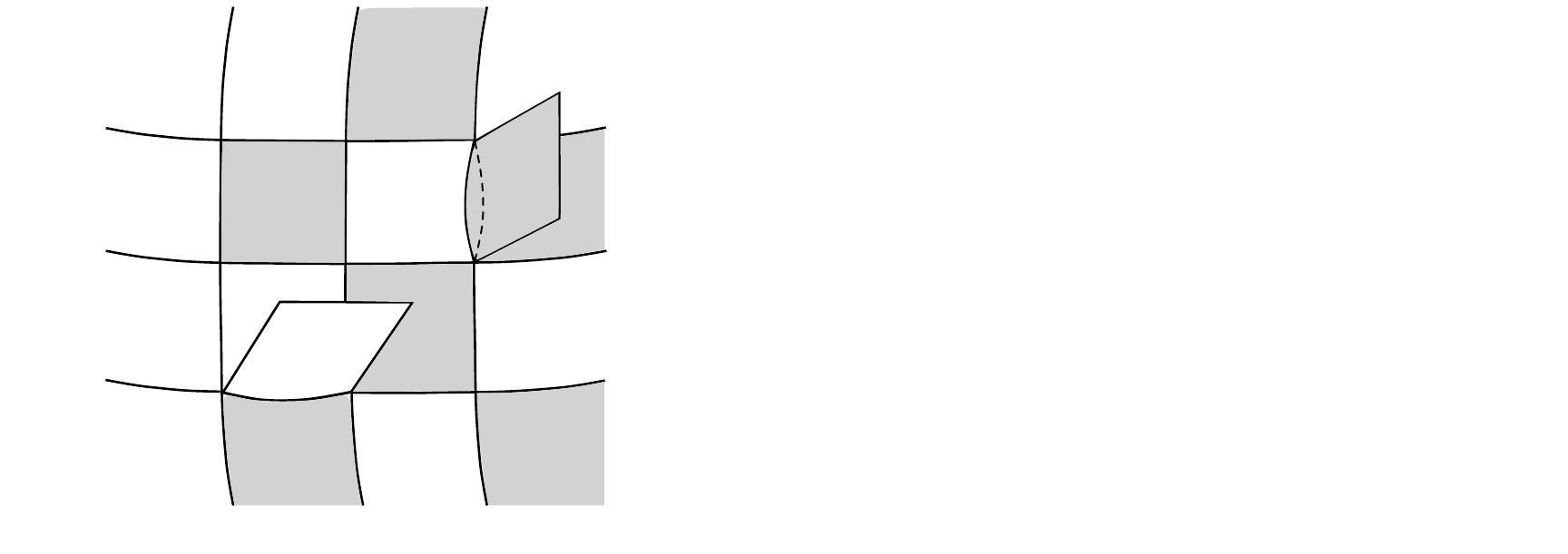}}%
    \put(0.46873106,0.22077917){\color[rgb]{0,0,0}\makebox(0,0)[lt]{\lineheight{1.25}\smash{\begin{tabular}[t]{l}$\widehat{\LL}$\end{tabular}}}}%
    \put(0,0){\includegraphics[width=\unitlength,page=2]{blowupwice_intro.pdf}}%
    \put(0.52334102,0.00551233){\color[rgb]{0,0,0}\makebox(0,0)[lt]{\lineheight{1.25}\smash{\begin{tabular}[t]{l}$A$\end{tabular}}}}%
    \put(0,0){\includegraphics[width=\unitlength,page=3]{blowupwice_intro.pdf}}%
    \put(0.02273029,0.00551233){\color[rgb]{0,0,0}\makebox(0,0)[lt]{\lineheight{1.25}\smash{\begin{tabular}[t]{l}$A$\end{tabular}}}}%
    \put(0,0){\includegraphics[width=\unitlength,page=4]{blowupwice_intro.pdf}}%
    \put(0.9280344,0.18035344){\color[rgb]{0,0,0}\makebox(0,0)[lt]{\lineheight{1.25}\smash{\begin{tabular}[t]{l}$\PP$\end{tabular}}}}%
    \put(-0.00190093,0.18290889){\color[rgb]{0,0,0}\makebox(0,0)[lt]{\lineheight{1.25}\smash{\begin{tabular}[t]{l}$\PPh$\end{tabular}}}}%
  \end{picture}%
\endgroup%

\caption{An example of a  map $\widehat{\La}$ obtained from $\La_4$ by gluing in flaps.}\label{fig_Lattesblowup}
\end{figure}
%\end{center}

 In order to obtain a Thurston map $f\: \PP\ra \PP$  from this construction, we need to choose a  homeomorphism $\phi\:  \widehat{\PP}\ra \PP$.
Roughly speaking, $\phi$ is a homeomorphism that  identifies $ \widehat{\PP}$ with $\PP$ and fixes each corner of the pillow.
The precise choice of
$\phi$ is  somewhat technical and so we refer to Section~\ref{subsec:blownup-lattes} for the details.  
 Then, one easily observes that the postcritical set $\postf$ consists of the four corners of the pillow $\PP$. 
 The map $f$  is uniquely determined up to Thurston equivalence 
 (see Definition~\ref{def:Thurston-equiv-and-realized})   independently of the choice of $\phi$ under suitable  restrictions. 
 We refer to $f$  as a Thurston map {\em obtained from the
 $(n\times n)$-Latt\`es  map  by gluing $n_h$ horizontal 
 and $n_v$   vertical flaps to $\PP$}.

 Now the following statement is true. As we will explain, it   can be seen as a special case of our main result.

\begin{theorem}\label{thm:flapped_intro} Let $n\in \N$ with 
$n\ge 2$ and  $f\colon \PP\to \PP$ be a Thurston map obtained  from the $(n\times n)$-Latt\`es map $\La_n$ by gluing $n_h\ge 0$ horizontal and $n_v\ge 0$ vertical flaps to $\PP$, where  $n_h+n_v>0$. Then the map $f$ has a hyperbolic orbifold. It 
 has an obstruction if and only if $n_h=0$ or $n_v=0$.
In particular, if $n_h>0$ and $n_v>0$, then $f$ is realized by a rational map.  
\end{theorem}

If $n_h=n_v=0$, then no flaps were glued to $\PP$ and the map $f$
coincides with the original $(n\times n)$-Latt\`es map $\La_n$ (strictly speaking, only if we choose the homeomorphism $\phi$ used in the construction above to be the identity on $\PP$, as we may). 
Then $f=\La_n$ has a parabolic orbifold.
Therefore, Thurston's criterion  as formulated in Section~\ref{subsec:thurston-char} does not  apply. 

If  $n_h=0$ or $n_v=0$, but $n_h+n_v>0$ as in Theorem~\ref{thm:flapped_intro}, then it is immediate  to see that $f$ has an   obstruction (see Section~\ref{subsec:obstructed pillow})  and therefore $f$ cannot be realized by  a rational map.  So the interesting part of Theorem~\ref{thm:flapped_intro} is the claim  that if $n_h>0$ and $n_v> 0$, then $f$ has no obstruction. 

Even though Theorem~\ref{thm:flapped_intro} follows from our more  general statement formulated in Theorem~\ref{thm:blow-up-obstr}, we will give  a complete proof. We will argue by contradiction and assume  that a map 
$f$ with $n_h>0$ and $n_v> 0$ has an obstruction. 
In principle, there are infinitely many candidates represented by 
essential isotopy classes of Jordan curves $\alpha \sub \PP\setminus \postf$. These isotopy classes in turn are distinguished by different  slopes in $\widehat \Q=\Q\cup\{\infty\}$ (as will be explained in Section~\ref{subsec:isotopies}). For such an isotopy class represented by $\alpha $ to be an  obstruction, it has to be $f$-invariant 
in the sense that $f^{-1}(\alpha )$ should contain a component $\widetilde \alpha $ isotopic to $\alpha $ rel.\ $\postf$. It seems to be a very intricate problem to find all slopes in $\widehat \Q$ that give an 
invariant isotopy class for $f$. Since we have been able to decide this question only for very  simple maps $f$, we proceed 
in a more indirect manner. 

We assume that the Jordan curve $\alpha  \sub \PP\setminus \postf$ is $f$-invariant and  gives  an   obstruction. We then   
investigate  the mapping degrees of $f$ on components of $f^{-1}(\alpha )$ and   consider intersection numbers of some relevant curves together with a careful counting  argument. We  heavily use the fact  that the {\em horizontal}  and {\em vertical} Jordan curves  (see \eqref{eq:hori+vert}) are $f$-invariant. 
Ultimately, we arrive at a contradiction. See Section~\ref{sec:realizing-blownup-lattes} for the details of this argument.  

Our idea to use intersection numbers (as in  Lemma~\ref{lem:preimage-bound}) to control possible locations of obstructions and dynamics  on curves is not new (see, for example,  \cite[Theorem 3.2]{PT}, \cite[Section 8]{WanderingCurves}, and \cite{ParrySlope}). However, the previously available results  do not provide sharp enough estimates applicable in our situation.

One can think of Theorem~\ref{thm:flapped_intro} in the following way. Suppose that instead of directly passing from the Latt\`es map $\La_n$ to a map, let us now call it  $\widehat f$, obtained  by gluing $n_h>0$ horizontal 
and $n_v>0$ vertical flaps to $\PP$, we first create an intermediate map $f$ obtained by gluing $n_h>0$ horizontal, but  no vertical flaps. Then $f$ has a hyperbolic orbifold and an  obstruction given  by a ``horizontal" Jordan curve $\alpha$. In the passage from $f$ to $\widehat f$ we kill this obstruction, because we glue additional vertical flaps that serve as obstacles and increase the mapping degree on some pullbacks of $\alpha$.  Theorem~\ref{thm:blow-up-obstr} then
says that no other obstructions arise for 
$\widehat f$. Therefore, Theorem~\ref{thm:blow-up-obstr}  generalizes Theorem~\ref{thm:flapped_intro} if we interpret it in the way just described. The proof of Theorem~\ref{thm:blow-up-obstr} is based on a  refinement and extension  of the ideas that we use to establish Theorem~\ref{thm:flapped_intro}. 
The proof of Theorem~\ref{thm:blow-up-obstr} is based on the ideas that we use to establish Theorem~\ref{thm:flapped_intro},  but substantial refinements and extensions
are required.

\subsection{The global curve attractor problem} \label{subsec:intro_attractor}
The mapping properties of Jordan curves play an important role in Thurston's characterization of rational maps. 
The original proof of this statement associates with a given 
Thurston map $f\:\Sp\to\Sp$ with a hyperbolic orbifold a certain Teichm\"{u}ller space $\TC_f$ and a real-analytic map $\sigma_f\: \TC_f\to\TC_f$, called \emph{Thurston's pullback map}. One can  show that  the map $f$ is realized by a rational map if and only if 
$\sigma_f$ has a fixed point \cite{DH_Th_char}.  This reduction to a fixed point problem in a Teich\-m\"uller space has also been successfully  applied by Thurston in the other contexts such as uniformization problems and the theory of $3$-manifoldsDH
(there is a rich literature on the subject; see, for example,  \cite{ThurstonSurfaces, Thurston3Manifolds, OtalHyperbolization,  SurfacesHomeo,  HubbardBook2}).  

In recent years, the pullback map $\sigma_f$ and its dynamical properties  have been subject to deeper investigations (see, for example, \cite{Pilgrim_Alg_Th,Selinger,Lodge_Boundary,Pullback}). In particular, Nikita Selinger showed in \cite{Selinger} that $\sigma_f$ extends to the \emph{Weil-Petersson boundary} of $\TC_f$. 
The behavior of $\sigma_f$ on this boundary is closely related to the 
behavior of Jordan curves under  pull-back by $f$. 
This in turn leads  to the following difficult open question in holomorphic dynamics, called the \emph{global curve attractor problem} (see \cite[Section 9]{Lodge_Boundary}).

 %. It appeared in a lecture by Kevin Pilgrim [25], and later in a paper by Russell Lodge 

\begin{conj*}\label{conj:Curve_Attr}
Let $f\: S^2\to S^2$ be a Thurston map  with a hyperbolic orbifold that is realized by a rational map. Then there exists a finite set $\mathscr{A}(f)$ of  Jordan curves in $S^2\setminus\postf$ such that for every   Jordan  curve $\gamma\sub  S^2\setminus\postf$,  all  pullbacks $\widetilde \gamma$ of $\gamma$ under $f^n$ are contained in $\mathscr{A}(f)$ up to isotopy 
rel.\ $\postf$  for all sufficiently large $n\in\N$.
\end{conj*}

A  set of Jordan curves $\mathscr{A}(f)$ as in this conjecture  is called a \emph{global curve attractor} of $f$.  We will give a solution of this  problem for  maps as in Theorem~\ref{thm:flapped_intro} with $n=2$ and $n_h, n_v\ge 1$.  
 Unfortunately, our methods only apply for $n=2$ and not for $n\ge 3$.

\begin{theorem}\label{thm:finite-curve-attr1}
 Let $f\: \PP\ra \PP$ be a Thurston map obtained from the  $(2\times 2)$-Latt\`{e}s map by gluing $n_h\geq 1$ horizontal and $n_v\geq 1$ vertical flaps to the pillow $\PP$. Then $f$ has a global curve attractor $\mathscr{A}(f)$. \end{theorem}

One can show that the Julia set of a rational map as provided by Theorem~\ref{thm:flapped_intro} 
is either  a Sierpi\'nski carpet or the  whole Riemann sphere depending on whether the map has periodic critical points or not  (see  Proposition~\ref{prop:Julia}). 
Accordingly,  Theorem~\ref{thm:finite-curve-attr1} provides the first examples of maps with  Sierpi\'{n}ski carpet Julia set for which an answer to the global curve attractor problem is known. In fact, we obtain  such maps  with arbitrary large degrees.

Recently, Belk-Lanier-Margalit-Winarski  proved the existence of a finite global curve attractor for  all postcritically-finite polynomials \cite{LifitingTrees}. The conjecture is also  known to be true for all \emph{critically fixed rational maps} (that is, rational maps for which each critical point is fixed) and some \emph{nearly Euclidean Thurston maps} (that is, Thurston maps with exactly four postcritical points and only simple critical points); see \cite{H_Tischler} and \cite{Lodge_Boundary,Origami}. In \cite{KelseyLodge} Gregory Kelsey and Russell Lodge verified the conjecture for all quadratic non-Latt\`{e}s maps with four postcritical points. However, for general postcritcally-finite rational maps the conjecture remains wide open.

%\textcolor{red}{There is a recent note by Kevin on the pullback map, which we should probably cite as well (not sure how to do it properly. https://www.semanticscholar.org/paper/On-the-pullback-relation-on-curves-induced-by-a-map-Pilgrim/14f4c4b902435b5e9d2f4b6905faeecb5c99d7fa?p2df}
%

%In the special case when $\#\postf=4$, we can describe the (essential) Jordan curves in $\Sp$ by their \emph{slopes} (that is, rational numbers) and the extension of $\sigma_f$ to the boundary may be then described by the \emph{slope map} $\mu_f$, which describes the dynamics of Jordan curves under the pullback. 
%The methods in this paper also allow us to make some partial progress 
%on a  difficult problem in rational dynamics, namely  the question whether 
%every  postcritically-finite rational map with a hyperbolic orbifold has 
%a {\em global curve attractor}  for the pullback operation  of essential Jordan curves.   

Since the maps we consider have four postcritical points, it is convenient to reformulate the  global curve attractor problem by introducing 
the {\em slope map} (it is closely related to the Thurston pull-back map $\sigma_f$ on the Weil-Petersson boundary of $\TC_f$).  To define it in the special case relevant for us, 
we  consider the {\em marked} pillow $(\PP,V)$, where  $V$ is  the set consisting of the four corners of $\PP$, and  assume that $f\:\PP\to \PP$ is a Thurston map with $\postf=V$. 
Up to topological conjugacy, every Thurston map with four postcritical points can be assumed to have this form. 

As we already mentioned, there is a bijective correspondence 
between isotopy classes $[\alpha]$ of essential Jordan curves $\alpha$ in $(\PP,V)$ and slopes $r/s\in  \widehat{\Q}$ (see Lemma~\ref{lem:isoclassesP}). We introduce the additional symbol  $\odot$ to represent peripheral Jordan curves in $(\PP,V)$. 
We now define the {\em slope map} $\mu_f\: 
\widehat{\Q}\cup \{\odot\} \ra  \widehat{\Q}\cup \{\odot\} $ associated with $f$ as follows. We set $\mu_f(\odot)\coloneq \odot$. This corresponds to the fact that each pullback  of a peripheral Jordan curve $\alpha$ in $(\PP,V)$ under $f$ is peripheral (see Corollary~\ref{cor:pull}~\ref{pull1}).  
If $r/s\in  \widehat{\Q}$ is an arbitrary slope, then we choose a Jordan curve $\alpha $ in  $(\PP,V)$ whose isotopy class $[\alpha]$ 
is represented by $r/s$. If all pullbacks of $\alpha$ under $f$ are peripheral, we set $\mu_f(r/s)\coloneq \odot$. Otherwise,
there exists an essential pullback $\widetilde \alpha$  
of $\alpha$ under $f$. Then the isotopy class  $[\widetilde \alpha]$
is independent of the choice of the essential pullback $\widetilde \alpha$ (see Corollary~\ref{cor:pull}~\ref{pull2}) and so it is represented by a unique slope
$r'/s'\in  \widehat{\Q}$. In this case, we set $\mu_f(r/s)\coloneq r'/s'$. In this way,  $\mu_f(x)\in  \widehat{\Q}\cup \{\odot\}$ is defined for all 
$x\in   \widehat{\Q}\cup \{\odot\}$.
Since the map $\mu_f$ has the same source and target, we can iterate it. If $n\in \N_0$, then we denote by $\mu^n_f$ the $n$-th iterate of $\mu_f$.  We will then prove  the following statement.

\begin{theorem}\label{thm:finite-curve-attr2}
 Let $f\: \PP\ra \PP$ be a Thurston map obtained from the  $(2\times 2)$-Latt\`{e}s map by gluing $n_h\geq 1$ horizontal and $n_v\geq 1$ vertical flaps to the pillow $\PP$. Then there exists   a finite set 
 $S\sub  \widehat{\Q}\cup \{\odot\}$ with the following 
 property: for each $x\in  \widehat{\Q}\cup \{\odot\}$ there exists
$N\in \N_0$  such that 
$  \mu^n_f(x)\in S$ for all $n\ge N$.\end{theorem}

Note that $\postf=V$ in this case; so our previous considerations apply and the map $\mu_f$ is defined. It is clear that the previous theorem leads to the solution of the global curve attractor problem for the maps $f$ considered:

\begin{proof}[Proof of Theorem~\ref{thm:finite-curve-attr1}
based on Theorem~\ref{thm:finite-curve-attr2}]
To obtain a finite attractor $\mathscr{A}(f)$, pick a Jordan curve in  each isotopy 
class represented by a slope in $S$ and add five Jordan curves 
that represent the  isotopy classes of peripheral Jordan curves in 
$(\PP,V)$ (one for null-homotopic curves and one for each corner of $\PP$).  
\end{proof}

For the proof of Theorem~\ref{thm:finite-curve-attr2} we will establish 
a certain monotonicity property of  the slope map $\mu_f$ for a  map $f$ as in the statement (see Proposition~\ref{prop:complexity-decreases}). Roughly speaking, this monotonicity means 
 that  up to isotopy 
rel.\   $\postf=V$ complicated  essential Jordan curves 
in $(\PP, V)$ get ``simpler" and ``less twisted" if we take 
 successive preimages under $f$  and eventually end up in the global curve attractor.

Our  methods again rely   on the consideration of intersection numbers.   The algebraic methods for solving the global curve attractor problem developed in \cite{Pilgrim_Alg_Th} (specifically,  \cite[Theorem 1.4] {Pilgrim_Alg_Th}) do  not apply  in general  for the maps considered in Theorem \ref{thm:finite-curve-attr1} (see the discussion in Section \ref{subsec:twists}). 

Some of our ideas  can also be used  for the study of the global dynamics of the slope map for 
Thurston maps  
that are not covered by  Theorem~\ref{thm:finite-curve-attr2}.
In particular, we are able to describe the iterative behavior of $\mu_f$ for a specific  obstructed Thurston map $f$ obtained by blowing up the $(2\times 2)$-Latt\`{e}s map (see Section \ref{sec:numerics} for the details). This provides an answer to a question by Kevin Pilgrim (see \cite[Question~4.4]{Pilgrim_Pullback}).

While it is straightforward to compute $\mu_f(x)$ for individual values $x\in  \widehat{\Q}\cup \{\odot\}$, we have been unable 
  to give an  explicit formula for  $\mu_f$  for the maps  $f$ we consider.  In general, these slope maps show very complicated behavior. Currently, very few explicit computations of slope maps  are known in the literature. Except for some very special situations (for example, when the slope map is constant, that is, when $\mu_f(x)=\odot$ for all $x\in  \widehat{\Q}\cup \{\odot\}$), we are only aware of  computations of slope maps  for nearly Euclidian Thurston maps in  \cite[Section 5]{Cannon_Nearly} and 
   \cite[Section 6]{Lodge_Boundary}.  See also \cite{NET_Decidability} for some general properties of the slope map $\mu_f$.

An undergraduate student at UCLA, Darragh Glynn, performed some computer experiments to compute 
$\mu_f$ for maps $f$ as in Theorem~\ref{thm:flapped_intro}  
for $n\ge 3$ (and $n_h,n_v\ge 1$ corresponding to the rational case). His results show that in these cases the map $\mu_f$ does not have the monotonicity property as for $n=2$,  but indicate that 
these maps $f$ still have a global curve attractor (see Section~\ref{sec:numerics} for more discussion).

\subsection{Organization of this paper} 
Our paper is organized as follows. In the next two  sections we review some background.  In Section 2, we fix notation and
state some basic definitions. We also 
  discuss  isotopy classes of Jordan curves  in  spheres with four marked points, how isotopy classes of such curves correspond to slopes in $\widehat \Q$, as well as   some relevant facts about intersection numbers.  Even though all of this is fairly standard, we give complete  proofs in the appendix,  because it is hard to track down this material in the literature with a detailed  exposition. 

 In Section~\ref{sec:thurston-maps} we  recall  some basics about Thurston maps and the relevant concepts for a precise formulation of  Thurston's characterization of rational maps for Thurston maps with four postcritical points---the only case relevant for us  (see Section~\ref{subsec:thurston-char}).  

We explain the blow-up procedure for arcs in  Section~\ref{sec:blow-up} and relate this to the procedure of gluing flaps to the pillow $\PP$ (see 
Section~\ref{subsec:blownup-lattes}). The proof of Theorem~\ref{thm:flapped_intro} is then given in 
Section~\ref{sec:realizing-blownup-lattes}. 

The proof of our main result, Theorem~\ref{thm:blow-up-obstr}, requires more preparation. This is the purpose of 
Section~\ref{sec:esscirc}. There we introduce the concept of {\em essential circuit length} that will allow us to formulate tight estimates for the number of essential pullbacks of a Jordan curve 
 under a Thurston map with four postcritical points. This is formulated in the rather technical Lemma~\ref{lem:preimage-bound-refined} which is of crucial importance 
though.  The proof of Theorem~\ref{thm:blow-up-obstr} is then given in Section~\ref{sec:elim-obstr}. 

 Section~\ref{sec:attractor} is devoted to the proof of Theorem \ref{thm:finite-curve-attr2}. In 
 Section~\ref{sec:Further_discussion},  we discuss some further directions related to this work.   As we already mentioned, the appendix is devoted to the discussion of isotopy classes and intersection numbers of Jordan curves in spheres with four marked points.

\medskip\noindent
{\bf Acknowledgments.} The authors would like to thank Kostya Drach, Dima Dudko,  Daniel Meyer, Kevin Pilgrim, and Dylan Thurston for various useful comments and remarks.  We  are grateful to  Darragh Glynn for allowing us to incorporate some of his numerical findings  in this paper. 

\section{Preliminaries} \label{sec:prelim}

In this section, we discuss  background relevant  for the rest of the paper. 

\subsection{Notation and basic concepts}
\label{subsec:basics}

We denote by $\N=\{1,2,\dots\}$ the set of natural numbers and  by 
$\N_0=\{0,1, 2, \dots \}$
the set of natural numbers
including $0$. 
The sets 
of integers, real numbers, and complex numbers are denoted by $\Z$, 
 $\R$, and $\C$, respectively. We write $i$ for  the imaginary 
unit in $\C$, and $\text{Im}(z)$ for the imaginary part of a complex number $z\in \C$.

   As usual, 
 $\R^2\coloneq \{(x,y):x,y\in \R\}$ is the Euclidean plane and $\CDach\coloneq \C\cup\{\infty\}$  is
the Riemann sphere. Here and elsewhere, we write $A\coloneq B$
for emphasis when an object $A$ is defined to be another object $B$.  When we consider two objects $A$ and $B$, and there is a natural
identification between them that is clear from the context, we
write 
$A\cong B$.\index{$\cong$} 
For example, $\R^2\cong \C$ if we identify a point $(x,y)\in \R^2$ with $x+ iy  \in \C$.
  We will freely switch back and forth between  these different 
  viewpoints of $\R^2\cong \C$.  
    
   We use the notation  $\I \coloneq [0,1]\subset\R$ for  the closed unit interval,  $\D\coloneq  \{z\in\C: |z|<1\}$ for the open unit disk in $\C$,  and $\Z^2\coloneq \{ x+iy: x,y\in \Z\}$ for  the square lattice in $\C$. 
If $z,w\in \C$, then we write $[z,w]\coloneq \{ z+t(w-z): t\in \I\}$ for the line segment in $\C$ joining $z$ and $w$.  We also use the notation $[z_0,w_0)\coloneq [z_0, w_0]\setminus\{w_0\}$ and $(z_0, w_0)\coloneq
[z_0, w_0]\setminus \{z_0, w_0\}$.

The cardinality of a set $X$ is denoted by $\#X\in \N_0\cup\{\infty\}$ and the identity map on $X$ by $\id_X$.  If $X$ is a topological space and $M\sub X$, then $\cl(M)$ denotes the closure, $\inte(M)$ the interior, and $\partial M$ the boundary  of $M$ in $X$. 

Let $f\colon X\to Y$ be a map between sets $X$ and $Y$. If
$M\sub  X$, then 
$f|M$
stands for the restriction of $f$ to $M$. If $N\sub Y$, then 
$f^{-1}(N)\coloneq \{x\in X : f(x)\in N\}$ is the preimage of
$N$ in $X$. Similarly,  $f^{-1}(y)\coloneq
\{x\in X : f(x)=y\}$ is the preimage of a point $y\in Y$. 

Let $f\: X\ra X$ be a map.  For $n\in \N$ we denote by 
\[f^n\coloneq \underbrace{f\circ \dots \circ f}_{\text{$n$ factors} }\] the $n$-th iterate of $f$. It is convenient to define $f^0\coloneq \id_X$.  For $n\in \N_0$ we denote by   $f^{-n}(M)\coloneq \{x\in X
: f^n(x) \in M\}$ and $f^{-n}(p)\coloneq \{x\in X : f^n(x)=p\}$ the preimages of a set $M\subset X$ and a point $p\in X$ under  $f^n$, respectively.

A {\em surface} $S$ is a connected and oriented topological $2$-manifold. We denote its {\em Euler characteristic} by $\chi(S)$. 
Note that $\chi(S)\in \{2, 1,0, -1, \dots\} \cup \{-\infty\}$. 
Throughout this paper, we use the notation $\Sp$ for a {\em (topological) $2$-sphere}, that is, $\Sp$ indicates a surface homeomorphic to the Riemann sphere~$\CDach$. An 
{\em annulus} is a surface homeomorphic to $\{ z\in \C: 1<|z|<2\}$.

A \emph{Jordan curve $\alpha$}  in a surface $S$  is the image $\alpha =\eta(\partial
\D)$ of a (topological)  embedding $\eta\: \partial \D \ra S$ of 
 the unit circle $\partial\D=\{z\in\C: |z|=1\}$ into  $S$.
An \emph{arc $e$} in~$S$ is the image $e=\iota(\I)$ of an 
embedding $\iota\:\I\to S$.  Then $\iota(0)$ and $\iota(1)$ are the {\em endpoints} of $e$, and we define  $\partial e \coloneq 
\{\iota(0),\iota(1)\}$. The set  $\inter(e)\coloneq e \setminus \partial e$ is called the \emph{interior} of $e$. The notions of endpoints and interior of $e$ only depend on $e$ and not on the choice of the embedding $\iota$.  Note that the notation $\partial e$ and 
$\inte(e)$ is ambiguous, because it should not be confused with the boundary and interior of $e$ as a subset of $S$. For arcs $e$ in a surface $S$, we will only use $\partial e$ and 
$\inte(e)$ with the meaning just defined. 

A subset $U$ of  a surface $S$ is called an {\em open} or {\em closed  Jordan region} if there exists a topological  embedding 
$\eta\: \cl(\D)=\{z\in \C: |z|\le 1\}\ra S$ such that $U=\eta(\D)$ or  
$U=\eta(\cl(\D))$, respectively. 
In both cases, $\partial U=\eta(\partial \D)$ is a Jordan curve in $S$. 
A {\em crosscut} $e$ in an open or closed Jordan region $U$ is an arc 
$e\sub \cl(U)$ such that $\inte(e)\sub \inte(U)$ and $\partial e\sub \partial U$.

A {\em path} $\ga$ in a surface $S$ is a continuous map $\ga\: [a,b]\ra S$, where $[a,b]\sub \R$ is a compact (non-degenerate) interval. As is common, we will use the same notation $\ga$ for the image $\ga([a,b])$ of the path if no confusion can arise. 
The path $\ga$ {\em joins} two sets $M,N\sub S$ if $\ga(a) \in M$ and $\ga(b)\in N$, or vice versa. 
 A {\em loop} in $S$  {\em based at $p\in S$} is a path 
$\ga\: [a,b] \ra S$ such that $\ga(a)=\ga(b)=p$. The loop $\ga$ is called {\em simple} if  $\ga$ is injective on $[a,b)$. So essentially a simple loop is a Jordan curve run through with some parametrization.    

 Let $M,N,K$ be subsets of a surface  $S$. We say  that {\em $K$ separates $M$ and $N$} if every path in $S$ joining 
  $M$ and $N$  meets  $K$. Note that here $K$ is not necessarily disjoint from $M$ or $N$. We say that $K$ separates a point 
  $p\in S$ from a set $M\sub S$ if $K$ separates $\{p\}$ and $M$.

Let  $Z\sub S$ be  a finite set of points in a surface $S$.  Then we refer to the  pair $(S, Z)$  as a \emph{marked surface}, and the points in $Z$ as the  \emph{marked points} in $S$. The most important  case for us will be when $S=S^2$ is a $2$-sphere and $Z\sub S^2$ consists of four points.

 A \emph{Jordan curve $\alpha $ in a marked surface $(S,Z)$} is a Jordan curve $\alpha \sub S\setminus Z$. An \emph{arc $e$ in $(S,Z)$} is an arc $e\sub S$ with 
$\partial e\sub Z$ and  $\inter(e)\sub S\setminus Z$. We say that a Jordan curve $\alpha$ in a marked sphere $(\Sp, Z)$ is  \emph{essential} if each of the two connected components of $\Sp\setminus \alpha$ contains at least two points of $Z$; otherwise,  we say that  $\alpha$ is \emph{peripheral}. 

Let $(\Sp,Z)$ be a marked sphere with $\#Z = 4$. A \emph{core arc} of an essential Jordan curve $\alpha$ in $(\Sp,Z)$ is an arc in $(\Sp,Z)$ that is contained in one of the two  connected 
components of $\Sp\setminus \alpha$ and joins the two points in $Z$ that lie in this  component. 

Let $A$ be an annulus. Then a {\em core curve}  of $A$ is a Jordan curve $\beta\sub A$ such that under some homeomorphism $\varphi\: A\ra A'$ the curve $\beta'=\varphi(\beta)$ separates the boundary components of $A'=\{z\in \C: 1<|z|<2\}$.

\subsection{Branched covering maps}
\label{subsec:brcovmaps}
Let $X$ and $Y$ be  surfaces. Then a continuous map $f\colon X\to Y$ is called a {\em branched covering map} if for each point 
$q\in Y$ there exists an open set  $V\sub Y$ homeomorphic to $\D$ with $q\in V$ that is {\em evenly covered} in the following sense: for some index set $J\ne \emptyset$ we can write $f^{-1}(V)$ as a disjoint union 
\begin{equation}\label {eq:evencov}
 f^{-1}(V)= \bigcup_{j\in J}U_j
 \end{equation} 
of open sets $U_j\sub X$ such that $U_j$ contains precisely one point $p_j\in f^{-1}(q)$. Moreover, we require that for each
 $j\in J$ there exists $d_j\in \N$ and  orientation-preserving 
homeomorphisms $\varphi_j\: U_j\ra \D$ with $\varphi_j(p_j)=0$ and   $\psi_j\: V\ra \D$ with $\psi_j(q)=0$ such that 
$$ (\psi_j \circ f \circ \varphi_j^{-1})(z)=z^{d_j}$$ 
for all $z\in \D$ (see \cite[Section A.6]{THEBook} for more background on branched covering maps). For given $f$, the number 
$d_j$ is uniquely determined by $p=p_j$, and called the {\em local degree} of $f$ at $p$ and denoted by $\deg(f,p)$. A point $p\in X$ with $\deg(f,p) \geq 2$ is called a \emph{critical point} of $f$. The set of all critical points of $f$ is a discrete set in $X$  and denoted by $\critf$. 
If  $f$ is a branched covering map, then it is a covering map 
 (in the usual sense) from $X\setminus
f^{-1}\big(f(\critf)\big)$ onto $Y\setminus f(\critf)$.

In the following,  suppose $X$ and $Y$ are compact surfaces, and  $f\: X\ra Y$ is a branched covering map.  Then $\critf\sub X$ is a finite set. Moreover, if  
 $\deg(f)\in \N$ denotes  the topological degree of $f$, then $$\sum_{p\in f^{-1}(q)} \deg(f,p)=\deg(f)$$
for each  $q\in Y$.

If $\ga\:[a,b]\ra Y$ is a path, then we call a path $\widetilde \ga\: [a,b]\ra X$ 
a {\em lift} of $\ga$ (under $f$) if $f\circ  \widetilde \ga=\ga$. 
Every path $\ga$ in $Y$ has a lift $\widetilde \ga$ in $X$ (see \cite[Lemma A.18]{THEBook}), but in general $\widetilde \ga$ is not unique. If $\ga([a,b))\sub 
Y\setminus f(\critf)$ and   $x_0\in f^{-1}(\ga(a))$, then there exists a unique 
lift  $\widetilde \ga\: [a,b]\ra X$ of $\ga$ under $f$ with $\widetilde \ga(a)=x_0$. This easily follows from standard existence and uniqueness theorems for lifts under covering maps (see \cite[Lemma A.6]{THEBook}). 

If $e\sub Y$ is an arc, then an arc $\widetilde e\sub X$ is called 
a {\em lift} of $e$ (under $f$) if $f|\widetilde e$ is a homeomorphism 
of $\widetilde e$ onto $e$. It easily follows from the existence and uniqueness statements for lifts of paths just discussed that if 
$e$ is an arc in $(Y, f(\critf))$, $y_0\in \inte(e)$, and $x_0\in f^{-1}(y_0)$, then there exists a unique lift $\widetilde e\sub X$ of $e$ with $x_0\in \widetilde e$.

Let $V\sub Y$ be an open and connected set and $U\sub f^{-1}(V)$ be a (connected) component of $f^{-1}(V)$. Then $f|U\: U \ra V$ is also a branched covering map. Each point $q\in V$ has the same number $d\in \N$ of preimages under $f|U$ counting local degrees.  We set $\deg(f|U)\coloneq d$. 
If the Euler characteristic $\chi(V)$ is finite, then  $\chi(U)$ is also 
finite and we have the 
{\em Riemann-Hurwitz formula}
\begin{equation}\label{eq:RH}
\chi(U)+\sum_{p\in U \cap \critf}(\deg(f,p)-1)=\deg(f|U)\cdot \chi(V). 
\end{equation}

\subsection{Planar embedded graphs}
\label{subsec:graphs}

 A \emph{planar embedded graph} in a sphere $S^2$ is a pair $G=(V,E)$, where $V$ is a finite set of points in  $\Sp$ and $E$ is a finite set of arcs in $(\Sp,V)$ with pairwise disjoint interiors. The sets $V$ and $E$ are called the \emph{vertex} and \emph{edge sets} of $G$, respectively. Note that our notion of a planar embedded graph does not allow \emph{loops}, that is, edges that connect a vertex to itself, but it does allow \emph{multiple edges}, that is, distinct edges that join the same pair of vertices. The \emph{degree} of a vertex $v$ in $G$, denoted $\deg_G(v)$, is the number of edges of $G$ incident to~$v$. Note that $2\cdot\#E = \sum_{v\in V} \deg_G(v)$.

The \emph{realization} of $G$ is the subset $\GC$ of $\Sp$ given by 
\[\GC \coloneq V \cup \bigcup_{e\in E} e.\]
A \emph{face} of $G$ is a connected component of $\Sp\setminus\GC$.
Usually, we conflate a  planar embedded graph $G$ with its 
realization $\GC$. Then it is understood that $\GC$ contains a finite set $V\sub \GC$ of distinguished points that  are the vertices 
of the graph. Its  edges are  the closures  of the components of  $\GC\setminus V$.

A \emph{subgraph} of a planar embedded graph $G=(V,E)$ is a planar embedded graph $G'=(V',E')$ with $V'\subset V$ and $E'\subset E$.   A \emph{path of length $n$ between vertices $v$ and $v'$} in $G$ is a sequence $v_0,e_0,v_1,e_1,\dots,e_{n-1},v_n$, where $v_0=v$, $v_n=v'$, and $e_k$ is an edge incident to the vertices $v_{k}$ and $v_{k+1}$ for  $k=0,\dots,n-1$. A path that does not repeat vertices is called a \emph{simple path}.

A path $v_0,e_0, v_1,e_1,\dots,e_{n-1},v_n$ with $v_0=v_n$ and $n\geq 2$ is called a \emph{circuit of length $n$} in $G$ and is denoted by $(e_0,e_1,\dots,e_{n-1})$. 
Such a circuit is called a \emph{simple cycle} if all vertices $v_k$, $k = 0,\dots,n-1$, are distinct.

A planar embedded  graph $G$ is called \emph{connected} if any two distinct vertices of $G$ can be joined  by a path in $G$. Equivalently,  $G$ is connected if its realization $\GC$ is connected as a subset of $\Sp$. Note that if $G$ is connected, then each face of $G$ is simply-connected.

As follows from \cite[Lemma 4.2.2]{DiestelGraph}, the topological boundary $\partial U$ of each face $U$ of $G$ may be viewed as the realization of a subgraph of $G$. Moreover, a walk around any connected component of the boundary $\partial U$ traces a circuit $(e_0,e_1,\dots,e_{n-1})$ in $G$ such that each edge of $G$ appears zero, one, or two times in the sequence $e_0,e_1,\dots,e_{n-1}$. We will say that the circuit $(e_0,e_1,\dots,e_{n-1})$ \emph{traces} (a connected component of) the boundary $\partial U$.  If $U$ is simply-connected, then $\partial U$ is connected, and the length of the (essentially unique) circuit that bounds $U$ is called the \emph{circuit length} of $U$ in $G$. 

A planar embedded graph $(V,E)$ is called {\em  bipartite} if we can split $V$ into two disjoint subsets $V_1$ and $V_2$ such that 
each edge $e\in E$ has one endpoint in $V_1$ and one in $V_2$.

\subsection{The Euclidean square pillow} \label{subsec:pillow}
As discussed in the introduction, we consider a {\em square pillow}  $\PP$ obtained from gluing two identical copies of the unit square $\I^2\subset \R^2$  along their boundaries by the identity map. Then  $\PP$ is a topological $2$-sphere.  We equip $\PP$  with the induced path metric that agrees with the Euclidean metric on each of the two copies of the unit square. We call this metric space $\PP$ the \emph{Euclidean square pillow}.  The vertices and edges of the unit square $\I^2$ in $\PP$ are called the \emph{vertices} and \emph{edges} of $\PP$. One copy of $\I^2$ in $\PP$ is called the {\em front} and the other copy the {\em back side} of $\PP$.   In a dynamical context, we also refer to these two copies of 
$\I^2$ as the $0$-{\em tiles} of $\PP$. We color the  front side of $\PP$ white, and its  back side black. Finally, we equip $\PP$ with the  orientation that agrees with the standard orientation 
on the front side $\I^2$ of $\PP$ (represented by the 
positively-oriented standard flag
$((0,0), \I\times \{0\}, \I^2)$; see  \cite[Appendix~A.4]{THEBook}).

We label the vertices and edges of $\PP$ in  counterclockwise order by $A,B,C,D$ and $a,b,c,d$, respectively, so that $A\in \PP$ corresponds to the vertex $(0,0)\in \I^2$ and the edge $a\subset \PP$ corresponds to  $[0,1]\times\{0\}\subset \I^2$. 
 Then $a$ has the endpoints $A$ and $B$. We can view the boundary 
 $\partial \I^2$ of $\I^2$ as a planar embedded graph in $\PP$ with the vertex set $V\coloneq \{A, B, C, D\}$ and the edge set $E\coloneq \{a,b,c,d\}$. We call $a$ and $c$ the \emph{horizontal} edges,  and $b$ and $d$ the  \emph{vertical} edges of $\PP$; see Figure~\ref{fig_pillownotation}.

%\begin{center}
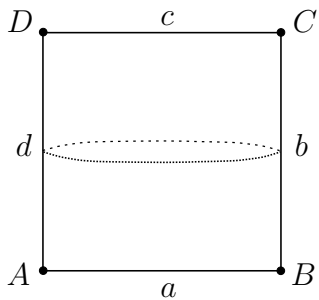
\begin{figure}[t]
\centering
\def\svgwidth{0.25\columnwidth}
%% Creator: Inkscape 1.0.1 (c497b03c, 2020-09-10), www.inkscape.org
%% Accompanies image file 'pillownotation.pdf' (pdf, eps, ps)
\begingroup%
  \makeatletter%
  \providecommand\color[2][]{%
    \errmessage{(Inkscape) Color is used for the text in Inkscape, but the package 'color.sty' is not loaded}%
    \renewcommand\color[2][]{}%
  }%
  \providecommand\transparent[1]{%
    \errmessage{(Inkscape) Transparency is used (non-zero) for the text in Inkscape, but the package 'transparent.sty' is not loaded}%
    \renewcommand\transparent[1]{}%
  }%
  \providecommand\rotatebox[2]{#2}%
  \newcommand*\fsize{\dimexpr\f@size pt\relax}%
  \newcommand*\lineheight[1]{\fontsize{\fsize}{#1\fsize}\selectfont}%
  \ifx\svgwidth\undefined%
    \setlength{\unitlength}{138.52847365bp}%
    \ifx\svgscale\undefined%
      \relax%
    \else%
      \setlength{\unitlength}{\unitlength * \real{\svgscale}}%
    \fi%
  \else%
    \setlength{\unitlength}{\svgwidth}%
  \fi%
  \global\let\svgwidth\undefined%
  \global\let\svgscale\undefined%
  \makeatother%
  \begin{picture}(1,0.92596021)%
    \lineheight{1}%
    \setlength\tabcolsep{0pt}%
    \put(0,0){\includegraphics[width=\unitlength,page=1]{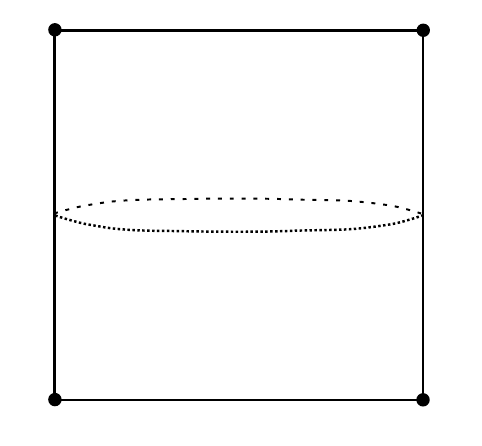}}%
    \put(0.91770271,0.86453135){\color[rgb]{0,0,0}\makebox(0,0)[lt]{\lineheight{1.25}\smash{\begin{tabular}[t]{l}$C$\end{tabular}}}}%
    \put(0.91156374,0.0473318){\color[rgb]{0,0,0}\makebox(0,0)[lt]{\lineheight{1.25}\smash{\begin{tabular}[t]{l}$B$\end{tabular}}}}%
    \put(0.49178598,0.00911037){\color[rgb]{0,0,0}\makebox(0,0)[lt]{\lineheight{1.25}\smash{\begin{tabular}[t]{l}$a$\end{tabular}}}}%
    \put(0.49178598,0.88618752){\color[rgb]{0,0,0}\makebox(0,0)[lt]{\lineheight{1.25}\smash{\begin{tabular}[t]{l}$c$\end{tabular}}}}%
    \put(0.92490862,0.46389186){\color[rgb]{0,0,0}\makebox(0,0)[lt]{\lineheight{1.25}\smash{\begin{tabular}[t]{l}$b$\end{tabular}}}}%
    \put(-0.00311208,0.86453135){\color[rgb]{0,0,0}\makebox(0,0)[lt]{\lineheight{1.25}\smash{\begin{tabular}[t]{l}$D$\end{tabular}}}}%
    \put(-0.00311208,0.0473318){\color[rgb]{0,0,0}\makebox(0,0)[lt]{\lineheight{1.25}\smash{\begin{tabular}[t]{l}$A$\end{tabular}}}}%
    \put(0.02617914,0.46389186){\color[rgb]{0,0,0}\makebox(0,0)[lt]{\lineheight{1.25}\smash{\begin{tabular}[t]{l}$d$\end{tabular}}}}%
  \end{picture}%
\endgroup%

\caption{The Euclidean square pillow $\PP$.} \label{fig_pillownotation}
\end{figure}
%\end{center}

The pillow $\PP$ is an example of a \emph{Euclidean polyhedral surface}, that is, a surface obtained by gluing Euclidean polygons along boundary edges by using isometries. Note that the metric on $\PP$ is locally flat except at its vertices, which are Euclidean conic  singularities. So $\PP$ is an \emph{orbifold} (see, for example, \cite[Appendix E]{Milnor_Book} and \cite[Appendix~A.9]{THEBook}).

An alternative description for the pillow $\PP$ can be given as follows.  We consider the unit square $\I^2\subset \R^2\cong\C$ and map it to the upper half-plane in $\CDach$ by a conformal map, normalized
so that the vertices $0,1,1+i, i$ are mapped to $0, 1, \infty, -1$, respectively. By  Schwarz reflection, this map can be extended to a meromorphic function $\wp\:\C\to  \CDach$. Then $\wp$ is a  \emph{Weierstrass $\wp$-function} (up to a postcomposition with a M\"{o}bius transformation) that  is doubly periodic with respect to the lattice $2\Z^2\coloneq \{2k+2ni: k,n\in \Z\} \subset\C$.  Actually, 
 for $z, w\in \C$ we have 
 \begin{equation}\label{eq:wpeq} 
\wp(z)=\wp(w) \text{ if and only if } z-w\in 2\Z^2 \text { or } z+w\in 2\Z^2. 
\end{equation} 
%We will write the last condition in \eqref{eq:wpeq}  in a  more convenient form as 
%$z\pm w\in \Z^2$. 

We can push forward the Euclidean metric on $\C$  to the Riemann sphere $\CDach$ by $\wp$. With respect to this metric, called the \emph{canonical orbifold metric} for $\wp$, the sphere $\CDach$ is isometric to the  Euclidean square pillow $\PP$. In the following, we identify the pillow $\PP$ with $\CDach$ by the orientation-preserving isometry that maps the vertices $A,B,C,D$ to 
$0, 1, \infty, -1$, respectively. Then  we can consider
$\wp\:\C\to  \CDach \cong \PP$  as a map onto the pillow $\PP$.  Actually,
 $\wp$ is  the \emph{universal orbifold covering map} for $\PP$ (see \cite[Section A.9] {THEBook} for more background). A very intuitive description of this map can be  given if we 
 color the  squares  $[k,k+1]\times [n,n+1]$, $k,n\in \Z$,
 in checkerboard manner black and white so that $[0,1]\times [0,1]$ is white.  Restricted to such a square $S$, the map  $\wp$ is an  isometry  that sends
 $S$ to the white $0$-tile  of $\PP$ if $S$ is white, and to the  black  
 $0$-tile $\PP$ if $S$ is black; 
 see Figure~\ref{fig_Weierstrass} for an illustration. Here, the  points in the complex plane $\C$ marked by a black dot (on the left) are mapped to $A$ by $\wp$ and are elements of  $\wp^{-1}(A)= 2\Z^2$.

%\begin{center}
\begin{figure}[t]
\centering
\def\svgwidth{0.81\columnwidth}
%% Creator: Inkscape 1.0.1 (c497b03c, 2020-09-10), www.inkscape.org
%% Accompanies image file 'covering.pdf' (pdf, eps, ps)
\begingroup%
  \makeatletter%
  \providecommand\color[2][]{%
    \errmessage{(Inkscape) Color is used for the text in Inkscape, but the package 'color.sty' is not loaded}%
    \renewcommand\color[2][]{}%
  }%
  \providecommand\transparent[1]{%
    \errmessage{(Inkscape) Transparency is used (non-zero) for the text in Inkscape, but the package 'transparent.sty' is not loaded}%
    \renewcommand\transparent[1]{}%
  }%
  \providecommand\rotatebox[2]{#2}%
  \newcommand*\fsize{\dimexpr\f@size pt\relax}%
  \newcommand*\lineheight[1]{\fontsize{\fsize}{#1\fsize}\selectfont}%
  \ifx\svgwidth\undefined%
    \setlength{\unitlength}{459.40224387bp}%
    \ifx\svgscale\undefined%
      \relax%
    \else%
      \setlength{\unitlength}{\unitlength * \real{\svgscale}}%
    \fi%
  \else%
    \setlength{\unitlength}{\svgwidth}%
  \fi%
  \global\let\svgwidth\undefined%
  \global\let\svgscale\undefined%
  \makeatother%
  \begin{picture}(1,0.32085516)%
    \lineheight{1}%
    \setlength\tabcolsep{0pt}%
    \put(0,0){\includegraphics[width=\unitlength,page=1]{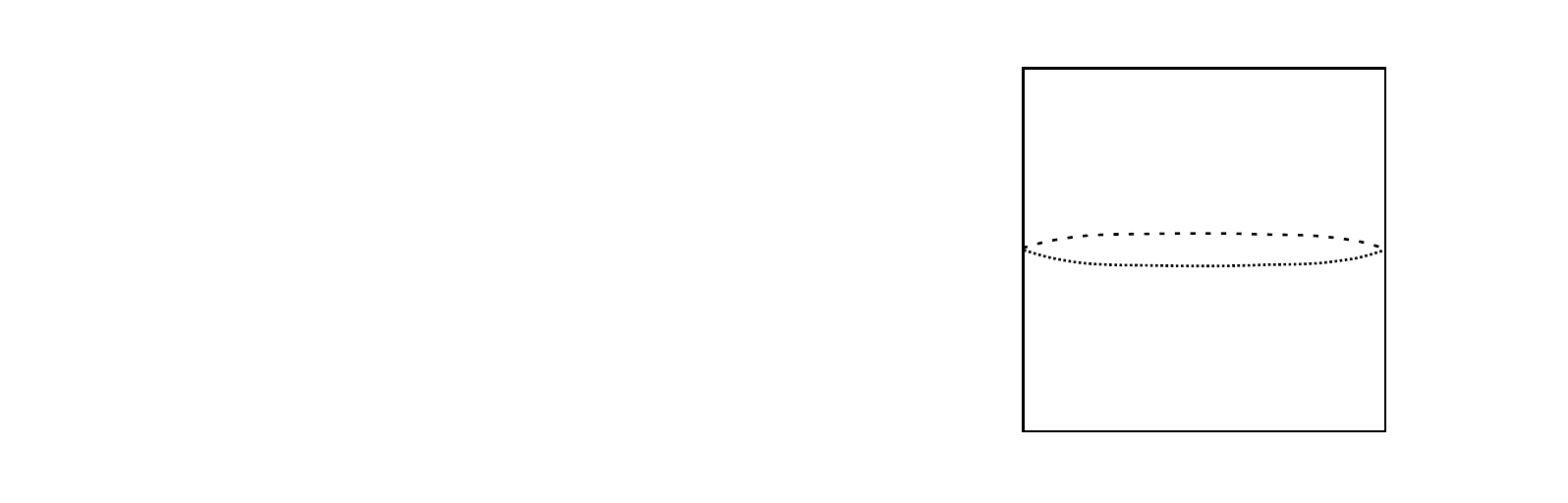}}%
    \put(0.90296098,0.14917373){\color[rgb]{0,0,0}\makebox(0,0)[lt]{\lineheight{1.25}\smash{\begin{tabular}[t]{l}$\PP$\end{tabular}}}}%
    \put(0,0){\includegraphics[width=\unitlength,page=2]{covering.pdf}}%
    \put(0.53577959,0.19443477){\color[rgb]{0,0,0}\makebox(0,0)[lt]{\lineheight{1.25}\smash{\begin{tabular}[t]{l}$\wp$\end{tabular}}}}%
    \put(0,0){\includegraphics[width=\unitlength,page=3]{covering.pdf}}%
    \put(0.61338751,0.02622958){\color[rgb]{0,0,0}\makebox(0,0)[lt]{\lineheight{1.25}\smash{\begin{tabular}[t]{l}$A$\end{tabular}}}}%
    \put(0,0){\includegraphics[width=\unitlength,page=4]{covering.pdf}}%
    \put(-0.00256323,0.15133545){\color[rgb]{0,0,0}\makebox(0,0)[lt]{\lineheight{1.25}\smash{\begin{tabular}[t]{l}$\C$\end{tabular}}}}%
    \put(0,0){\includegraphics[width=\unitlength,page=5]{covering.pdf}}%
  \end{picture}%
\endgroup%

\caption{The map $\wp\colon  \C\to \PP$.}\label{fig_Weierstrass}
\end{figure}
%\end{center}

\subsection{Isotopies and  intersection numbers}
\label{subsec:isotopies}

Let $X$ and $Y$ be topological spaces. 
Then a continuous map $H\: X\times\I\to Y$ is called 
a {\em homotopy} from $X$ to $Y$. For $t\in \I$ we denote by $H_t\coloneq 
H(\cdot,t)\: X  \ra Y$ the {\em time}-$t$ map of the homotopy. The homotopy $H$ is called an {\em isotopy} if $H_t$ is a homeomorphism from $X$ onto $Y$ for each $t\in \I$. If $Z\sub X$, then a homotopy   $H\: X\times\I\to Y$ is said to be a  \emph{homotopy relative to $Z$} (abbreviated ``$H$ is a homotopy rel.\ $Z$'') if $H_t(p) = H_0(p)$ for all $p\in Z$ and $t\in\I$. In other words, the image of each point in $Z$ remains fixed during the homotopy $H$. Isotopies rel.\ $Z$ are defined in a similar way.

Two homeomorphisms $h_0, h_1 \: X \ra  Y $ are called \emph{isotopic} (\emph{rel.\ $Z\subset X$}) if there exists an isotopy $H\: X\times\I\to Y$ (rel.\ $Z$) with $H_0 = h_0$ and $H_1 = h_1$. 
Given $M,N,Z\subset X$, we say that \emph{$M$ is isotopic to $N$ rel.\ $Z$} (or \emph{$M$ can be isotoped into $N$ rel.~$Z$}), denoted by $M\sim N$ rel.\ $Z$,  if there exists an isotopy $H\:X\times\I \to X$ rel.\ $Z$ with $H_0 = \id_{X}$ and $H_1(M) = N$. Recall that  $\id_{X}$ is the identity map on $X$.

Let $(S,Z)$ be a marked surface (with a finite, possibly empty set $Z\sub S$ of marked points). If  $\alpha$ is a Jordan curve in 
$(S,Z)$,  then 
its {\em isotopy class} $ [\alpha]$ (with $(S,Z)$ understood) consists of all Jordan curves $\beta$ in $(S,Z)$ such that $\alpha \sim \beta$  rel.\ $Z$.

The following statement gives a sufficient  condition for  two Jordan curves in 
 $(S,Z)$ to be isotopic  rel.\ $Z$. 

\begin{lemma}\label{lem:isocrit} Let $\alpha$ and $\beta$ be disjoint 
Jordan curves in a marked surface  $(S,Z)$. Suppose there is an annulus $U\sub S \setminus Z$ such that $\partial U=\alpha \cup \beta$. Then 
$\alpha$ and $\beta$ are isotopic rel.\ $Z$. 
\end{lemma} 

\begin{proof}This is standard and we will only give a sketch of the proof. Since Jordan curves in  surfaces are {\em tame}, one can slightly enlarge the annulus $U$ to an annulus $U'\sub S^2\setminus Z$ that contains $\alpha$ and $\beta$. Then $\alpha$ can be isotoped into  $\beta$ by an isotopy on $U'$ that is the identity near $\partial U'$. This isotopy on $U'$  can be extended to an isotopy on $S^2$ rel.\ $Z$ that isotopes $\alpha$ into $\beta$.   
\end{proof} 

%
%\begin{proof} Our assumptions imply that there exists a homeomorphism $H\: \partial \D\times \I \ra \cl(U)$ with 
%$H( \partial \D\times\{0\})=\alpha$ and    $H( \partial \D\times\{1\})
%=\beta$. Since $\cl(U)=U\cup \alpha \cup \beta\sub S\setminus Z$, we can consider $H$ as a homotopy $H\:\partial \D\times \I\ra S\setminus Z$ with 
%$H_0(\partial \D)=\alpha$ and $H_1(\partial \D)=\beta$. 
%The statement now  follows from Lemma~\ref{lem:hom=iso}. 
%\end{proof} 

Let $(S,Z)$ be a marked surface.
If $\alpha$ and $\beta$ are arcs or Jordan curves in $(S, Z)$, we define 
their (unsigned)  \emph{intersection number}
as 
\[\ins(\alpha, \beta)\coloneq 
\inf\{\#(\alpha'\cap \beta'): \text{$\alpha\sim \alpha'$ rel.\ $Z$ and $\beta\sim \beta'$ rel.\ $Z$}\}.\]
The relevant marked surface  $(S,Z)$ here will be understood from the context, and we suppress it from our notation for intersection numbers. If we want to emphasize it, we will say that we consider intersection numbers in $(S,Z)$. 
The intersection number is always finite, because  we can always reduce to the case when  $\alpha$ and $\beta$ are piecewise geodesic with respect to some Riemannian metric on $S$ (see \cite[Lemma A.8]{BuserGeometry}). 
 If $\alpha$ and $\beta$
satisfy $\ins(\alpha \cap \beta) =\# (\alpha\cap \beta)$, then we say that  $\alpha$ and $\beta$ are in {\em minimal position} (in their isotopy classes rel.\ $Z$).

Suppose $\alpha$  and $\beta$ are arcs or Jordan curves 
 in  $(S, Z)$.  
Then we say that  $\alpha$ and $\beta$   {\em meet  transversely
at a point} $p\in \alpha \cap \beta\cap (S\setminus Z)$  (or $\alpha$
 {\em crosses} $\beta$ at $p$) if $p$ is an isolated point in $\alpha\cap \beta$ and  if the following 
condition is  true for 
 a (small) arc $\sigma\sub
\alpha$ containing $p$ as an interior point  such that $\sigma\cap \beta=\{p\}$: let $\sigma^L$ and $\sigma^R$ be  the two subarcs of $\sigma$ into which $\sigma$ is split by $p$, then with suitable orientation of $\beta$ near $p$ the arc $\sigma^L$ lies to the left and $\sigma^R$ to the right of $\beta$. 
We say that $\alpha$ and $\beta$ {\em meet transversely} or have {\em transverse intersection} if the set $\alpha \cap \beta$ is finite and if $\alpha$ and $\beta$   meet  transversely
at each point $p\in \alpha \cap \beta\cap (S\setminus Z)$.

\begin{lemma} \label{lem:transverse} 
Suppose $\alpha$ and $\beta$ are  Jordan curves or arcs in 
a marked surface $(S,Z)$.
  If $\alpha$ and $\beta$ are in minimal position, then $\alpha$ and $\beta$ meet transversely. 
\end{lemma} 

\begin{proof} This  is  essentially a standard fact (see, for example, \cite[pp.~416--417]{BuserGeometry}), and we will only give an outline of the proof.

 Since $ \# (\alpha\cap \beta)=\ins(\alpha \cap \beta)$, the set 
$\alpha \cap \beta$ consists of finitely many isolated points. To reach a contradiction, suppose that $\alpha$ and $\beta$ do not meet transversely at some point  $p\in \alpha \cap \beta\cap (S\setminus Z)$. Then there exists an   arc $\sigma\sub
\alpha$ containing $p$ as an interior point  such that 
$\sigma\cap \beta=\{p\}$ and with the following property: if    $\sigma_1$ and $\sigma_2$   denote the two subarcs of $\sigma$ into which $\sigma$ is split by $p$, then $\sigma_1$ and $\sigma_2$ lie on the same side of $\beta$ (equipped with some orientation locally near $p$).
In other words, $\alpha$ touches $\beta$ locally near  $p$ from one side and does not cross $\beta$ at $p$. 

We can then modify the curve $\alpha$ near $p$ by an isotopy   that pulls the subarc $\sigma$ away from $\beta$ so that the new curve $\alpha$ does not have the intersection point $p$ with $\beta$ while no new intersection points  of $\alpha$ and $\beta$  arise.  This contradicts our assumption
that for the original curve $\alpha$ we have  
$\# (\alpha\cap \beta)=\ins(\alpha \cap \beta)$.
\end{proof}

\subsection{Jordan curves in spheres with four marked points}
\label{subsec:curvesonsph}

If   $(S^2,Z)$ is a marked sphere where $Z\sub S^2$ consists   of exactly four points, then,  up to homeomorphism, we may assume that $S^2$ is equal to the pillow $\PP$, and $Z=V=\{A,B,C,D\}$ consists of the four vertices of $\PP$.  
We will freely switch back and forth between a general marked sphere $(S^2, Z)$ with $\#Z=4$ and $(\PP,V)$.

We need  some statements about  isotopy classes of Jordan curves and arcs  in $(\PP,V)$ and their intersections numbers. They are  ``well known", but unfortunately we have been unable to  track down a comprehensive account  in the literature. Accordingly, we will provide a complete treatment. This may be  of independent interest apart from the main objective of the paper. We will give the statements in this section, but will provide the details of the proofs in the appendix.

As we will see,   there is a natural way to define a bijection between the set of  isotopy classes $[\ga]$ of essential  Jordan curves 
$\gamma$ in $(\PP, V)$ and  the set of \emph{extended rational numbers} $\widehat{\Q}\coloneq \Q\cup\{\infty\}$.  Throughout this paper, whenever we write $r/s\in  \widehat{\Q}$, we  assume that 
$r\in \Z$ and $s\in \N_0$ are two relatively
prime integers. We allow $s=0$ here, in which case we assume $r= 1$. Then 
$r/s= 1/0\coloneq \infty\in  \widehat{\Q}$.  

We say that a  (straight) line  $\ell\sub  \C$ has {\em 
slope}  $r/s\in\widehat{\Q}$ if it is given as 
$$\ell =\{ z_0+(s+ ir)t: t\in \R\}\sub \C$$ for some $z_0\in \C$. We use the notation $\ell_{r/s}(z_0)$ for the unique line in $\C$  with slope 
$r/s$ passing through $z_0\in \C$, and the notation  $\ell_{r/s}$ (when the point $z_0$ is not important) for any line in $\C$ with slope $r/s$.

   Let 
  $\ell_{r/s}\sub \C$ be any  line with slope $r/s \in \widehat{\Q}$.
If $\ell_{r/s}$   does not contain any  point in the lattice $\Z^2=\wp^{-1}(V)$ and so  $\ell_{r/s}\sub \C \setminus \Z^2$, then 
  $\tau_{r/s}\coloneq \wp(\ell_{r/s})$ is a Jordan curve   in $\PP\setminus V$. Actually,   $\tau_{r/s}$ is a simple closed geodesic in 
the Euclidean square pillow $\PP$ (see Figure~\ref{fig_line} for an illustration).  If  $\ell_{r/s}$  contains a point in $\Z^2$, then 
$\xi_{r/s} \coloneq \wp(\ell_{r/s})$ is a geodesic arc in 
$(\PP, V)$. 

It is easy to see that every simple closed geodesic or geodesic arc $\tau$ in $(\PP, V)$  has the form $\tau=\wp(\ell_{r/s})$ 
for a line $\ell_{r/s}\sub \C$ with some slope $r/s\in \widehat \Q$.
In the following, we use the  notation $\tau_{r/s}$ for a simple  closed geodesic and $\xi_{r/s}$ for a geodesic arc obtained in this way.

  It  follows from \eqref{eq:wpeq} that for fixed $r/s \in \widehat{\Q} $ we obtain precisely two distinct arcs $\xi_{r/s}$ and $\xi'_{r/s}$  
of the  form $\wp(\ell_{r/s}(z_0))$ depending on $z_0\in \Z^2$.  For each simple closed geodesic $\tau_{r/s}$ the arcs $\xi_{r/s}$ and $\xi'_{r/s}$ are core arcs of $\tau_{r/s}$ lying  in different components of $\PP\setminus \tau_{r/s}$ (see the appendix for more details). In particular, $\tau_{r/s}$ is always  an essential Jordan curve in $(\PP, V)$.

 It turns out that the  isotopy classes of essential 
Jordan curves  in $(\PP,V)$ are closely related to the  simple closed geodesics 
$\tau_{r/s}$.

%\begin{center}
\begin{figure}[t]
\centering
\def\svgwidth{0.82\columnwidth}
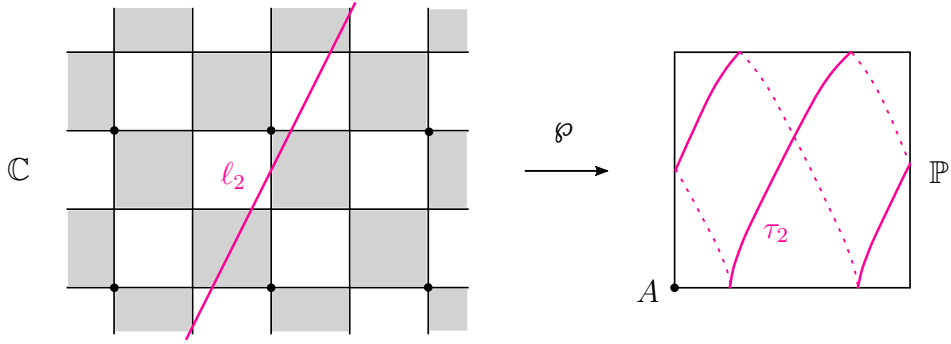 
\caption{A line $\ell_{2}$ and the corresponding  Jordan curve $\tau_2=\wp(\ell_2)$ in $\PP$.}\label{fig_line}
\end{figure}
%\end{center}

\begin{lemma}\label{lem:isoclassesP} 
Let $\ga$ be  an essential Jordan curve in $(\PP,V)$. 
Then there exists a unique slope $r/s\in  \widehat{\Q}$ with the following property. Let $\ell_{r/s}$ be any  line in $\C$ with slope $r/s$ and  $\ell_{r/s}\sub \C\setminus \Z^2$, and set $\tau_{r/s}\coloneq \wp(\ell_{r/s})$. 
Then $\tau_{r/s}$ is an essential  Jordan curve in $(\PP,V)$ with $\ga\sim 
\tau_{r/s}$ rel.~$V$. Moreover, the map $[\ga]\mapsto r/s$ gives a bijection between  isotopy classes $[\ga]$ of essential Jordan curves $\ga$ in 
$(\PP,V)$ and  slopes   $r/s\in \widehat{\Q}$.
\end{lemma} 

While this is  well known  (see, for example,  \cite[Proposition~2.6]{FarbMargalit}
or \cite[Pro\-po\-sition~2.1]{KeenSeries}), we find the available proofs too sketchy. This is the reason why we provide a  detailed proof  in the appendix.  Implicit in Lemma~\ref{lem:isoclassesP}  is the fact that the isotopy class $[\tau_{r/s}]$ of $\tau_{r/s} =\wp(\ell_{r/s})$ only depends on $r/s$ and not on the specific choice of the line $\ell_{r/s}$ with $\ell_{r/s} \sub  \C\setminus \Z^2$ (see Lemma~\ref{lem:twogeod} for an explicit statement).

Recall  that $a$, $c$ denote the horizontal,  and $b$, $d$  the vertical edges of  $\PP$.  In the following, we denote by $\alpha^h= \tau_{0}$ a {\em horizontal} essential Jordan curve  in $(\PP,V)$ (corresponding to slope $0$ and separating the 
edges $a$ and $c$ of $\PP$) and by $\alpha^v =\tau_{\infty}$ a 
{\em vertical} essential Jordan
curve  in $(\PP,V)$ (corresponding to slope $\infty$ and separating $b$ from $d$). To be specific,  
we set 
\begin{equation}\label {eq:hori+vert} 
\alpha^h\coloneq  \wp(\R\times\{1/2\}) \text{ and } 
 \alpha^v\coloneq\wp(\{1/2\}\times\R). 
\end{equation}

The following lemma summarizes the intersection properties of essential Jordan curves and arcs in $(\PP,V)$.

\begin{lemma}\label{lem:i-properties-curve}
Let $\alpha$ and $\beta$ be essential  Jordan curves in $(\PP,V)$ and $r/s, r'/s'\in \widehat \Q$ be the unique slopes such that $\alpha\sim \tau_{r/s}$ and $\beta \sim \tau_{r'/s'}$ rel.\ $V$, where 
$\tau_{r/s}$ and $\tau_{r'/s'}$ are   simple closed geodesics in $(\PP,V)$ with  slopes $r/s$ and $r'/s'$, respectively.  Let $\xi$ be a core arc of $\beta$, and $\xi_{r'/s'}$ be a geodesic arc in $(\PP, V)$ with slope $r'/s'$. 
  Then the following statements are true for intersection numbers in $(\PP, V)$: 
\begin{enumerate}[label=\text{(\roman*)},font=\normalfont,leftmargin=*]

\smallskip 
\item \label{item:i1} If $r/s=r'/s'$, then $\ins(\alpha,\beta)=0$, and if  $r/s\ne r'/s'$, then 
$\ins(\alpha, \beta )= \#(\tau_{r/s}\cap \tau_{r'/s'}) =2|rs'-sr'|>0$.  

\smallskip 
\item \label{item:i2} 
$\ins(\alpha, \xi)= \#(\tau_{r/s}\cap \xi_{r'/s'}) =\frac 12 \ins(\alpha, \beta) =|rs'-sr'|$. 

\smallskip 
\item  \label{item:i3} $\ins(\alpha,a)= 
\#(\tau_{r/s}\cap a )=|r|, \  \ins(\alpha,c)=
\#( \tau_{r/s}\cap c)=|r|$. 

\smallskip 
\item  \label{item:i4} $\ins(\alpha, b)=
\#(\tau_{r/s}\cap b)=s,\  \ins(\alpha, d)=\#(\tau_{r/s}\cap d)=s$.

\smallskip
\item  \label{item:i5} $\ins( \alpha, \alpha^h) =2|r|$ and 
$\ins( \alpha, \alpha^v) =2s$.
\end{enumerate}
\end{lemma}

\begin{figure}[t]
\centering
\def\svgwidth{0.24\columnwidth}
%% Creator: Inkscape 1.0.1 (c497b03c, 2020-09-10), www.inkscape.org
%% Accompanies image file 'intersectionnumbers.pdf' (pdf, eps, ps)
\begingroup%
  \makeatletter%
  \providecommand\color[2][]{%
    \errmessage{(Inkscape) Color is used for the text in Inkscape, but the package 'color.sty' is not loaded}%
    \renewcommand\color[2][]{}%
  }%
  \providecommand\transparent[1]{%
    \errmessage{(Inkscape) Transparency is used (non-zero) for the text in Inkscape, but the package 'transparent.sty' is not loaded}%
    \renewcommand\transparent[1]{}%
  }%
  \providecommand\rotatebox[2]{#2}%
  \newcommand*\fsize{\dimexpr\f@size pt\relax}%
  \newcommand*\lineheight[1]{\fontsize{\fsize}{#1\fsize}\selectfont}%
  \ifx\svgwidth\undefined%
    \setlength{\unitlength}{149.79003534bp}%
    \ifx\svgscale\undefined%
      \relax%
    \else%
      \setlength{\unitlength}{\unitlength * \real{\svgscale}}%
    \fi%
  \else%
    \setlength{\unitlength}{\svgwidth}%
  \fi%
  \global\let\svgwidth\undefined%
  \global\let\svgscale\undefined%
  \makeatother%
  \begin{picture}(1,1.13856356)%
    \lineheight{1}%
    \setlength\tabcolsep{0pt}%
    \put(0,0){\includegraphics[width=\unitlength,page=1]{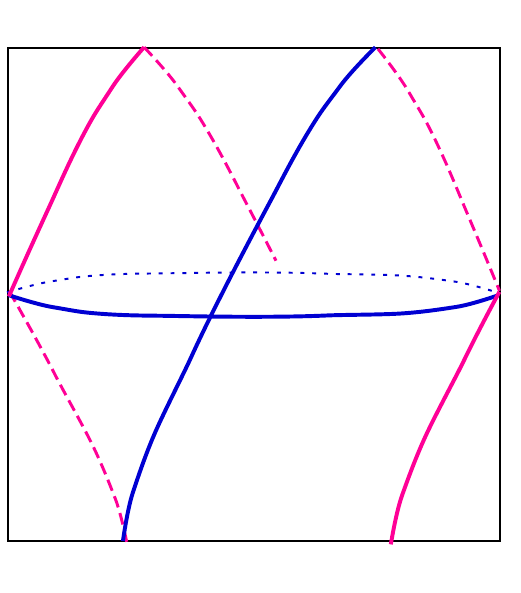}}%
    \put(0.14460725,0.74191551){\color[rgb]{0,0,0}\makebox(0,0)[lt]{\lineheight{1.25}\smash{\begin{tabular}[t]{l}\textcolor{anisred}{$\tau_2$}\end{tabular}}}}%
    \put(0,0){\includegraphics[width=\unitlength,page=2]{intersectionnumbers.pdf}}%
    \put(0.48867947,1.08945689){\color[rgb]{0,0,0}\makebox(0,0)[lt]{\lineheight{1.25}\smash{\begin{tabular}[t]{l}$c$\end{tabular}}}}%
    \put(0.49189278,0.01124842){\color[rgb]{0,0,0}\makebox(0,0)[lt]{\lineheight{1.25}\smash{\begin{tabular}[t]{l}$a$\end{tabular}}}}%
    \put(0,0){\includegraphics[width=\unitlength,page=3]{intersectionnumbers.pdf}}%
    \put(0.71356924,0.42776311){\color[rgb]{0,0,0}\makebox(0,0)[lt]{\lineheight{1.25}\smash{\begin{tabular}[t]{l}\textcolor{anisblue}{$\alpha^h$}\end{tabular}}}}%
    \put(0,0){\includegraphics[width=\unitlength,page=4]{intersectionnumbers.pdf}}%
  \end{picture}%
\endgroup%

\caption{Counting intersections of $\tau_2$ with the horizontal curve $\alpha^h$ and the  horizontal edges $a$ and $c$.}\label{fig_curveintersections}
\end{figure}

We will prove this lemma in the appendix.
Note that  Figure~\ref{fig_curveintersections}  illustrates   \ref{item:i3} and \ref{item:i5} when $\alpha=\tau_2$. It follows from the statement  that $\tau_{r/s}$ for 
$r/s\ne 0, \infty$ is in minimal position with each  of the curves $a,b,c,d, \alpha^h, \alpha^v$.

 Let $\ga$ be a Jordan curve or an arc in a  surface $S$,   and $M_1, M_2\sub S$ be two disjoint sets with $0<\#(\ga\cap M_j)<\infty $ for $j=1,2$. 
 We say that the points 
in $\ga\cap M_1\ne \emptyset $ and $\ga\cap M_2\ne \emptyset $ {\em alternate} on $\ga$ if any two points in one of the sets are separated  by the other, that is,  any subarc $\sigma\sub \ga$ with both endpoints in either  of the sets 
$\ga\cap M_1$ or $\ga\cap M_2$ must 
contain a point in the other set. 

More intuitively, this  situation when the points in  $\ga\cap M_1$ and $\ga\cap M_2$ alternate on an arc $\ga$ can be described as follows.  Suppose we traverse  $\ga$ in some (injective) parametrization starting from one of its endpoints.   
 Then we will   first meet a point in either  $M_1$ or 
  $M_2$, say  
in $M_1$. Then as we continue along $\ga$, we will meet a point in $M_2$,  then a point in $M_1$, etc. A similar remark applies when $\ga$ is a Jordan curve. Note that is this case $\#(\ga\cap M_1)=
 \#(\ga\cap M_2)$.

\begin{lemma} \label{lem:alter} 
Let $\tau=\wp(\ell_{r/s})$ be a simple closed geodesic or a geodesic arc in $(\PP,V)$  obtained from a line $\ell_{r/s}\sub \C$ with  slope $r/s\in \widehat \Q$.  If $r/s\ne 0$, then the sets $a\cap \tau$ and
 $c\cap\tau $ are non-empty and finite,  and the points in  $a\cap \tau$ and
 $c\cap\tau $  alternate on $\tau$.  
\end{lemma} 

A similar statement is true if $r/s\ne \infty$ and we replace $a,c$ with $b,d$, respectively. Lemma~\ref{lem:alter} is related to a similar  statement in a more general setting  that is the key  to proving  Lemma~\ref{lem:isoclassesP} 
 (see Lemma~\ref{lem:alternate8}). 

\begin{proof} Define $\om\coloneq s+ir\in \Z^2$. Suppose first that 
$\tau=\wp(\ell_{r/s})$ is a simple closed geodesic. Then 
$ \tau=\wp ([z_0, w_0])$, where  $z_0\in \ell_{r/s} \sub \C\setminus \Z^2$ and $w_0=z_0+2\om$, and  the map 
$u\in [0,1]\mapsto  \wp(uz_0+ (1-u)w_0)$  provides a parametrization of $\tau$ as a simple loop as follows from \eqref{eq:wpeq}.

Now the set  $\wp^{-1}(a\cup c)$ consists 
precisely of the lines $\ell_0(ni)=\{ z\in \C: \text{Im} (z)=n\}$, $n\in \Z$. These lines alternate  in the sense that $\wp$ maps  
$\ell_0(ni)$ onto $a$ or $c$  depending on whether  $n\in \Z$ is even or odd, respectively. 
Since $r/s\ne 0$, we have $r\in \Z\setminus\{0\}$, and so $\text{Im} (w_0-z_0)=\text{Im}(2\om)=2r$ is a non-zero even integer.  This implies  
that the line segment $ [z_0, w_0]\sub 
\ell_{r/s}$  has non-empty intersections (consisting of finitely many points)  with each of the sets 
$\wp^{-1}(a)$ and $\wp^{-1}(c)$. Moreover, the points in these intersections alternate on 
the segment $[z_0, w_0]$. From this together with the fact that  
 $$\#(\wp^{-1}(a)\cap [z_0, w_0))=|r|=\#(\wp^{-1}(c)\cap [z_0, w_0)) $$ 
 the statement follows (the latter fact is needed to argue that the points in $a\cap \tau$ and $c\cap \tau$ alternate on the simple {\em closed} geodesic  $\tau$).

 If $\tau$ is a geodesic arc, then there exists $z_0\in \Z^2$  such that $ \tau=\wp ([z_0, w_0])$, where  $w_0=z_0+\om $. The map $\wp$ sends  $[z_0, w_0]$ homeomorphically  onto $\tau$. Again, $ [z_0, w_0]$ has  non-empty intersections consisting of finitely many points  with each of the sets $\wp^{-1}(a)$ and $\wp^{-1}(c)$. Moreover, the points in the sets $\wp^{-1}(a)\cap [z_0, w_0]\ne \emptyset $ and $\wp^{-1}(c)\cap[z_0, w_0]\ne \emptyset$ alternate on 
the segment $[z_0, w_0]$.   The statement also follows in this case. 
\end{proof}

We conclude this section with a statement related to the previous considerations  formulated for an arbitrary sphere with four marked points.
\begin{lemma}\label{lem:i-properties-curve-sphere}
Let $(\Sp, Z)$ be a marked sphere with $\# Z = 4$, and 
$\alpha, \gamma$ be essential Jordan curves in $(\Sp, Z)$. Suppose that 
 $a_\alpha$ and $c_\alpha$ are core arcs of $\alpha$ that lie in different components of $\Sp\setminus\alpha$. Then the following statements are true: 
\begin{enumerate}[label=\text{(\roman*)},font=\normalfont,leftmargin=*]

\item \label{item:i1-sp} 
$\ins(\alpha, \gamma )= 2 \ins(a_\alpha, \gamma) = 2 \ins(c_\alpha, \gamma)$.

\smallskip 
\item \label{item:i2-sp} 
If  $\ins(\alpha, \gamma)>0$, then  
there exists  a Jordan curve $\gamma'$ in $(\Sp,Z)$ with $ \gamma' \sim \gamma$ rel.\ $Z$ such that $\gamma'$ is in minimal position with $\alpha$, $a_\alpha$, $c_\alpha$ and the points in 
$a_\alpha\cap \ga'\ne \emptyset$ and $c_\alpha\cap \ga'\ne \emptyset$ alternate on $\ga'$. 
\end{enumerate}
\end{lemma}
\begin{proof}
We may identify the marked sphere $(\Sp, Z)$ with the pillow $(\PP,V)$ by a homeomorphism that sends $\alpha$, $a_\alpha$, $c_\alpha$ to $\alpha^h$, $a$, $c$, respectively. Then, by Lemma~\ref{lem:isoclassesP}, the curve  $\ga$ is isotopic to a simple closed geodesic $\tau_{r/s}$ with slope $r/s\in \widehat \Q$. Statement \ref{item:i1-sp}
then follows from 
Lemma~\ref{lem:i-properties-curve}~\ref{item:i3} and~\ref{item:i5}.

If  $\ins(\alpha, \ga)=\ins(\alpha^h,\tau_{r/s})=2|r|>0$, then we can choose
$\ga'=\tau_{r/s}$ in \ref{item:i2-sp}. Indeed, then  $\ga'=\tau_{r/s}\sim \ga$ rel.\ $V=Z$, and  
 $\ga'=\tau_{r/s}$ is in minimal position with $\alpha^h=\tau_0$, $a_\alpha=a$, $c_\alpha=c$ as follows from Lemma~\ref{lem:i-properties-curve}~\ref{item:i1} and~\ref{item:i3}. Since $r/s\ne 0$ in this case, the statement about alternation follows  from Lemma~\ref{lem:alter}.
\end{proof}

\section{Thurston maps}
\label{sec:thurston-maps}

Here we provide a very brief summary of some relevant definitions and facts.   For more details we   refer the reader to
\cite[Chapter 2]{THEBook}.   

Let $f\colon S^2\to S^2$ be a branched covering map of a  topological $2$-sphere
$S^2$. A point $p\in\Sp$ is called \emph{periodic} (for $f$) if $f^n(p)=p$ for some $n\in\N$.   Recall that $\critf$ denotes the set of all critical points of $f$. The union 
$$\postf=\bigcup_{n\in \N} {f^{n}(\critf)}$$ of the orbits of critical points is called the \emph{postcritical set} of $f$. Note that
\[ f(\postf)\subset \postf \subset f^{-1}(\postf).\]
The map $f$ is said to be \emph{postcritically-finite} if its postcritical set $\postf$ is finite,  in other words, if every critical point of $f$ has a finite orbit under iteration.

\longhide{
Let $f\colon S^2\to S^2$ be a
\emph{branched covering map} of the topological $2$-sphere
$S^2$. That is, $f$ is a continuous surjective map and locally at each point $p\in\Sp$ the map $f$ can be written as $z\mapsto z^d$
for some $d\in \N$ in orientation-preserving
homeomorphic coordinates in domain and target. The integer
$d \geq 1$ is uniquely determined by $f$ and $p$, and called the
\emph{local degree} of the map $f$ at $p$, denoted by
$\deg(f,p)$. Let $\deg(f)\in \N$ be the topological degree of $f$. Then $$\sum_{q\in f^{-1}(p)} \deg(f,q)=\deg(f)$$
for every $p\in \Sp$. %; see \cite[Section 2.2]{Hatcher}.
%For a subset $A\subset S^2$, we denote the restriction of $f$ to $A$ by $f|A$ and its topological degree by $\deg(f|A)$ or $\deg(f:A\to f(A))$. \textcolor{blue}{MH: I would move this to notations}
A point $c\in\Sp$ with $\deg(f,c) \geq 2$ is called a \emph{critical point} of $f$. The set of all critical points of $f$ is finite and denoted by $\critf$. The union 
$$\postf=\bigcup_{n=1}^{\infty} {f^{n}(\critf)}$$ of the orbits of critical points is called the \emph{postcritical set} of $f$. Note that
\[ f(\postf)\subset \postf \subset f^{-1}(\postf).\]
The map $f$ is said to be \emph{postcritically-finite} if its postcritical set $\postf$ is finite,  in other words,  the orbit of every critical point of $f$ is finite.
}

\begin{definition}
  \label{def:Thurston-map} 
  A \emph{Thurston map} is a   post\-criti\-cally-finite branched covering map 
  $f\colon S^2 \to S^2$ of topological degree $\deg(f)\geq 2$.
\end{definition}

Natural examples of Thurston maps are given by  \emph{rational Thurston maps}, that is, post\-critically-finite rational maps on the Riemann sphere $\CDach$.

The \emph{ramification function}
of a Thurston map $f\colon\Sp\to\Sp$ is a function
$\alpha_{f} \colon \Sp \to \N\cup\{\infty\}$ such that 
$\alpha_{f}(p)$ for $p\in S^2$ is the lowest common multiple of all local
degrees $\deg(f^n,q)$, where $q\in f^{-n}(p)$ and $n\in \N$ are
arbitrary. In particular,  $\alpha_{f}(p) = 1$ for  $p\in \Sp \setminus
\postf$ and $\alpha_{f}(p) \ge 2$ for $p\in \postf$.

\begin{definition}
  \label{def:Thurston-equiv-and-realized} 
Two Thurston maps $f\colon S^2\to S^2$ and $g\colon
 \widehat{S}^2\to \widehat{S}^2$, where $ \widehat{S}^2$ is
another topological $2$-sphere,
are called \emph{Thurston equivalent} if
there are homeomorphisms $h_0,h_1\colon S^2\to  \widehat{S}^2$
that are isotopic rel.\ $\postf$ such that $h_0 \circ f = g
\circ h_1$. 
\end{definition}

We say that a Thurston map is  \emph{realized} (by a rational map) if it is Thurston equivalent to a rational map. Otherwise, we say that it is \emph{obstructed}.

The \emph{orbifold} $\mathcal{O}_f$ associated with a Thurston map $f$ is the pair
$(S^2,\alpha_f)$. The \emph{Euler
  characteristic} of $\mathcal{O}_f$ is
\begin{equation} \label{eq:orb-char}
  \chi(\mathcal{O}_f) \coloneq 2 -
  \sum_{p\in\postf}\biggl(1-\frac{1}{\alpha_f(p)}\biggr). 
\end{equation}
Here we set $1/\infty\coloneq 0$. 

The Euler characteristic of the orbifold $\mathcal{O}_f$ satisfies $\chi(\mathcal{O}_f) \leq 0$. We call $\mathcal{O}_f$
\emph{hyperbolic} if $\chi(\mathcal{O}_f)<0$, and \emph{parabolic} if $\chi(\mathcal{O}_f)=0$.

If $f\: S^2\ra S^2$ is a Thurston map, then 
$f(C_f\cup P_f)\sub P_f$, which implies $C_f\cup P_f\sub f^{-1}(P_f)$. The reverse inclusion is related to the parabolicity  of $\mathcal{O}_f$.

\begin{lemma}\label{lem:DHLem2}
Let $f\:S^2\ra S^2$ be a Thurston map. If $f$ has a parabolic orbifold, then $f^{-1}(P_f)=C_f\cup P_f$.  Moreover,  conversely, if 
 $\#\postf\ge 4$ and   
$f^{-1}(\postf)\sub \critf \cup \postf$, then $f$ has a parabolic 
orbifold. \end{lemma}

The second part  follows from \cite[Lemma 2]{DH_Th_char}, but  for the convenience of the reader we will  provide the simple proof. Here the assumption $\#\postf\ge 4$ cannot be omitted
as some examples  with $\#P_f=3$ show (such as the Thurston map arising from the ``barycentric subdivision rule"; see  \cite[Example 12.21]{THEBook}).

\begin{proof} Let $\alpha_f$ be the ramification function of $f$.
 Then 
$p\in P_f$ if and only if  $\alpha_f(p)\ge 2$.

First suppose that  $f$ has a parabolic orbifold. Then $\alpha_f(q)\cdot\deg(f,q)=
\alpha_f(f(q))$ for all $q\in S^2$ (see \cite[Proposition 2.14]{THEBook}). So if $q\in f^{-1}(P_f)$, then $f(q)\in P_f$ which implies
$$\alpha_f(q)\cdot\deg(f,q)=\alpha_f(f(q))\ge 2. $$ 
This is only possible if  $\alpha_f(q)\ge 2$ in which case $q\in P_f$, or if $\deg(f,q)\ge 2$ in which case $q\in C_f$. Hence $q\in C_f\cup P_f$, and so $f^{-1}(P_f)\sub C_f\cup P_f$. Since the reverse inclusion is true for all Thurston maps, we see that $f^{-1}(P_f)=C_f\cup P_f$ if $f$ has a parabolic orbifold. 

For the converse suppose that $f\: S^2\ra S^2$ is an arbitrary Thurston map with  $\#\postf\ge 4$ and   
$f^{-1}(\postf)\sub \critf \cup \postf$.   Let $d\coloneq\deg(f)\ge 2$. Note that $f^{-1}(P_f)\sub (C_f\setminus P_f)\cup P_f$ by our hypotheses. 

Each point $p\in S^2$ has precisely $d$ preimages counting multiplicities, that is, 
\[d=\sum_{q\in f^{-1}(p)}\deg(f,q).\]
Furthermore, since $C_f\subset f^{-1}(P_f)$ and $\deg(f,q)\ge 2$ for $q\in S^2$ if and only if $q\in  C_f$, the Riemann-Hurwitz formula implies
\begin{align*}
 \#(C_f\setminus P_f)  \le  \#C_f&\le \sum_{c\in C_f}(\deg(f,c)-1)
 =\sum_{q\in f^{-1}(P_f)}(\deg(f,q)-1) \\
 &= 
2d-2.
\end{align*}

It follows that 
\begin{align*}
 d\cdot \#P_f &= \sum_{q\in f^{-1}(P_f)}\deg(f,q)=  \sum_{q\in f^{-1}(P_f)}(\deg(f,q)-1)+ \#f^{-1}(P_f)\\
&= (2d-2)+  \#f^{-1}(P_f)\le (2d-2) +\# (C_f\setminus P_f)+\# P_f\\
&\le 4(d-1) +\#P_f. 
\end{align*} 
Hence $(d-1)\cdot \#P_f\le 4(d-1)$, and so $4\le \#P_f\le 4$. This implies  
$\#P_f=4$ and that all the previous inequalities must be equalities. 
In particular, $\# (C_f\setminus P_f)=2d-2$, which shows  that at all critical points the  local degree of $f$  is equal to $2$ and no critical points  belong to $P_f$.  

As a consequence,  under iteration of $f$ the orbit of any point $p\in S^2$  passes through at most one critical point.  It follows that  we have $\alpha_f(p)=1$ for  $p\in S^2\setminus P_f$ and $\alpha_f(p)= 2$ for  $p\in  P_f$. This implies that the Euler characteristic  
(see \eqref{eq:orb-char}) of the orbifold $\mathcal{O}_f$ associated with $f$ is equal to 
\begin{equation*}
  \chi(\mathcal{O}_f) = 2 - \sum_{p\in  P_f}\left(1-\frac{1}{\alpha_f(p)}\right) = 2-(1/2+1/2+1/2+1/2)=0. 
\end{equation*}
We conclude that  $f$ has a parabolic orbifold. \end{proof}

\subsection{The $(n\times n)$--Latt\`{e}s map}
 \label{subsec:lattes-maps}
 In general, a  \emph{Latt\`{e}s map} is a rational Thurston
map with parabolic orbifold that does not have periodic critical
points. Here we provide  the analytic definition for the  Latt\`{e}s maps that we use in this paper and interpret this from a more geometric perspective. See \cite{Milnor-Lattes} 
and \cite[Chapter 3]{THEBook} for a general discussion of Latt\`{e}s maps.

Let $\PP\cong \CDach$ be the  Euclidean square pillow 
and $\wp\colon \C\to \PP$ be its universal orbifold covering map,  as discussed  in Section \ref{subsec:pillow}. Fix a natural number $n\geq 2$. It follows from $\eqref{eq:wpeq}$ that there is a unique (and well-defined) map $\La_n\colon \PP\to \PP$ such that 
 \begin{equation}\label{eq:2Lattes} 
 \La_n(\wp(z))=\wp(nz) \text{ for } z\in \C. 
\end{equation} 
We call $\La_n$ the  \emph{$(n\times n)$-Latt\`{e}s map}. In fact, $\La_n$ is a rational map under the identification $\PP\cong\CDach$ as discussed in Section \ref{subsec:pillow}.

% OLD LOCATION OF FIGURE LATTES MAP

%\begin{center} 
%\begin{figure}[h]
%\def\svgwidth{0.6\columnwidth}
%\input{Pillow4a.pdf_tex}
%\caption{The $4\times 4$-Lattes map.}\label{fig_Lattes}
%\end{figure}
%\end{center}

Alternatively, we can describe the map $\La_n$ in a combinatorial fashion as follows. Recall that the front side of $\PP$ is colored white, and  the back side black. These are the two $0$-tiles  of~$\PP$, and we subdivide each of them
 into $n^2$  squares of sidelength $1/n$. 
We refer to  these small squares as $1$-{\em tiles} (with $n$ understood), and  color them in a checkerboard fashion black and white so that the $1$-tile $S$ in the white side of $\PP$  with the vertex $A$ on its boundary is colored white. We map $S$ to the white side of the pillow $\PP$ by an orientation-preserving Euclidean similarity (that scales by the factor $n$) so that the vertex $A$ is fixed. If we extend this map by reflection to the whole pillow, we get the $(n\times n)$-Latt\`{e}s map $\La_n$ (see Figure \ref{fig_Lattes} for $n=4$). The map $\La_n$ sends each black or white $1$-tile homeomorphically (by a similarity) onto the $0$-tile in $\PP$ of the same color.

Based on this combinatorial description,  it is easy to see that each critical point of $\La_n$ has local degree $2$ and that the postcritical set of $\La_n$ coincides with the set of vertices of $\PP$, that is, $P_{\La_n}=\{A,B,C,D\}=V$. One can also check that for the ramification function of $\La_n$ we have $\alpha_{\La_n}(p)=2$ for each $p\in P_{\La_n}$. Substituting this into $\eqref{eq:orb-char}$, we see that $\chi(\mathcal{O}_{\La_n})=0$.  Thus, $\La_n$ has a parabolic orbifold.

\subsection{Thurston's characterization of rational maps}\label{subsec:thurston-char}

Thurston maps can  often be described  from  a combinatorial viewpoint as the Latt\`es map $\La_n$ in  Section~\ref{subsec:lattes-maps} (see, for instance,  \cite{CFKP} and \cite[Chapter~12]{THEBook}). The question whether a given Thurston map $f$ can be realized by a rational map is usually  difficult to answer except in  some special cases.  William Thurston provided a sharp, purely topological criterion that answers this question. The formulation and proof   of this celebrated result can be found in \cite{DH_Th_char}. In this section, we introduce the necessary concepts  and formulate the result only when $\#\postf=4$, which is the relevant case for  this paper.

In the following, let $f\:\Sp\to\Sp$ be a Thurston map. The map $f$ defines a natural pullback operation on  Jordan curves in $(\Sp,\postf)$: a \emph{pullback} of a Jordan curve $\gamma\subset\Sp\setminus\postf$ under $f$ is a connected component $\widetilde{\gamma}$ of $f^{-1}(\gamma)$. 
Since $f$ is a covering map from $S^2\setminus f^{-1}(\postf)$ 
onto $S^2\setminus \postf$, each pullback $\widetilde{\gamma}$
of $\ga$ is a Jordan curve in $(\Sp,\postf)$.  Moreover, 
$f|\widetilde{\gamma}\: \widetilde{\gamma}\ra \ga$ is a covering map. For some $k\in \N$ with $1\le k\le \deg(f)$ each point $p\in \ga$ has precisely $k$ distinct preimages in $ \widetilde{\gamma}$.  
Here $k$ is the (unsigned) mapping degree of 
$f|\widetilde{\gamma}$ which we denote  by $\deg(f\:\widetilde{\gamma} \to \gamma)$.

Recall that a Jordan curve  $\gamma\subset\Sp\setminus\postf$ is called \emph{essential} if each of the two connected components of $\Sp\setminus\gamma$ contains at least two points from $\postf$, and is called \emph{peripheral} otherwise.

We will need the following standard facts. 

\begin{lemma}\label{lem:pull} 
 Let $f\: S^2\ra S^2 $ be a Thurston map and let $\gamma$ and $\gamma'$   be Jordan curves in $(S^2, \postf)$  with $\gamma'\sim \gamma$  rel.\ $\postf$.  Then  there is a bijection $\widetilde \ga\leftrightarrow \widetilde \ga'$ 
between the pullbacks $\widetilde \ga$ of $\ga$ and the pullbacks
$\widetilde \ga'$  of $\ga'$ under $f$ such that for all pullbacks corresponding under this bijection we have 
$\widetilde \ga\sim  \widetilde \ga'$ rel.\ $\postf$ and 
$\deg(f\:\widetilde \ga\ra \ga)=\deg(f\:\widetilde \ga'\ra \ga')$.  
\end{lemma} 
For the proof see \cite[Lemma 6.9]{THEBook}. 
A consequence of this statement is that the isotopy classes of curves in $f^{-1}(\gamma)$ rel.\  $\postf$ only depend on the isotopy class $[\ga]$ rel.\ $ \postf$ and not on the specific choice of $\ga$.

\begin{cor}\label{cor:pull} 
Let $f\: S^2\ra S^2 $ be a Thurston map, and $\gamma$ be Jordan curve in $(S^2, \postf)$.
 \begin{enumerate}[label=\text{(\roman*)},font=\normalfont]
\item \label{pull1} If $\ga$ is peripheral, then every pullback of $\gamma$ under $f$ is also peripheral. 

\smallskip 
\item \label{pull2} Suppose that $\#\postf = 4$ and let $\widetilde \gamma$ be a pullback of $\gamma$ under $f$. If $\ga$ and  $\widetilde \ga$ are essential, 
then the isotopy class $[\widetilde \ga]$ rel.\ $ \postf$ only depends on the isotopy class  $[\ga]$ rel.\ $ \postf$ and not on the specific  choice of $\ga$ and 
its essential pullback $\widetilde \ga$. 
\end{enumerate}
\end{cor}
\begin{proof} \ref{pull1}  Since $\gamma$ is peripheral, $\gamma$ can be isotoped (rel.\ $\postf$) into a Jordan curve $\gamma'$ inside a small open Jordan region $V\subset \Sp$ such that $\# (V \cap \postf) \leq 1$ and $V$ is evenly covered by the branched covering map $f$ as in \eqref{eq:evencov}. 

Then for each component $U_j$ of $f^{-1}(V)$, the map $f|U_j\:U_j\to V$ is given by $z\in \D\mapsto z^{d_j}\in \D$ for some $d_j\in \N$ after orientation-preserving  homeomorphic coordinate changes in the source  and target.  This implies that  $\# (U_j \cap \postf) \leq 1$ and that 
each pullback of  $\gamma'$ in $U_j$ is peripheral. 
Hence  
 all pullbacks of $\gamma'$ under $f$  are peripheral and  the same is true for the pullbacks of $\gamma$ as follows 
from  Lemma \ref{lem:pull}.

\smallskip 
\ref{pull2} Suppose $\widetilde \ga$ and $\widetilde \ga'$ are two distinct essential pullbacks of $\ga$ under $f$. Since these are components of $f^{-1}(\ga)$, the Jordan curves 
 $\widetilde \ga$ and $\widetilde \ga'$ are disjoint.  Then the 
 set $S^2\setminus (\widetilde \ga \cup \widetilde \ga')$ is a disjoint union $S^2\setminus (\widetilde \ga \cup \widetilde \ga')=V\cup U\cup V' $,
 where $V,V'\sub S^2$ are Jordan regions and $U\sub S^2$ is an annulus with $\partial U=\widetilde \ga \cup \widetilde \ga'$.
  Since  $\widetilde \ga$ and  $\widetilde \ga'$ are essential,  both $V$ and $V'$ must contain at least two postcritical points. Now  $\#\postf=4$, and so  $U\cap \postf =\emptyset$.   Lemma~\ref{lem:isocrit}
then implies that 
 $\widetilde \ga$ and $\widetilde \ga'$ are isotopic rel.\ $\postf$.

 It follows that the isotopy class 
$[\widetilde \ga]$ rel.\ $\postf$ does not depend on the choice of the essential pullback $\widetilde \ga$ of $\ga$.  At the same time, Lemma~\ref{lem:pull} implies that $[\widetilde \ga]$ only depends on the isotopy class $[\ga]$, as desired.
\end{proof}

For a general Thurston map $f$ the concept of an invariant multicurve  is important to decide whether $f$ is  realized or  obstructed. By definition, a {\em multicurve} is a non-empty finite family $\Gamma$ of essential Jordan curves in $\Sp\setminus \postf$ that are pairwise disjoint and pairwise non-isotopic rel.\ $\postf$.

 Suppose now that $\#\postf=4$. Then any two essential Jordan curves in $\Sp\setminus \postf$ are either isotopic rel.~$\postf$ or have a non-empty intersection (as follows from the argument in the proof of Corollary~\ref{cor:pull}~\ref{pull2}).   Thus, in this case each multicurve $\Gamma$ consists of a single essential Jordan curve $\gamma$ in $\Sp\setminus \postf$.  We say that an essential Jordan curve $\ga$ in $\Sp\setminus \postf$ is \emph{$f$-invariant} if each essential pullback of $\gamma$ under $f$ is isotopic to  $\gamma$  rel.\ $\postf$.

\begin{definition}\label{def:obstruction} Let $f\:\Sp\to\Sp$ be a Thurston map with $\#\postf=4$ and let $\gamma \sub \Sp\setminus \postf $ be an essential  $f$-invariant Jordan curve. We denote by $\gamma_1,\dots ,\gamma_n$ for $n\in \N_0$  all the essential pullbacks of $\gamma$ under $f$ and  define
\begin{equation}\label{eq:Thurst_coeff}
  \lambda_f(\gamma)\coloneq \sum_{j=1}^{n} \frac{1}{\deg(f\:\gamma_j\to \gamma)}.
\end{equation}
Then $\gamma$ is called a \emph{(Thurston) obstruction} for $f$
if $\lambda_f(\gamma) \geq 1$. 
\end{definition}
Note that if $n=0$, then  the sum in \eqref{eq:Thurst_coeff} is the empty sum and so   $\lambda_f(\gamma)=0$. It immediately follows from 
Lemma~\ref{lem:pull} that whether or not $\ga$ is an obstruction for $f$ only depends on the isotopy class $[\ga]$ 
rel.~$\postf$.

The following theorem gives a criterion when a Thurston map $f$ with 
$\#\postf=4$ is realized. 

\begin{theorem}[Thurston's criterion]\label{thm:Thurston}  Let $f\: S^2\ra S^2$ be a Thurston map with $\#\postf=4$ and suppose that 
$f$ has a hyperbolic orbifold. Then $f$ is realized by a rational map if and only if $f$ has no  obstruction.
\end{theorem} 

With a suitable definition  of an obstruction (as an invariant multicurve that satisfies certain mapping properties) this statement is also true for  general  Thurston maps with a hyperbolic orbifold; see  \cite{DH_Th_char} or 
\cite[Section 2.6]{THEBook}.

The example of the $(n\times n)$-Latt\`{e}s map $\La_n\:\PP\to \PP$ with  $n\ge 2$ shows that Theorem~\ref{thm:Thurston} is false if $f$ has a parabolic orbifold.
Indeed, let  $\gamma$ be any essential Jordan curve in $(\PP, P_{\La_n})$, where $P_{\La_n}=V=\{A,B,C,D\}$ consists of the vertices of the pillow $\PP$. Since only the isotopy class $[\ga]$ rel.~$V$ matters, by Lemma~\ref{lem:isoclassesP}  we may 
assume without loss of generality   that $\gamma=\tau_{r/s}=\wp(\ell_{r/s}(z_0))$ with $z_0\in \C$,   $r/s\in\widehat\Q$, and $\ell_{r/s}(z_0)\sub \C \setminus \Z^2$.  Here  $r$ and $s$ are relatively prime integers and so there exist $p,q\in \Z$ such that $pr+qs=1$. Let $\widetilde\om \coloneq -p+iq$. Using \eqref{eq:wpeq} and \eqref{eq:2Lattes}  one can verify that  under $\La_n$ the curve $\gamma$ has exactly $n$ distinct pullbacks  
\begin{equation}\label{eq:lattes-preimage}
\ga_j=\wp( \ell_{r/s}((z_0+2j\widetilde \om)/n)), \quad j=1,\dots,n. 
\end{equation}
 Moreover, each  curve $\gamma_j$ is isotopic to $\gamma$ rel.\ $P_{\La_n}=V$ and $\deg(\La_n\:\gamma_j\to\gamma)=n$ for all $j=1,\dots, n$; see Figure \ref{fig:L2preimage} for an illustration.
Thus, $\lambda_{\La_n}(\gamma)=1$ and $\ga$ is an obstruction.

%\begin{center}
%\vspace{12pt}
\begin{figure}%[h]
\centering
\def\svgwidth{0.7\columnwidth}
%% Creator: Inkscape 1.0.1 (c497b03c, 2020-09-10), www.inkscape.org
%% Accompanies image file 'L2preimagecurve.pdf' (pdf, eps, ps)
\begingroup%
  \makeatletter%
  \providecommand\color[2][]{%
    \errmessage{(Inkscape) Color is used for the text in Inkscape, but the package 'color.sty' is not loaded}%
    \renewcommand\color[2][]{}%
  }%
  \providecommand\transparent[1]{%
    \errmessage{(Inkscape) Transparency is used (non-zero) for the text in Inkscape, but the package 'transparent.sty' is not loaded}%
    \renewcommand\transparent[1]{}%
  }%
  \providecommand\rotatebox[2]{#2}%
  \newcommand*\fsize{\dimexpr\f@size pt\relax}%
  \newcommand*\lineheight[1]{\fontsize{\fsize}{#1\fsize}\selectfont}%
  \ifx\svgwidth\undefined%
    \setlength{\unitlength}{500.10323066bp}%
    \ifx\svgscale\undefined%
      \relax%
    \else%
      \setlength{\unitlength}{\unitlength * \real{\svgscale}}%
    \fi%
  \else%
    \setlength{\unitlength}{\svgwidth}%
  \fi%
  \global\let\svgwidth\undefined%
  \global\let\svgscale\undefined%
  \makeatother%
  \begin{picture}(1,0.31993907)%
    \lineheight{1}%
    \setlength\tabcolsep{0pt}%
    \put(0,0){\includegraphics[width=\unitlength,page=1]{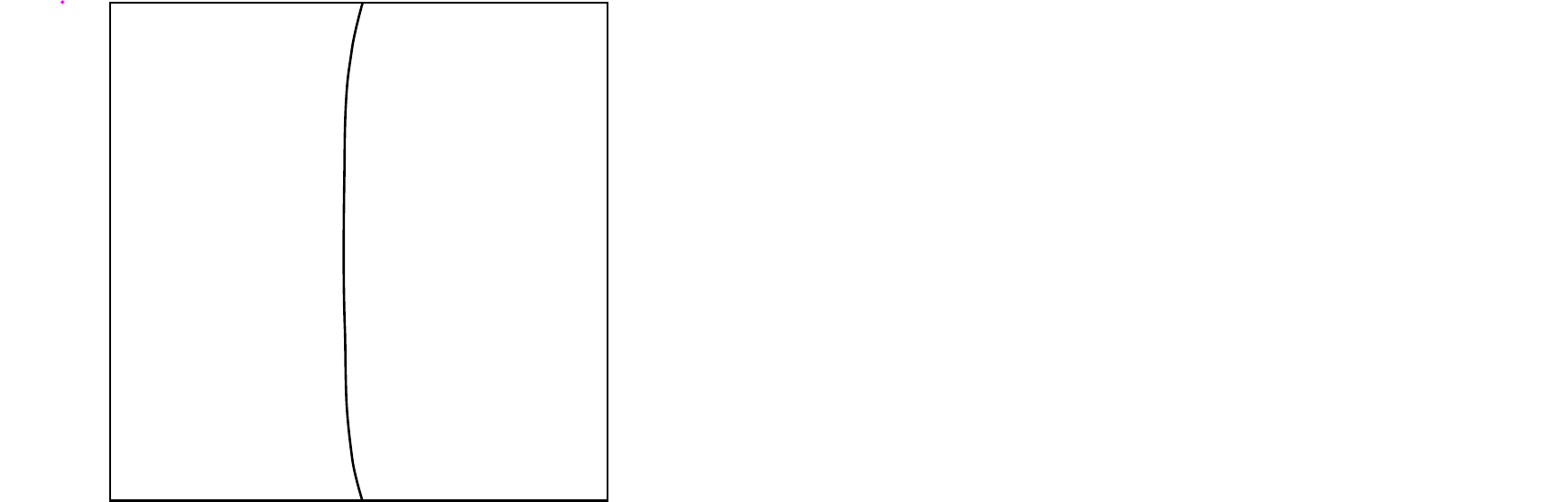}}%
    \put(0.46684308,0.19897753){\color[rgb]{0,0,0}\makebox(0,0)[lt]{\lineheight{1.25}\smash{\begin{tabular}[t]{l}$\LL_2$\end{tabular}}}}%
    \put(0,0){\includegraphics[width=\unitlength,page=2]{L2preimagecurve.pdf}}%
    \put(0.68551157,0.06401587){\makebox(0,0)[lt]{\lineheight{1.25}\smash{\begin{tabular}[t]{l}\textcolor{anisgreen}{$\tau_2$}\end{tabular}}}}%
    \put(0,0){\includegraphics[width=\unitlength,page=3]{L2preimagecurve.pdf}}%
    \put(-0.00189116,0.15084293){\color[rgb]{0,0,0}\makebox(0,0)[lt]{\lineheight{1.25}\smash{\begin{tabular}[t]{l}$\PP$\end{tabular}}}}%
    \put(0.92840451,0.14440111){\color[rgb]{0,0,0}\makebox(0,0)[lt]{\lineheight{1.25}\smash{\begin{tabular}[t]{l}$\PP$\end{tabular}}}}%
  \end{picture}%
\endgroup%
 
\caption{The two pullbacks of a curve $\ga=\tau_2$ under $\mathcal{L}_2$.}\label{fig:L2preimage}
\end{figure}

\section{Blowing up arcs}
\label{sec:blow-up}

Here, we describe the operation of ``blowing up arcs'', originally introduced by Kevin Pilgrim and Tan Lei in \cite[Section 2.5]{PT}. This operation allows us to define and modify various Thurston maps and plays a crucial role in this paper. We will first describe the general construction and then illustrate it  for Latt\`es maps. As we will explain, the  procedure of ``gluing a flap"  that we introduced  in Section~\ref{sec:blowupL}  can be viewed as a special case of  blowing  up arcs for  Latt\`es maps.

The construction of blowing up arcs will involve a finite collection $E$ of  arcs  with pairwise disjoint interiors in a $2$-sphere $S^2$. We denote by $V$  the set of endpoints of these arcs and consider $(V,E)$ as an embedded graph in $S^2$. In the construction we will make various topological choices. 
The following general statement  guarantees that we do not run into topological difficulties. In the formulation  we equip $S^2$ with a ``nice" metric $d$ so that  
$(S^2,d)$ is isometric to  $\CDach$ carrying the spherical metric with length element $ds=2|dz|/(1+|z|^2)$.  
\begin{prop}\label{prop:geod-arcs}
Let $G=(V,E)$  be a planar embedded graph in $S^2$ and  $\GC\sub S^2$ be its realization.  Then there exists a planar embedded graph $G'=(V,E')$ in $S^2$ with the same vertex set such that its  realization  $\GC'$ is  isotopic to $\GC$ rel.\ $V$  and such that each edge of $G'$ is a piecewise geodesic
 arc in $(\Sp,d)$.
\end{prop}
 An outline of the proof  is given in \cite[Chapter~I,~\S 4]{Bollobas}; the proposition  also follows from \cite[Lemma A.8]{BuserGeometry}.

\subsection{The general construction}\label{subsec:blow-up-general}
Before we provide a formal definition, we give some rough idea of how to  ``blow up'' arcs. In the following,  $f\:\Sp\to\Sp$ is a fixed Thurston map. Let $e$ be an arc in $\Sp$ such that the restriction $f|e$ is a homeomorphism onto its image. We cut the sphere $\Sp$ open along $e$ and glue in a closed Jordan region  $D$ along the boundary. In this way we obtain a new 
$2$-sphere on which we can define a  branched covering map $\widehat{f}$ as follows: $\widehat{f}$ maps the complement of $\inter(D)$ in the same way as the original map  $f$ and it maps $\inter(D)$ to the complement of $f(e)$ by a homeomorphism that matches the map $f|e$. We say that $\widehat{f}$ is obtained from $f$ by {\em blowing up} the arc $e$  with multiplicity $1$.

Now we proceed to give a rigorous definition of the blow-up operation in the general case, where several arcs $e$ are blown up simultaneously with possibly different multiplicitities $m_e\ge 1$ resulting in a new Thurston map $\widehat f$.  
To this end, let 
$E$ be a finite set of arcs in $(\Sp,f^{-1}(\postf))$ with pairwise disjoint interiors  
such that the restriction $f|e\: e\to f(e)$ is a homeomorphism for each $e\in E$. In this case,  we say that $E$ satisfies the \emph{blow-up conditions}. 

We assume that each arc $e\in E$ has an assigned {\em multiplicity} $m_e\in \N$.  Since each $e\in E$ is an arc in $(\Sp,f^{-1}(\postf))$, its interior $\inte(e)$  is disjoint from $f^{-1}(\postf) \supset \postf$ and so  $\inte(e)$ does not contain any critical or postcritical point of $f$.

For each arc $e\in E$, we choose an open Jordan region $W_e\sub S^2$ so that the following conditions  hold: 

\begin{enumerate}[label=(\text{A\arabic*)},font=\normalfont]
%leftmargin=*,noitemsep]

\item The open  Jordan regions $W_e$, $e\in E$, are pairwise disjoint. 

\smallskip 
\item  For  distinct arcs $e_1,e_2\in E$, we have 
$\cl{(W_{e_1})} \cap \cl({W_{e_2})}=\partial e_1\cap \partial e_2$.

\smallskip
\item  $\inter(e)\subset W_e$ and $\partial e\subset \partial W_e$ for each $e\in E$. 

\smallskip
\item  $\cl(W_e)\cap f^{-1}(\postf) = e\cap f^{-1}(\postf)=\partial e$ for each $e\in E$.

\smallskip
\item $f|\cl(W_{e})$ is a homeomorphism onto its image for each $e\in E$. 
\end{enumerate}

The existence of such a choice (and also of the choices below) can easily be justified based on Propositon~\ref{prop:geod-arcs} and we will skip the details.

Let $e\in E$ and  $W_e$ be chosen as above.
Then we choose a closed Jordan region  $D_e$ so that $e$ is a crosscut in $D_e$ and $D_e\setminus \partial e \sub W_e$. The two endpoints  of $e$ lie on  the Jordan curve $\partial D_e$ and partition it into two arcs, which 
we denote by $\partial D_e^+$ and  $\partial D_e^-$.  
One can think of $D_e$ as the resulting region if $e$ has been ``opened up".  This is illustrated in the left and middle parts of Figure~\ref{fig:blowupsphere1}.

%\begin{center}
%\vspace{12pt}
%\begin{figure}[ht!]
\begin{figure}[t]
\centering
\def\svgwidth{0.9\columnwidth}
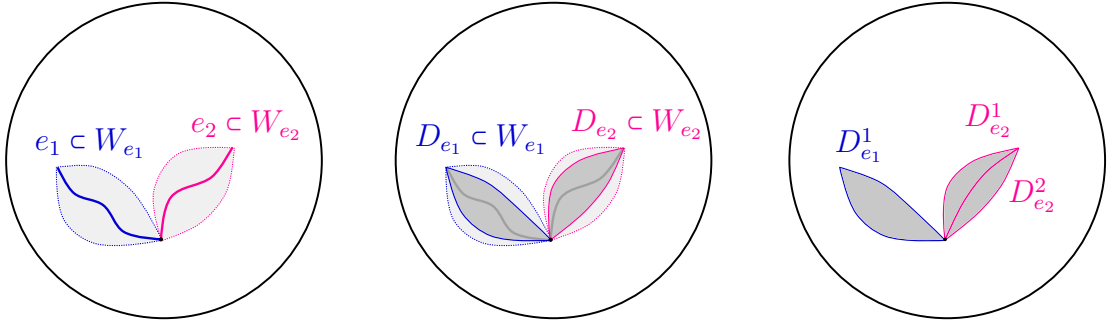
\caption{Setup for blowing up the arcs $e_1$ and $e_2$ (in  the sphere on the left) with the multiplicities $m_{e_1}=1$ and $m_{e_2}=2$.}\label{fig:blowupsphere1}
\end{figure}
%\end{center}

In order to define the desired Thurston map $\widehat f$, we want to collapse $D_e$ back to $e$. For this we choose  
a continuous map 
$h\: \Sp\times\I\to \Sp$ 
 with the following properties:  

\begin{enumerate}[label=(\text{B\arabic*}),font=\normalfont]
%,leftmargin=*,noitemsep]

\item\label{cond:B1} $h$ is a {\em pseudo-isotopy}, that is,  $h_t\coloneq h(\cdot, t)$ is a homeomorphism on $\Sp$ for  each $t\in [0,1)$.

\smallskip
\item\label{cond:B2} $h_0$ is the identity map on $\Sp$. 

\smallskip
\item\label{cond:B3}$h_t$ is the identity map on $\Sp\setminus  \bigcup_{e\in E}{W_e}$  for each  $t\in[0,1]$.

\smallskip 
\item\label{cond:B4} $h_1$ is a homeomorphism of $S^2\setminus \bigcup_{e\in E}{D_e}$ onto $S^2\setminus \bigcup_{e\in E}{e}$, and 
$h_1$ maps  $\partial D_e^+$ and  $\partial D_e^-$ homeomorphically onto $e$ for each $e\in E$. 
\end{enumerate}

It is easy to  see that if we equip $S^2$ with some metric, then  the set $h_t(D_e)$ Hausdorff converges to $e$ as $t\to 1^-$. This implies that $h_1(D_e)=e$. So 
intuitively, the deformation process described by $h$ 
collapses each closed Jordan region
$D_e$ to $e$  at time $1$   so that the points   in $\Sp\setminus  \bigcup_{e\in E}{W_e}$  remain fixed.

We now make yet another choice. For a fixed  arc $e\in E$, let
 $m=m_e$. We choose $m-1$ crosscuts  $e^1,\dots,e^{m-1}$ in $D_e$ with the same endpoints as  $e$ such that these crosscuts have pairwise disjoint interiors. We set $e^0\coloneq \partial D_e^+$ and  
$e^{m}\coloneq \partial D_e^-$. The arcs $e^0,\dots,e^{m}$
subdivide the closed Jordan region $D_e$ into $m$ closed Jordan regions $D_e^1, \dots, D_e^m$,  called \emph{components} of $D_e$. This is illustrated in the right-hand part of Figure~\ref{fig:blowupsphere1}.

We may assume that the labeling is such that 
$\partial D_e^k=e^{k-1}\cup e^k$ for $k=1, \dots, m$. 
For each $k=1, \dots, m$, we now choose a continuous map 
$\varphi_k\: D_e^k \ra S^2$ with the following properties:  
\begin{enumerate}[label=(\text{C\arabic*}),font=\normalfont]
%,leftmargin=*,noitemsep]

\item\label{cond:C1} $\varphi_k$ is an orientation-preserving  homeomorphism of $\inte(D_e^k)$ onto
 $S^2\setminus f(e)$ and maps $ e^{k-1}$ and  $ e^k$ homeomorphically onto $f(e)$.   
  
\smallskip
\item\label{cond:C2}  $\varphi_1|e^0=f\circ h_1|e^0$, $\varphi_m|e^m=f\circ h_1|e^m$, and   $\varphi_k|e^k=\varphi_{k+1}|e^k$ for $k=1, \dots, m-1$.  
\end{enumerate}
Note that, by the earlier discussion, $h_1$ maps $e^0=\partial D^+_e$ and $e^m=\partial D^-_e$ homeomorphically onto $e$ and $f$ is a homeomorphism of $e$ onto $f(e)$. These choices of the maps 
$\varphi_k$  depend on~$e$, but we suppress this in our notation for simplicity.

A map $\widehat{f}\:\Sp\to\Sp$  can now be defined as follows: 

\begin{enumerate}[label=(\text{D\arabic*}),font=\normalfont]

%\begin{center}
%\vspace{12pt}
%\begin{figure}[ht!]
\begin{figure}[b]
\centering
\def\svgwidth{0.63\columnwidth}
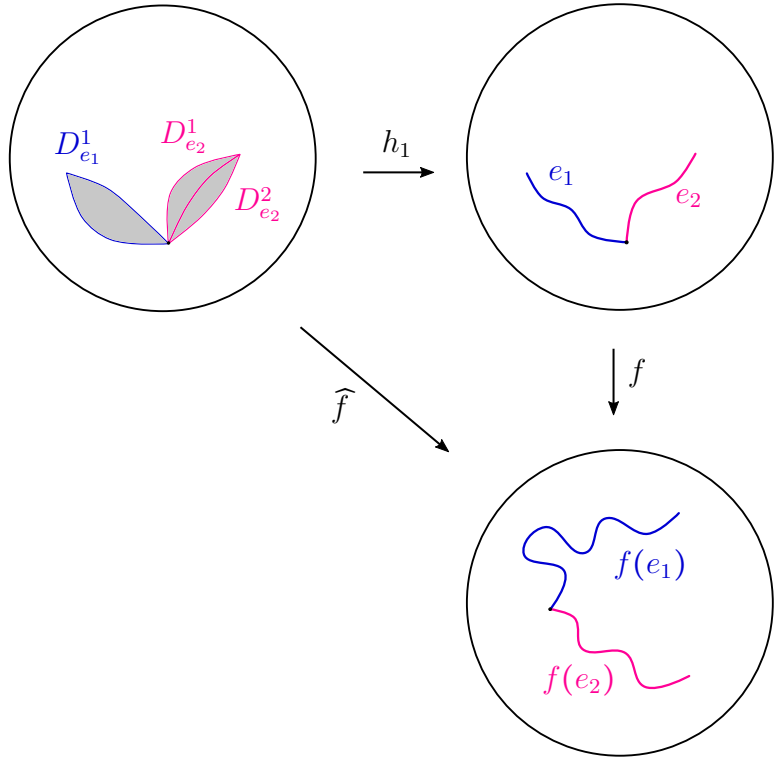
\caption{The map $\widehat f$ is obtained from $f$ by blowing up the arcs $e_1$ and $e_2$ with multiplicities $m_{e_1}=1$ and $m_{e_2}=2$.} \label{fig:blowupsphere2}
\end{figure}
%\end{center}

\item\label{cond:D1}   If $p\in \Sp\setminus\bigcup_{e\in E} \inter(D_e)$, we set $\widehat{f}(p)=f(h_1(p))$.

\smallskip
\item\label{cond:D2} 
  If $p\in D_e$ for some $e\in E$, then $p$ lies in one of the components $D_e^k$ of  $D_e$ and we set $\widehat{f}(p)=\varphi_k(p)$. 
\end{enumerate}
The matching conditions (C2) above immediately imply that $\widehat{f}$ is well-defined and continuous.

\begin{definition}\label{def:blowing-up-f} We say that the  map 
 $\widehat{f}\: S^2\ra S^2$ as described above is obtained from 
 the Thurston map $f$ by \emph{blowing up each arc $e\in E$ with multiplicity  $m_e$}. \end{definition}

 Figure~\ref{fig:blowupsphere2} illustrates the construction of $\widehat{f}$. Here, we blow up the arcs $e_1$ and $e_2$ from Figure~\ref{fig:blowupsphere1} with multiplicities  $m_{e_1}=1$ and $m_{e_2}=2$. The arcs $f(e_1)$ and $f(e_2)$ share an endpoint (since $e_1$ and $e_2$ do), but in general they could have more points in common or even coincide. For simplicity,   we chose to draw them with disjoint interiors.

    By construction, $\widehat f$  acts in a similar way as $f$ outside the closed Jordan regions $D_{e_1}=D^1_{e_1}$ and $D_{e_2}=D^1_{e_2}\cup D^2_{e_2}$.  More precisely,  $\widehat f=f \circ h_1$ on $\Sp\setminus (\inter(D_{e_1}) \cup \inter(D_{e_2}))$, where $h_1$ collapses the closed Jordan regions $D_{e_1}$ and $D_{e_2}$ onto $e_1$ and $e_2$, respectively. At the same time, $\widehat f$ maps $\inter(D_{e_1})$  homeomorphically onto $\Sp\setminus f(e_1)$, and each of the regions  $\inter(D^1_{e_2})$
    and  $\inter(D^2_{e_2})$  homeomorphically onto $\Sp\setminus f(e_2)$.

%\begin{center}
%\vspace{12pt}
\begin{figure}[b]
\centering
\def\svgwidth{0.81\columnwidth}
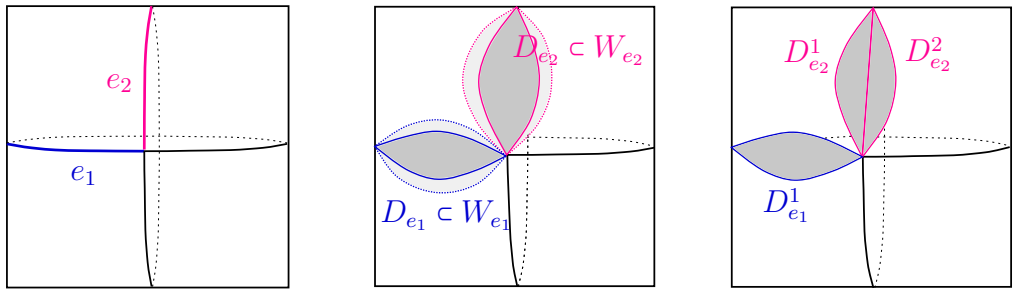
\caption{Setup for blowing up the arc set $E=\{e_1, e_2\}$ (on the left pillow) with $m_{e_1}=1$ and $m_{e_2}=2$.}\label{fig:generablowup}
%\caption{Setup for blowing up the arcs $e_1$ and $e_2$ (on the left pillow) with $m_{e_1}=1$ and $m_{e_2}=2$.}\label{fig:generablowup}
\end{figure}
%\end{center}

%\begin{center}
%\vspace{12pt}
\begin{figure}[b]
\centering
\def\svgwidth{0.7\columnwidth}
\hspace{0.9cm}
%% Creator: Inkscape 1.0.1 (c497b03c, 2020-09-10), www.inkscape.org
%% Accompanies image file 'fhatgeneral.pdf' (pdf, eps, ps)
\begingroup%
  \makeatletter%
  \providecommand\color[2][]{%
    \errmessage{(Inkscape) Color is used for the text in Inkscape, but the package 'color.sty' is not loaded}%
    \renewcommand\color[2][]{}%
  }%
  \providecommand\transparent[1]{%
    \errmessage{(Inkscape) Transparency is used (non-zero) for the text in Inkscape, but the package 'transparent.sty' is not loaded}%
    \renewcommand\transparent[1]{}%
  }%
  \providecommand\rotatebox[2]{#2}%
  \newcommand*\fsize{\dimexpr\f@size pt\relax}%
  \newcommand*\lineheight[1]{\fontsize{\fsize}{#1\fsize}\selectfont}%
  \ifx\svgwidth\undefined%
    \setlength{\unitlength}{392.41090716bp}%
    \ifx\svgscale\undefined%
      \relax%
    \else%
      \setlength{\unitlength}{\unitlength * \real{\svgscale}}%
    \fi%
  \else%
    \setlength{\unitlength}{\svgwidth}%
  \fi%
  \global\let\svgwidth\undefined%
  \global\let\svgscale\undefined%
  \makeatother%
  \begin{picture}(1,0.30927224)%
    \lineheight{1}%
    \setlength\tabcolsep{0pt}%
    \put(0,0){\includegraphics[width=\unitlength,page=1]{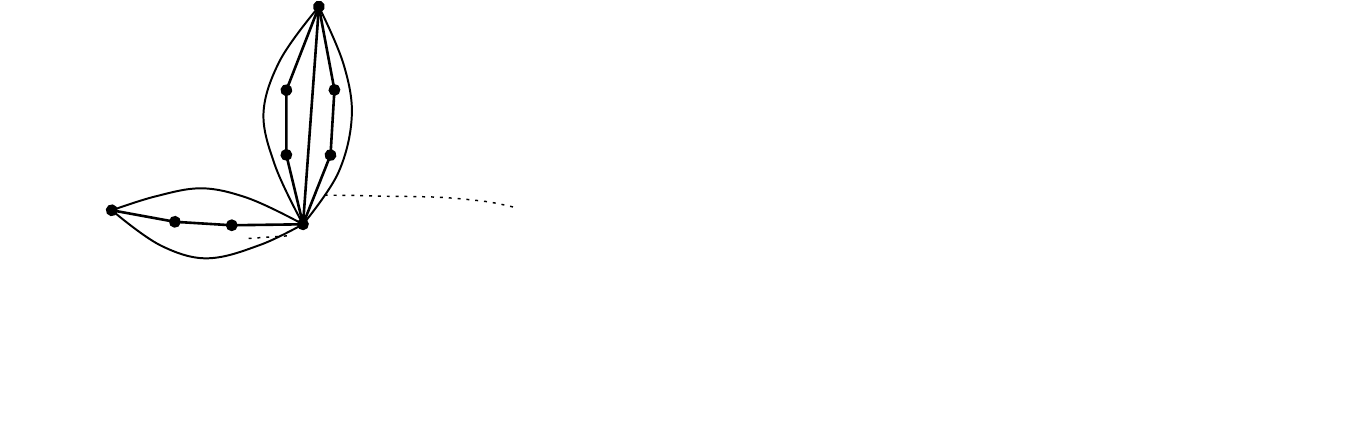}}%
    \put(0.45828998,0.1906773){\color[rgb]{0,0,0}\makebox(0,0)[lt]{\lineheight{1.25}\smash{\begin{tabular}[t]{l}$\widehat{f}$\end{tabular}}}}%
    \put(0,0){\includegraphics[width=\unitlength,page=2]{fhatgeneral.pdf}}%
    \put(-0.00241016,0.14189295){\color[rgb]{0,0,0}\makebox(0,0)[lt]{\lineheight{1.25}\smash{\begin{tabular}[t]{l}$\PP$\end{tabular}}}}%
    \put(0.90875602,0.14033517){\color[rgb]{0,0,0}\makebox(0,0)[lt]{\lineheight{1.25}\smash{\begin{tabular}[t]{l}$\PP$\end{tabular}}}}%
  \end{picture}%
\endgroup%

\caption{The map $\widehat{f}$ obtained from $f=\mathcal{L}_2$ by blowing up the arcs in the set $E=\{e_1, e_2\}$ illustrated in Figure~\ref{fig:generablowup} with $m_{e_1}=1$ and $m_{e_2}=2$.} \label{fig:fhatgeneral}
\end{figure}
%\end{center}

In the next section, we want to relate   ``blowing up arcs" with ``gluing flaps" as discussed in the introduction. In order to set this up,  
we consider the $(2\times 2)$-Latt\`{e}s map $f=\La_2$.  We choose two edges $e_1$ and $e_2$ of a $1$-tile in $\PP$ as shown in the pillow on the left in Figure \ref{fig:generablowup}. Note that $f$ sends $e_1$ and $e_2$ homeomorphically onto the edges $c$ and $b$ of $\PP$, respectively.  Thus the set $E=\{e_1,e_2\}$ satisfies the blow-up conditions. Figure~\ref{fig:generablowup} illustrates the setup for blowing up these arcs $e_1$ and $e_2$ with the multiplicities $m_{e_1}=1$ and $m_{e_2}=2$. The resulting map $\widehat f\:\PP\to \PP$ is shown in Figure~\ref{fig:fhatgeneral}. The points marked by a dot
 on the left  pillow $\PP$ (the domain of the map) correspond to the preimage points $\widehat f^{-1}(V)$. The pillow on the left is subdivided into closed Jordan regions alternately colored black and white. The map $\widehat f$  sends each of these closed Jordan regions $U$ homeomorphically onto the back side or  front side of the pillow $\PP$  depending on whether $U$ is black or white.

 The following statement summarizes the main properties of  maps $\widehat{f}$ as in  Definition~\ref{def:blowing-up-f}.

\begin{lemma}\label{lem:blowup} Let $f\: S^2\ra S^2$ be a Thurston map and  $E$ be a set of arcs in $(S^2, f^{-1}(\postf))$ satisfying the blow-up conditions. 
Suppose  $\widehat{f}\: S^2\ra S^2$ is the map obtained by blowing up each arc $e\in E$ with multiplicity  $m_e\in \N$. 

Then $\widehat f$ is a Thurston map with $P_{\widehat f}=\postf$.  Moreover, the map $\widehat f$  is uniquely determined up to Thurston equivalence independently of the choices  in the above construction. More precisely, up to  Thurston equivalence $\widehat f$ depends only on  the original map $f$, the isotopy classes of the arcs in $E$ rel.\ $f^{-1}(\postf)$, and the multiplicities $m_e$ for $e\in E$. 
\end{lemma}

\begin{proof} By construction $\widehat f$ is an orientation-preserving local homeomorphism near each point $p\in S^2\setminus f^{-1}(\postf)$. By considering the  number of 
preimages of a generic point in $S^2$ we see that the topological degree of $\widehat{f}$ is equal to  $\deg(f)+ \sum_{e\in E} m_{e}>0.$ 
The fact that  $\widehat{f}\:\Sp\to\Sp$ is a branched covering map can now be deduced from  \cite[Corollary A.14]{THEBook}. 

We have $\deg(\widehat f, p)=1$ for $p\in S^2\setminus f^{-1}(\postf)$ and  
$$\deg(\widehat f, p) =\deg(f,p)+ \sum_{\{e\in E:\ p\in e\}} m_e $$
for $p\in  f^{-1}(\postf)$.  
This implies  $\critf \sub C_{\widehat f} \sub f^{-1}(\postf)$. Since 
on the set $f^{-1}(\postf)$ the maps $f$ and $\widehat f$ agree, 
it follows that $P_{\widehat f}=\postf$. So $\widehat f$ has a finite postcritical set, and we conclude that 
$\widehat f$ is indeed a Thurston map.

We  omit a detailed justification of the second claim that  $\widehat f$ is uniquely determined up to Thurston equivalence
by $f$, the isotopy classes of the arcs $E$, and their multiplicities.  
 A proof can be given along the lines of  \cite[Proposition 2]{PT}.  
 \end{proof} 
 
 \begin{rem}\label{rem:hyporb} If in the previous statement $E\ne \emptyset$ and $P_f\ge 3$, then $\widehat f$ has a hyperbolic orbifold. To  see this, pick an arc $e\in E$. Then $f(e)$ has at most two points in common with $P_f$, and so we can find a point $p\in P_f\setminus f(e)\sub P_{\widehat f}$. Then it follows from the construction of $\widehat 
 f$ that there exists a point $q$ in the interior of the region $D_e$ associated with $e$ such that $\widehat f(q)=p$ and $\deg(\widehat f, q)=1$.
 Then $q\not\in C_{\widehat f}$, but we also have 
 $$q\in \inte(D_e)\sub 
 S^2\setminus f^{-1}(P_f)\sub S^2\setminus P_f=S^2\setminus P_{\widehat f}.$$  This shows that $q\in \widehat f^{-1}(p)\sub  
  \widehat f^{-1}(P_{\widehat f})$, but $q\not\in  C_{\widehat f}\cup 
  P_{\widehat f}$, and so $\widehat f$ must have  a hyperbolic orbifold by the first part of Lemma~\ref{lem:DHLem2}.  
   \end{rem}

\subsection{Blowing up the $(n\times n)$-Latt\`{e}s map}\label{subsec:blownup-lattes} Let $\PP$ be  the  Euclidean square pillow and $ \La_n\:\PP\to \PP$ be the 
$(n\times n)$-Latt\`{e}s map for fixed  $n\geq 2$.  We denote by $\CC\sub \PP$ the common boundary of the two sides of $\PP$. The set $\CC$  may be viewed as a planar embedded graph with the vertex set $V=
\{A,B,C,D\}$ and the edge set $\{a, b,c,d\}$ in the notation from Section~\ref{subsec:pillow}. Let $\widetilde{\CC}\coloneq  \La_n^{-1}(\CC)\sub \PP$ be the preimage of $\CC$ under $\La_n$, viewed as a planar embedded graph with the  vertex set $\La_n^{-1}(V)$. In the next section, we will study the question whether a  Thurston map is realized by a rational map if it is obtained from $\La_n$ by blowing up 
 edges of $\widetilde{\CC}$. In order to facilitate this discussion,  we will provide a more concrete combinatorial model for these maps.

By the definition of the map $\La_n$, the graph $\widetilde{\CC}$ subdivides the pillow $\PP$ into $2n^2$ $1$-tiles, which are squares of sidelength $1/n$. The edges of the embedded graph 
$\widetilde{\CC}$ are precisely the sides of these squares. We call them the $1$-{\em edges} of $\PP$ (for given $n$).  The map $\La_n$ sends each $1$-edge $e$ of $\PP$ homeomorphically onto one of the edges $a$, $b$, $c$, $d$ of~$\CC$.  We call  $e$ {\em horizontal} if  $\La_n$ maps it onto  $a$ or  $c$, and {\em vertical} if $\La_n$ maps it onto $b$ or $d$.

We take two disjoint copies of the Euclidean square $[0,1/n]^2\sub \R^2$ and identify the points on three of their sides, say the sides $\{0\}\times [0, 1/n]$, $[0,1/n] \times \{1/n\}$, and $\{1/n\}\times [0,1/n]$. We call the  object obtained a \emph{flap} $F$. Note that it is homeomorphic to the  closed unit disk and has two ``free" sides corresponding to the two copies of  $[0,1/n] \times \{0\}$ in $F$.

% THIS IS IN THE INTRODUCTION NOW!
%\begin{center}
%\begin{}[h]
%\includegraphics[scale=0.10]{skizze6.jpg}
%\caption{Construction of $\Phat$: Gluing in a flap.}
%\end{figure}
%\end{center}

We can cut the pillow $\PP$ open along one of the edges of $\widetilde{\CC}$ and glue in a flap $F$ to the pillow by identifying each copy of  $[0,1/n]\times \{0\}$ in the flap with one side of the slit by an isometry (see Figure~\ref{fig_flap}). In this way, we get a new polyhedral surface homeomorphic to $\Sp$. One can also glue multiple copies of the flap to the slit by an isometry and obtain  a polyhedral surface $\widehat \PP$ homeomorphic to $\Sp$. This can 
be  described  more concretely as follows. Let $e$ be an edge in $\widetilde \CC$ and $F_1, \dots, F_m$ be $m\geq 1$ copies of the flap. For each $k=1,\dots,m$, we denote the two copies of $[0,1/n]\times \{0\}$ in the flap $F_k$ by $e_k'$ and $e_k''$. We now  construct a new polyhedral surface $\widehat \PP$ in the following way: 
 \begin{enumerate}[label=\text{(\roman*)},font=\normalfont]
\item First, we cut the original pillow $\PP$ open along the edge $e$.

\smallskip
\item Then, for each $k=1,\dots, m-1$, we identify the edge $e_k''$ of $F_k$ with the edge $e_{k+1}'$ of $F_{k+1}$ by an isometry. We get a polyhedral surface $D_e$ homeomorphic to a closed disk, whose boundary consists of two edges $e_1'$ and $e_m''$.

\smallskip
\item  Finally, we glue the disk $D_e$ to the pillow $\PP$ cut open along $e$  by identifying the edges $e_1'$ and $e_m''$ in $\partial D_e$ with the two sides of the slit by an isometry so that $e_1'$ and $e_m''$ are identified with different sides of the slit. We obtain a polyhedral surface $\widehat \PP$ that  is homeomorphic to a $2$-sphere. 
\end{enumerate}
More generally, we can cut open $\PP$ simultaneously along several  edges $e$ of 
$\widetilde{\CC}$  and, by the method described, glue $m_e\in \N$   copies of the flap   to the slit obtained from each edge $e$.    If these edges $e$ of $\widetilde{\CC}$ with their multiplicities $m_e$ are given,  then there  is essentially only one way of gluing flaps  so that the resulting object is   a polyhedral surface homeomorphic to $\Sp$.

Let $\widehat{\PP}$ be the  polyhedral surface obtained from $\PP$ by gluing a total number of $n_h\geq 0$ {\em horizontal flaps} (i.e., 
flaps  glued along horizontal edges of $\widetilde{\CC}$) and a total number of $n_v\geq 0 $ {\em vertical flaps} (i.e., flaps  glued along vertical edges of $\widetilde{\CC}$). We call this surface a \emph{flapped pillow}. We denote by  $E$  the set of 
 all edges in $\widetilde\CC$ along which flaps were glued  and  by $m_e$, $e\in E$, the corresponding  multiplicities. See the left part of Figure \ref{fig:Lhatlattes} for an example of a flapped pillow $\widehat \PP$ obtained by gluing one horizontal and two vertical flaps at the edges $e_1$ and $e_2$ from Figure \ref{fig:generablowup}.

% \begin{center}
%\vspace{12pt}
\begin{figure}[t]
\centering
\def\svgwidth{0.75\columnwidth}
%% Creator: Inkscape 1.0.1 (c497b03c, 2020-09-10), www.inkscape.org
%% Accompanies image file 'fhatlattes.pdf' (pdf, eps, ps)
\begingroup%
  \makeatletter%
  \providecommand\color[2][]{%
    \errmessage{(Inkscape) Color is used for the text in Inkscape, but the package 'color.sty' is not loaded}%
    \renewcommand\color[2][]{}%
  }%
  \providecommand\transparent[1]{%
    \errmessage{(Inkscape) Transparency is used (non-zero) for the text in Inkscape, but the package 'transparent.sty' is not loaded}%
    \renewcommand\transparent[1]{}%
  }%
  \providecommand\rotatebox[2]{#2}%
  \newcommand*\fsize{\dimexpr\f@size pt\relax}%
  \newcommand*\lineheight[1]{\fontsize{\fsize}{#1\fsize}\selectfont}%
  \ifx\svgwidth\undefined%
    \setlength{\unitlength}{399.69863751bp}%
    \ifx\svgscale\undefined%
      \relax%
    \else%
      \setlength{\unitlength}{\unitlength * \real{\svgscale}}%
    \fi%
  \else%
    \setlength{\unitlength}{\svgwidth}%
  \fi%
  \global\let\svgwidth\undefined%
  \global\let\svgscale\undefined%
  \makeatother%
  \begin{picture}(1,0.35038629)%
    \lineheight{1}%
    \setlength\tabcolsep{0pt}%
    \put(0,0){\includegraphics[width=\unitlength,page=1]{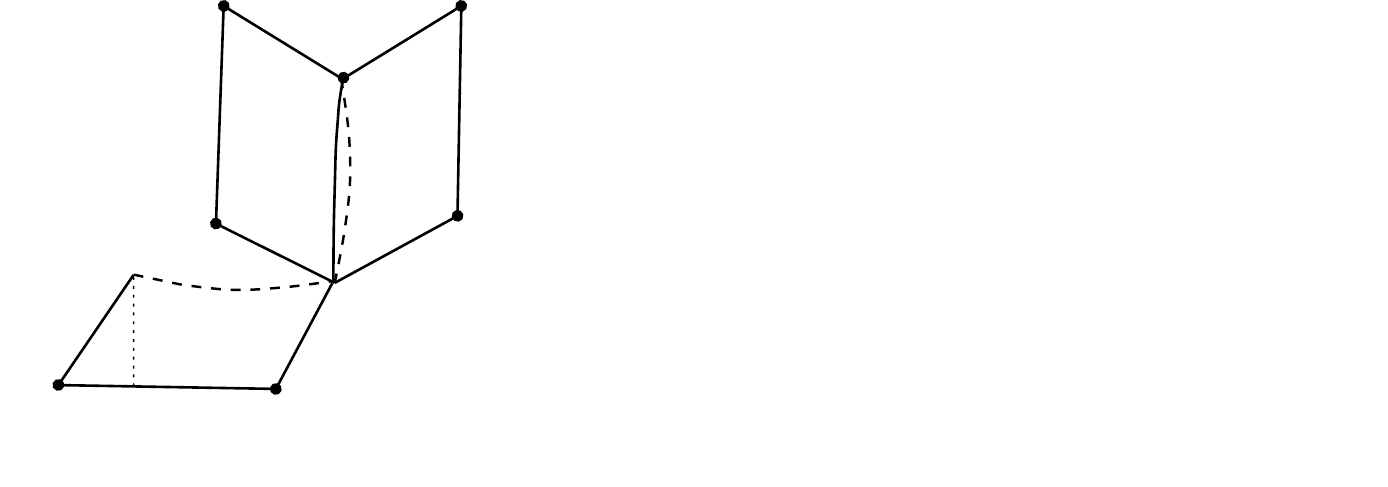}}%
    \put(0.46511054,0.18403043){\color[rgb]{0,0,0}\makebox(0,0)[lt]{\lineheight{1.25}\smash{\begin{tabular}[t]{l}$\widehat{\La}$\end{tabular}}}}%
    \put(0,0){\includegraphics[width=\unitlength,page=2]{fhatlattes.pdf}}%
    \put(-0.00236622,0.14027317){\color[rgb]{0,0,0}\makebox(0,0)[lt]{\lineheight{1.25}\smash{\begin{tabular}[t]{l}$\PPh$\end{tabular}}}}%
    \put(0.91041968,0.14031406){\color[rgb]{0,0,0}\makebox(0,0)[lt]{\lineheight{1.25}\smash{\begin{tabular}[t]{l}$\PP$\end{tabular}}}}%
  \end{picture}%
\endgroup%
 
\caption{A branched covering map $\widehat{ \La}\: \widehat{\PP}\ra \PP$ induced by the flapped pillow $\widehat \PP$ on the left.} \label{fig:Lhatlattes}
\end{figure}
%\end{center} 
 
The polyhedral surface  $\widehat \PP$ is naturally subdivided into 
   $$2(n^2+n_h+n_v)=2n^2+2\sum_{e\in E} m_e$$  squares  of  sidelength $1/n$, called the \emph{$1$-tiles} of the flapped pillow $\widehat \PP$. The vertices and the edges of these squares are called the \emph{$1$-vertices} and \emph{$1$-edges} of $\widehat \PP$.  There is a natural  path metric on $\widehat{\PP}$ that agrees with 
 the Euclidean metric on each $1$-tile. The surface $\widehat{\PP}$
 equipped with this metric  is locally Euclidean with  conic singularities at some of the $1$-vertices. Such a conic singularity arises at a $1$-vertex $v\in \widehat \PP$ if $v$ is contained in $k_v\ne 4$ distinct $1$-tiles.

 We will assume that $\widehat \PP$ has at least one flap, that is, $n_h+n_v\geq 1$. Let $F_j$, $j=1,\dots, n_h+n_v$, be the collection of flaps glued to $\PP$. Each flap $F_j$ consists of two $1$-tiles in $\widehat \PP$. We call the four $1$-vertices that belong to $F_j$ the \emph{vertices} of the flap $F_j$. The boundary $\partial F_j$ is a Jordan curve in $\widehat \PP$ composed of two $1$-edges $e'_j$ and $e''_j$, which we call the \emph{base edges} of $F_j$. The $1$-edge in $F_j$ that is opposite to the base edges is called the \emph{top edge} of the flap $F_j$. Note that 
 $\partial e'_j=\partial e''_j$ consists of two vertices of $F_j$.

 We now define the {\em base} $B(\widehat \PP)\sub \widehat \PP$ of the flapped pillow as 
\begin{equation}\label{eq:base-pillow}
B(\widehat \PP)\coloneq \widehat \PP \setminus \biggl (\bigcup_{j=1}^{n_h+n_v} (F_j\setminus \partial e'_j) \biggr). \end{equation}
In other words, $B(\widehat \PP)$ is obtained from $\widehat \PP$ 
by removing all flaps $F_j$ from $\widehat \PP$, except that we keep  the  two vertices in $\partial e'_j\sub F_j$ from each flap.   
 There is a natural identification 
 \begin{equation}\label{eq:base-ident}
B(\widehat \PP)\cong \PP\setminus \bigcup_{e\in E} \inter(e)\sub \PP.  \end{equation}
This means that we can consider the base $B(\widehat \PP)$ both as a subset of $\widehat \PP$ and  of $\PP$. Figure~\ref{fig:basepillow} illustrates these two viewpoints.  This is slightly imprecise, but this 
point of view will be   extremely convenient in the following. 

%\begin{center}
%\vspace{12pt}
\begin{figure}[b]
\def\svgwidth{0.52\columnwidth}
%% Creator: Inkscape 1.0.1 (c497b03c, 2020-09-10), www.inkscape.org
%% Accompanies image file 'basepillow.pdf' (pdf, eps, ps)
\begingroup%
  \makeatletter%
  \providecommand\color[2][]{%
    \errmessage{(Inkscape) Color is used for the text in Inkscape, but the package 'color.sty' is not loaded}%
    \renewcommand\color[2][]{}%
  }%
  \providecommand\transparent[1]{%
    \errmessage{(Inkscape) Transparency is used (non-zero) for the text in Inkscape, but the package 'transparent.sty' is not loaded}%
    \renewcommand\transparent[1]{}%
  }%
  \providecommand\rotatebox[2]{#2}%
  \newcommand*\fsize{\dimexpr\f@size pt\relax}%
  \newcommand*\lineheight[1]{\fontsize{\fsize}{#1\fsize}\selectfont}%
  \ifx\svgwidth\undefined%
    \setlength{\unitlength}{273.75918466bp}%
    \ifx\svgscale\undefined%
      \relax%
    \else%
      \setlength{\unitlength}{\unitlength * \real{\svgscale}}%
    \fi%
  \else%
    \setlength{\unitlength}{\svgwidth}%
  \fi%
  \global\let\svgwidth\undefined%
  \global\let\svgscale\undefined%
  \makeatother%
  \begin{picture}(1,0.42920368)%
    \lineheight{1}%
    \setlength\tabcolsep{0pt}%
    \put(0,0){\includegraphics[width=\unitlength,page=1]{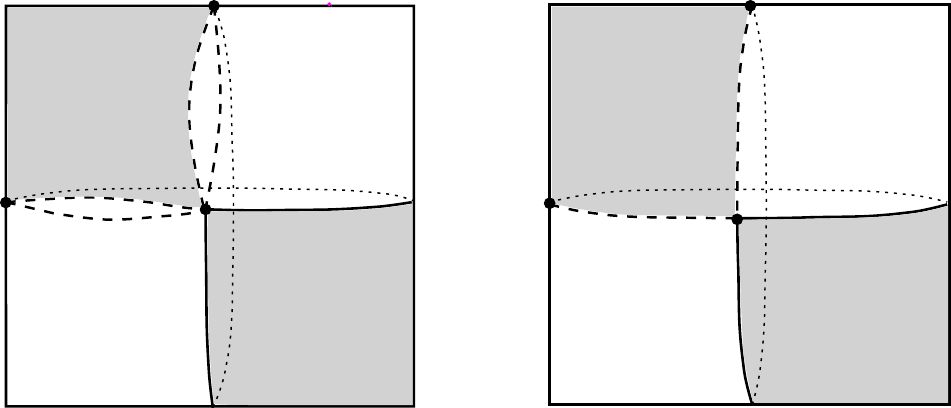}}%
    \put(0.49324384,0.20448652){\makebox(0,0)[lt]{\lineheight{1.25}\smash{\begin{tabular}[t]{l}$\cong$\end{tabular}}}}%
  \end{picture}%
\endgroup%

\caption{The base $B(\widehat \PP)$ of the flapped pillow $\widehat \PP$ from Figure \ref{fig:Lhatlattes} depicted in two different ways: as a subset of $\PPh$ and as the subset $\PP\setminus \bigcup_{e\in E} \inter(e)$ of $\PP$.} \label{fig:basepillow}
\end{figure}
%\end{center} 

We choose the  orientation on  $\widehat \PP$ so that the induced orientation on $B(\widehat \PP)$ considered as a subset of 
$\widehat \PP$ coincides with the orientation on 
$B(\widehat \PP)$ considered as a subset of the oriented sphere
$\PP$ (if we represent  orientations on surfaces by flags as described in \cite[Section~A.4]{THEBook}, then we simply pick 
a positively-oriented flag contained in $B(\widehat \PP)\sub \PP$  and declare  it to be positively-oriented in
$\widehat \PP\supset B(\widehat \PP)$).

The set $ B(\widehat \PP)\cong \PP\setminus \bigcup_{e\in E}
\inter(e)$
contains the vertex set $\La_n^{-1}(V)\sub \PP$ of  the graph $
\widetilde \CC =\La_n^{-1}(\CC)\sub \PP$. This means that we   can naturally view each vertex of  $\widetilde \CC $  also as a $1$-vertex in $\widehat \PP$. Let $\widehat{A},\,\widehat{B},\,
\widehat{C},\,\widehat{D}$ be the $1$-vertices of $\widehat{\PP}$ that correspond to the vertices $A,\,B,\,C,\,D$ of the original pillow, respectively. We set $\widehat{V}\coloneq\{\widehat{A},\widehat{B},\widehat{C},\widehat{D}\}$ and call $\widehat{A},\,\widehat{B},\,\widehat{C},\,\widehat{D}$ the \emph{vertices} of $\widehat{\PP}$.

%\begin{center}
%\begin{figure}[h]
%\includegraphics[scale=0.10]{skizze7.jpg}
%\caption{The mapping $\tilde{f}\colon \Phat\to P$ for $n=4$ (illustration only shows one face of each pillow.)}\label{fig1}
%\end{figure}
%\end{center}

Recall from Section \ref{subsec:lattes-maps} that the faces of the 
embedded graph $\widetilde{\CC}\sub \PP$ are colored black and white in a checkerboard manner. This coloring induces  a checkerboard coloring on the $1$-tiles of the flapped pillow $\widehat{\PP}$. The original map $ \La_n\:\PP\to \PP$ can now be naturally extended to a continuous map $\widehat{ \La}\:\widehat{\PP}\to \PP$ by reflection so that it preserves the coloring: $\widehat{ \La}$ maps each  $1$-tile of $\widehat{\PP}$ by a Euclidean similarity (scaling distances by the factor $n$) onto the $0$-tile of $\PP$ with the same color; see Figure \ref{fig:Lhatlattes} for an illustration.  On the base  $B(\widehat \PP)$  the map $\widehat \La$ agrees with  the original $(n\times n)$-Latt\`{e}s map $\La_n$
(if we consider $B(\widehat \PP)$ as a subset of $\PP$ by the identification~\eqref{eq:base-ident}).

It is clear that   $\widehat{ \La}\: \widehat{\PP}\ra \PP$ is a branched covering map.
This map is essentially the  Thurston map obtained from $\La_n$ by blowing up each arc $e\in E$ with muliplicity $m_e$. 
To make this more precise, we need a suitable  identification of the source  $\widehat{\PP}$ with the target $\PP$ of $\widehat{ \La}$ so that 
we obtain  a self-map on $\PP$.  For this we choose  a natural homeomorphism $\phi\:  \widehat{\PP}\to \PP$, which we will now define.

We view the set $\widehat \CC \coloneq  \widehat{\La}^{-1} (\CC)\subset \widehat \PP$ as a planar embedded graph, whose vertices and edges are precisely the $1$-vertices and the $1$-edges of the flapped pillow $\widehat \PP$.  Each $1$-edge of $\widehat \PP$ is homeomorphically mapped by $\widehat \La$ onto one of the edges of $\PP$. Similarly as before,  the $1$-edges of $\widehat \PP$ that are mapped by $\widehat \La$ onto $a$ or $c$ are called \emph{horizontal}, while the $1$-edges of $\widehat \PP$ that are mapped by $\widehat \La$ onto $b$ or $d$ are called \emph{vertical}.

There is a simple path of length $n$ in the graph $\widehat \CC$ that connects the vertices $\widehat A$ and $\widehat B$. Clearly, any such path consists only of horizontal $1$-edges in $\widehat \PP$. We denote by $\widehat a$ the realization of the chosen path in the sphere $\widehat{\PP}$, which is an arc in $(\widehat{\PP}, \widehat V)$. The arc $\widehat a$ may not be uniquely determined (namely, if flaps have been glued to slits obtained from edges $e\sub a$), but any two such arcs are isotopic rel.\ $\widehat V$. We define $\widehat b$, $\widehat c$,  $\widehat d$  in a similar way and  call $\widehat a,\, \widehat c$ the \emph{horizontal edges}, and $\widehat b,\,  \widehat d$ the \emph{vertical edges} of $\widehat{\PP}$.

We now choose an orientation-preserving homeomorphism 
 $\phi\:  \widehat{\PP}\to \PP$ that sends $\widehat{A},\,\widehat{B},\,\widehat{C},\, \widehat{D}$ to $A,\,B,\,C,\, D$, and $\widehat a, \,\widehat b, \,\widehat c, \,\widehat d$ to $a, \,b, \,c, \,d$, respectively. We define $f\coloneq\widehat{ \La}\circ \phi^{-1}$, which is a self-map on $\PP$. Clearly,  $f$ is a branched covering map on $\PP$. To refer to this map, 
 we  say that  $f\colon \PP \to \PP$ is  {\em obtained from the $(n\times n)$-Latt\`es map $\La_n$ by gluing $n_h$ horizontal and $n_v$ vertical flaps to $\PP$}. More informally, we call both  maps 
 $f$ and $\widehat \La$  a {\em blown-up Latt\`es map}. 
 
A point in $\widehat{\PP}$ is a critical point for $\widehat{ \La}\:\widehat{\PP}\to \PP$ if and only if it is on the boundary of at least four $1$-tiles subdividing $\widehat{\PP}$. This implies  that the set $C_{\widehat \La}$ of critical points of $\widehat \La$ consists of $1$-vertices of $\widehat \PP$ and that  each critical point of $\La_n$ is also a critical point for $\widehat{ \La}$ (recall that we can view each point in $\La_n^{-1}(V)\supset C_{\La_n}$ as a $1$-vertex in $\widehat \PP$). Moreover, if a $1$-vertex of $\widehat \PP$ is a critical point of 
$\widehat \La$, but not of $\La_n$, then it must be one of the points in $\widehat V$. For example, $\widehat A\in\widehat V$ is a critical point of $\widehat{ \La}$  if and only if a flap was  glued to an edge of $\widetilde{\CC}$ incident to  $A\cong \widehat A$. In any case, since $\widehat \La$ sends the $1$-vertices of $\widehat \PP$ to the vertices of $\PP$, the postcritical set of $f=\widehat{ \La}\circ \phi^{-1}$ coincides with the vertex set $V$. Thus, $f$ is a Thurston map.  

Since we assumed that  $\widehat \PP$ contains at least one flap
(that is, $n_h+n_v>0$), the orbifold of the Thurston map $f$ is hyperbolic. Indeed, each  $1$-vertex of $\widehat \PP$ that is a critical point of $\La_n$ is also  a critical point of $\widehat \La$ with the same or larger local degree.
Since we glued at least one flap, there is at least one $1$-vertex $v$ contained in six or more $1$-tiles of $\widehat \PP$. 
Then $\deg(\widehat \La,v)=\deg(f,v') \ge 3$, where $v'=\phi(v)$. 
Now $$X\coloneqq f(v')= \widehat \La(v)\in V=\{A,B,C,D\}=\postf, $$ 
 and so for the ramification function $\alpha_f$ of $f$ we have 
 $\alpha_f(X)\ge 3$. On the other hand, for all other points $Y\in  
 V=\postf$ we have $\alpha_{f}(Y)\ge \alpha_{\La_n}(Y)\ge 2$. 
 It then follows from  \eqref{eq:orb-char} that 
 $\chi(\mathcal{O}_f)<0$, and so $f$ has indeed a hyperbolic orbifold.

The homeomorphism $\phi$  chosen in the definition of $f$ 
 is not unique, but any two such 
 homeomorphisms are isotopic rel.\ $\widehat V$ (this easily  follows from \cite[Theorem A.5]{BuserGeometry}). This implies that  $f$ is uniquely determined up to Thurston equivalence. This  map may be viewed (up to Thurston equivalence) as the result of the blowing up operation introduced in Section~\ref{subsec:blow-up-general} applied to the edges $e \in E$ with the multiplicities $m_e$.  In particular, if we run the procedure for the map $\widehat{ \La}$ indicated in 
Figure~\ref{fig:Lhatlattes}, then we obtain the map $\widehat f$ illustrated in Figure~\ref{fig:fhatgeneral} (up to Thurston equivalence).

 %\begin{center}
%\vspace{12pt}
%\begin{figure}[h]
%\def\svgwidth{0.6\columnwidth}
%\input{generalblowup.pdf_tex}
%\caption{TO BE REPLACED} \label{fig:bigassblowupfigure}
%\end{figure}
%\end{center}

\section{Realizing  blown-up Latt\`{e}s maps}
\label{sec:realizing-blownup-lattes}
The goal of this section is to determine when a blown-up Latt\`{e}s map is realized by a rational map.  In particular, we will apply Thurston's criterion to prove Theorem~\ref{thm:flapped_intro}. The strategies and techniques used in the  proof will  highlight the main ideas needed for establishing the more  general Theorem~\ref{thm:blow-up-obstr}.

We fix $n\ge 2$, $n_h, n_v\ge 0$,  and follow the notation introduced in Section~\ref{subsec:blownup-lattes}.
In particular, we denote by 
$\widehat{\PP}$ a flapped pillow with $n_h$ horizontal and $n_v$ vertical flaps, by $\widehat{ \La}\:\widehat{\PP}\to \PP$ the respective blown-up $(n\times n)$-Latt\`es map, and by 
$\phi\:\widehat{\PP}\to \PP$ the identifying homeomorphism.  Then 
 $f\:\PP\to \PP$  given as   $f=\widehat{ \La}\circ \phi^{-1}$ is the Thurston map under consideration. We will assume that  $n_h+n_v>0$, and so  $\widehat \PP$ has at least one flap. In this case, $f$ has a hyperbolic orbifold as we have seen,  and so we can apply Thurston's criterion.  For this we consider   essential Jordan curves $\gamma$ in $(\PP,\postf)=(\PP,V)$ and study their  (essential) pullbacks under $f$.

If $\ga$ is such  a curve, then the homeomorphism $\phi$ sends the pullbacks of $\gamma$ under $\widehat \La$ to the pullbacks of $\gamma$ under $f$. So in order to understand the isotopy types and mapping properties of the pullbacks under $f$, we will instead look at  the pullbacks of $\ga$ under $\widehat \La$. In particular, if $\widehat \gamma$ is a pullback of $\gamma$ under $\widehat \La$, then $\deg(\widehat \La\: \widehat \gamma \to \gamma) = \deg(f\: \phi(\widehat \gamma) \to \gamma)$ and $\phi(\widehat \gamma)$ is essential in $(\PP,\postf)=(\PP,V)$ if and only if $\widehat \gamma$ is essential in $(\widehat \PP, \widehat V)$, where $\widehat V$ denotes the vertex set of the flapped pillow $\widehat \PP$.

 Since  the mapping $\widehat \La\:\widehat \PP \to \PP$ is a similarity map on each $1$-tile of  $\widehat \PP$, the preimage $\widehat \La^{-1} (\ga)$ of a Jordan curve $\ga$ in $(\PP,\postf)$ (or of any subset $\ga$ of $\PP$), can  be obtained in the following intuitive way: we rescale and copy the part of $\ga$ that belongs to the white $0$-tile of  $\PP$ into each white $1$-tile of  $\widehat \PP$ and
the part of $\ga$ that belongs to the black  $0$-tile  of $\PP$ into each black  $1$-tile.

\subsection{The horizontal and  vertical curves}\label{subsec:obstructed pillow} 

Recall that $\alpha^h$ and $\alpha^v$ (see \eqref{eq:hori+vert}) denote the Jordan curves in $(\PP,V)=(\PP,\postf)$ that separate the two horizontal and the two vertical edges of $\PP$, respectively. These two curves  are invariant under $f$ and will play a crucial role in the  considerations of this section. 

%\begin{center}\begin{figure}[h]
%\def\svgwidth{0.6\columnwidth}
%\input{Pillow9b.pdf_tex}
%\caption{Pullbacks of $\alpha^h$ for a blown-up 
%$(4\times 4)$-Latt\`es map with $n_h=n_v=1$.}
%\label{fig:horivertiflap} 
%\end{figure}
%\end{center}

%\begin{center}
\begin{figure}[t]
\def\svgwidth{0.72\columnwidth}
%% Creator: Inkscape 1.0.1 (c497b03c, 2020-09-10), www.inkscape.org
%% Accompanies image file 'blowupwice.pdf' (pdf, eps, ps)
\begingroup%
  \makeatletter%
  \providecommand\color[2][]{%
    \errmessage{(Inkscape) Color is used for the text in Inkscape, but the package 'color.sty' is not loaded}%
    \renewcommand\color[2][]{}%
  }%
  \providecommand\transparent[1]{%
    \errmessage{(Inkscape) Transparency is used (non-zero) for the text in Inkscape, but the package 'transparent.sty' is not loaded}%
    \renewcommand\transparent[1]{}%
  }%
  \providecommand\rotatebox[2]{#2}%
  \newcommand*\fsize{\dimexpr\f@size pt\relax}%
  \newcommand*\lineheight[1]{\fontsize{\fsize}{#1\fsize}\selectfont}%
  \ifx\svgwidth\undefined%
    \setlength{\unitlength}{502.25221772bp}%
    \ifx\svgscale\undefined%
      \relax%
    \else%
      \setlength{\unitlength}{\unitlength * \real{\svgscale}}%
    \fi%
  \else%
    \setlength{\unitlength}{\svgwidth}%
  \fi%
  \global\let\svgwidth\undefined%
  \global\let\svgscale\undefined%
  \makeatother%
  \begin{picture}(1,0.32399527)%
    \lineheight{1}%
    \setlength\tabcolsep{0pt}%
    \put(0,0){\includegraphics[width=\unitlength,page=1]{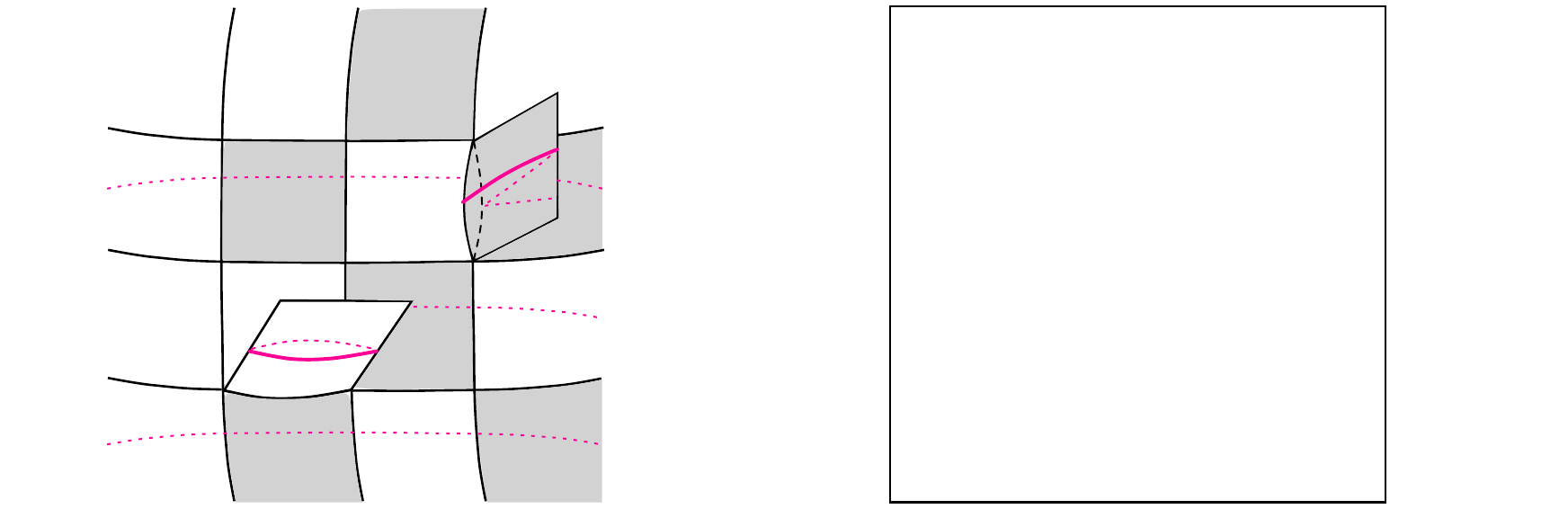}}%
    \put(0.46637644,0.19211183){\color[rgb]{0,0,0}\makebox(0,0)[lt]{\lineheight{1.25}\smash{\begin{tabular}[t]{l}$\widehat{\LL}$\end{tabular}}}}%
    \put(0,0){\includegraphics[width=\unitlength,page=2]{blowupwice.pdf}}%
    \put(0.76183716,0.11432192){\color[rgb]{0,0,0}\makebox(0,0)[lt]{\lineheight{1.25}\smash{\begin{tabular}[t]{l}\textcolor{anisred}{$\alpha^h$}\end{tabular}}}}%
    \put(0,0){\includegraphics[width=\unitlength,page=3]{blowupwice.pdf}}%
    \put(-0.00188306,0.15346579){\color[rgb]{0,0,0}\makebox(0,0)[lt]{\lineheight{1.25}\smash{\begin{tabular}[t]{l}$\PPh$\end{tabular}}}}%
    \put(0.92871085,0.15539289){\color[rgb]{0,0,0}\makebox(0,0)[lt]{\lineheight{1.25}\smash{\begin{tabular}[t]{l}$\PP$\end{tabular}}}}%
  \end{picture}%
\endgroup%

\caption{Pullbacks of $\alpha^h$ for a blown-up 
$(4\times 4)$-Latt\`es map with $n_h=n_v=1$.}
\label{fig:horivertiflap} 
\end{figure}
%\end{center}

\begin{lemma}\label{lem:blown-lattes-horizontal}
Let $f=\widehat \La \circ \phi^{-1} \colon \PP \to \PP$ be a Thurston map obtained from the $(n\times n)$-Latt\`es map, $n\ge 2$,  by gluing $n_h\ge 0$ horizontal and $n_v\ge0$ vertical flaps to $\PP$. Then the following statements are true:
\begin{enumerate}[label=\text{(\roman*)},font=\normalfont]
\item The Jordan curve $\alpha^h$ has $n+n_h$ pullbacks under $f$.  Exactly $n$ of these pullbacks are essential. Each of these essential pullbacks  is isotopic to~$\alpha^h$ rel.\ $\postf$. 

\smallskip
\item If $\widetilde{\alpha}$ is one of the $n$ essential pullbacks of $\alpha^h$, then $\deg(f\:\widetilde{\alpha}\to\alpha^h)=n+n_{\widetilde{\alpha}}$, where  $n_{\widetilde{\alpha}}\ge 0$ is the number 
of distinct vertical flaps in $\widehat \PP$ that $
\phi^{-1} (\widetilde{\alpha})$ meets.  

\smallskip
\item We have 
$$\displaystyle \lambda_f(\alpha^h)=\sum_{\widetilde{\alpha}} \frac{1}{n+n_{\widetilde{\alpha}}},$$ where the sum is taken over all essential pullbacks $\widetilde \alpha$ of $\alpha^h$ under~$f$. 
\end{enumerate}
Analogous statements are true for the curve $\alpha^v$. 
\end{lemma}

\begin{proof} 
Figure \ref{fig:horivertiflap} illustrates the proof. It is obvious that the Jordan curve   $\alpha^h$ has exactly $n+n_h$ distinct pullbacks under $\widehat{\La}$. Among them, there are $n$ essential pullbacks $\widehat \alpha_1,\dots, \widehat \alpha_n$ that separate the two horizontal edges of $\widehat \PP$ and thus are isotopic to each other relative to the vertex set $\widehat V$ of $\widehat \PP$.  For each $j=1,\dots, n$, the image $\alpha_j\coloneq \phi(\widehat \alpha_j)$ is isotopic to $\alpha^h$. Moreover, we have $\deg(\widehat \La\: \widehat \alpha_j\to\alpha^h)=n+n_{{\alpha}_j}$. The other $n_h$ pullbacks of $\alpha^h$ under $\widehat \La$ are each contained in one of the horizontal flaps and thus are peripheral in $(\widehat \PP, \widehat V)$.  Consequently, $$\lambda_f(\alpha^h)=\sum_{j=1}^n \frac{1}{\deg(f\: {\alpha}_j\to\alpha^h)}=\sum_{j=1}^n \frac{1}{\deg(\widehat \La\: \widehat \alpha_j\to\alpha^h)}=\sum_{j=1}^n \frac{1}{n+n_{{\alpha}_j}}.$$
This completes the proof of the lemma for the curve $\alpha^h$. The proof for the curve $\alpha^v$ follows from similar considerations. 
\end{proof}

The following corollary is an immediate consequence of the previous lemma.

\begin{cor}\label{cor:the horizontal curve}
Let $f=\widehat \La \circ \phi^{-1} \colon \PP \to \PP$ be a Thurston map obtained from the $(n\times n)$-Latt\`es map, $n\ge 2$,  by gluing $n_h\ge 0$ horizontal and $n_v\ge0$ vertical flaps to $\PP$. Then  $\alpha^h$ is an obstruction (for $f$) if and only if $n_v=0$,  and 
$\alpha^v$ is an obstruction  if and only if $n_h=0$.
\end{cor}

\begin{proof} Let us first suppose that $n_v=0$, that is, $\widehat{\PP}$ does not have any vertical flaps. Then by Lemma \ref{lem:blown-lattes-horizontal}, $\alpha^h$ has $n$ essential pullbacks under $f$,  each of which is mapped onto $\alpha^h$ with degree $n$ (this is illustrated in Figure~\ref{fig:horiflap} in a special case). Consequently, $\lambda_f(\alpha^h)= n\cdot (1/n)=1$, which means $\alpha^h$ is an  obstruction for $f$.  

If  $n_v>0$, the flapped pillow $\widehat \PP$ has at least one vertical flap. Then $n_{\widetilde{\alpha}}>0$ for at least one essential pullback $\widetilde{\alpha}$ of $\alpha^h$. Lemma \ref{lem:blown-lattes-horizontal} implies that
$ \lambda_f(\alpha^h)\leq (n-1)\frac{1}{n}+\frac{1}{n+1}<1, $
and so  $\alpha^h$ is not an obstruction for $f$.

The proof for the vertical curve $\alpha^v$ is completely analogous. 
\end{proof}

The above corollary can be read as follows: the obstruction $\alpha^h$ for the $(n\times n)$-Latt\`es map can be eliminated by gluing a vertical flap to $\PP$. Similarly, the obstruction  $\alpha^v$ can be eliminated by gluing a horizontal flap. We will show momentarily that if both of these obstructions are eliminated (that is, if there are both horizontal and vertical flaps) then no other obstructions are present and so the map $f$ is realized. 

%\begin{center}
\begin{figure}[t] 
%\def\svgwidth{0.7\columnwidth}
%\input{Pillow8c.pdf_tex}

%\vspace{24pt}

\def\svgwidth{0.71\columnwidth}
%% Creator: Inkscape 1.0.1 (c497b03c, 2020-09-10), www.inkscape.org
%% Accompanies image file 'blowuponce.pdf' (pdf, eps, ps)
\begingroup%
  \makeatletter%
  \providecommand\color[2][]{%
    \errmessage{(Inkscape) Color is used for the text in Inkscape, but the package 'color.sty' is not loaded}%
    \renewcommand\color[2][]{}%
  }%
  \providecommand\transparent[1]{%
    \errmessage{(Inkscape) Transparency is used (non-zero) for the text in Inkscape, but the package 'transparent.sty' is not loaded}%
    \renewcommand\transparent[1]{}%
  }%
  \providecommand\rotatebox[2]{#2}%
  \newcommand*\fsize{\dimexpr\f@size pt\relax}%
  \newcommand*\lineheight[1]{\fontsize{\fsize}{#1\fsize}\selectfont}%
  \ifx\svgwidth\undefined%
    \setlength{\unitlength}{500.62163773bp}%
    \ifx\svgscale\undefined%
      \relax%
    \else%
      \setlength{\unitlength}{\unitlength * \real{\svgscale}}%
    \fi%
  \else%
    \setlength{\unitlength}{\svgwidth}%
  \fi%
  \global\let\svgwidth\undefined%
  \global\let\svgscale\undefined%
  \makeatother%
  \begin{picture}(1,0.32468064)%
    \lineheight{1}%
    \setlength\tabcolsep{0pt}%
    \put(0,0){\includegraphics[width=\unitlength,page=1]{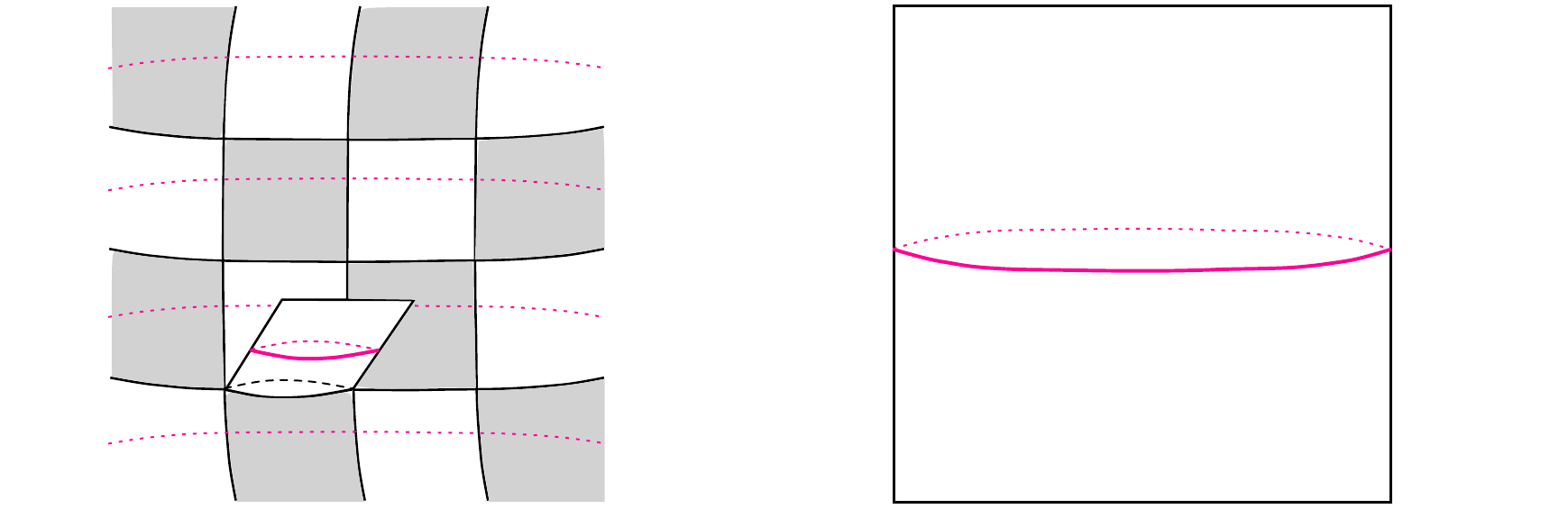}}%
    \put(0.76512657,0.11469411){\color[rgb]{0,0,0}\makebox(0,0)[lt]{\lineheight{1.25}\smash{\begin{tabular}[t]{l}\textcolor{anisred}{$\alpha^h$}\end{tabular}}}}%
    \put(0,0){\includegraphics[width=\unitlength,page=2]{blowuponce.pdf}}%
    \put(0.46915009,0.19390429){\color[rgb]{0,0,0}\makebox(0,0)[lt]{\lineheight{1.25}\smash{\begin{tabular}[t]{l}$\widehat{\LL}$\end{tabular}}}}%
    \put(0,0){\includegraphics[width=\unitlength,page=3]{blowuponce.pdf}}%
    \put(-0.0018892,0.1555235){\color[rgb]{0,0,0}\makebox(0,0)[lt]{\lineheight{1.25}\smash{\begin{tabular}[t]{l}$\PPh$\end{tabular}}}}%
    \put(0.92847865,0.15624925){\color[rgb]{0,0,0}\makebox(0,0)[lt]{\lineheight{1.25}\smash{\begin{tabular}[t]{l}$\PP$\end{tabular}}}}%
  \end{picture}%
\endgroup%
 
\caption{Pullbacks of $\alpha^h$ for a blown up 
$(4\times 4)$-Latt\`es map with $n_h=1$ and $n_v=0$.}
\label{fig:horiflap} 
\end{figure}
%\end{center}

\subsection{Ruling out other obstructions} \label{subsec:non-obstructed pillow}

Now we discuss what happens with the essential curves in $(\PP,\postf)$ that are not isotopic to the horizontal curve $\alpha^h$ or the vertical curve $\alpha^v$.

\begin{theorem}\label{thm:flapped_proven}
Let $f=\widehat \La \circ \phi^{-1} \colon \PP \to \PP$ be a Thurston map obtained from the $(n\times n)$-Latt\`es map, $n\ge 2$,  by gluing $n_h\ge 0$ horizontal and $n_v\ge0$ vertical flaps to $\PP$ and assume 
that $n_h+n_v>0$.  If  $\gamma\sub \PP\setminus  \postf $ is an essential Jordan curve that is not isotopic to either $\alpha^h$ or $\alpha^v$, then $\gamma$ is not an obstruction for $f$.
\end{theorem}

Before we turn to the proof of this theorem, we first record how it  implies Theorem~\ref{thm:flapped_intro} stated in the introduction.

\begin{proof} [Proof of Theorem~\ref{thm:flapped_intro}]
Let $n$, $f$, $n_h$, and  $n_v$ with $n_h+n_v>0$  be as in the statement. We have seen in Section~\ref{subsec:blownup-lattes}
that then $f$ has a hyperbolic orbifold. If $n_h=0$ or $n_v=0$, then 
by Corollary~\ref{cor:the horizontal curve} the curve $\alpha^v$ or the curve $\alpha^h$ is an obstruction, respectively. 

If $n_h>0$ and 
$n_v>0$, then $f$ has no obstruction as follows from 
Corollary~\ref{cor:the horizontal curve} and Theorem~\ref{thm:flapped_proven}. Since $f$ has a hyperbolic orbifold, in this case $f$ is realized by a rational map according to  Theorem~\ref{thm:Thurston}.
 \end{proof}

Corollary~\ref{cor:the horizontal curve} and  Theorem~\ref{thm:flapped_proven} also imply that if $n_h=0$, then 
$\alpha^v$ is the only obstruction for $f$ (up to isotopy rel.\ $\postf$). Similarly, $\alpha^h$ is the only obstruction if $n_v=0$.

Before we go into the details, we  will give an outline for the proof of 
Theorem~\ref{thm:flapped_proven}.  
We argue  by contradiction and assume that $f$ has an 
 obstruction given by an essential Jordan curve $\gamma$ in $(\PP,\postf)$
 that is isotopic to neither $\alpha^h$ nor $\alpha^v$ rel.\ $\postf$. 
 Then  $\lambda_f(\gamma)\geq 1$. Let $ \gamma_1,\dots,  \gamma_k$ for some  $k\in \N$  be all the essential pullbacks of $\gamma$ under $f$, which must be isotopic to $\gamma$ rel.\ 
 $\postf$. 

Using  facts about intersection numbers and the mapping properties of $f$, one can  show that for the number of essential pullbacks of $\gamma$ we have $k \le n$ and that the corresponding mapping degrees satisfy $\deg(f\:  \gamma_j\to \gamma)\geq n$ for all $j=1,\dots, k$. Since $\lambda_f(\gamma)\geq 1$, it follows  that there are exactly $k=n$ essential pullbacks and that $\deg(f\:  \gamma_j\to \gamma)=n$ for each $j=1, \dots, n$. 

This in turn  implies that none of the essential pullbacks $\widehat \gamma_j\coloneq \phi^{-1}( \gamma_j)$ of $\gamma$ under $\widehat \La$ goes over a flap in $\widehat \PP$.
Then  all the $n$ pullbacks $\widehat \gamma_1,\dots, \widehat \gamma_n$ belong to the base $B(\widehat \PP)$ of $\widehat \PP$. This means that each $\widehat \gamma_j$ can  be thought of as a pullback of $\gamma$ under the original $(n\times n)$-Latt\`es map $\La_n$. However, there are only $n$ pullbacks of $\gamma$ under $\La_n$, which cross all the edges of the graph $\La_n^{-1}(\CC)$, where $\CC$ is the common boundary of the $0$-tiles in $\PP$. Consequently, the pullbacks $\widehat \gamma_1,\dots, \widehat \gamma_n$ cross all the $1$-edges in the closure of the base $B(\widehat \PP)$. It follows that one of the pullbacks $\widehat \gamma_j$ must cross one of the base edges of a flap $F$ in $\widehat \PP$, which would necessarily mean that  $\widehat \gamma_j$ goes over the flap $F$. This gives the desired contradiction and Theorem~\ref{thm:flapped_proven} follows.

In the remainder  of this section we will fill in the details for this outline.  First, we establish  several general facts about degrees and intersection numbers. 

Let $n\in \N$ and $f\: X\ra Y$ be a map between two sets $X$ and $Y$. We say that  $f$  is {\em at most $n$-to-$1$}  if $\#f^{-1}(y)\le n$ for each $y\in Y$. 
 We say that $f$ is  {\em $n$-to-$1$}
 if $\#f^{-1}(y)= n$ for each $y\in Y$.

\begin{lemma}\label{lem:preimage-count}
Let $f\: X\ra Y$ be a map between two sets $X$ and $Y$.
Suppose $M\sub X$, $N\sub Y$, and  $f|M\:M\to f(M)$ is at most  $n$-to-$1$  for some $n\in \N$. Then 
$$\#(M \cap f^{-1}(N)) \le  n \cdot \#(f(M)\cap N).$$
\end{lemma}

\begin{proof} The map $f$ sends each point in $M \cap f^{-1}(N)$
 to a point in $f(M)\cap N$. Moreover, 
 each point in $f(M)\cap N$  has at most  $n$ preimages in 
 $M\cap f^{-1}(N)$ under $f$. The statement follows. 
\end{proof}

\begin{lemma}\label{lem:preimage-bound} Let 
$f\: S^2\ra S^2$ be a Thurston map with $\#\postf=4$, and $\ga$ be an essential 
Jordan curve in $(S^2, \postf)$. Suppose that $\widetilde \alpha$ is an essential Jordan curve or an arc in $(S^2,\postf)$ such that
\begin{enumerate}[label=\text{(\roman*)},font=\normalfont,leftmargin=*]

\smallskip 
\item \label{item:bd1}
 $f(\widetilde\alpha)$ and $\widetilde \alpha$ are isotopic  rel.\ $\postf$, 

\smallskip 
\item \label{item:bd2} the map $f|\widetilde \alpha: \widetilde \alpha \to f( \widetilde \alpha)$ is at most $n$-to-$1$, where $n\in \N$,

\smallskip 
\item  \label{item:bd3} $\ins(f(\widetilde \alpha), \gamma)=
\#(f(\widetilde \alpha)\cap \gamma)
>0$. 
\end{enumerate}

Then $k \le n$,  where  $k\in \N_0$ denotes  the number of distinct pullbacks of $\gamma$ under $f$ that are isotopic to $\gamma$
rel.\ $\postf$.  Moreover, if $\widetilde \alpha$ meets a peripheral pullback of $\ga$, then $k<n$.
\end{lemma}

 In the formulation and the ensuing proof  intersection numbers are with respect to $(S^2, P_f)$.  Note that since $f(\widetilde\alpha)$ and $\widetilde \alpha$ are isotopic  rel.\ $\postf$ by assumption, $f(\widetilde\alpha)$ is of the same type as $\widetilde \alpha$, that is,  a Jordan curve or an arc in $(S^2,\postf)$.

\begin{proof} 
Let  $\gamma_1,\dots, \gamma_k$ be the distinct  pullbacks of $\gamma$ under $f$ that are isotopic to $\gamma$ rel.\ $\postf$. Since $f|\widetilde \alpha\:\widetilde \alpha\ra  f(\widetilde \alpha)$ is at most $n$-to-$1$, we can apply Lemma \ref{lem:preimage-count} and conclude that
$$\#(\widetilde \alpha \cap f^{-1}(\gamma))\le  n \cdot \#( f(\widetilde \alpha) \cap \ga)=
n\cdot  \ins (f(\widetilde \alpha) , \gamma).$$
On the other hand,
\begin{equation}\label{eq:degree-bound-curves}
n\cdot \ins (f(\widetilde \alpha) , \gamma)\ge \#(\widetilde \alpha \cap f^{-1}(\gamma))\geq \sum_{j=1}^k \#(\widetilde \alpha \cap \gamma_j  ) \geq \sum_{j=1}^k \ins(\widetilde \alpha, \gamma_j)=
k \cdot  \ins(f(\widetilde \alpha) , \gamma).
\end{equation} 
Since $\ins(f(\widetilde \alpha) ,\gamma)>0$, we see that $k \le n$. 
If $\widetilde \alpha$ meets a peripheral pullback of $\ga$, then the second inequality in \eqref{eq:degree-bound-curves} is strict and we actually have $k <n$. 
\end{proof}

The next result will lead to  the  strict inequality from Lemma~\ref{lem:preimage-bound} in the proof of Theorem~\ref{thm:flapped_proven}.

\begin{lemma}\label{lem:curves-over-flap} As before, 
let $\widehat \La\: \widehat \PP \to \PP$ be the blown-up $(n\times n)$-Latt\`es map, and suppose  $\gamma=\tau_{r/s}$ is a simple closed geodesic in $\PP$ with slope $r/s\in \widehat \Q\setminus \{0,\infty\}$. Let $\widehat \gamma$ be a pullback of $\gamma$ under $\widehat \La$. If $\widehat \gamma$ intersects the interior of a base edge of a flap $F$ in $\widehat \PP$, then $\widehat \gamma$ also intersects the top edge of~$F$.
\end{lemma}

\begin{proof} We have  $\gamma=\wp(\ell_{r/s})$, where $\ell_{r/s}\subset\C\setminus \Z^2$ is a straight line with slope $r/s\ne 0,\infty$.  Then 
$\ga$ is an essential Jordan curve in $(\PP, V)$. Let  $\widehat\gamma\subset \widehat \PP$ be as in the statement.  As in Section \ref{subsec:pillow}, 
we denote by  $a,b,c,d$  the edges of the pillow $\PP$.  Let $e'\sub \Phat$ be a base edge of a flap $F$ in $\widehat \PP$ such that $\widehat{\gamma}\cap \inter(e') \neq \emptyset$.  We will assume that $F$ is a horizontal flap. Then   $\widehat \La(e')=a$ or  $\widehat \La(e')=c$. We will make the further assumption that $\widehat \La(e')=a$. The other cases, when $\widehat \La(e')=c$ or when $F$ is a vertical flap, can be treated in a way that is completely analogous to the ensuing argument.

Let  $e''\sub F$ be the base edge of $F$ different from $e'$, and $\widetilde e$ be the top edge of $F$.
Then $\widehat \La(e'')=a$ and $\widehat \La(\widetilde e)=c$.   Moreover, 
\begin{equation}\label{eq:pre-on-flap} 
F\cap \widehat \La^{-1}(a)=e'\cup e''=\partial F \quad \text{and}\quad 
F\cap \widehat \La^{-1}(c) = \widetilde e.
\end{equation}

We have  $\gamma=\wp(\ell_{r/s})$ with $r/s\ne 0$, and so  by Lemma~\ref{lem:alter}  the sets 
$a\cap \ga$ and $c\cap \ga$ are finite and non-empty and the points in these  sets  alternate on $\gamma$. Since $\widehat \La$ is a covering map
from $\widehat \gamma$ onto  $\ga$,  we conclude  that the sets  $\widehat \gamma \cap  \widehat \La^{-1}(a) $ and $\widehat \gamma \cap  \widehat \La^{-1}(c)$  are also  finite and  non-empty and the points in  these sets alternate on $\widehat \gamma$.

We choose a point $p\in \widehat \gamma \cap \inte(e')$. Since $\gamma=\wp(\ell_{r/s})$, the curve $\gamma$ has transverse intersections with $\inter(a)$. It follows that the pullback $\widehat\gamma$ has a transverse intersection with $\inter(e')$ at $p$, and so  $\widehat \gamma $ crosses into the interior of the flap $F$ at $p$. Therefore, if we travel along $\widehat \gamma $ starting at $p\in \widehat \gamma \cap  \widehat \La^{-1}(a)$ and traverse into the interior of the flap $F$, we must meet   $\widehat \La^{-1}(c)$ before we possibly exit $F$  through its boundary  $\partial F=e'\cup e''=F\cap \widehat \La^{-1}(a)$  (see \eqref{eq:pre-on-flap}). Now  $F\cap \widehat \La^{-1}(c)=\widetilde  e$ and so this 
 implies that the pullback $\widehat \gamma$ meets the top edge $\widetilde e$ (see Figure \ref{fig_run_into_flap} for an illustration).  The statement follows. \end{proof}

%\begin{center}
%\vspace{10pt}
\begin{figure}[t]
\def\svgwidth{0.5\columnwidth}
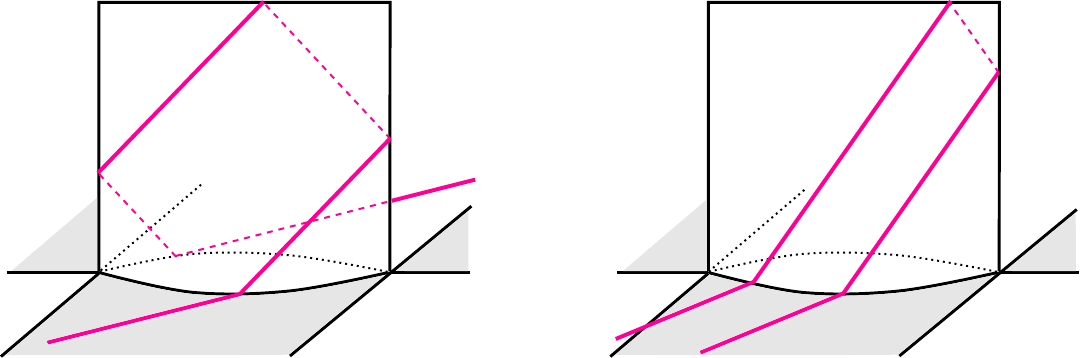
\caption{A pullback $\widehat \gamma$ going over a horizontal flap in $\widehat \PP$.} \label{fig_run_into_flap}
\end{figure}
%\end{center}

We are now ready to prove the main result of this section.

\begin{proof}[Proof of Theorem~\ref{thm:flapped_proven}] 
Let $f\: \PP\ra \PP$ be  a Thurston map as in the statement, 
obtained from 
the $(n\times n)$-Latt\`es map $\La_n$, $n\ge 2$,  by gluing $n_h\ge 0$ horizontal and $n_v\ge0$ vertical flaps, where  we assume $n_h+n_v>0$.  As described in the beginning of this section, then $f=\widehat \La \circ \phi^{-1}$, where $\widehat \La\: \widehat \PP \to \PP$ is a branched covering map  on the associated flapped pillow $ \widehat \PP$ and 
$\phi\:  \widehat \PP \ra \PP$ is an identifying homeomorphism as discussed in 
Section \ref{subsec:blownup-lattes}.  Note that $ \widehat \PP$
has at least one flap, since  $n_h+n_v>0$. 
Following the notation from Section \ref{subsec:blownup-lattes}, we denote by 
$\widehat a$, $\widehat b$, $\widehat c$, $\widehat d$ the arcs in 
  $(\widehat \PP, \widehat V)$ that  under $\phi$ correspond to the edges  
  $a$, $b$, $c$, $d$ of $\PP$, respectively.

We now argue by contradiction and assume that there exists  an essential Jordan curve $\gamma$ in $(\PP, \postf)=(\PP,V)$ that is not  isotopic to  $\alpha^h$ or $\alpha^v$ rel.\ $\postf=V$, but is   an obstruction for $f$, that is, $\lambda_f(\gamma) \geq 1$.
Since we can replace $\ga$ with  any curve in the same isotopy class, by Lemma~\ref{lem:isoclassesP} we may assume that $\gamma = \wp(\ell_{r/s})$ for a straight line $\ell_{r/s}\sub \C\setminus\Z^2$ with  slope  $r/s\in \widehat \Q$.  Since $\gamma$ is not isotopic to 
$\alpha^h$ or $\alpha^v$ rel.\ $V$, we have $r/s\ne 0,\infty$, and so $r,s\ne 0$.   Then it follows from  Lemma \ref{lem:i-properties-curve} ~\ref{item:i3} and~\ref{item:i4}
that \begin{align}\label{eq:gaamin} 
\#(a\cap \ga)= \ins(a, \ga)= |&r|=  \ins(c, \ga) =\#(c\cap \ga)>0,
\\
\#(b\cap \ga)=\ins(b,\ga) =&s = \ins (d,\ga) =\#(d\cap \ga) >0. \notag
 \end{align} 
 In particular, $\gamma\sub \PP\setminus V$ intersects the interiors of all  four edges of   $\PP$.
 
 Let $\gamma_1, \dots, \gamma_k$ for some $k\in \N$ denote all the pullbacks of $\gamma$ under $f$ that are isotopic to $\gamma$ rel.\ $\postf$. By construction of the blown-up map, the arc  $\widehat a=\phi^{-1}(a) $ consists of $n$ $1$-edges in $\widehat \PP$,  each of which is homeomorphically mapped onto $a$ by $\widehat \La$. This implies that the map $f|a\: a\to a$ is at most $n$-to-$1$. 
By \eqref{eq:gaamin}  we can apply Lemma~\ref{lem:preimage-bound} to $a$ and $\gamma$ and conclude   that $k\leq n$.

  By Lemma \ref{lem:blown-lattes-horizontal}, the horizontal curve $\alpha^h$ has $n$ distinct pullbacks under $f$ that are isotopic to $\alpha^h$ rel.\ $\postf$. Since 
  \begin{equation}\label{eq:ahgamin}
\ins(\gamma, \alpha^h)=\#(\gamma \cap \alpha^h) =2|r|>0
\end{equation} 
by Lemma~\ref{lem:i-properties-curve}~\ref{item:i1}, we can apply  Lemma~\ref{lem:preimage-bound} again, this time  to $\gamma_j$ and $\alpha^h$ (in the roles of $\widetilde \alpha$ and $\ga$, respectively),
  and conclude that  $n\leq \deg(f\colon \gam_j\to\gam)$ for all $j=1,\dots,k$.  

Then   
$$1\le \lambda_f(\gamma)=\sum_{j=1}^k\frac{1}{\deg(f\colon \gam_j\to\gam)}\le k/n\le 1. $$ 
Therefore,  $k=n$ and $\deg(f\colon \gam_j\to\gam)= n$ for each $j=1,\dots,n$. 

This shows that the curve $\gamma$ has exactly $n$ essential pullbacks under $\widehat \La$ given by $\widehat \gamma_1\coloneq \phi^{-1}(\gamma_1), \dots, \widehat \gamma_n\coloneq \phi^{-1}(\gamma_n)$. Here, the isotopy classes are considered with respect to the vertex set $\widehat V$ of $\widehat \PP$. We will now use the second part of Lemma~\ref{lem:preimage-bound} to show that none of these pullbacks goes over a flap in $\widehat \PP$. 

\begin{claim}
For each flap $F$ in  $\widehat \PP$ and each pullback $\widehat \gamma_j$, $j=1,\dots,n$, we have $F\cap \widehat \gamma_j = \emptyset$. 
\end{claim}

%\begin{center}
%\begin{figure}[h]
%\includegraphics[scale=0.11]{skizze11_neu.jpg}
%\caption{The curve $\ga_m$ and the flap $F$.}\label{fig_run_into_flap}
%\end{figure}
%\end{center} 

To see that the Claim is true, suppose some pullback $\widehat \gamma_j$ meets a flap $F$ in $\widehat \PP$. We may assume that $F$ is a horizontal flap; the other case, when $F$ is vertical, can be treated by similar considerations. Then $F$ contains a peripheral pullback $\widehat \alpha^h$ of $\alpha^h$ under $\widehat \La$, which separates the union $\partial F$ of the two base edges of $F$  from the top edge of $F$. We will first show that $\widehat \gamma_j$ intersects $\widehat\alpha^h$.

Note that $\partial F$ is a Jordan curve and that $\inter(F)$ does not contain any point from the vertex set $\widehat V$ of $\widehat \PP$. It follows that the curve $\widehat \gamma_j$ must intersect $\partial F$, because $\widehat \gamma_j$ is essential in $(\widehat \PP, \widehat V)$. Since the curve $\gamma$ does not pass through $\postf=V$, its pullback $\widehat \gamma_j$ under $\widehat \La$  does not pass through any $1$-vertex in $\widehat \PP$. Consequently, $\widehat \gamma_j$ must meet the interior of one of the two base edges of $F$, which compose the boundary $\partial F$. Lemma \ref{lem:curves-over-flap} now implies that $\widehat \gamma_j$ also meets the top edge of $F$. Therefore, $\widehat \gamma_j$ meets the peripheral pullback $\widehat \alpha^h$ in $F$. 

It follows that $\gamma_j=\phi(\widehat \gamma_j)$ meets the peripheral pullback $\phi(\widehat \alpha^h)$ of $\alpha^h$ under $f$.  Lemma \ref{lem:preimage-bound} now  implies that $n< \deg(f\colon \gamma_j\to \gam)=n$, which is a contradiction. This finishes the proof of the Claim.

\smallskip

The Claim implies that each essential pullback $\widehat \gamma_j$, $j=1,\dots, n$, belongs to the base  $B(\widehat \PP)$
of $\widehat \PP$. By  \eqref{eq:base-ident}
we can identify $B(\widehat \PP)$ with the subset $\PP\setminus \bigcup_{e\in E} \inter(e)$ of the original pillow $\PP$, where  $E$ denotes 
the non-empty subset of all $1$-edges of $\PP$ along which flaps were glued in the construction of $\widehat \PP$. 

 Under this identification, the map $\widehat \La$ on $B(\widehat \PP)$ coincides with the $(n\times n)$-Latt\`{e}s map $\La_n$. Thus we may view $\widehat \gamma_1, \dots, \widehat \gamma_n$ as pullbacks of $\gamma$ under $\La_n$ on the original pillow $\PP$.
 Now  $\gamma$ has exactly $n$ pullbacks under $\La_n$ (see \eqref{eq:lattes-preimage}). This implies that 
 $$  \La_n^{-1}(\ga)=\widehat  \ga_1\cup\dots \cup \widehat \ga_n. $$ 
 Since  $\gamma$ meets the interior of every edge of $\PP$, the set 
 $\La_n^{-1}(\gamma)=\widehat  \ga_1\cup\dots \cup \widehat \ga_n $ meets the interior of every $1$-edge of $\PP$.  This is impossible, because 
 $$\widehat  \ga_1\cup\dots \cup \widehat \ga_n \sub B(\widehat \PP) =\PP\setminus \bigcup_{e\in E} \inter(e)$$ does not meet the interior of  any $1$-edge in $E\ne \emptyset$. This is a contradiction and the statement follows. 
\end{proof}

\section{Essential circuit length}
\label{sec:esscirc}

In order to prove our main result, Theorem \ref{thm:blow-up-obstr}, we need some preparation, in particular  a refined version of Lemma \ref{lem:preimage-bound}.  We will also address the question how  blowing up arcs modifies the pullbacks of a curve $\alpha$ under natural restrictions on the blown-up arcs. 
First, we introduce some terminology and establish some auxiliary facts.

Let $U\sub S^2$ be an open and connected set, and $\sigma\sub S^2$ be an arc. We say that $\sigma$ is an  {\em arc in $U$ ending in $\partial U$} if there  exists an endpoint $p$ of $\sigma$ such that $\sigma\setminus\{p\}\sub U$ and $p\in \partial U$. 

Let $\GC$ be a connected planar embedded graph in $S^2$ 
 and $U$ be one of its faces. Then $U$ is simply connected,
and so we can find a homeomorphism $\varphi\:\D\to U$.
Since we want some additional properties of $\varphi$ here, it is easiest to equip $S^2$ with a complex structure and 
choose a conformal map $\varphi\:\D\to U$.

Since $\partial U$ is a finite union of edges of $\GC$, this set is locally connected and so the conformal map $\varphi$ extends to a surjective continuous map $\varphi\: \cl(\D) \ra \cl(U)$ \cite[Theorem 2.1]{Po}. This extension has the following property: if 
$\sigma$ is an arc in $U$ ending in $\partial U$, then $\varphi^{-1}(\sigma)$ is an arc in $\D$ ending in $\partial \D$ (see \cite[Proposition 2.14]{Po}).  For given $\GC$ and $U$,  we fix, once and for all,
such a map 
$\varphi=\varphi_{\GC, U}$  from $\cl(\D)$ onto  
$\cl(U)$. 

Let $(e_1,e_2,\dots,e_{n})$ be a circuit in $\GC$ that traces the boundary $\partial U$. Recall from Section \ref{subsec:graphs} that the number $n$ is called the circuit length of $U$ in $\GC$, and each edge $e\in\partial U$ appears exactly once or twice in the sequence $e_1,e_2,\dots,e_{n}$ depending on whether the face $U$ lies on one or both sides of $e$, respectively. 
Then there is a corresponding decomposition 
$\partial \D=\sigma_1\cup \dots \cup \sigma_n$ of the unit circle $\partial
\D$  into non-overlapping subarcs $\sigma_1, \dots, \sigma_n$ of $\partial \D$ such that 
$\varphi=\varphi_{\GC, U}$ is a homeomorphism of $\sigma_m$ onto $e_m$ for each $m=1, \dots, n$. 

Let $0<\eps < 1$. We say that a Jordan curve $\beta\sub U$ is an 
\emph{$\eps$-boundary of $U$ with respect to (wrt.) $\GC$} 
if $\beta'\coloneq \varphi^{-1}(\beta)\sub A_\eps\coloneq 
\{z\in \C: 1-\eps<|z|<1\}$, and $\beta'$ separates $0$ from  
$\partial \D$. Then  $\beta'$ is a core curve of the annulus 
$A_\eps$.

For the remainder of this section, $f\: S^2\ra S^2$ is  a Thurston map. All isotopies on $S^2$ are relative to $P_f$, and we consider intersection numbers in $(S^2, P_f)$. 

 Let  $e$ be an arc in $(S^2, \postf)$. Then we can naturally view the set $\GC\coloneq f^{-1}(e)$ as a planar embedded graph with the vertex set $f^{-1}(\partial e)$ and the edges given by the lifts of $e$ under $f$. Note that $\GC$ is bipartite.

\begin{lemma}\label{lem:choicebd} 
Let $f\: S^2\ra S^2$ be a Thurston map, $e$ be an arc and $\ga$ be a Jordan curve in $(S^2, \postf)$ with $\#(e\cap \ga)=\ins(e,\ga)$. Suppose that $\HC$ is a connected subgraph of $\GC=f^{-1}(e)$ and $U$ is a face of $\HC$. Let $2n$ with $n\in\N$ be the circuit length of $U$ in $\HC$. 
Then for each $0<\eps<1$ there exists an 
$\eps$-boundary $\beta$ of $U$ wrt.\  $\HC$ such that 
$\#(\beta\cap f^{-1}(\ga)) = 2n \cdot \ins(e,\ga).$
\end{lemma} 

Note that the circuit length of  $U$ in $\HC$ is even since $\HC$ is a bipartite graph. 

\begin{proof} Let $s\coloneq\ins(e,\ga)\in \N_0$. Then $e$ and $\ga$ have exactly 
$s$ distinct points in common, say $p_1, \dots, p_s\in e\cap \ga$.
Each point  $p_j$ lies  in $\inte(e)$, because 
$\partial e\sub \postf$ and $\ga\sub S^2\setminus \postf$.  
Since $e$ and $\ga$ are in minimal position, they meet transversely (see Lemma~\ref{lem:transverse}), that is, if we travel along $\ga$
towards one of the points $p_j$ (according to some orientation of $\ga$),
then near $p_j$ we stay on one side of $e$, but cross over  to the other side of $e$ if we pass $p_j$. 

This implies that we can find disjoint subarcs $\sigma_1, \dots, \sigma_s$ of $\ga$ such that each arc $\sigma_j$ contains 
$p_j$ in its interior, but contains no other point in $e\cap \ga$. Moreover, $p_j$ splits 
$\sigma_j$ into two non-overlapping subarcs $\sigma_j^L$ and 
$\sigma_j^R$ with the common endpoint $p_j$ so that with some fixed orientation of $e$ the arc  $\sigma_j^L$ lies to the left and 
$\sigma_j^R$ lies to the right of $e$. Note that if  $\ga'\coloneq
\ga\setminus(\sigma_1\cup \dots \cup \sigma_s)$, then  
$e\cap \ga'=\emptyset$.  

 For the given face $U$ of $\HC$ we fix a map
$\varphi=\varphi_{\HC,U}\: \cl(\D) \ra \cl(U)$ as discussed in the beginning of this section.  Let $(e_1,\dots,e_{2n})$ be a circuit in $\HC$ that traces the boundary $\partial U$. As we have already pointed out, 
the number of edges in the circuit is even, because $\HC$ is a bipartite graph. 
With suitable orientation of each  arc $e_m$, the face $U$ lies on the left of  
 $e_m$. If an arc appears twice in the list $e_1,\dots,e_{2n}$,  then it will carry  opposite orientations in its two occurrences.  

We want to investigate the set $f^{-1}(\ga)\cap \cl(U)$ near $\partial U$. 
Note that $f$ maps each arc $e_m$ homeomorphically onto $e$. 
This implies that 
$f$ is a homeomorphism on a suitable Jordan region that contains 
$e_m$ as a crosscut. It follows  that we can pull back the 
local picture near points in $e\cap \ga$ to a similar local picture for points in $e_m\cap f^{-1}(\ga)$.   So if we choose the arcs 
$\sigma_j$ small enough, as we may assume, and pull them back by $f$, then it is clear  that  $f^{-1}(\ga)\cap \cl(U)$ can be represented in the form $$ f^{-1}(\ga)\cap \cl(U)=K \cup \bigcup_{m=1}^{2n}  \bigcup_{j=1}^s 
\sigma_{m,j},$$
where $K$ has positive distance  to $\partial U=e_1\cup \dots \cup e_{2n}$ (with respect to some base metric on $S^2$). Moreover, each $\sigma_{m,j}$ is an arc in $U$ ending in 
$e_m\sub \partial U$ such that $f$ is a homeomorphism
 from $\sigma_{m,j}$ onto 
$\sigma_{j}^L$ or $\sigma_{j}^R$  depending on whether 
 $f|e_m\: e_m\ra e$ is orientation-preserving or orientation-reversing.  
 If we remove  from each arc  $\sigma_{m,j}$ its endpoint in $e_m$, then the  half-open arcs obtained are all disjoint. Two  arcs $\sigma_{m,j}$ and $\sigma_{m',j'}$  share 
 an endpoint precisely when $j=j'$ and they arise from edges $e_m$ and 
 $e_{m'}$ with the same underlying set, but with opposite orientations. In this case,  $f$ sends one of these arcs to $\sigma^L_j$, and the other one to  $\sigma^R_j$.  
 
 Since $K$ has positive distance to $\partial U$, it is clear that 
 if $\beta$ is an $\eps$-boundary of $U$ wrt.\ $\HC$ for $\eps>0$ small enough (as we may assume), then $\beta\cap K=\emptyset$. 
 So in order to control $\#(\beta\cap f^{-1}(\ga))$, we  have to   worry only about the intersections of $\beta$ with the arcs $\sigma_{m,j}$.
 
 Note that there are exactly $2n\cdot s=2n\cdot \ins(e,\ga)$ of these arcs. If we pull them back by the map $\varphi$, then we obtain 
 pairwise disjoint arcs in $\D$ ending in $\partial \D$. The statement now follows from the following fact,  whose precise justification we leave to the reader: if $\alpha_1, \dots, \alpha_M$
 with $M\in \N_0$  are pairwise disjoint arcs in $\D$ ending in $\partial\D$, then for each $0<\eps<1$ there exists a core curve $\beta'$  of 
 the annulus  $A_\eps= 
\{z\in \C: 1-\eps<|z|<1\}$ such that 
$ \# (\beta'\cap (\alpha_1\cup \dots \cup \alpha_M))=M$.  
\end{proof}

Now we are ready to provide a refined version of Lemma \ref{lem:preimage-bound}.

\begin{lemma}\label{lem:preimage-bound-refined}
Let $f\:S^2\ra S^2$ be a Thurston map with $\#\postf=4$, 
$\alpha$ and $\gamma$ be essential Jordan curves in 
$(S^2, \postf)$, $c$  be a  core arc of $\alpha$, and assume that $\#(c \cap\ga) =\ins(c, \ga)>0$.

Let  $\HC$ be a connected subgraph of $\GC\coloneq f^{-1}(c)$, 
  and $U$ be  a face of $\HC$ such that for 
small enough $\eps>0$   each 
$\eps$-boundary $\beta$ of $U$ wrt.\ $\HC$ is isotopic to $\alpha$ rel.\ $\postf$. Let 
$2n$ with $n \in \N$ be  the circuit length of $U$ in $\HC$.  

Then $k\le n$, where $k\in N_0$ denotes  the number of  pullbacks of $\gamma$ under $f$ that are isotopic to $\gamma$ rel.\ $\postf$. 
Moreover, if $\partial U\subset \HC$ meets a peripheral pullback of $\ga$ under $f$, then $k<n$. 
\end{lemma}

\begin{proof}
 Let $\gamma_1,\dots, \gamma_k$ be the pullbacks of $\gamma$ under $f$ that are isotopic to $\gamma$ rel.\ $\postf$. 
Then by  Lemma~\ref{lem:choicebd}, for sufficiently small $\eps > 0$, we can find an $\eps$-boundary $\beta$ of $U$ wrt.\ $\HC$ such that 
$\beta\sim \alpha$ rel.\ $\postf$ and 
$$ \#(\beta\cap f^{-1}(\gamma)) = 2n \cdot \ins(c, \ga)= n\cdot \ins (\alpha, \ga),$$
where the last equality follows from Lemma \ref{lem:i-properties-curve-sphere}.
Hence, we have
\begin{equation}\label{eq:degree-bound-refined}
 n\cdot \ins (\alpha, \ga) = \#( \beta\cap  f^{-1}(\gamma))
\geq \sum_{j=1}^k \#( \beta \cap\gamma_j ) \geq \sum_{j=1}^k \ins(\beta, \gamma_j)=k \cdot  \ins(\alpha,\gamma).
\end{equation} 
Since $\ins(\alpha,\gamma)=2\cdot \ins(c, \ga)>0$, we conclude  that $k\le n$, as desired. 

To see the second statement, we have to revisit the proof of Lemma~\ref{lem:choicebd}. There we identified $2n\cdot  \ins(c,\gamma)=n\cdot \ins(\alpha,\gamma) $ distinct  arcs $\sigma$  in $U$ ending in $\partial U$ (they were called $\sigma_{m,j}$ in the proof). These arcs were subarcs of $f^{-1}(\ga)$ and accounted for all possible intersections of $\beta$ with $f^{-1}(\ga)$ for sufficiently small $\eps>0$; with a suitable choice of $\beta$ each of these arcs $\sigma$ gave precisely one such intersection point. Now if a peripheral pullback 
$ \widetilde \ga\sub f^{-1}(\ga)$ of $\gamma$ under
$f$ meets $\partial U$, then one of these arcs $\sigma$  is a subarc   
of  $ \widetilde \ga$. 
It follows that the first inequality in  \eqref{eq:degree-bound-refined} must be strict and so $k<n$. 
 \end{proof}

For the rest of this section, we fix  a Thurston map $f\: S^2\ra S^2$  with $\#\postf=4$, an essential Jordan curve $\alpha $ in $(S^2,  \postf)$, and core arcs   
 $a$ and $c$ of $\alpha$ that lie in different components of $\Sp\setminus\alpha$. We can view the set $\GC\coloneq f^{-1}(a\cup c)$ as a planar embedded graph in $S^2$ with the set of vertices $f^{-1}(\postf)$ and the edge set consisting of the lifts of $a$ and $c$ under $f$. Then $\GC$ is a bipartite graph.

Let $U$ be the unique connected component of $\Sp\setminus (a\cup c)$. Then $U$ is an annulus and $\alpha$ is its core curve.  The connected components $\widetilde{U}$ of $f^{-1}(U)$  are precisely the complementary components of   $\GC=f^{-1}(a \cup c)$ in $S^2$. It easily follows from the Riemann-Hurwitz formula (see \eqref{eq:RH}) that each $\widetilde{U}$ is an annulus, and that 
$f\:  \widetilde{U} \ra U$ is a covering map. Moreover,  each such annulus contains precisely one pullback $\widetilde \alpha$ of $\alpha$ under $f$.  

 This setup is illustrated in Figure \ref{fig_sphere1}. The points marked in black indicate the four postcritical points of $f$.  The  sphere on the right   contains  two core arcs $a$ and $c$ of a Jordan curve $\alpha$ in $(\Sp,P_f)$. 
On the left the lifts of $a$ and $c$ under  $f$ are shown in blue and magenta colors, respectively, and the pullbacks of $\alpha$ in green.

%\begin{center}
%\vspace{12pt}
\begin{figure}[t]
\def\svgwidth{0.72\columnwidth}
%% Creator: Inkscape 1.0.1 (c497b03c, 2020-09-10), www.inkscape.org
%% Accompanies image file 'spherepics1.pdf' (pdf, eps, ps)
\begingroup%
  \makeatletter%
  \providecommand\color[2][]{%
    \errmessage{(Inkscape) Color is used for the text in Inkscape, but the package 'color.sty' is not loaded}%
    \renewcommand\color[2][]{}%
  }%
  \providecommand\transparent[1]{%
    \errmessage{(Inkscape) Transparency is used (non-zero) for the text in Inkscape, but the package 'transparent.sty' is not loaded}%
    \renewcommand\transparent[1]{}%
  }%
  \providecommand\rotatebox[2]{#2}%
  \newcommand*\fsize{\dimexpr\f@size pt\relax}%
  \newcommand*\lineheight[1]{\fontsize{\fsize}{#1\fsize}\selectfont}%
  \ifx\svgwidth\undefined%
    \setlength{\unitlength}{506.2276715bp}%
    \ifx\svgscale\undefined%
      \relax%
    \else%
      \setlength{\unitlength}{\unitlength * \real{\svgscale}}%
    \fi%
  \else%
    \setlength{\unitlength}{\svgwidth}%
  \fi%
  \global\let\svgwidth\undefined%
  \global\let\svgscale\undefined%
  \makeatother%
  \begin{picture}(1,0.35185679)%
    \lineheight{1}%
    \setlength\tabcolsep{0pt}%
    \put(0,0){\includegraphics[width=\unitlength,page=1]{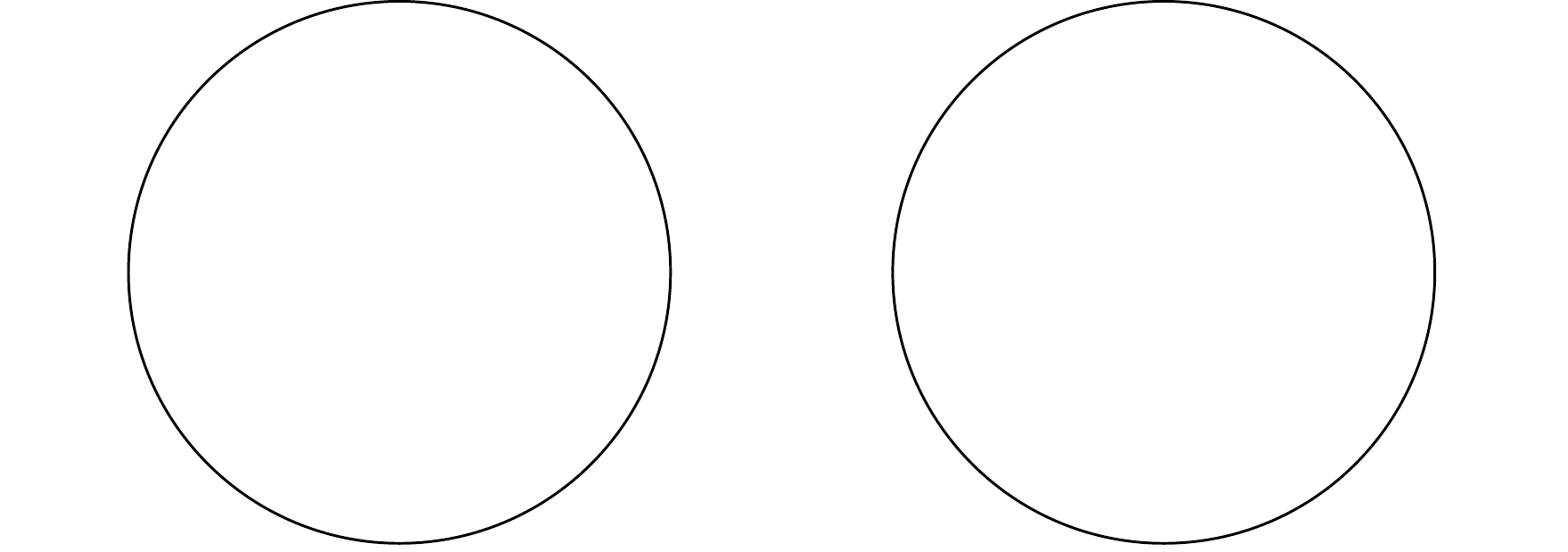}}%
    \put(0.49477032,0.21099379){\color[rgb]{0,0,0}\makebox(0,0)[lt]{\lineheight{1.25}\smash{\begin{tabular}[t]{l}$f$\end{tabular}}}}%
    \put(0,0){\includegraphics[width=\unitlength,page=2]{spherepics1.pdf}}%
    \put(0.76290601,0.11792558){\makebox(0,0)[lt]{\lineheight{1.25}\smash{\begin{tabular}[t]{l}\textcolor{anisgreen}{$\alpha\subset U$}\end{tabular}}}}%
    \put(0.73683407,0.29068824){\makebox(0,0)[lt]{\lineheight{1.25}\smash{\begin{tabular}[t]{l}\textcolor{anisred}{$c$}\end{tabular}}}}%
    \put(0.73327529,0.04681112){\makebox(0,0)[lt]{\lineheight{1.25}\smash{\begin{tabular}[t]{l}\textcolor{anisblue}{$a$}\end{tabular}}}}%
    \put(0.00109931,0.25743576){\makebox(0,0)[lt]{\lineheight{1.25}\smash{\begin{tabular}[t]{l}\textcolor{anisgreen}{$\talpha\subset \tU$}\end{tabular}}}}%
    \put(0,0){\includegraphics[width=\unitlength,page=3]{spherepics1.pdf}}%
  \end{picture}%
\endgroup%

\caption{A Thurston map $f$. The sphere on the right shows a Jordan curve $\alpha$ in $(\Sp, P_f)$ and two core arcs $a$ and $c$. The  sphere on the left shows the pullbacks of $\alpha$ under $f$ and the planar embedded graph $\GC=f^{-1}(a\cup c)$.}\label{fig_sphere1}
\end{figure}
%\end{center}

We call a connected component  $ \widetilde{U}$ of   $f^{-1}(U)=S^2\setminus \GC$ {\em essential} or {\em peripheral}, depending on whether the unique pullback $\widetilde \alpha$ of $\alpha$ contained in $ \widetilde{U}$  is 
essential or peripheral in $(S^2, \postf)$, respectively. 
Each  boundary $\partial\widetilde{U}$ has exactly two connected components. One of them is mapped  by $f$ to $a$ and the other one to $c$; accordingly, we denote them by $\partial_a  \widetilde U$ and $\partial_c \widetilde U$, respectively. 
Then we have 
$$\partial  \widetilde U=\partial_a  \widetilde U\cup \partial_c  \widetilde U, \quad 
\partial  \widetilde U_a=f^{-1}(a)\cap \partial \widetilde U, \quad \text{and} \quad   \partial  \widetilde U_c=f^{-1}(c)\cap \partial  \widetilde U. $$ 

The sets $\partial_a \widetilde{U}$ and $\partial_c \widetilde{U}$ are subgraphs of $\GC$. Since $\widetilde{U}$ is a connected subset of  $S^2\setminus \GC\sub 
S^2\setminus  \partial_a \widetilde{U}$, there exists a unique face $V_a$ of $\partial_a \widetilde{U} $ (considered as a subgraph of $\GC$) such that 
$\widetilde U\sub V_a$. 
Similarly, there exists a unique face  $V_c$ of 
$\partial_c \widetilde U$ with $\widetilde U\sub V_c$. 
By definition, the {\em circuit length} of $\partial_a \widetilde U$ or 
 of $\partial_c \widetilde U$  is the circuit length of $V_a$ in $\partial_a \widetilde{U}$  or of $V_c$ in $\partial_c \widetilde{U}$, respectively.

Then the following statement is true.

\begin{lemma} \label{lem:deccirc} The circuit lengths of $\partial_a \widetilde{U}$ and $\partial_c \widetilde{U}$  are both equal to 
$2\cdot \deg(f\:  \widetilde{U} \ra U)$. 
\end{lemma} 

We call the identical  circuit lengths of $\partial_a \widetilde{U}$ and $\partial_c \widetilde{U}$ the 
{\em circuit length of  $\widetilde{U}$} (for fixed 
 $f$, $\alpha$, $a$, and $c$). 

\begin{proof} It is clear that the  
subgraph $\partial_a \widetilde{U}$ of $\GC$ is bipartite, and so 
$\partial_a \widetilde{U}$ has even  circuit length  $2n$ with $n\in\N$. Let  $d\coloneq \deg(f\:  \widetilde{U} \ra U)$.  It is enough to show that $2n=2d$, because the roles of  $\partial_a \widetilde{U}$
and  $\partial_c \widetilde{U}$ are symmetric, and so the same identity will then  also be true for the circuit length of  $\partial_c \widetilde{U}$. 

To see that $2n=2d$, we use a similar idea as in the proof of 
Lemma~\ref{lem:choicebd}.   We choose a point $p\in \inte(a)$ and an arc $\sigma\sub S^2\setminus c$ with $p\in \inte(\sigma)$ that 
meets  $a$ transversely in $p$, but has no other point with $a$ in common. Then $p$ splits 
$\sigma$ into two non-overlapping subarcs $\sigma^L$ and 
$\sigma^R$ with the common endpoint $p$ so that with some fixed orientation of $a$ the arc  $\sigma^L$ lies to the left and 
$\sigma^R$ lies to the right of $a$.

Let $(e_1,\dots,e_{2n})$ be a circuit in $\partial_a \widetilde U$ that traces the boundary $\partial V_a=\partial_a \widetilde U$,  where $V_a$ is the unique face of $\partial_a \widetilde U$ with $\widetilde U\sub V_a$. With suitable orientation of each arc $e_m$,
the face $\widetilde U$ lies on the left of $e_m$.  
We now consider the set $f^{-1}(\sigma)\cap \cl(\widetilde U)$ near 
$\partial_a \widetilde U$. If we choose $\sigma$ small enough (as we may), then as in the proof of Lemma~\ref{lem:choicebd} 
we see that  
\begin{equation}\label{eq:sigmapull} 
 f^{-1}(\sigma)\cap \cl( \widetilde U)= \bigcup_{m=1}^{2n} \sigma_{m},\end{equation}
where each $\sigma_{m}$ is an arc in $\widetilde U$ ending in 
$e_m\sub \partial_a \widetilde U$ such that $f$ is a homeomorphism
 from $\sigma_{m}$ to 
$\sigma^L$ or $\sigma^R$  depending on whether 
 $f|e_m\: e_m\ra a$ is orientation-preserving or orientation-reversing.  
 If we remove  from each arc  $\sigma_{m}$ its endpoint in $e_m$, then the  half-open arcs obtained are all disjoint. On the other hand,
since $f\: \widetilde U \to U$ is a $d$-to-$1$ covering map, there are precisely $d$ distinct  lifts of $\sigma^L\setminus \{p\}$ and $d$ 
distinct lifts of $\sigma^R\setminus \{p\}$  under $f$ contained in 
$\widetilde U$. These must be precisely the half-open arcs obtained from $\sigma_{m}$, $m=1, \dots, 2n$. It follows that 
$2n=2d$, as desired.
\end{proof} 

Suppose  $\widetilde U$ is  an essential  component of $f^{-1}(U)$.
We consider a circuit in  $\partial \widetilde{U}\sub \GC$ and denote by $\HC\sub \GC$ the underlying  graph of the circuit. In the following, we will often conflate the circuit with its underlying graph $\HC$, where we think of $\HC$ as traversed as a circuit in some way.   
Since $\widetilde{U}$ is a connected set in $S^2\setminus \GC\sub
S^2\setminus \HC$, there exists a unique face $V$ of $\HC$  such that 
$\widetilde U\sub V$. By definition, for $0<\eps<1$, an $\eps$-boundary $\beta$ of 
$\widetilde U$ wrt.\ $\HC$ is an $\eps$-boundary  of $V$  wrt.\ $\HC$. 
This is an  abuse of terminology, because even for small $\eps>0$ such an  $\eps$-boundary $\beta$ may not lie in $\widetilde U$, but it is convenient in the following.  Note that for small enough $\eps>0$ such $\eps$-boundaries  for fixed  $\HC$ have the same isotopy type rel.\ $\postf$. 

By definition the {\em essential circuit length} of $\widetilde{U}$ 
is the minimal length of all circuits  $\HC$  in $\partial \widetilde{U}$ 
such that for all small enough $\eps>0$ each $\eps$-boundary of $\widetilde{U}$ 
wrt.\  $\HC$ is isotopic to a core curve of  $\widetilde{U}$ rel.\ $\postf$. As we will see momentarily, if we run through $\partial_a \widetilde{U}$ and $\partial_c \widetilde{U}$ as circuits, then they have this property  and so the essential circuit length of $\widetilde{U}$ is well-defined. 
  We call a circuit $\HC$ in $\partial \widetilde U$ that realizes the 
essential circuit length an {\em essential circuit} for $\widetilde{U}$.

\begin{lemma} \label{lem:esscirc} We have the inequality 
$$ \text{circuit length of  $\widetilde{U}$} \ge 
 \text{essential circuit length of  $\widetilde{U}$}. $$  
  \end{lemma} 
 
For example, consider the annulus $\widetilde U$ containing the pullback $\widetilde \alpha$ in Figure \ref{fig_sphere1}. Then the circuit length of $ \widetilde U$ equals $6$, while the essential circuit length of $ \widetilde U$ equals $4$. 
 
  \begin{proof} Consider $\partial_a \widetilde{U}$ as a circuit in $\partial \widetilde{U}$. Let $V_a$ be the unique face of $\partial_a \widetilde U$ that contains $\widetilde U$. Then, for sufficiently small $\eps>0$, each  $\eps$-boundary $\beta$ of $V_a$ wrt.\ $\partial_a \widetilde U$ necessarily separates $\partial_a \widetilde{U}$ and $\partial_c \widetilde{U}$, and is thus  a core curve of    
$\widetilde{U}$. The statement follows.   \end{proof}

Let $\widehat f$ be a Thurston map obtained from $f$ by blowing up some set of arcs $E$ in $S^2\setminus f^{-1}(\postf)$.
Under certain natural assumptions on the arcs in $E$, we want to 
 describe the components of $\widehat f^{-1}(U)$ and their properties in terms of the components of $f^{-1}(U)$.  We first formulate suitable conditions that allow such a comparison. 
  
  \begin{definition}[$\alpha$-restricted blow-up conditions]\label{def:blowup_conditions} Let  $f\: S^2\ra S^2$  be 
a Thurston map    with $\#\postf=4$, $\alpha $ be  an essential Jordan curve  in $(S^2,  \postf)$, and  
 $a$ and $c$ be core arcs of $\alpha$ that lie in different components of $\Sp\setminus\alpha$. 
Suppose 
 $E\ne \emptyset$ is  a finite set of arcs in $(\Sp,f^{-1}(\postf))$ satisfying the blow-up conditions, that is, the interiors of the arcs in $E$ are disjoint and $f\: e\to f(e)$ is a homeomorphism for each $e\in E$. 
 
 We say that $E$ satisfies the $\alpha$-\emph{restricted blow-up conditions} if  
 \begin{equation} \label{eq:alphares} 
\ins(f(e), \alpha)=\#(f(e)\cap \alpha) = 1\quad \text{and} \quad f(\inte(e)) \cap a =  \emptyset =f(\inte(e))\cap c 
\end{equation} for each $e\in E$. 
 \end{definition}

  In other words, for each $e\in E$ the arc $f(e)$ is in minimal position with respect to $\alpha$ and  intersects $\alpha$ only once, and $f(\inter(e))=\inter(f(e))$ belongs to the annulus $U=S^2\setminus (a\cup c)$. Note that the endpoints of $f(e)$ lie in $\postf\sub a\cup c=\partial U$; see Figure \ref{fig_sphere2} for an illustration.
  
%  \begin{center}
%\vspace{12pt}
\begin{figure}[t]
\def\svgwidth{0.76\columnwidth}
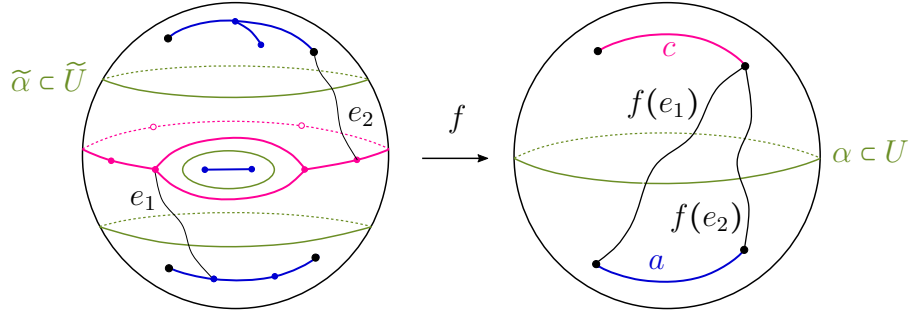
\caption{The Thurston map $f$ from Figure \ref{fig_sphere1} and a set $E=\{e_1,e_2\}$ of arcs  in $(\Sp,f^{-1}(\postf))$ satisfying the  {$\alpha$-restricted} blow-up conditions.}\label{fig_sphere2}
\end{figure}
%\end{center}
  
 Condition~\eqref{eq:alphares} is somewhat artificial, because it depends not only on $\alpha$, but also on the choices of $a$ and $c$. One can show that up to isotopy  it can be replaced by the more natural condition  $\ins(f(e),\alpha)=1$ for all $e\in E$. Since the justification of this claim involves some topological machinery that is beyond the scope of the paper, we prefer to work with 
 \eqref{eq:alphares}. 
 
Now the following statement is true.

%\begin{center}
%\vspace{12pt}
\begin{figure}[b]
\def\svgwidth{0.72\columnwidth}
%% Creator: Inkscape 1.0.1 (c497b03c, 2020-09-10), www.inkscape.org
%% Accompanies image file 'spherepics3.pdf' (pdf, eps, ps)
\begingroup%
  \makeatletter%
  \providecommand\color[2][]{%
    \errmessage{(Inkscape) Color is used for the text in Inkscape, but the package 'color.sty' is not loaded}%
    \renewcommand\color[2][]{}%
  }%
  \providecommand\transparent[1]{%
    \errmessage{(Inkscape) Transparency is used (non-zero) for the text in Inkscape, but the package 'transparent.sty' is not loaded}%
    \renewcommand\transparent[1]{}%
  }%
  \providecommand\rotatebox[2]{#2}%
  \newcommand*\fsize{\dimexpr\f@size pt\relax}%
  \newcommand*\lineheight[1]{\fontsize{\fsize}{#1\fsize}\selectfont}%
  \ifx\svgwidth\undefined%
    \setlength{\unitlength}{511.96044545bp}%
    \ifx\svgscale\undefined%
      \relax%
    \else%
      \setlength{\unitlength}{\unitlength * \real{\svgscale}}%
    \fi%
  \else%
    \setlength{\unitlength}{\svgwidth}%
  \fi%
  \global\let\svgwidth\undefined%
  \global\let\svgscale\undefined%
  \makeatother%
  \begin{picture}(1,0.34357609)%
    \lineheight{1}%
    \setlength\tabcolsep{0pt}%
    \put(0,0){\includegraphics[width=\unitlength,page=1]{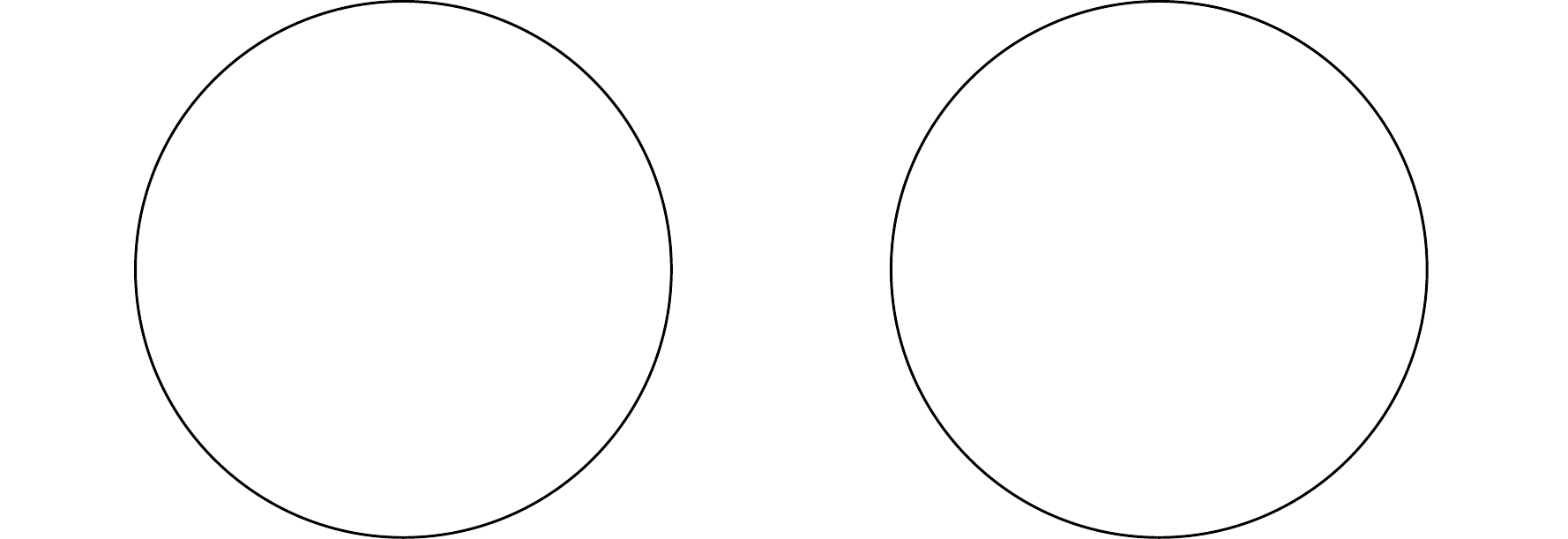}}%
    \put(0.49457356,0.20429049){\color[rgb]{0,0,0}\makebox(0,0)[lt]{\lineheight{1.25}\smash{\begin{tabular}[t]{l}$\fhat$\end{tabular}}}}%
    \put(0,0){\includegraphics[width=\unitlength,page=2]{spherepics3.pdf}}%
    \put(0.75970674,0.11226441){\makebox(0,0)[lt]{\lineheight{1.25}\smash{\begin{tabular}[t]{l}\textcolor{anisgreen}{$\alpha\subset U$}\end{tabular}}}}%
    \put(0.73392675,0.28309257){\makebox(0,0)[lt]{\lineheight{1.25}\smash{\begin{tabular}[t]{l}\textcolor{anisred}{$c$}\end{tabular}}}}%
    \put(0.73040782,0.04194627){\makebox(0,0)[lt]{\lineheight{1.25}\smash{\begin{tabular}[t]{l}\textcolor{anisblue}{$a$}\end{tabular}}}}%
    \put(-0.00021761,0.24821947){\makebox(0,0)[lt]{\lineheight{1.25}\smash{\begin{tabular}[t]{l}\textcolor{anisgreen}{$\halpha\subset \hU$}\end{tabular}}}}%
    \put(0,0){\includegraphics[width=\unitlength,page=3]{spherepics3.pdf}}%
  \end{picture}%
\endgroup%

\caption{A Thurston map $\widehat f$ obtained from the Thurston map $f$ in Figures~\ref{fig_sphere1} and \ref{fig_sphere2} by blowing up the arcs $e_1$ and $e_2$ with multiplicities $m_{e_1} = 2$ and $m_{e_2}=1$, respectively. }\label{fig_sphere3}
\end{figure}
%\end{center}

 \begin{lemma} \label{lem:blowup/annuli}
 Let  $f\: S^2\ra S^2$  be 
a Thurston map    with $\#\postf=4$, $\alpha $ be  an essential Jordan curve  in $(S^2,  \postf)$, 
 $a$ and  $c$ be core arcs of $\alpha$ that lie in different components of $\Sp\setminus\alpha$.   Suppose a set  $E$   of  arcs   in $(\Sp,f^{-1}(\postf))$ 
 satisfies the  $\alpha$-restricted blow-up conditions and we are 
 given multiplicities $m_e\in \N$ for $e\in E$.
 
Then the Thurston map $\widehat f$ obtained from $f$ by blowing up each arc $e\in E$ with multiplicity $m_e$ can be constructed so that it satisfies the following conditions: 
 \begin{enumerate}[label=\text{(\roman*)},font=\normalfont]
\item \label{blow1} $\GC=f^{-1}(a\cup c)$ is a subgraph of $\widehat 
\GC\coloneq\widehat f^{-1}(a\cup c)$. 
\item \label{blow2} Each complementary annulus $\widehat U$ of  $\widehat 
\GC$ is contained in a unique complementary annulus $\widetilde U$ 
of $\GC$. Moreover, the assignment $\widehat U\mapsto 
\widetilde U$ is a bijection between the complementary annuli of 
$\widehat \GC$ and of $\GC$.  
  \setcounter{mylistnum}{\value{enumi}}
\end{enumerate}
Let $\widehat U$ and $\widetilde U$ be corresponding annuli as in 
\textnormal{\ref{blow2}}. Then the following statements are true: 
 \begin{enumerate}[label=\text{(\roman*)},font=\normalfont]
  \setcounter{enumi}{\value{mylistnum}}
  \item\label{blow3}  The core curves of $\widehat U$ and of $\widetilde U$ 
 are isotopic rel.\ $\postf=P_{\widehat f}$. In particular,  
 $\widehat U$ is essential if and only if $\widetilde U$  is essential. 
  \item\label{blow4}  If $\widehat U$ (and hence also $\widetilde U$) is essential, then the essential circuit lengths of  $\widehat U$ and  $\widetilde U$ are the same. Moreover, if $\HC$ is an essential circuit for 
  $\widehat U$, then $\HC$ is an essential circuit for
   $\widetilde U$. In particular,  $\HC\sub \partial  \widetilde U\sub\partial  \widehat U$. 
 \end{enumerate}  \end{lemma}

In \ref{blow1} it is understood that the planar embedded graph  
$\GC=f^{-1}(a\cup c)\sub S^2$ has the vertex set $f^{-1}(P_f)$ and 
that $\widehat \GC=\widehat f^{-1}(a\cup c)$ has the  vertex set $\widehat{f}^{-1}(P_{\widehat{f}})\supset f^{-1}(P_f)$.

To illustrate the  lemma, we consider the Thurston map $\widehat f$
that is in indicated in Figure~\ref{fig_sphere3}  and obtained from the Thurston map $f$ in Figure \ref{fig_sphere1} by blowing up the arcs $e_1$ and $e_2$ in Figure \ref{fig_sphere2} with multilicities 
$m_{e_1}=2$ and $m_{e_2}=1$. Here, the  lifts of $a$ and $c$ under the blown-up map $\widehat f$ are shown in blue and magenta colors on the left sphere, respectively.   Comparing these figures, we 
  immediately see that in this case  the statements  of  the lemma are true. 
 
 \begin{proof}  Let  $e\in E$ be arbitrary. Since $E$ satisfies the conditions in Definition~\ref{def:blowup_conditions}, the set $f(e)$ is an arc in $(S^2, 
 P_f)$, and so $f(e)$  has its endpoints in $P_f$. By  \eqref{eq:alphares} the arc $f(e)$  meets $\alpha$ precisely once and is in minimal position with respect to $\alpha$. In particular, $f(e)$ meets $\alpha$ transversely by Lemma~\ref{lem:transverse}. 
 This implies that  the endpoints of $f(e)$  lie in different core arcs  of $\alpha$, and so one endpoint of $f(e)$ lies in $a$ and the other one in $c$. It follows that $e$ has one endpoint in 
 $f^{-1}(a)$  and the other one in $f^{-1}(c)$. 
 
The  set $\inter(e)$ belongs to 
 a unique annulus  $\widetilde U$ obtained as a  complementary 
 component of  $\GC=f^{-1}(a\cup c)$. Then one endpoint of $e$ is in $\partial_a \widetilde U=f^{-1}(a)\cap \partial \widetilde U$ and the other one in $\partial_c \widetilde U=f^{-1}(c)\cap \partial \widetilde U$. In the blow-up construction described in Section \ref{subsec:blow-up-general}, we can choose the open Jordan region $W_e$ so that $W_e \subset \widetilde U$ for each 
 $e\in E$ (of course, here the annulus  $\widetilde U$ depends on $e$). Now we make choices of the subsequent ingredients in  the blow-up construction as discussed in Section \ref{subsec:blow-up-general}. That is, for each fixed arc $e\in E$, we choose a closed Jordan region $D_e$ inside $W_e$. It is  is subdivided into $m=m_e$ components $D_e^1,\dots, D_e^m$. In addition, we also choose a pseudo-isotopy $h\: \Sp\times\I\to \Sp$ satisfying conditions \ref{cond:B1}--\ref{cond:B4}, as well as maps 
  $\varphi_k\:D_e^k\to \Sp$, $k=1,\dots, m=m_e$, satisfying conditions \ref{cond:C1} and \ref{cond:C2}. Let $\widehat f$ be the Thurston map obtained by blowing up each  arc  $e\in E$ with  multiplicity $m_e$ according to these choices. We claim that $\widehat f$ satisfies all the conditions in the statement. 
 
 It immediately follows from \ref{cond:B3} and the definition of $\widehat f$  that $\GC=f^{-1}(a\cup c)$ is a subgraph of $\widehat \GC\coloneq\widehat f^{-1}(a\cup c)$, and so statement  \ref{blow1} is true. 

We have $\widehat \GC\setminus \GC \subset \bigcup_{e\in E} D_e$. Condition \ref{cond:C1} now implies that for each $e\in E$ and each $k=1,\dots, m=m_e$, the set $\widehat \GC \cap D^k_e$ consists of two disjoint edges, one of which is homeomorphically mapped onto $a$ and the other one onto $c$ by $\widehat f$. We will call these edges  $a$-{\em sticks} and  $c$-{\em sticks}, respectively. Each closed Jordan region  $D_e$ contains exactly $m_e$ $a$-sticks, which have a common endpoint in $\partial e \cap f^{-1}(a)$, and exactly $m_e$ $c$-sticks with a common endpoint in $\partial e \cap f^{-1}(c)$. The edge set of $\widehat \GC$ consists of all the edges of $\GC$ together with all the $a$-sticks  and $c$-sticks.

Each complementary component $\widehat U$ of $\widehat \GC$
is equal to a unique complementary component $\widetilde U$ of $\GC$ with all the $a$- and $c$-sticks removed that are 
contained in $\cl(\widetilde U)$. Statement  \ref{blow2} follows. Furthermore, since $\postf\cap \widetilde U = \emptyset$, statement \ref{blow3} follows as well. 

To prove \ref{blow4}, let $\widetilde U$ and $\widehat U$ be corresponding essential annuli as in \ref{blow2}. Viewing  
$\partial \widetilde U$ as a subgraph of $\GC $ and $\partial  \widehat U$ as a subgraph of $\widehat \GC$, we see that $\partial \widetilde U$ is a subgraph of $\partial \widehat U$. The additional  edges of $\partial \widehat U$ are exactly the $a$- and $c$-sticks contained in
 $\cl(\widetilde U)$.  It  follows from the definition that the essential circuit length of $\widehat U$ is greater than or equal to  the essential circuit length of $\widetilde U$. 

Let $\HC$ be an essential circuit for $\widehat U$ and suppose 
 $\HC$ contains an $a$- or  $c$-stick $\sigma$.  Then  one of the endpoints of $\sigma$ has  degree $1$ in $\partial \widehat U$, and so $\sigma$ must appear in two consecutive positions in the circuit $\HC$. Omitting these two occurrences of $\sigma$ from $\HC$ we get a shorter  circuit $\HC'$ in $\partial \widehat U$ such that for all small enough $\eps>0$ each $\eps$-boundary of $\widehat U$ wrt.\ $\HC'$ is isotopic to a core curve of $\widehat{U}$ rel.\ $P_{\widehat f} =\postf$. This contradicts the choice of $\HC$, and it follows that  $\HC$ does not contain any $a$- or  $c$-sticks. Consequently, $\HC\subset\partial \widetilde  U \subset \partial  \widehat U$, and the definition of the essential circuit length together with \ref{blow3} imply that $\HC$ is an essential circuit for $\widetilde  U$. Statement  \ref{blow4} follows.
 \end{proof}

\section{Eliminating obstructions by blowing up arcs}
\label{sec:elim-obstr}
The goal of this section is to show that the blow-up surgery can be applied to an obstructed Thurston map $f$ with four postcritical points in such a way that the resulting map $\widehat{f}$ is realized by a rational map. The precise formulation is given in Theorem~\ref{thm:blow-up-obstr} (see also Remark~\ref{rem:killobs}). We will prove this statement by contradiction. For this  we assume that  
$\widehat{f}$ has an obstruction, and will carefully analyze some related mapping degrees.
This leads  to a very tight situation, where in some inequalities we actually have equality. From this we want to conclude that 
$f$ has a parabolic orbifold, in contradiction  to our hypotheses in   
 Theorem~\ref{thm:blow-up-obstr}. We first formulate a related criterion  for  parabolicity.

\begin{lemma}\label{lem:parabolicity}
Let $f\:S^2\ra S^2$ be a Thurston with $\#\postf=4$. Suppose there exists a Jordan curve $\alpha$ in $(S^2, \postf)$ such that  the following conditions are true:
\begin{enumerate}[label=\text{(\roman*)},font=\normalfont]
\item $\alpha$ is an obstruction for $f$.

\smallskip
\item $\alpha$ has no peripheral pullbacks under $f$. 

\smallskip  
\item If  we choose core arcs $a$ and $c$ of $\alpha$ in different components of $S^2\setminus \alpha$ and consider the graph $\GC=f^{-1}(a\cup c)$, then $\GC$ has precisely 
$n\in \N$ essential complementary components $U_1, \dots, U_n$ with core curves isotopic to $\alpha$ rel.\ $\postf$. Moreover, we assume  the essential circuit length of $U_j$ is  equal to $2n$ for each $j=1, \dots, n$. 
\end{enumerate} 
Then $f$ has a parabolic 
orbifold. \end{lemma}

\begin{proof}  Each $U_j$ contains precisely one pullback $\alpha_j$ of $\alpha$ under $f$. The curves $\alpha_1, \dots, \alpha_n$ are all the pullbacks of  $\alpha$ under $f$.
Then it follows from our assumptions that 
\begin{align*}
2\deg(f\: \alpha_j \ra \alpha)&=2\deg(f\:U_j\ra U)\\ 
&= \text {circuit length of $U_j$} \quad \text{(by 
Lemma~\ref{lem:deccirc})} \\
&\ge \text {essential circuit length of $U_j$} \quad \text{(by 
Lemma~\ref{lem:esscirc})} \\
&= 2n,
\end{align*} 
and so $\deg(f\: \alpha_j \ra \alpha)\ge n$ for $j=1, \dots, n$. On the other hand, 
$\alpha$ is an obstruction for $f$, and so
$$ 1\le \lambda_f(\alpha)= \sum_{j=1}^n \frac{1} 
{\deg(f\: \alpha_j \ra \alpha)}\le n/n=1. $$
It follows that we have equality in all previous inequalities. In particular, 
$$\text {circuit length of $U_j$}=
 \text {essential circuit length of $U_j$} =2n$$
 for $j=1, \dots, n$. 
 
We want to apply the second part of Lemma~\ref{lem:DHLem2}, that is, we want to show that $f^{-1}(\postf)\sub \critf  \cup \postf$. To see this, let 
 $v\in f^{-1}(\postf)$ be arbitrary. Then $v$ is a vertex of $\GC=f^{-1}(a\cup c)$. If $v$ is incident with two or more edges in $\GC$, then $v\in \critf$.

Otherwise, $v$ is the endpoint of precisely one edge $e$ in $\GC$, and so $\deg_{\GC}(v)=1$. 
We claim that then $v\in \postf$; to see this, we argue by contradiction and assume that $v\not \in \postf$. Since $\alpha$ has no peripheral pullbacks, we have 
$$ e\sub \GC =\bigcup_{j=1}^n \partial U_j, $$
and so $e\sub \partial U_j$ for some $U_j$. 
Then $e$ is contained in a circuit $(e_1,\dots, e_{2n})$ of length $2n$ that traces one of the components of $\partial U_j$. Since $\deg_{\GC}(v)=1$, the circuit must traverse $e$ twice with opposite orientations, that is, the edge $e$ appears precisely  in  two consecutive entries in the cycle $(e_1,\dots, e_{2n})$.  Erasing these two occurrences from the cycle, we obtain a new circuit in $\partial U_j\subset \GC$ with length $2n-2$. Let   $\HC$ denote the   underlying subgraph of $\GC$ corresponding to this shortened circuit   and let $U$ be the face of $\HC$ that contains $U_j$. Since the endpoint $v$ of $e$ does not belong to $\postf$, for every small enough  $\eps>0$ the  $\eps$-boundary of $\HC$  wrt.\  $U$  is isotopic  to the curve $\alpha_j$ rel.\ $\postf$.   
Then the essential circuit length of $U_j$ is  $\le 2n-2$, contradicting the fact that 
$2n$ is the essential circuit length of  $U_j$ by our hypotheses. 
So we must have $v\in \postf$.

It follows that $f^{-1}(\postf)\sub \critf  \cup \postf$, and so $f$ has a parabolic orbifold by Lemma~\ref{lem:DHLem2}.  
\end{proof}

 We are now ready to prove our main result. 
\begin{proof}[Proof of Theorem \ref{thm:blow-up-obstr}] 
 As in the statement, suppose  $f\:\Sp\to\Sp$ is a Thurston map with $\#\postf=4$ and a hyperbolic orbifold. We assume  that $f$ has an obstruction given by a Jordan curve $\alpha$ in $(S^2, \postf)$. We choose  core arcs $a$ and $c$ for $\alpha$ that lie in different components of $S^2\setminus \alpha$, and assume that $E\ne \emptyset$ is a finite set of arcs in $(\Sp,f^{-1}(\postf))$ that satisfies the $\alpha$-restricted blow-up conditions as in 
 Definition~\ref{def:blowup_conditions}. 
 
We assume that we obtained a Thurston map $\widehat{f}\: S^2\ra S^2$  by blowing up arcs in $E$ (with some multiplicities) so  that $\lambda_{\widehat{f}}(\alpha)<1$. Then 
$P_{\widehat{f}}=\postf$ and $\widehat{f}$ has a hyperbolic orbifold (see Lemma~\ref{lem:blowup} and Remark~\ref{rem:hyporb}). 
Up to replacing $\widehat{f}$ with a Thurston equivalent map, we may also assume 
that the statements in Lemma~\ref{lem:blowup/annuli}
are true for the map $\widehat{f}$.  
We now argue by contradiction and assume that  $\widehat{f}$ is not realized by a rational map. 
Then by Theorem~\ref{thm:Thurston} the map $\widehat{f}$ has an  obstruction given by a Jordan curve $\gamma$ in $(S^2, \postf)$.

  We set $U=S^2\setminus(a\cup c)$.
 Since $E$ satisfies the $\alpha$-restricted blow-up conditions, we have  $\#(f(e)\cap \alpha) = 1$ and $f(\inte(e)) \cap a = f(\inte(e))\cap c = \emptyset$ for each $e\in E$. In other words, $f(e)$ intersects $\alpha$ only once and $\inter(f(e))$ belongs to $U$. Then each arc in $E$ intersects only one pullback of $\alpha$ and only once.

 Since $\gamma$ is  an obstruction for $\widehat f$, but $\alpha $ is not, the curves $\alpha$ and $\ga$  are not isotopic  rel.\ $P_{\widehat{f}}=\postf$. So  we have $\ins(\alpha, \ga)>0$ for intersection numbers in $(S^2, P_f)$ as follows from Lemma~\ref{lem:isocrit}. By 
Lemma~\ref{lem:i-properties-curve-sphere}~\ref{item:i1-sp} we have 
   \begin{equation}\label{eq:lemprereq1}
  \ins(a,\ga)= \ins(c,\ga) =\tfrac12 \ins(\alpha, \ga)>0.
   \end{equation}
As follows from  Lemma~\ref{lem:i-properties-curve-sphere}~\ref{item:i2-sp}, by replacing  $\ga$ with an isotopic curve rel.\ $P_{\widehat{f}}=\postf$ if necessary, we may also assume 
 that 
 \begin{equation}\label{eq:lemprereq2}
 \# (\alpha\cap \ga)=\ins(\alpha,\ga),\   \# (a\cap \ga)=\ins(a,\ga),\   \# (c\cap \ga)=\ins(c,\ga)
\end{equation}
   and that the points in the non-empty and finite  sets   $a\cap \ga$ and $c\cap \ga$ alternate on $\ga$. 
   
 We denote by  $\alpha_1,\dots,\alpha_n$ with $n\in \N$ the pullbacks of $\alpha$ under $f$ that are isotopic to $\alpha$ rel.\ $\postf$.  
Now we consider the graphs $ \GC= f^{-1}(a\cup c)$ and $\widehat \GC=\widehat f^{-1}(a\cup c)$ as in Section~\ref{sec:esscirc}. By Lemma~\ref{lem:blowup/annuli}, $\GC$ is a subgraph of $\widehat{\GC}$. Moreover, the following 
facts are true for their complementary components. Each  $\alpha_j$ is a core curve in an essential  annulus 
$U_j$ that is a component of $S^2\setminus \GC$. Each 
$U_j$ contains precisely one component $\widehat U_j$
of $S^2\setminus \widehat \GC$. This component is essential and 
contains precisely one essential pullback $\widehat \alpha_j$  of $\alpha$ under $\widehat f$.  The essential circuit length of $U_j$ is the same as the essential circuit length of $\widehat U_j$. 
The curves $\widehat \alpha_1, \dots, \widehat \alpha_n$ are precisely all the distinct essential pullbacks of $\alpha$ under $\widehat f$. They  are isotopic to $\alpha$ rel.\ $\postf$.

Let   $\gamma_1,\dots,\gamma_k$ with $k\in \N$ be the pullbacks of $\gamma$ under $\widehat{f}$ that are isotopic to $\gamma$ rel.\ $P_{\widehat{f}}=\postf$.
Applying Lemma~\ref{lem:deccirc} and Lemma~\ref{lem:preimage-bound-refined} (for the latter  \eqref{eq:lemprereq1} and  \eqref{eq:lemprereq2}
are important)   to an essential circuit 
for  $\widehat U_j$, we see that  
\begin{align} \label{eq:circlenf} 
\deg(f\: \alpha_j\ra \alpha)&= \deg(f\: U_j\ra U)\\
&=\tfrac12\cdot\text {circuit length of $U_j$} \notag 
\\ &\ge \tfrac12\cdot  \text {essential circuit length of $U_j$} 
\notag
\\& =  \tfrac12\cdot  \text {essential circuit length of $\widehat U_j$}\notag \\
&\ge k \notag
\end{align}
for $j=1, \dots, n$.

On the other hand, $\alpha$ has  $n$ distinct essential 
pullbacks $\widehat \alpha_1, \dots, \widehat \alpha_n$ under $\widehat f$, and so Lemma~\ref{lem:preimage-bound} implies that 
$\deg(\widehat f\:\ga_m\to\ga)\ge n $ for $m=1,\dots, k$ (again 
\eqref{eq:lemprereq1} and \eqref{eq:lemprereq2} are used here). 
Since $\alpha$ and $\ga$ are obstructions for
$f$ and  $\widehat{f}$, respectively, we conclude that 
\begin{align*}1&\leq\lambda_{f}(\alpha)=\sum_{j=1}^n \frac{1}{\deg(f\:\alpha_j\to\alpha)}\leq n/k
 \end{align*} 
and
\begin{align*} 1&\leq\lambda_{\widehat{f}}(\gamma)=\sum_{m=1}^k \frac{1}{\deg(\widehat{f}\:\gamma_m\to\gamma)}\leq {k}/{n}.
 \end{align*} 
It follows that  $k=n$, which forces 
$\deg(f\:\alpha_j\to\alpha)=\deg(\widehat{f}\:\gamma_j\to\gamma)=n$ for $j=1,\dots,n$.
If we combine this with \eqref{eq:circlenf}, then we also see 
that 
\begin{align}
  \text {circuit length of $U_j$} &=  \text {essential circuit length of $U_j$}  \label{eq:circuit_length_relation} \\ &= \text {essential circuit length of $\widehat U_j$} =2n \notag
  \end{align}  
for $j=1, \dots, n$. 

We now want to apply Lemma~\ref{lem:parabolicity} to our map $f$ and the obstruction $\alpha$. In order to verify the hypotheses of  Lemma~\ref{lem:parabolicity}, it  remains to show that $\alpha$ has no peripheral pullbacks under $f$, or equivalently, no peripheral pullbacks under $\widehat f$ 
(see Lemma~\ref{lem:blowup/annuli}~\ref{blow3}).  

We argue by contradiction and assume that  
$\alpha$ has some peripheral pullbacks under $\widehat f$.
Then there  exists at least one peripheral  annulus  in the complement of $\widehat \GC$. Such an annulus is disjoint  from each annulus $\widehat U_j$. We can then travel from a point $p$ of such a peripheral annulus to a point in the set  $\widehat U_1 \cup 
\dots \cup \widehat U_n$ along an arc $\sigma$  in $S^2\setminus \widehat f^{-1}
(P_{\widehat{f}})$ that crosses each edge in the graph $\widehat \GC$ transversely. Then there is a first point $q$ on $\sigma$ 
where we enter the closure $M$ of $\widehat U_1 \cup 
\dots \cup \widehat U_n$. The point   $q$ is necessarily an 
 interior point of an edge $e$ of $\widehat \GC$ contained in the boundary $\partial \widehat U_j$ for some $j\in \{1, \dots, n\}$.
 Interior points of the subarc of $\sigma$ between $p$ and $q$ that are close to $q$ do not lie in $M\cup \widehat \GC$. Hence such points must belong to a peripheral component $\widehat U$ 
 of $\widehat \GC$. Then necessarily $e\sub \partial \widehat U$. 

In other words, there exists an edge $e$ in the graph 
$ \widehat \GC$ that belongs to the boundary of an essential annulus $\widehat U_j$ and a peripheral annulus $\widehat U$. 
Clearly, $\widehat f(e)=a$ or  $\widehat f(e)=c$.  In the following, we will assume that $\widehat f(e)=c$, that is, $e\subset \partial_c \widehat U_j\cap \partial_c \widehat U$; the other case, $\widehat f(e)=a$,  is completely analogous.

Since $\ins(c,\ga) = \#(c \cap \ga)>0$, there exists a pullback 
$\widehat \ga$ of $\ga$ under $\widehat f$ that meets $e$ transversely. Consequently, this pullback
$\widehat \ga$ meets both  $\widehat U_j$ and $\widehat U$.
Since $\widehat f\:\widehat \ga \ra \ga$ is a covering map and the points in $a\cap \ga \ne \emptyset$ and $c\cap \ga \ne \emptyset$ alternate on $\ga$, the points in $\widehat f^{-1}(a) \cap  \widehat \ga\ne \emptyset $ 
and $\widehat f^{-1}(c) \cap \widehat \ga\ne \emptyset $ alternate on 
$\widehat \ga$. This implies that the curve $\widehat \ga$ also meets the sets   $\partial_a \widehat U_j$ and $\partial_a \widehat U$, and hence  both components of the boundary of  $\widehat U_j$ and of $\widehat U$. We conclude   that   $\widehat \ga$ meets the core curve 
$\widehat \alpha_j$ of $\widehat U_j$ and the core curve 
$\widehat \alpha$ of $\widehat U$. Note that  $\widehat \alpha$ is a peripheral pullback of $\alpha$ under $\widehat f$. In order to show  that this is impossible, we consider two cases.

\smallskip
{\em Case 1:} $\widehat \ga$ is an essential pullback of $\ga$ under $\widehat f$, say $\widehat \ga=\ga_m$ for some $m\in \{1, \dots, n\}$. Then Lemma~\ref{lem:preimage-bound} (for the map $\widehat f$, and  $\alpha$, $\widehat \ga$ in the roles of $\ga$, $\widetilde \alpha$, respectively) shows 
$n<\deg(\widehat f\: \widehat \ga \ra \ga)$, because $\alpha$ has $n$ essential pullbacks and $ \widehat \ga $ meets a peripheral pullback of $\alpha$. On the other hand, we know that 
$\deg(\widehat f\: \widehat \ga \ra \ga)=\deg(\widehat f\: \ga_m \ra \ga)=n$. This is a contradiction. 

\smallskip
{\em Case 2:} $\widehat \ga$ is a peripheral  pullback of $\ga$ under $\widehat f$. Let  $\HC\coloneq \partial_c  U_j $.  Then it follows from Lemma \ref{lem:blowup/annuli} that 
 $\HC \subset \partial_c  \widehat U_j$. Moreover,  \eqref{eq:circuit_length_relation} implies that  $\HC$ (considered as a circuit) realizes the essential circuit lengths of $U_j$ and $\widehat U_j$, which are both equal to $2n$.

 Now 
  $e\sub \partial_c  U_j=\HC$, and so $\HC$ meets the peripheral pullback $\widehat \ga$ of $\ga$ under $\widehat f$. The second part of 
 Lemma~\ref{lem:preimage-bound-refined} applied to $\HC$ implies that the number $k$ of essential pullbacks of $\ga$ under $\widehat f$ is less than $n$, contradicting $k=n$.

 \smallskip To summarize, these contradictions show that 
  $\alpha$ has no peripheral pullbacks under $\widehat f$, and hence no peripheral pullbacks under $f$ by
  Lemma~\ref{lem:blowup/annuli}~\ref{blow3}.    So we can apply 
 Lemma~\ref{lem:parabolicity} and conclude that $f$ has a parabolic orbifold. This is yet another contradiction, because $f$ has a hyperbolic orbifold by our hypotheses.  This shows that our initial assumption that $\widehat f$ has an obstruction is false. Hence $\widehat f$ is realized by a rational map. This completes the proof of Theorem~\ref{thm:blow-up-obstr}. 
 \end{proof}

 \begin{rem} \label{rem:killobs} 
Let $f\: S^2 \ra S^2$ be a Thurston map with $\#P_f=4$
and a hyperbolic orbifold, and suppose $f$ has an  obstruction represented  by a Jordan curve $\alpha$ in $(S^2, P_f)$. Then there  always exist a set of arcs $E\neq \emptyset$ in $(S^2, f^{-1}(P_f))$ satisfying the $\alpha$-restricted blow-up conditions and  multiplicities $m_e$, $e\in E$, such  that the corresponding blown-up map $\widehat f$ satisfies $\lambda_{\widehat f}(\alpha) < 1$. We are then exactly in the setup of Theorem~\ref{thm:blow-up-obstr}. 

To see this, we first fix some core arcs $a$ and $c$ of $\alpha$ lying in different components of $S^2\setminus \alpha$. We now choose an arc $e_0$ in $(S^2,P_f)$ with 
$\ins(e_0,\alpha)=\#(e_0 \cap \alpha) = 1$ and $\inte(e_0)\subset \Sp\setminus (a\cup c)$. Let $E$ be the (non-empty) set of all lifts of $e_0$ under $f$. Then $E$ is a set of arcs in $(S^2, f^{-1}(P_f))$ and  it is clear that 
$E$ satisfies the $\alpha$-restricted blow-up conditions.
If $\widetilde \alpha$ is any pullback of $\alpha $ under $f$, then there exists at least one arc $ e \in E$ that meets 
$\widetilde \alpha$ (necessarily in an interior point of $e$).  Blowing up the arc $e$ with some multiplicity $m_e\in \N$ increases the mapping degree for the corresponding pullback $\widehat \alpha$ under  $\widehat f$ by $m_e$ and does not change the isotopy class of this pullback, that is, 
$[\widehat \alpha] = [\widetilde \alpha]$ rel.\ $P_{\widehat f}= P_f$ (this easily follows from Lemma~\ref{lem:blowup/annuli} and its proof).  Note that each   pullback $\widehat \alpha$ of $\alpha$ under $\widehat f$  corresponds to a   pullback $\widetilde \alpha$ of $\alpha$  under $f$ (this is essentially   Lemma~\ref{lem:blowup/annuli}~\ref{blow2}).   
It follows that  if we choose the multiplicities $m_e$, $e\in E$, large enough, then for the Thurston map $\widehat f$  we will have 
$\lambda_{\widehat f}(\alpha) < 1$ and so $\alpha$ is
 not an obstruction for  $\widehat f$. By Theorem~\ref{thm:blow-up-obstr} the map $\widehat f$
is actually realized by a rational map. So  by a suitable blow-up operation an obstructed
Thurston map  $f$  (with $P_f=4$ and a hyperbolic orbifold) can be turned into a Thurston map $\widehat f$ that is realized.  \end{rem}

\section{Global curve attractors}\label{sec:attractor}

In this section we will prove Theorem~\ref{thm:finite-curve-attr2}. 
We consider the pillow $\PP$ with its vertex set $V= \{A,B,C,D\}$.  For the remainder of this section, $f\:\PP\to \PP$ is  a Thurston map   obtained from the  $(2\times 2)$-Latt\`{e}s map by gluing $n_h\geq 1$ horizontal and $n_v\geq 1$ vertical flaps to $\PP$. Then $f$ is Thurston equivalent to a rational map  by Theorem \ref{thm:flapped_intro}.
In the following, all isotopies on $\PP$ are con\-sidered relative to $\postf=V$.

In order to prove Theorem~\ref{thm:finite-curve-attr2}, 
we want to show that Jordan curves in $(\PP,V)$ 
 are getting ``less twisted'' under taking preimages under $f$. To formalize this, we define the \emph{complexity}  $\| x \|$ of $x\in  \widehat{\Q}\cup \{\odot\}$ as $\Vert x \Vert\coloneq 0$ for $x= \odot$ 
 and $\| x\|\coloneq |r|+s$ for $x=r/s\in \widehat{\Q}$. 
Recall that $\odot$ represents the isotopy classes of all peripheral curves, and that for a slope  $r/s\in \widehat{\Q}$  we use  the convention that the numbers $r\in \Z$ and $s\in \N_0$  are relatively prime and that $r=1$ if $s=0$.  Note that $\| x \|=0$ for 
 $x\in  \widehat{\Q}\cup \{\odot\}$ if and only if $x=\odot$. 
 
 The complexity admits a natural interpretation in terms of intersection numbers. To see this, recall that  
 $\alpha^h$ and $\alpha^v$ (see \eqref{eq:hori+vert}) represent simple closed geodesics  in $(\PP, V)$ that separate the two horizontal and the two vertical edges of $\PP$, respectively. Suppose the slope $r/s \in \widehat{\Q}$ corresponds to  the isotopy class $[\ga]$ of a (necessarily essential)  Jordan curve 
 $\ga$ in $(\PP,V)$. Then by Lemma~\ref{lem:i-properties-curve}~\ref{item:i5},  
\[\| r/s \|=|r|+s= \tfrac12 \ins(\gamma,\alpha^h)+ \tfrac12 \ins(\gamma,\alpha^v).\]
Moreover, if $\ga$ is peripheral, then 
$\ins(\gamma,\alpha^h)+ \ins(\gamma,\alpha^v)=0$,  which agrees with the fact that $\| {\odot} \|=0$.

As we will see, under the slope map $\mu_f$ (as defined in Section \ref{subsec:intro_attractor}) complexities 
do not increase, and actually strictly decrease  unless the slope  belongs to a certain finite set.  More precisely, we will show the following statement.

\begin{prop}\label{prop:complexity-decreases}
Let $f\:\PP\to \PP$ be a Thurston map obtained from the  $(2\times 2)$-Latt\`{e}s map by gluing $n_h\geq 1$ horizontal and $n_v\geq 1$ vertical flaps to the pillow $\PP$. Then the following statements are  true:
\begin{enumerate}[label=\text{(\roman*)},font=\normalfont]

\item 
 $\|\mu_f(x) \| \leq \|x \|$ for all  
  $x\in  \widehat{\Q}\cup \{\odot\}$.

\smallskip
\item $\|\mu_f(x) \| < \|x \|$  for all  
 $x\in  \widehat{\Q}\cup \{\odot\}$ with   $\|x\|> 8.$  
\end{enumerate}
\end{prop}

  Since the set $ \{x \in  \widehat{\Q}\cup \{\odot\} : \|x\|\leq 8\}$
  is finite, we actually have the strict inequality in (i) with at most finitely many exceptions.  The proof of the proposition will show that $\|\mu_f(x) \| =\|x \|$ if and only if $\mu_f(x)=x$ (see Remark~\ref{rem:mufix}). 
As we will see below, Theorem~\ref{thm:finite-curve-attr2}  easily follows from Proposition~\ref{prop:complexity-decreases}.

Before we proceed with the proof of this proposition, 
 we will establish  several auxiliary results.
 As in Section \ref{subsec:pillow},   $a, b, c, d$ are the edges of the pillow $\PP$, and $\wp\:\C\to \PP$  denotes the Weierstrass function that is doubly periodic with respect to the lattice $2\Z^2$.

We are interested in simple closed geodesics and geodesic arcs $\tau$
in $(\PP, V)$. Recall that every such geodesic has the form $\tau=\wp(\ell_{r/s})$ 
for a line $\ell_{r/s}\sub \C$ with slope $r/s\in \widehat \Q$. 
If $\ell_{r/s}\sub \C\setminus \Z^2$, then $\tau=\wp(\ell_{r/s})$ is a simple closed geodesic  in $(\PP,V)$, that is,  $\tau \sub \PP\setminus V$. If $\ell_{r/s}$ contains a point in $\Z^2$,
then $\tau=\wp(\ell_{r/s})$ is a geodesic  arc in $(\PP,V)$, that is, its interior  lies in 
 $\PP\setminus V$ and  its endpoints are in $V$. 

\begin{lemma}\label{lem:1-edge-crossing}
 Let $\tau$ be a simple closed geodesic or a geodesic arc in $(\PP,V)$ with slope $r/s\in\widehat{\Q}$. 
We consider the $1$-edges of $\PP$ with respect to the $(n\times n)$-Latt\`{e}s map $\La_n$, $n\ge 2$, that is, the lifts of the edges $a,b,c,d$ of $\PP$ under $\La_n$.  
 Then the following statements are  true:
\begin{enumerate}[label=\text{(\roman*)},font=\normalfont]
\item If $|r|>2n$, then $ \tau$ intersects the interior of every horizontal $1$-edge of $\PP$. 

\smallskip 
\item If $s>2n$,  then $ \tau$ intersects the interior of every vertical $1$-edge of $\PP$. 
\end{enumerate}
\end{lemma}

\begin{proof}
We will only show the first part  of the statement. The proof of the second part is completely analogous. 

Let $\tau$ be a simple closed geodesic or a geodesic arc in $(\PP,V)$ with slope $r/s\in\widehat{\Q}$ where $|r|>2n$. Suppose  that $e$ is a horizontal $1$-edge and  $\widetilde e$ is a lift of $e$ under $\wp$. Then the arc  $\widetilde e$ is a line segment of length $1/n$ contained in a line $\ell_0\sub \C$  parallel to the real axis. In order to show that $\tau$ meets $\inte(e)$, it suffices to represent the given geodesic $\tau$ in the form $\tau=\wp(\ell_{r/s})$  for a line $\ell_{r/s}\sub \C$ with
 $ \ell_{r/s} \cap\inte(\widetilde e) \ne \emptyset$.

For this  we   choose $p, q\in \Z$ 
 such that $pr+qs=1$   and define $\om\coloneq 
2(s+ir) $ and $\widetilde \om \coloneq 2(-p+iq)$. The numbers  $\om$ and $\widetilde \om$ form a basis of the period lattice $2\Z^2$ 
of $\wp$. In particular, if $\tau=\wp(\ell_{r/s}(z_0))$ for some $z_0\in \C$, then $\tau = \wp(\ell_{r/s}(z_0+j \widetilde \om))$ for all $j\in \Z$. The lines $\ell_{r/s}(z_0+j \widetilde \om)$, $j\in \Z$,
 are parallel and equally spaced. Actually, two consecutive lines in this family differ by a translation by  $\widetilde \om$.
 Since $r\ne 0$, these lines are not parallel to the real axis and so they will cut out subsegments of equal length on  the line $\ell_0$ that contains 
 $\widetilde e$. To determine the length of these segments, we
 write $\widetilde \om$ in the form 
 \begin{equation}\label{eq: spacing} 
\widetilde \om=u+v  \om
\end{equation}
 with $u,v\in \R$.
It is easy to see that \eqref{eq: spacing} implies that 
 $u=-2/r$ (multiply  \eqref{eq: spacing} by the complex conjugate of $\om$ and take imaginary parts), and so the lines in our family cut $\ell_0$ into subsegments  of length  $|u|=2/|r|$. Since 
 $|u|=2/|r|<1/n$ by our hypotheses, one of these lines meets  $\inte(\widetilde e)$. This implies that $\tau\cap \inte(e)\ne \emptyset$, as desired. 
\end{proof}

We now want to see what happens to a geodesic arc  $\xi$ in 
$(\PP,V)$ if we take preimages under a map $f$ as in
 Proposition~\ref{prop:complexity-decreases}.
Unless $\xi$ has  slope in a finite exceptional set,  suitable sets $ \HC$ in the preimage 
 $f^{-1}(\xi)$  will meet the interior of a flap glued to the pillow $\PP$, and consequently a peripheral pullback of the horizontal curve 
 $\alpha^h\sub \PP$ or of the vertical curve $\alpha^v\sub \PP$.
  We will formulate some relevant statements in a slightly more general situation. We first introduce some terminology.

  Suppose $Z\sub S^2$ consists of four distinct points. We say that $K\sub S^2$ {\em essentially separates~$Z$} if we can split $Z$ into two pairs (that is, into disjoint subsets $Z_1,Z_2\sub Z$ consisting of two points each) such that $K$ separates $Z_1$ and $Z_2$. Note that   $K$ trivially  has this property if $K\cap Z$ consists of two or more points.

Now  let $n\in \N$, $n\ge 2$, and consider the $(n\times n)$-Latt\`es map $\La_n\: \PP\ra \PP$. 
 If $\xi$ is a geodesic arc in $(\PP,V)$, then the preimage 
 $\La_n^{-1}(\xi)$ is a disjoint union of simple closed geodesics  and geodesic arcs  in  $(\PP, V)$. Note that each connected component of $\La_n^{-1}(\xi)$ essentially separates $V$,  but no proper subset of such a component does. It follows that if  $K\sub \La_n^{-1}(\xi)$ is a connected set, then it essentially separates $V$   if and only if $K$ is a simple closed geodesic or a geodesic arc in
 $(\PP,V)$.

Let $\widehat \La\: \widehat \PP \to \PP$ be a branched covering map obtained from the $(n\times n)$-Latt\`es map by gluing flaps to $\PP$. As in  Section~\ref{subsec:blownup-lattes}, we denote by $\widehat V$ the vertex set and by $B(\widehat \PP)$ the base of the flapped pillow $\widehat \PP$. By construction, $\widehat \La$ maps each $1$-tile of $\widehat \PP$ by a Euclidean similarity (with scaling factor $n$) onto a $0$-tile   of $\PP$. We also recall that we can naturally view the base  $B(\widehat \PP)$ as a subset of $\PP$ (see \eqref{eq:base-ident}) and, with such an identification, the map $\widehat \La$ coincides with $\La_n$ on $B(\widehat \PP)$.

 Suppose that a geodesic arc $\xi$ in $(\PP,V)$ joins two distinct points $X,Y\in V$. We consider $\widehat \GC\coloneq\widehat \La^{-1}(\xi)$ as a planar embedded  graph in $\Phat$ 
 with the set of vertices $\widehat \La^{-1}(\{X,Y\})$ and the edges given by the lifts of $\xi$ under $\widehat \La$.

\begin{lemma}\label{lem:arcs-over-flaps}
Let $\widehat \La\: \widehat \PP \to \PP$ be a branched covering map obtained from  the $(n\times n)$-Latt\`es map with $n\ge 2$ by  gluing $n_h\geq 1$ horizontal and $n_v\geq 1$ vertical flaps to $\PP$. Suppose $\xi$ is a geodesic arc in $(\PP,V)$ with slope $r/s \in \widehat \Q \setminus \{0,\infty\}$ and $\widehat \xi$ is a lift of $\xi$ under $\widehat \La$.

Let $F$ be a flap in $\widehat \PP$ with the base edges $e'$ 
and $e''$. If $$\widehat \xi \cap \left(\inter(F) \cup \inter (e') \cup \inter (e'')\right) \neq \emptyset,$$ then $\widehat \xi$ meets a base edge and the top edge of the flap $F$.
\end{lemma} 

\begin{proof} The proof is similar to the proof of Lemma \ref{lem:curves-over-flap}. Recall that  $a,b,c,d$ denote the edges of the pillow $\PP$. Suppose $\xi\subset \PP$, $\widehat\xi \subset \widehat \PP$, and $e',e''\subset F$ are as in the statement of the lemma. Let $\widetilde e\sub F$ be the top edge of $F$.

Without loss of generality we will assume that $F$ is a horizontal flap. Then   $\widehat \La(e')=a$ or  $\widehat \La(e')=c$. We will make the further assumption that $\widehat \La(e')=a$. The other cases, when $\widehat \La(e')=c$ or when $F$ is a vertical flap, can be treated in a way that is completely analogous to the ensuing argument. Then $\widehat \La(e'')=a$ and $\widehat \La(\widetilde e)=c$.   Moreover, 
\begin{equation}\label{eq:pre-on-flap-for-arc} 
\widehat \La^{-1}(a\cup c)\cap F=e'\cup e'' \cup \widetilde e.
\end{equation}

Since $\xi$ is a geodesic arc in $(\PP, V)$ with slope $r/s\ne 0$,  by Lemma \ref{lem:alter} the sets  $a\cap \xi $ and $c\cap \xi$ are  non-empty and finite,  and the points  in these sets  alternate on $\xi$.  We claim that there is a point $p\in \widehat \xi \cap \inter(F)$. By our hypotheses this can only  fail if  $\widehat \xi$ meets either $\inter(e')$ or $\inter(e'')$ in a point $q$. Since the arc   $\xi$ has a transverse intersection with $\inter(a)$ at $\widehat \La (q)$, the arc  $\widehat \xi$ has a transverse intersection with $\inter(e')$ or $\inter(e'')$ at $q$. Then  $\widehat \xi$  meets $\inter(F)$ in a point $p$ in any case. 

Since the points in  $a\cap \xi \ne \emptyset $ and $c\cap \xi\ne \emptyset $  alternate on $\xi$,  the points in $\widehat \xi \cap  \widehat \La^{-1}(a)\ne \emptyset $ and $\widehat \xi \cap  \widehat \La^{-1}(c)\ne \emptyset$  alternate on $\widehat \xi$. Note that $F\cap \widehat \La^{-1}(c)=\widetilde   e$ and $F\cap \widehat \La^{-1}(a)= \partial F = e' \cup e''$. So, if we trace the arc $\widehat \xi$ starting from $p$ in two different directions,  we must meet a base edge of $F$ in one direction and the top edge of $F$ in the other direction. The statement follows.
\end{proof}

 Now the following fact is true. 
 
 \begin{lemma}\label{lem:curves-on-flapped-pillow}
Let $\widehat{\La}\:\Phat\to \PP$ be a branched covering map obtained  from the $(n\times n)$-Latt\`{e}s map  with $n\ge 2$ by gluing $n_h\geq 1$ horizontal and $n_v\geq 1$ vertical flaps to $\PP$.  Suppose that $\xi$ is a geodesic arc in $(\PP,V)$ with slope $r/s\in \widehat \Q$ and $ \HC$ is any  connected subgraph of $ \widehat \GC=\widehat{\La}^{-1}(\xi)$  that essentially separates 
$\widehat V\sub \Phat$. If $|r|+s> 4n$, then $\HC$ meets a base edge and the top edge of a flap in $\Phat$. 
\end{lemma}

\begin{proof} 
Suppose $\xi$ is a geodesic arc in $(\PP,V)$ with  slope 
$r/s\in \widehat \Q$, where $|r|+s> 4n$, and  $ \HC$ is   a connected subgraph  of $\widehat \GC=\widehat {\La}^{-1}(\xi)$ that essentially separates the vertex set $\widehat V$ of $\widehat \PP$. 
Note that then $r/s\ne 0,\infty$, which will allow us to apply Lemma~\ref{lem:arcs-over-flaps}. Each edge of the graph  
$\widehat \GC=\widehat {\La}^{-1}(\xi)$ is a lift $\widehat \xi$ of $\xi$ as in this lemma.

We now argue by contradiction and suppose that there is no flap $F$ in $\widehat \PP$ such  that 
$\HC$ meets both a base edge and the top edge of $F$. By the definition of $B(\widehat \PP)$ (see \eqref{eq:base-pillow}) and  Lemma~\ref{lem:arcs-over-flaps}, this means that each edge of $\HC$, and thus the graph $ \HC$ itself, is contained in $B(\widehat \PP)$. Consequently,  we can consider 
$ \HC$ as a connected  subset of $\PP\supset B(\widehat \PP)$. On  
$B(\widehat \PP)$ the maps $\widehat \La$ and $\La_n$ are identical.  Therefore, we can also regard  $ \HC$ as a connected subset 
of $\La_n^{-1}(\xi)$.  

The set  $\HC$, now considered as a subset of $\PP$, essentially separates $V\sub \PP$.
To see this, let $\widehat V_1, \widehat V_2\sub \widehat V\sub \Phat$ be two pairs of vertices separated by $ \HC$ in   $\Phat$. We can identify $\widehat V_1$ with a pair $V_1$ and 
 $\widehat V_2$ with a pair  $V_2$ of vertices of $\PP$. We claim that $V_1$ and $V_2$   are separated by $\HC$ in $\PP$. Indeed, if this was not the case, then we could find a path $\beta$ in $\PP$ that joins $V_1$ and $V_2$ without meeting 
$\HC$. This path can be modified as follows to a path $\widehat \beta$ in 
$\Phat$ that joins  $\widehat V_1$ and $\widehat V_2$ and does not meet $\HC\sub \Phat$: if $\beta$ 
meets some $1$-edge $e$ of $\PP$ to which one or several  flaps are glued, then on $\beta$ there is   a first  point $p\in e$ and a 
last  point $q\in e$. We now replace the part of $\beta$ between $p$ and $q$  by a path that joins  points corresponding to $p$ and $q$ in $\widehat \PP$, travels on these flaps,  and   does  not meet $\HC$.  If we make such replacements for all these edges $e$ consecutively, then we obtain a path
$\widehat \beta$ that joins  $\widehat V_1$ and $\widehat V_2$, but is disjoint from $\HC$.  But such a path $\widehat \beta$
cannot exist, because $ \HC$ separates $\widehat V_1$ and $\widehat V_2$ in $\Phat$. 

We see that  $ \HC\sub \La_n^{-1}(\xi)$  indeed essentially separates $V$. Since $\HC$ is connected, the discussion above  (after the definition of essential separation) implies that   $ \HC$ is  a 
simple closed geodesic or a geodesic arc 
 in $(\PP,V)$ with slope $r/s$. Since $|r|+s>4n$, either $|r|>2n$ or $s>2n$. Thus, by Lemma \ref{lem:1-edge-crossing}, the geodesic  
$\HC$ meets each horizontal $1$-edge of $\PP$ in the first case or each vertical $1$-edge of $\PP$  in the second case. Since $n_h\geq 1$ and $n_v\geq 1$, in either case, $\HC$ must meet the interior of a $1$-edge along which a flap  is glued and hence cannot be a subset of $B(\widehat \PP)$. This is a  contradiction  and the lemma follows. 
\end{proof}

 \begin{rem}\label{rem:betameetsperi}
 Suppose we are in the setup of Lemma \ref{lem:curves-on-flapped-pillow}. Then the connected set $\HC$ meets a base edge   and the top edge of a flap, say a horizontal flap $F$.  Then there exists a peripheral pullback $\widehat \alpha$ of the horizontal curve
 $\alpha^h$ under the map $ \widehat \La$ that is contained in $F$.  Let $e'$ and $e''$ be the base edges of $F$, and $\widetilde e$ be the top edge of $F$.   
 Then the   curve $\widehat \alpha$ separates $\partial F = e'\cup e''$ from  $\widetilde e\sub F$. Since $\HC$ is connected and meets both  $\widetilde e$ and 
 $e'\cup e''$, we conclude that  $\HC \cap \widehat \alpha \ne \emptyset$. If $\beta$ is a connected set that traces $\HC$ closely, then it will also have points close to 
 $\widetilde e$ and close to  $e'\cup e''$. Again this will imply that 
 $\beta\cap \widehat \alpha \ne \emptyset$. This remark will become important in the proof of Proposition~\ref{prop:complexity-decreases}. 
 \end{rem}

 A completely analogous statement  to Lemma~\ref{lem:curves-on-flapped-pillow}  is true (with a very similar proof) if we assume that $\xi$ is a simple closed geodesic  in $(\PP,V)$ and $ \HC$~is an essential  pullback of $\xi$ under~$\widehat \La$.

  We now turn to the proof of Proposition~\ref{prop:complexity-decreases}.

\begin{proof}[Proof of Proposition \ref{prop:complexity-decreases}]
Let  $f\:\PP\to \PP$ be  a Thurston map obtained from the  $(2\times 2)$-Latt\`{e}s map by gluing $n_h\geq 1$ horizontal and $n_v\geq 1$ vertical flaps to $\PP$. Then $\postf=V$, where $V=\{A,B,C,D\}$ is the set of vertices  of $\PP$, and $A$ is the  unique point in $\postf=V$ that is fixed by $f$. 

To prove the first statement, let  $x\in  \widehat{\Q}\cup \{\odot\}$
be arbitrary. If $x=\odot$, then $\mu_f(\odot)=\odot$ and  
$\|\mu_f(\odot)\|=\|{\odot}\|=0$.  
So  in the following we will assume that $x=r/s\in   \widehat{\Q}$. 
Let $\gamma\sub \PP\setminus V$ be  a simple closed geodesic  with   slope $r/s\in\widehat {\Q}$. Then 
 $\gamma$ is an essential Jordan curve and so  each of the two  complementary components of $\ga$ in $\PP$ contains 
 precisely two postcritical points of $f$.  Let $\xi$ and $\xi'$ be core arcs of $\gamma$  belonging to  different  components of $\PP\setminus\gamma$.   Here we may assume that $\xi$ and $\xi'$ are geodesic arcs in $(\PP,V)$ with slope $r/s$. 
 
As before, we denote by  $\alpha^h$ and $\alpha^v$  simple closed geodesics  in $(\PP, V)$ that separate the two horizontal and the two vertical edges of $\PP$, respectively. 
Then, by Lemma \ref{lem:i-properties-curve}, we have:
\begin{align*}
\ins(\gamma,\alpha^h)&=2 \ins(\xi, \alpha^h)=\#(\gamma\cap\alpha^h)=2|r|,\\
\ins(\gamma,\alpha^v)&=2 \ins(\xi, \alpha^v)=\#(\gamma\cap\alpha^v)=2s, \\
 \|x\|  &=|r|+s = \tfrac12 \ins(\gamma, \alpha^h)+ \tfrac12  \ins(\gamma, \alpha^v). 
\end{align*}

We call  a point $p\in \PP$  a $1$-{\em vertex} if $f(p) \in 
\postf=\{A,B,C,D\}$. 
We say that a $1$-vertex is of {\em type} $A$, $B$, $C$, or $D$ if it is a preimage of $A$, $B$, $C$, or $D$ under $f$, respectively.

Without loss of generality,  we may assume that the core arc $\xi$ connects the point $A$ with a point $X\in\{B,C,D\}$. Then $\xi'$ joins the two points in $\{B,C,D\}\setminus X$. Let $\GC=f^{-1}(\xi\cup \xi')$, which we view as a planar embedded graph with the set of vertices $f^{-1}(V)$. Note that the degree of a vertex $p$ in $\GC$ is equal to the local degree of the map $f$ at $p$.  In addition, the graph $\GC$ has the following  properties: 

 \begin{enumerate}[label=\text{(P\arabic*)},font=\normalfont]

\item\label{G-prop:bipartite} $\GC$ is a bipartite graph. In particular, $1$-vertices of type $A$ are connected only to $1$-vertices of type $X$ and vice versa.

 \smallskip 
\item\label{G-prop:typeA} Each postcritical point of $f$ is a $1$-vertex of type $A$. If a $1$-vertex of type $A$ has degree $\ge 2$ in $\GC$, then it must be a postcritical point. 
\end{enumerate}

The analog of  \ref{G-prop:bipartite} is valid for arbitrary
Thurston maps with four postcritical  points. To see that 
\ref{G-prop:typeA} is true, note that the $(2\times 2)$-Latt\`es map sends each of the four vertices of $\PP$ to $A$. This remains true if we glue any number of flaps to $\PP$. Moreover, gluing additional flaps can only create additional preimages of $A$ of degree $1$ in $\GC$.

If every pullback  of $\gamma$ under $f$ is peripheral, then 
$\mu_f(x)=\odot$, and so
\begin{equation}\label{eq:tricasesig}
 \| \mu_f(x)\|=\|{\odot}\|=0< |r|+s=\|x\|. 
 \end{equation}

Suppose $\gamma$ has an essential pullback $\widetilde{\gamma}$ under $f$. Then $\mu_f(x)\in \widehat \Q$ is the slope cor\-res\-ponding to the isotopy class of $\widetilde \ga$.
 By the discussion in Section \ref{sec:esscirc}, the pullback $\widetilde{\gamma}$ belongs to a unique component $\widetilde{U}$ of $\PP\setminus \GC$. We use the notation   $\partial_\xi\widetilde{U}\coloneq f^{-1}(\xi)\cap\partial \widetilde{U}$. Then 
$\partial_\xi\widetilde{U}$ is a subgraph of $\GC$ that only contains $1$-vertices of type $A$ and $X$. 
Since $\widetilde{\gamma}$ is essential,  $\partial_\xi\widetilde{U}$ satisfies: 
 \begin{enumerate}
\item[\mylabel{U-prop:essential}{(P3)}] $\#(\partial_\xi\widetilde{U}\cap \postf) \leq 2$. 
\end{enumerate}

 Our goal now is to simplify the pullback $\widetilde{\gamma}$ using an isotopy depending on the combinatorics of $\partial_\xi\widetilde{U}$. More precisely, we will  construct a curve $\beta$ that is isotopic to $\widetilde{\gamma}$, but has fewer intersections with $\alpha^h$ and $\alpha^v$.
In order to obtain a suitable curve $\beta$, we  now distinguish several cases that exhaust all possibilities.

 %\begin{center}
%\vspace{12pt}
\begin{figure}[t]
\def\svgwidth{0.7\columnwidth}
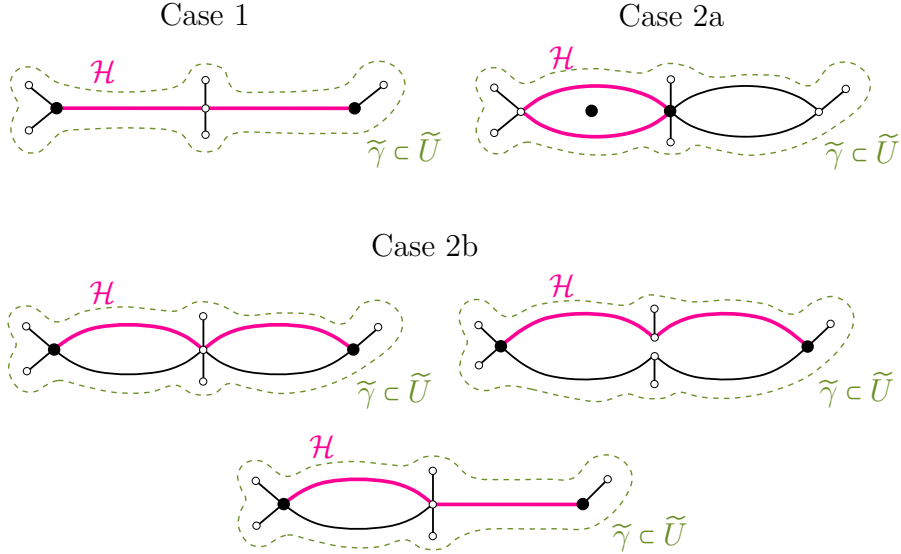
\caption{Different combinatorial types of the graph $\partial_\xi \widetilde U$. The subgraph in magenta corresponds to the appropriate choice of $\HC$ in each case. The vertices in black indicate the postcritical points.
}\label{fig:attractorcases} 
\end{figure}
%\end{center} 

\smallskip
\emph{Case 1:} $\partial_\xi\widetilde{U}$ does not contain any simple  cycle. Then $\partial_\xi\widetilde{U}$  is a tree and, since $\widetilde{\gamma}$ is essential, there are exactly two postcritical points in $\partial_\xi\widetilde{U}$. These are $1$-vertices of type $A$  by \ref{G-prop:typeA}.  Let  $\HC\sub \partial_\xi\widetilde{U}$ be the unique 
simple path  that joins these two postcritical  points in $\partial_\xi\widetilde{U}$.  By   \ref{G-prop:bipartite} the path $\HC$ 
must have length $\ge 2$, because the endpoints of $\HC$ have type $A$ and the vertices of types $A$ and $X$ alternate on $\HC$.

If the length of $\HC$ was $\ge 3$, then $\HC$ would contain 
at least one additional point $p$ of type $A$ apart from its endpoints. 
Then $\deg_\HC(p)=2$, so $\deg_\GC(p)\geq 2$, which means $p$ must be a postcritical point by  \ref{G-prop:typeA}. But then $\HC\sub \partial_\xi\widetilde{U}$ contains at least three postcritical points, which contradicts 
\ref{U-prop:essential}. We conclude that $\HC$ has length $2$; see Figure \ref{fig:attractorcases} (Case 1).

 Let $\widehat{U}=S^2\setminus \HC$. Then the annulus between $\widetilde{\gamma}$ and $\HC$ contains no postcritical points  of  $f$, and hence  for sufficiently small $\epsilon$ each $\epsilon$-boundary $\beta$ of $\widehat{U}$ wrt.\ $\HC$ is isotopic to~$\widetilde{\gamma}$ 
 as follows  from Lemma~\ref{lem:isocrit}.

\smallskip
\emph{Case 2:} $\partial_\xi \widetilde{U}$ contains a simple cycle. Then by \ref{G-prop:bipartite} one of  the vertices of such a  cycle must be of type $A$. Since this  vertex has degree equal to $2$ in the cycle, and hence degree $\ge 2$ in $\GC$, it must be a postcritical point by  \ref{G-prop:typeA}.  It follows that  
$\#(\partial_\xi\widetilde{U}\cap \postf)\geq 1$.  So by  \ref{U-prop:essential}, either $\#(\partial_\xi\widetilde{U}\cap \postf) = 1$ or $\#(\partial_\xi\widetilde{U}\cap \postf) = 2$.

\smallskip
\emph{Case 2a:}  $\#(\partial_\xi\widetilde{U}\cap \postf) = 1$. 
Since $\widetilde{\gamma}$ is essential, there are exactly two postcritical points in  the component of $S^2\setminus \widetilde{\gamma}$ that contains $\partial_\xi\widetilde{U}$.
One of them belongs to $\partial_\xi\widetilde{U}$, while the other one belongs 
to a face of $\partial_\xi\widetilde{U}$ disjoint from $\widetilde U$. This postcritical point then necessarily belongs to a face of a simple cycle $\HC$
in $\partial_\xi\widetilde{U}$.  
This simple cycle $\HC$ then necessarily contains the unique postcritical point 
in $\partial_\xi\widetilde{U}$ as we have seen above. 
Moreover, $\HC$ must have length $2$, because otherwise $\HC$ has an  even length $\ge 4$ by  \ref{G-prop:bipartite}. But then $\HC$ contains another $1$-vertex of type $A$ with degree $\ge 2$, which is necessarily a postcritical point  by  \ref{G-prop:typeA}. Then $\HC\sub \partial_\xi \widetilde{U}$ contains at least two postcritical points, which contradicts our assumption for this case.  So  $\HC$   has indeed length $2$; see Figure~\ref{fig:attractorcases} (Case 2a).

Let $\widehat{U}$ denote the face of $\HC$ that contains $\widetilde{U}$. 
Then again the annulus between $\widetilde{\gamma}$ and $\HC$ contains no postcritical points  of  $f$, and hence each $\epsilon$-boundary $\beta$ of $\widehat{U}$ wrt.\ $\HC$ is isotopic to~$\widetilde{\gamma}$ for sufficiently small $\epsilon$. 

\smallskip
\emph{Case 2b:}  $\#(\partial_\xi\widetilde{U}\cap \postf) = 2$. Let $\HC$ be a simple path in $\partial_\xi\widetilde{U}$ that joins the two postcritical points in $\partial_\xi\widetilde{U}$. By the same reasoning as in Case 1, $\HC$ has length $2$; see Figure \ref{fig:attractorcases} (Case 2b).  Let 
$\widehat{U}=S^2\setminus\HC$.  Since $\widetilde \gamma$ is essential, there are no postcritical points in the annulus between $\widetilde \gamma$ and $\HC$. Thus, each $\epsilon$-boundary $\beta$ of $\widehat{U}$ wrt.\ $\HC$ is isotopic to~$\widetilde{\gamma}$  for sufficiently small $\epsilon$.

\smallskip 
Note that in all cases $\HC$ essentially separates $V=\postf$, because in all cases $\HC$ separates the pairs of  points in $V$ contained in different complementary components of $\widetilde \ga$. Moreover, by our choice the  circuit length of $\widehat U$  is equal to $4$ in Cases 1 and 2b, and equal to $2$ in Case 2a. So in each case it is $\le 4$. Since $\xi$ and $\alpha^h$ are in minimal position as follows from Lemma~\ref{lem:i-properties-curve}, we can apply  Lemma~\ref{lem:choicebd} to the face $\widehat U$ of $\HC$. Hence for each sufficiently small $\eps>0$ we can always  find an $\eps$-boundary $\beta$ of $\widehat U$ wrt.\ $\HC$ that is isotopic to~$\widetilde{\gamma}$  and satisfies  $\#(\beta\cap f^{-1} (\alpha^h))\leq 4 \ins (\xi, \alpha^h)$.

Let $\widetilde{\alpha}_{1}$  and $\widetilde{\alpha}_{2}$ be the two pullbacks of $\alpha^h$ under $f$  that are isotopic to $\alpha^h$ (there are exactly two such pullbacks by Lemma \ref{lem:blown-lattes-horizontal}). Then in all cases we  have \begin{align}
2\ins(\widetilde \ga, \alpha^h)&=
 2\ins(\beta, \alpha^h)=\ins(\beta,\widetilde{\alpha}_{1})+ \ins(\beta,\widetilde{\alpha}_{2})  \label{eq:compl-decr-H} \\ & \leq  \# (\beta\cap\widetilde{\alpha}_{1}) + \#(\beta\cap\widetilde{\alpha}_{2}) \notag \\&\leq  \#(\beta\cap f^{-1} (\alpha^h)) \notag \\
 & \leq 4 \ins (\xi, \alpha^h) \notag \\
 & = 2 \ins (\gamma, \alpha^h).\notag  
 \end{align}
 Thus, $\ins(\widetilde \ga, \alpha^h) \leq \ins(\gamma, \alpha^h)$.
 The same reasoning (with a possibly different choice of $\beta$) also shows
$\ins(\widetilde \ga, \alpha^v) \leq \ins(\gamma, \alpha^v).$
Combining  these inequalities, we conclude: 
\begin{equation}\label{eq:compl-eq}
\|\mu_f(x)\|	=  \tfrac12  \ins(\widetilde{\gamma}, \alpha^h)+ \tfrac12  \ins(\widetilde{\gamma}, \alpha^v) \leq  \tfrac12  \ins(\gamma, \alpha^h)+  \tfrac12   \ins(\gamma, \alpha^v) = \|x\|.
\end{equation}
This completes the proof of the first part of the statement. 

Note that the second inequality  in \eqref{eq:compl-decr-H} is strict if $\beta$ intersects a peripheral pullback of~$\alpha^h$.
A similar statement  is also true 
 for the analogous inequality for the curve $\alpha^v$. 
 We now  assume that $x=r/s\in \widehat \Q$  satisfies $\|x\|>8$.
 We will argue that then either inequality  \eqref{eq:compl-decr-H}
 or the analogous inequality  for $\alpha^v$ is strict. This  will lead to 
 $\|\mu_f(x)\|< \|x\|. $
 
 To see this, first  note that   $f=\widehat \La \circ \phi^{-1}$, where $\widehat \La\:   \Phat \ra \PP $ is the associated branched covering map obtained by blowing up the $(2\times 2)$-Latt\`es map and $\phi\: \Phat \ra \PP$ is a suitable homeomorphism (see Section~\ref{subsec:blownup-lattes} for the details). 
 We proceed as in the first part of the proof and again represent $x$ by a simple closed geodesic $\ga$ in $(\PP,V)$ with slope $x=r/s$.    We may assume 
that $\ga$ has an essential pullback $\widetilde \ga$ under $f$, because otherwise we have the desired strict inequality by 
\eqref{eq:tricasesig}.

We choose a geodesic core arc $\xi$ in $(\PP,V)$  and a connected set 
$\HC\sub f^{-1}(\xi)$  as before and define $\widehat \HC\coloneq \phi^{-1} (\HC)$. Then
$\widehat \HC$ is a connected subset of $\widehat \La^{-1}(\xi)$ that essentially separates~$\widehat V$, where $\widehat V=
\phi^{-1}(V)$ is the set of vertices of the flapped pillow $\Phat$. 
Since $\|x\|=|r|+s>8$, 
we can apply  Lemma~\ref{lem:curves-on-flapped-pillow} (with $n=2$) and conclude that 
the set $\widehat\HC$ will meet a base edge and the top edge of some flap $F$ in $\Phat$. We will assume that $F$ is a horizontal flap, the case of a vertical flap being completely analogous. 

If $\eps$ is small enough, then the $\eps$-boundary $\beta$ constructed above traces $\HC$ very closely in the sense that for each point in $\HC$, there is a nearby point in $\beta$. The same is true  for $\widehat \beta\coloneq \phi^{-1} (\beta)$ and $\widehat \HC$. Using Remark \ref{rem:betameetsperi}, this implies that if $\eps$ is sufficiently small (as we may assume), then 
 $\widehat \beta$ will meet the peripheral pullback $\widehat \alpha$ of $\alpha^h$ under $\widehat \La$ that is contained in the horizontal flap $F$.  Consequently, $\beta$ meets the peripheral pullback $\phi
(\widehat \alpha)$ of $\alpha^h$ under $f$. As we already pointed out, this leads to a strict inequality in \eqref{eq:compl-decr-H} and thus also in \eqref{eq:compl-eq}. The statement follows. 
 \end{proof}

 The proof of Theorem~\ref{thm:finite-curve-attr2}  is now easy.
  
 \begin{proof}[Proof of Theorem~\ref{thm:finite-curve-attr2}] 
Let $f\:\PP \ra \PP$ be a Thurston map as in the statement.
 Then Proposition~\ref{prop:complexity-decreases} implies that 
if  $x\in  \widehat \Q \cup\{\odot\} $ is arbitrary, then 
 the complexities of  the elements $x,\, \mu_f(x),\, \mu^2_f(x),\, \dots$ of the  orbit of $x$  under iteration of $\mu_f$ strictly decrease until this orbit eventually reaches the finite set $S\coloneq \{u \in  \widehat{\Q}\cup \{\odot\} : \|u\|\leq 8\}$. From this point on, the orbit of $x$ stays in $S$. The statement follows.
 \end{proof}

 \begin{rem}\label{rem:mufix} 
 The proofs of  Theorems~\ref{thm:finite-curve-attr1} and~\ref{thm:finite-curve-attr2}
 show that a  global curve attractor $\AC(f)$ for $f$ can be obtained from Jordan curves 
 corresponding to slopes in  the finite set $S= \{x \in  \widehat{\Q}\cup \{\odot\} : \|x\|\leq 8\}$. Actually,  \eqref{eq:compl-decr-H} and \eqref{eq:compl-eq} imply that $\|\mu_f(x)\|=\|x\|$ if and only if $\mu_f(x)=x$. Therefore, the {\em minimal} global curve attractor  
 $\AC(f)$
 corresponds to  the set $\{x \in  \widehat{\Q}\cup \{\odot\} : \mu_f(x)=x\} \subset S$. 
In other words,  the minimal  $\AC(f)$ consists of  peripheral curves and essential curves that are invariant under $f$
(up to isotopy). 
 \end{rem}

In principle, a global curve attractor for a map $f$ as in  Theorem~\ref{thm:finite-curve-attr2} depends on the locations of the flaps. By Remark~\ref{rem:mufix} for each concrete case  one can easily determine the exact attractor by checking if a slope 
$x \in  \widehat{\Q}$ with   $\|x\|\leq 8$ is invariant. For example, 
by using a computer  program written by Darragh Glynn, we verified  that for the map $f$ corresponding to the flapped pillow in Figure \ref{fig:flabsatcorner} (with one horizontal flap and one vertical flap
glued at the two $1$-edges of $\PP$ incident to the vertex $B$) the invariant slopes are $0$, $\infty$,  $1$, $-1$.

%\begin{center}
%\vspace{10pt}
\begin{figure}[t]
\def\svgwidth{0.31\columnwidth}
%% Creator: Inkscape 1.0.1 (c497b03c, 2020-09-10), www.inkscape.org
%% Accompanies image file 'flabsatcorner.pdf' (pdf, eps, ps)
\begingroup%
  \makeatletter%
  \providecommand\color[2][]{%
    \errmessage{(Inkscape) Color is used for the text in Inkscape, but the package 'color.sty' is not loaded}%
    \renewcommand\color[2][]{}%
  }%
  \providecommand\transparent[1]{%
    \errmessage{(Inkscape) Transparency is used (non-zero) for the text in Inkscape, but the package 'transparent.sty' is not loaded}%
    \renewcommand\transparent[1]{}%
  }%
  \providecommand\rotatebox[2]{#2}%
  \newcommand*\fsize{\dimexpr\f@size pt\relax}%
  \newcommand*\lineheight[1]{\fontsize{\fsize}{#1\fsize}\selectfont}%
  \ifx\svgwidth\undefined%
    \setlength{\unitlength}{263.2842057bp}%
    \ifx\svgscale\undefined%
      \relax%
    \else%
      \setlength{\unitlength}{\unitlength * \real{\svgscale}}%
    \fi%
  \else%
    \setlength{\unitlength}{\svgwidth}%
  \fi%
  \global\let\svgwidth\undefined%
  \global\let\svgscale\undefined%
  \makeatother%
  \begin{picture}(1,0.93250957)%
    \lineheight{1}%
    \setlength\tabcolsep{0pt}%
    \put(0,0){\includegraphics[width=\unitlength,page=1]{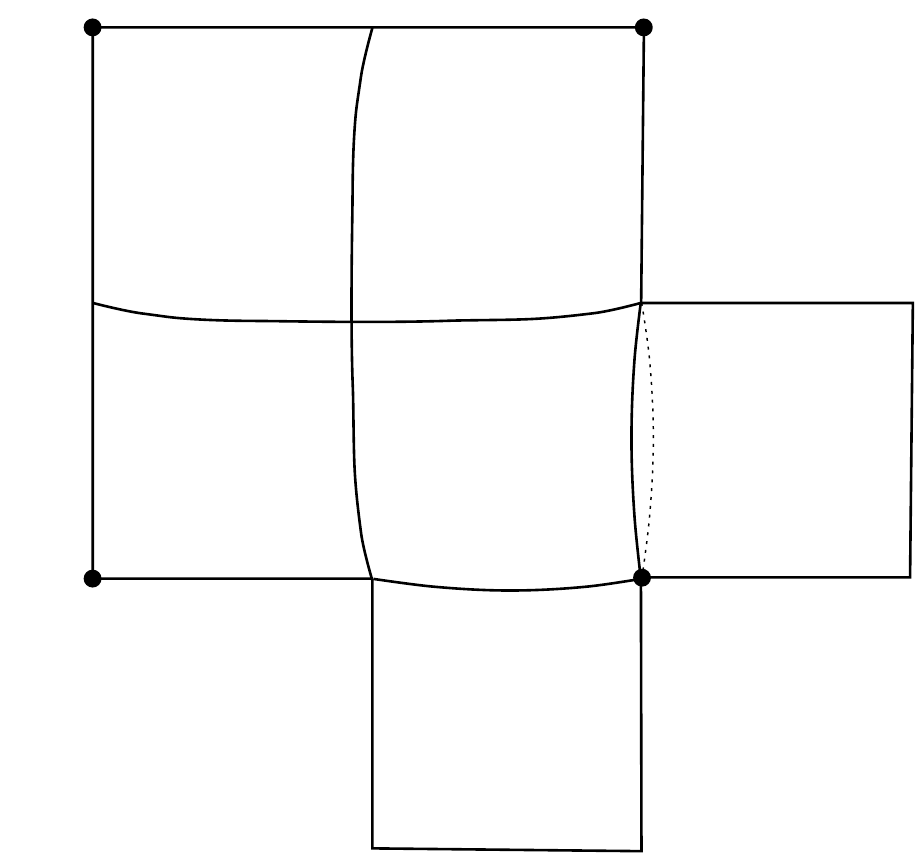}}%
    \put(0.01791479,0.23653157){\makebox(0,0)[lt]{\lineheight{1.25}\smash{\begin{tabular}[t]{l}$A$\end{tabular}}}}%
    \put(0.71833322,0.22204621){\makebox(0,0)[lt]{\lineheight{1.25}\smash{\begin{tabular}[t]{l}$B$\end{tabular}}}}%
    \put(-0.00355942,0.89319496){\makebox(0,0)[lt]{\lineheight{1.25}\smash{\begin{tabular}[t]{l}$D$\end{tabular}}}}%
    \put(0.73542504,0.88048566){\makebox(0,0)[lt]{\lineheight{1.25}\smash{\begin{tabular}[t]{l}$C$\end{tabular}}}}%
    \put(0,0){\includegraphics[width=\unitlength,page=2]{flabsatcorner.pdf}}%
  \end{picture}%
\endgroup%

\caption{A flapped pillow.}\label{fig:flabsatcorner}
\end{figure}
%\end{center}

\section{Further discussion}\label{sec:Further_discussion}

In this section, we briefly discuss  some additional topics related to the investigations in this paper.

\subsection{Julia sets of blown-up Latt\`{e}s maps}
An obvious question is what we can say about the Julia sets of the
rational maps provided by Theorem~\ref{thm:flapped_intro}
(for the definitions of the Julia and Fatou sets of rational maps   and other basic notions in  complex dynamics, see \cite{Milnor_Book}).

\begin{prop} \label{prop:Julia} Let $g\: \CDach \ra \CDach$ be a rational map that is Thurston equivalent to a map $f\: \PP\ra \PP$ obtained from the $(n\times n)$-Latt\`{e}s map $\La_n$ with $n\ge 2$ by gluing $n_h\ge 1$ horizontal and $n_v\ge 1$ vertical flaps to the pillow $\PP$. Then the following statements are true: 

\begin{enumerate}[label=\text{(\roman*)},font=\normalfont]

\item The Julia set of $g$ is equal to $\CDach$ if $n$ is even and the vertex $A$ is not contained in a flap, or if $n$ is odd and none of the points in $V$ is contained in a flap. 

\smallskip
\item  The Julia set of $g$ is equal to a Sierpi\'nski carpet in $\CDach$  if $n$ is even and $A$ is  contained in a flap, or if $n$ is odd and at least one of the points in $V$ is contained in a flap.
\end{enumerate}
\end{prop}

Obviously, these cases cover all possibilities and so the Julia set 
of $g$ is either the whole Riemann  sphere $\CDach$ or a Sierpi\'nski carpet, i.e.,  
a subset of $\CDach$ homeomorphic to the standard $1/3$-Sierpi\'nski carpet fractal. 
As we will see, in the first case the map $g$ has no periodic critical points, while it has periodic critical points (namely critical fixed points) in the second case.

\begin{proof}
Let $g$ be a rational map as in the statement. To see what the Julia set of $g$ is, we will check whether $g$ has periodic critical points or not, and verify in the former case that $g$ has no Levy arcs (see below for the definition). These conditions are invariant under Thurston equivalence and  therefore it is enough to consider the map $f$.  Then $\postf=V$, where $V=\{A,B,C,D\}$ is the set of vertices  of $\PP$. By definition of the $(n\times n)$-Latt\`{e}s 
map $\La_n$, for each $X\in V$ we have 
$\La_n(X)=A$ if $n$ is even and $\La_n(X)=X$ if $n$ is odd. 
Since $f|V$ agrees with $\La_n|V$, this implies that for  each $X\in V$ we also have 
$f(X)=A$ if $n$ is even and $f(X)=X$ if $n$ is odd. 

Since the orbit of each critical point under iteration of $f$ passes 
through the set $P_f=V$, this shows that any  periodic
critical point of $f$ must be equal to  the point $A$ if $n$ is even or must belong to $V$ if 
$n$ is odd. Now for $X\in V$ we have  $\deg_{\La_n}(X)=1$  and so
 $$\deg_{f}(X)=n_X+\deg_{\La_n}(X)=n_X+1,$$ where $n_X\in \N_0$ is the number of flaps that contain $X$. 
 These considerations show that $f$, and hence also $g$, has no periodic critical points in case (i). Hence the Julia set of $g$ is the whole Riemann sphere $\CDach$ in this case (see \cite[Corollary 19.8]{Milnor_Book}).
 
 In case (ii), the map $f$, and hence also $g$,  has a critical fixed point, and so the  Fatou set of $g$ is non-empty. To show that its Julia set is a Sierpi\'nski carpet, 
we  use the following criterion that follows from \cite[Lemma 4.16]{BD_Exp}:  the Julia set of $g$ is a Sierpi\'{n}ski carpet if and only if $g$, or equivalently $f$,  has no Levy arcs. Here a \emph{Levy arc} of $f$ is a path $\alpha$ in $\PP$ satisfying the following conditions:
\begin{enumerate}[label=(\text{L\arabic*)},font=\normalfont]

\item $\alpha$ is an arc in $(\PP, V)$,   or $\alpha$ is a simple loop based at a point $X\in V$ such that $\alpha \setminus \{X\}\sub \PP\setminus V$ and each component of $\PP\setminus  \alpha$ contains at least one point in $V$. 

\smallskip
\item There exist  $k\in\N$ and a lift $\widetilde \alpha$ of $\alpha$ under $f^k$ such that $\alpha$ and $\widetilde \alpha$ are isotopic rel.\ $V$. 
\end{enumerate}

Now suppose that $f$ has a Levy arc $\alpha$ with $\widetilde \alpha$ and $k\in\N$ as in (L2). Then $f^k|\widetilde \alpha$ is a $1$-to-$1$ map and either $\ins(\alpha, \alpha^h)>0 $ or $\ins(\alpha, \alpha^v)>0$. Without loss of generality, we may assume that $\ins(\alpha, \alpha^h)=\#(\alpha\cap \alpha^h)>0$. If $\alpha$ is an arc in $(\PP, V)$, then we can apply Lemma \ref{lem:preimage-bound} with $\gamma\coloneq \alpha^h$, $f\coloneq f^k$, and $\widetilde \alpha\coloneq \widetilde \alpha$ and conclude that the number of distinct pullbacks of $\alpha^h$ under $f^k$ that are isotopic to $\alpha^h$ is at most $1$.
This is also true if $\alpha $ is a  simple loop as in (L1) by the  argument in the proof of  Lemma~\ref{lem:preimage-bound}.
 
 We reach  a  contradiction, because it  follows from Lemma \ref{lem:blown-lattes-horizontal} that $\alpha^h$  has  $n^k>1$ such pullbacks.  Consequently, $f$ and $g$ do not have any Levy arcs and so the Julia set of $g$ is a  Sierpi\'{n}ski carpet.
\end{proof}

\subsection{The global curve attractor problem}\label{sec:numerics}  We were able to prove the existence of a finite global curve attractor only for blown-up $(n\times n)$-Latt\`{e}s maps with  $n=2$. The proof of Theorem~\ref{thm:finite-curve-attr2} crucially relies on 
 Proposition \ref{prop:complexity-decreases}, which says that the (naturally defined) complexity of curves does not increase under the pullback operation. The latter statement is false in general for blown-up $(n\times n)$-Latt\`{e}s maps with $n\geq 3$. 
 
 Numerical  computations  by Darragh Glynn suggest that for some blown-up $(3\times 3)$-Latt\`{e}s map $f$ one can have infinitely many slopes $x\in \widehat {\Q}$ such that $\|\mu_f(x) \| > \| x \|$.   For example, consider the map $f$ obtained from the $(3\times 3)$-Latt\`{e}s map by blowing up once the horizontal and vertical edges incident to the vertex $B$ of $\PP$. Then one can prove the following general relation for the slope map $\mu_f$:
  $$\mu_f(r/s)=r'/s' \Rightarrow \mu_f(r/(s+24r))=r'/(s'+22r').$$
Based on this one can show that $\|\mu_f(x) \| > \| x\|$ for all 
$$x\in \{1/(m+24k) :m\in\{7,8,9,15,16,17\},  k\in \N_0\}. $$   Actually, it seems that in this case the slope map $\mu_f$  has orbits
 with  arbitrarily  many strict increases  of complexity.  For instance, we have
two  jumps of complexity for the orbit of slope $1/9$ under $\mu_f$:
$$1/9	 \to 	3/25 \to 	3/23  \to 	1/7 \to 	3/19 \to 	3/17 \to 1/5 \to 	1/5.$$
%and  $7$ ``jumps'' of complexity for the slope $1/191$:
%
%\begin{align*}
%&1/191	\to 	1/175  \to  	3/481  \to 1/147  \to 	1/135 \to 3/371 \to 1/113  \to 	3/311  \to 	1/95  \to 	1/87  \to 	3/239  \to \\
%& 1/73 \to 	1/67 \to 	1/61  \to 	1/56  \to 	3/154  \to 	1/47  \to 	1/43  \to 1/39 \to 3/107  \to   \to 3/91 \to 	3/83\to \\
%&  1/25  \to 1/23 \to 	1/21  \to 	1/19  \to 	1/17 \to  	3/47  \to 3/43  \to 	1/13  \to 1/12  \to 	1/11  \to  1/10  \to \odot
%\end{align*}
Note that this orbit stabilizes at the fixed point $1/5$ of $\mu_f$.  The numerical computations by Darragh Glynn also show that there are examples of blown-up $(n\times n)$-Latt\`{e}s maps with $n\geq 5$  for which the slope map
has periodic cycles of  length $\geq 2$. 

It is natural  to ask what one can say  about the behavior of the slope map $\mu_f$ for an obstructed Thurston map $f$ (with $\#P_f=4$). It was already observed in \cite{Pullback} that for a blown-up $(2\times 2)$-Latt\`{e}s map $f$ with only vertical flaps glued to the  pillow $\PP$, there are infinitely many (non-isotopic) invariant essential Jordan curves. Indeed, for such a map $f$ the curve  $\alpha^v$ is $f$-invariant and satisfies $\lambda_f(\alpha^v)=1$. 
One can use this to show  that $f$ commutes with $T^2$ (up to isotopy rel.~$P_f$), where $T$ is a {\em Dehn twist} about $\alpha^v$. This implies that  each curve $T^{2n}(\alpha^h)$ is $f$-invariant. In fact, it is easy to verify directly 
that each essential Jordan curve with slope $x\in\Z \cup \{\infty\}$ is $f$-invariant, or equivalently,  that  $\mu_f(x)=x$ for $x\in \Z \cup \{\infty\}$.

However, for such a  blown-up $(2\times 2)$-Latt\`{e}s map $f$ with only vertical flaps glued to the pillow $\PP$, the general behavior of the slope map $\mu_f$ under iteration has not been analyzed before.  The considerations in the proof of the first part of Proposition \ref{prop:complexity-decreases} also  apply in this situation. In particular, \eqref{eq:compl-decr-H} and \eqref{eq:compl-eq} are still true and  show  that the orbit of an arbitrary $x\in\widehat \Q \cup \{\infty\}$ under $\mu_f$ eventually lands in a fixed point of $\mu_f$. Moreover, results in  Section \ref{sec:attractor} provide a 
method to determine all fixed slopes for $\mu_f$.

The easiest case is the map $f$ obtained from the $(2\times 2)$-Latt\`es map $\La_2$ by gluing  at least one vertical   flap  to each of the four vertical $1$-edges in the ``middle" of the pillow $\PP$. If $\xi$ is  a geodesic arc in $(\PP,V)$ with an endpoint in $A\in V$ and slope $x\in \Q\setminus \Z$, then each component of 
$\La_2^{-1}(\xi)$ must pass through the interior of one of the four vertical $1$-edges in the middle of $\PP$. 
Consequently, we can apply the considerations  in the proof of 
Lemma~\ref{lem:curves-on-flapped-pillow} and in the  second part of the proof of Proposition \ref{prop:complexity-decreases}  and conclude that $\| \mu_f(x)\| < \| x\| $ for $x\in \Q \setminus \Z$. Thus, the orbit of each $x\in\widehat \Q \cup
 \{\odot \}$ under $\mu_f$ eventually lands in $\Z \cup \{\infty, \odot \}$ (that is, in a fixed point of $\mu_f$). Since the map $f$ is easily seen to be \emph{expanding} (see \cite[Definition 2.2 and Theorem 14.1]{THEBook}), this provides an answer to a question raised by Kevin Pilgrim.  
 
 %(see \cite[Question~4.4]{Pilgrim_Pullback}). 

\subsection{Twisting problems}\label{subsec:twists}  To this date,  many natural problems related to Thurston equivalence remain rather  mysterious 
and are often very  difficult to solve. \emph{Twisting problems} are examples of this nature. 

To explain this, suppose we are given a rational Thurston map $f\:\CDach\to\CDach$. Let $\phi\:\CDach\to\CDach$ be an orientation-preserving  homeomorphism that fixes the postcritical set $P_f$ pointwise. We now consider the branched covering map $g\coloneq \phi\circ f$ on $\CDach$, called the $\phi$-{\em twist}  of $f$.  Then  $C_g=C_f$ and $g$ has the same dynamics on $C_f$ as $f$. In particular, $P_g=P_f$; so $g$ has a finite postcritical set and 
is a  Thurston map. 

This leads us to the natural questions: \emph{Is $g$ realized? And if yes, to which rational map is $g$  equivalent depending on the isotopy type of $\phi$?} In fact, there are only finitely many rational maps $g$ (up to M\"obius conjugation)  that can arise in this way from a fixed map~$f$.  A famous instance of this question, called the ``twisted rabbit problem'', was solved by Laurent Bartholdi and Volodymyr Nekrashevych in \cite{BarNekr_Twist} (see also \cite{Lodge_Boundary,LifitingTrees}). 

In our context we can  ask  which twists of maps  as in 
Theorem \ref{thm:flapped_intro}  are realized. We do not have an answer to this question, but it seems that this leads to non-trivial and difficult problems. 
 For example, consider  the blown-up $(2\times 2)$-Latt\`{e}s map $f\:\PP\to\PP$ corresponding to the flapped pillow $\widehat \PP$ as in Figure \ref{fig:flabsatcorner}. Then $\widehat \PP$  has one horizontal and one vertical flap, and so $f$ is realized  by Theorem \ref{thm:flapped_intro}. One can check that the Jordan curve $\gamma\coloneq \wp(\ell_{3/13})$ has exactly two essential pullbacks $\gamma_1, \gamma_2 \sim \wp(\ell_{1/3})$ under $f$ with $\deg(f\:\gamma_1\to\gamma)=1$ and $\deg(f\:\gamma_2\to\gamma)=2$.  We now choose an orientation-preserving  homeomorphism $\phi\:\PP\to\PP$ that maps $\gamma$ onto $\gamma_1$, while fixing each point in $V=P_f$.
Then  the curve $\gamma_1$ is an obstruction for the twisted map $g\coloneq \phi\circ f$ with $\lambda_g( \gamma_1)=3/2$.  An analogous construction applies to some other essential Jordan curves, for instance, with slopes $3/23$ and $3/49$, and gives twists of $f$ with an obstruction.  

It follows from this discussion that the \emph{mapping class biset} associated with the map $f$ above is not \emph{contracting} (see \cite{BD_Algo,BD_Exp} for the definitions). Thus, the algebraic methods for solving the global curve attractor problem developed in \cite{Pilgrim_Alg_Th} (see, specifically, 
\cite[Theorem 1.4]{Pilgrim_Alg_Th}) do not apply in general for the maps considered 
in  Theorem \ref{thm:flapped_intro}.

\subsection{Thurston maps with more than four postcritical points} While in this paper we only discuss the case of Thurston maps $f\:\Sp\to\Sp$ with $\#\postf=4$, it is natural to ask if one can adapt Theorem~\ref{thm:blow-up-obstr} to the case when $\#\postf>4$. The main difficulty is that an obstruction in this case  is in general not given by a unique essential Jordan curve in $(\Sp,\postf)$, but by a {\em multicurve}. Of course, this fact complicates the analysis of pullback properties of curves and their intersection numbers. However, we expect that one can naturally generalize our result for an arbitrary Thurston map: given an obstructed Thurston map $f$ one can eliminate all possible multicurve obstructions  by successively applying the blow-up operation and  obtain a Thurston  map that is realized.

\subsection{Other combinatorial constructions of rational maps} The dynamical behavior of curves under the pullback operation is an important topic in holomorphic dynamics. While in this paper we only study the realization and the global curve attractor problems,  one is led to similar considerations,  for example,   in the study of {\em iterated monodromy groups} (see \cite{HM_growth}).   For these investigations  it is important  to have explicit  classes of  rational maps at hand that are constructed in combinatorial fashion and against which conjectures can be tested or which lead to the discovery of general phenomena. The  maps provided by Theorem~\ref{thm:blow-up-obstr}
may be useful in this respect. Another interesting class of maps worthy of further investigation are  Thurston maps constructed from tilings of the Euclidean or hyperbolic plane as in \cite[Example~12.25]{THEBook}.

\section{Appendix: Isotopy classes of Jordan curves in spheres with four marked points}\label{sec:Appendix}
In this appendix, we will provide  proofs for Lemmas~\ref{lem:isoclassesP} and~\ref{lem:i-properties-curve}.
Our presentation is rather detailed.   We need some additional  auxiliary  facts that we will discuss first. 
Throughout, we will rely on the notation and terminology established in  Section~\ref{sec:prelim}.

In the following, we will consider a marked sphere  $(S^2, Z)$, where $Z\sub S^2$ consists  of precisely four points.  If $M\sub S^2$ and $\alpha$ is a Jordan curve in $(S^2, Z)$, then we say that $\alpha$ is in {\em minimal position with the set} $M$ if $\#(\alpha\cap M)\le  
\#(\alpha'\cap M)$ for all Jordan curves $\alpha'$  in $(S^2, Z)$  with $\alpha\sim \alpha'$ rel.\ $Z$. 

Let $\alpha$ and $\beta$ be Jordan curves or arcs in  
$(S^2,Z)$.  
We say that subarcs $\alpha'\sub \alpha $ and $\beta'\sub \beta$ 
form a {\em bigon} $U$ in $(S^2,Z)$ if  $\alpha'$ and $\beta'$ have the same endpoints, but  disjoint interiors  and if $U\sub S^2$ is an open  Jordan region
 with $\partial U=
 \alpha'\cup \beta'$   and $U\sub S^2\setminus Z$.

\begin{lemma} \label{lem:nobigons}  Let $\ga$ be a Jordan curve  in a marked sphere 
$(\Sp, Z)$  with $\# Z = 4$, and let  $a$ and $c$ be  disjoint arcs in $(\Sp,Z)$.  Suppose $\ga$ is in minimal position with the set  $a\cup c$. 
Then $\ga$ meets each of the arcs $a$ and $c$ transversely and  no subarcs of $\ga$  and of $a$ or $c$ form a bigon $U$ in $(S^2, Z)$  with $U\cap (\ga \cup a \cup c)=\emptyset$. 
\end{lemma} 

\begin{proof} These facts are  fairly standard  in contexts like this (see, for example,
\cite[Section 1.2.4]{FarbMargalit}), and so we will only give an outline of the proof.   

Our assumptions imply that $\ga$ meets each arc $a$ and $c$ transversally and has only finitely many intersections with $a\cup c$ (Lemma~\ref{lem:isocrit} and its proof apply {\em mutatis mutandis} to our situation).  We now argue by contradiction and assume that 
 a subarc $\ga'\sub \ga $ and a subarc $\sigma$ of $a$ or $c$, 
form a  bigon $U$ in $(S^2,Z)$ with $U\cap (\ga \cup a \cup c)=\emptyset$.  Note that then $\cl(U)\sub S^2\setminus Z$. Hence  we can modify the curve $\ga$ near $U$  by an isotopy  in $S^2\setminus Z$ that pulls the subarc $\ga'$ of $\ga$ through $U$ and away from $\sigma$  so that the new Jordan curve $\ga$ does not intersect $\sigma\sub a\cup c$  and  no new intersection points with  $a\cup c$  arise. 
This leads to a contradiction, because the original curve $\ga$ was in minimal position with $a\cup c$.  
\end{proof}

A topological space $D$ is called a {\em closed topological disk} if there exists a homeomorphism $\eta\: \cl(\D)\ra D$ of the closed unit disk $\cl(\D)\sub \C$ onto $D$. This is an abstract version of the notion of a closed Jordan region contained in a surface. The set $\partial D\coloneq \eta(\partial \D)$ is a Jordan curve independent of $\eta$ and called the {\em boundary} of $D$. 
The {\em interior} of $D$ is defined as $\inte(D)\coloneq \eta(\D)= D\setminus \partial D$. 

Similarly as for closed Jordan regions, an arc $\alpha$ contained in a  closed topological disk
$D$ is called a {\em crosscut} (in $D)$ if $\partial \alpha\sub \partial D$ and $\inte(\alpha)\sub \inte(D)$. A crosscut $\alpha$ splits 
$D$ into two compact and connected sets $S$ and $S'$ called the {\em sides} of $\alpha$ (in $D)$ such that $D=S\cup S'$ and $S\cap S'=\alpha$. 
With suitable orientations of $\alpha $ and $D$, one side of $\alpha$ lies on the left and the other side on the right of $\alpha$.
Each non-empty connected set $c\sub D$ that does not meet $\alpha$ is contained in precisely one side of $\alpha$.

If $\alpha_1, \dots, \alpha_n$ for  $n\in \N$  are pairwise disjoint 
crosscuts in a  closed topological disk
$D$, then we can define  an abstract  graph $G=(V,E)$ in the following way: we consider each  component $U$ of $D\setminus (\alpha_1 \cup \dots \cup\alpha_n)$ as a vertex of $G$. We join two distinct vertices represented by components $U$ and $U'$ by an edge, if one of the crosscuts $\alpha_j$ is contained in the boundary of both $U$ and $U'$.  
Accordingly, the edges of $G$ are in bijective correspondence with the crosscuts $\alpha_1, \dots, \alpha_n$. 
 
 \begin{lemma}\label{lem:chords} Let   $n\in \N$ and $\alpha_1, \dots, \alpha_n$ be  pairwise disjoint 
crosscuts in a  closed topological disk
$D$, and let $G$ be  the graph obtained from the components 
of the set $D\setminus (\alpha_1 \cup \dots \cup \alpha_n)$ as described. The the following statements are true: 

\begin{enumerate}[label=\text{(\roman*)},font=\normalfont,leftmargin=*]

\item \label{i:chord1} 
The graph $G$ is a finite tree with at least two vertices. 

\smallskip 
\item \label{i:chord2} 
Let $c\sub D$ be a   connected set and suppose that $c\cap \alpha_k=\emptyset$ for some $k\in \{1, \dots, n\}$. Then there exists 
$m\in \{1,\dots, n\}$ and a side $S$ of $\alpha_m$ such that 
$S\setminus \alpha_m $ is disjoint from all the 
sets $c$, $\alpha_1, \dots, \alpha_n$.  
\end{enumerate}
\end{lemma} 

If $\alpha=\alpha_m$ and  $S$ are as in \ref{i:chord2}, then there exists a subarc  $\beta$ of $\partial D$ with the same endpoints as $\partial \alpha\sub D$ such that $\partial S=\alpha\cup \beta$. 
Then $U\coloneq\inte(S)$ is an open Jordan region bounded by  the union of the arcs $\alpha$ and $\beta$ whose only common points are their endpoints. This region  $U$ does not meet $c$ nor any of the arcs $\alpha_1, \dots, \alpha_n$. In the proof of Lemma~\ref{lem:alternate8} we will use such a region $U$ to obtain a bigon in an appropriate context.

\begin{proof} \ref{i:chord1} This is intuitively clear, and we leave the details to the reader. By induction on the number $n$ of 
crosscuts one can  show that $G$ is a finite connected graph with at least two vertices. Since a crosscut 
splits $D$ into two sides, it easily follows that the removal of any  edge from $G$ disconnects
  it.  Hence $G$ cannot contain any simple cycle and must be a tree.  
 
 \smallskip
  \ref{i:chord2} The graph $G$ is a tree; so if we remove  the edge corresponding 
  to the crosscut  $\alpha_k$ from $G$, then  we obtain two disjoint non-empty subgraphs $G_1$ and $G_2$ of $G$. The connected components of $D\setminus (\alpha_1 \cup \dots \cup \alpha_n)$ corresponding 
 to the vertices of $G_1$ are contained  in one side $S'$ of $\alpha_k$, while the other 
 connected components of $D\setminus (\alpha_1 \cup \dots \cup \alpha_n)$ corresponding to the vertices of $G_2$ lie in the other side $S''$ of $\alpha_k$. Since $c$ is connected and does not meet $\alpha_k$, it must be contained in one of the sides of $\alpha_k$, say $c\sub S''$. Then $c$ is disjoint from $S'$ and hence from all the sets that correspond to vertices in $G_1$. 
 
 The tree $G$ has a {\em leaf} $v$ in $G_1\ne \emptyset$, i.e., 
 there exists a vertex $v$ of $G_1$  such that $v$ is connected to the rest of $G$ by precisely one edge.  Then the connected component of $D\setminus (\alpha_1 \cup \dots \cup \alpha_n)$ corresponding to $v$ has exactly one of the crosscuts, say  $\alpha_m$ with $m\in \{1,\dots, n\}$, on its boundary.  Then this component has the form $S\setminus \alpha_m$, where $S$ is the unique side of $\alpha_m$ contained in $S'$. Then $S\setminus \alpha_m$ is disjoint from $c\sub S''$ and from all the crosscuts $\alpha_1, \dots, \alpha_n$.  
  \end{proof} 

We can now prove a statement  that is  the key  to the understanding of isotopy classes of Jordan curves in a sphere with four marked points.

 \begin{lemma}\label{lem:alternate8}
 Let $\ga$ be a Jordan curve  in a marked sphere 
$(\Sp, Z)$  with $\# Z = 4$, and let  $a$ and $c$ be  disjoint arcs in $(\Sp,Z)$.  Suppose $\ga$ is in minimal position with the set  $a\cup c$.  Then  the sets $a\cap \ga$ and  $c\cap \ga$ are 
 non-empty and finite, and the points in these sets alternate on $\ga$ unless 
$\ga$ is peripheral or $\ga\cap (a\cup c) =\emptyset$. 
\end{lemma}

\begin{proof}
In the given setup,  each of the disjoints sets $\partial a$ and $\partial c$ contains two points in $Z$. Since  $\#Z=4$, it follows that $a$ connects two of the points in $Z$, while $c$ connects the other two points.  

We may assume that $\ga$ is essential and that at least one of the arcs $a$ or $c$ meets $\ga$, say $a\cap \ga\ne \emptyset$, because otherwise we are in an exceptional situation as in the statement.

If none of the arcs $a$ and $c$ meets $\ga$ in more than two points, then  $\# (a \cap \gamma  ) =1$ and $\# (c \cap \gamma  ) \le 1$. Now $a$ and $\ga$ meet transversely by   
 Lemma~\ref{lem:nobigons}. This implies that the endpoints of 
 $a$ lie in different  components of $S^2\setminus \ga$. Since $\ga$ is essential, each of these components contains  precisely two points of $Z=\partial a\cup\partial c$. Hence the endpoints of $c$  
 also lie in different components of $S^2\setminus \ga$. This implies that $c\cap \gamma\ne \emptyset$ and so $\# (c \cap \gamma  ) =1$ in the case
 under consideration. So both $a$ and $c$ meet $\ga$ in exactly one point. It  follows that the statement is true in this case.

We are reduced to the situation where at least one of the arcs $a$ or  $c$ meets $\ga$ in at least two (but necessarily finitely many) points, say  
$n\coloneqq \# (a \cap \gamma ) \geq 2$.
We now endow $\gamma$ and $a$  with some orientations.
 With the  given orientation we denote the initial point of $a$    by $x_0$ and its terminal point by $x_1$.  Let $y_1, \dots, y_{n}, y_{n+1}=y_1$ denote  the $n\ge 2$ intersection points of $\gamma$ with $a$ that we encounter while traversing $\gamma$ once  starting from some point  in $\gamma \setminus a $.  
 The same $n$ points also appear on $a$. We denote them by $p_1, \dots, p_n$ in the order they appear if we traverse $a$ starting from $x_0$. For $k=1, \dots, n$, we denote by 
  $\gamma[y_k, y_{k+1}]$ the subarc of $\ga$ obtained from  traversing  $\ga$ with the given orientation from $y_k$ to $y_{k+1}$.

\smallskip 
{\em Claim.} $c\cap \gamma[y_k, y_{k+1}]\ne \emptyset$ for each 
$k=1, \dots, n$.

\smallskip 
To see this, we argue by contradiction and assume  that $c\cap \gamma[y_k, y_{k+1}] = \emptyset$ for some $k\in \{1, \dots, n\}$.  Our goal now is  to show that some subarcs of $a$ and  $\gamma$ form a bigon $U$ in $(S^2, Z)$ with $U\cap (\ga\cup a\cup c)=\emptyset$. This is a  contradiction with  Lemma~\ref{lem:nobigons}, because $\ga$ and $a\cup c$ are in minimal position.  

 To produce such bigon $U$,  we want to apply Lemma~\ref{lem:chords}. In order to do this, we slit the sphere $S^2$ open along the arc $a$. This results in a closed topological disk $D$ whose boundary $\partial D$ consists of  two copies $a^+$ and $a^-$ of the arc $a$. The set $S^2\setminus \inte(a)$ 
can be identified with  $\inte(D)$, while each point in $\inte(a)$ is doubled into  one corresponding point in $a^+$ and one in $a^-$.

The arcs $a^+$ and  $a^-$ have their endpoints $x_0$ and $x_1$ in common. We identify $a^+$ with the original arc $a$ with the same orientation. Then  we can think of the intersection points 
$p_1, \dots, p_n$ of $\ga$ with $a$ as lying on $a^+=a$, while each of the points $p_j$ has a corresponding point $q_j$ on $a^{-}$.

Each arc $\ga[y_j, y_{j+1}]$  corresponds to a  crosscut 
$\ga_j$ in $D$ for $j=1, \dots, n$. These crosscuts have their endpoints in the set $P\cup Q$, where $P\coloneq \{ p_1, \dots, p_n\}$ and $Q\coloneq \{ q_1, \dots, q_n\}$. Moreover, the crosscuts $\ga_1, \dots, \ga_n$  are pairwise disjoint. Indeed, the only possible common intersection point of two of these arcs
could be a common endpoint of two consecutive arcs $\ga_j$ and 
$\ga_{j+1}$  (where $\ga_{n+1}\coloneq \ga_1$) corresponding 
to $y_{j+1}\in a$; but in the process of creating $D$, the point $y_{j+1}$ is doubled into the points $p_{\ell}$ and $q_{\ell}$
for some $\ell \in \{1, \dots, n\}$.  Since $\ga$ meets $a$ transversely, one of these points will be the terminal point of $\ga_j$, while the other one will be the initial point of $\ga_{j+1}$, and so actually $\ga_j\cap \ga_{j+1}=\emptyset$. 
 It follows that the hypotheses of Lemma~\ref{lem:chords} are satisfied.  

It is clear that the arc $c$, now considered as a subset of $D$, does not meet the crosscut $\ga_k$ corresponding to $\gamma[y_k,y_{k+1}]$. 
 Hence by Lemma~\ref{lem:chords}, there exists $m\in \{1, \dots, n\}$ and a side $S$ of $\ga_m$ in $D$ such that 
$S\setminus \ga_m$ is disjoint from $c$ and all the arcs $\ga_1, \dots, \ga_n$. Then there exists an arc $\beta\sub \partial 
D=a^+\cup a^-$ with the same endpoints as $\ga_m$ such that 
$\partial S=\ga_m\cup \beta$. The set  $S\setminus \ga_m$ is
disjoint from $\ga_1\cup \dots \cup \ga_n\supset P\cup Q$, and so the  arc $\beta$ has its endpoints in the set $P\cup Q$, but no other points in common with $P\cup Q$. 

This implies that neither $x_0$ nor $x_1$  are contained in 
$\beta$; indeed, suppose $x_0\in \beta$, for example. Then  the endpoints of $\beta$ and hence of $\ga_m$ are necessarily the points $p_1$ and $q_1$. Collapsing $D$ back to $S^2$, we see that the endpoints $y_m$ and $y_{m+1}$ of 
$\ga[y_m, y_{m+1}]$ are the same. This is a contradiction (here the assumption $n\ge 2$ is crucial).  We arrive at a similar contradiction (using the points $p_n$ and $q_n$) if we assume $x_1\in \beta$. It follows that $\beta\sub \inte(a^+)$ or $\beta\sub \inte(a^-)$. 

These considerations imply that  if we pass back to  $S^2$ by  identifying corresponding points in $a^+$ and $a^-$, then from $\inte(S)$ 
 we obtain a bigon $U\sub S^2$ bounded 
  by the subarc $\ga[y_m, y_{m+1}]$ of $\ga$ and a subarc 
  $\widetilde \beta$ of $a$, where  $U$ is disjoint from $ \ga \cup a\cup c$. This is impossible by 
 Lemma~\ref{lem:nobigons} since $\ga$ is in minimal position with $a\cup c$. This contradiction shows that the Claim is indeed true. 

 \smallskip 
  The Claim implies that 
 $c$ has at least $n\ge 2$  intersection points with $\ga$. Hence we can reverse the roles of $a$ and $c$ and get a similar statement as the Claim also for the arc $c$. This implies that the (finitely many) points in $a\cap \ga\ne \emptyset$ and $c\cap \ga\ne \emptyset$ alternate on $\ga$. \end{proof} 
 
  As in Section \ref{subsec:pillow},   we now consider the pillow $\PP$ with its set $V$ of vertices as the marked points, and 
  the Weierstrass function $\wp\:\C\to \PP$  that is doubly periodic with respect to the lattice $2\Z^2$. We will now revert to  the notation  $a$ and $c$ for the horizontal  edges  of $\PP$.  Recall that $\I=[0,1]$.

\begin{lemma}\label{lem:difference} 
Let $\alpha\: \I\ra  \C$ be a simple loop or a homeomorphic 
parametrization of an  arc. Suppose that   the endpoints $z_0=\alpha(0)$ and $w_0=\alpha(1)$ of $\alpha$ lie in $\C\setminus \wp^{-1}(a\cup c)$ and that  $\wp(z_0)=\wp(w_0)$. 
Suppose further that either $\alpha\cap \wp^{-1} (a\cup c)=\emptyset$, or all of the following conditions are true:   $\alpha$ meets each of the lines in $\wp^{-1} (a\cup c)$ transversely, we have 
 $$0<\# ( \alpha \cap \wp^{-1} (a))=\# ( \alpha \cap \wp^{-1} (c))<\infty,
 $$ and  the points in $ \alpha \cap \wp^{-1} (a) $  and 
 $ \alpha \cap \wp^{-1} (c)$  alternate on $\alpha$. 
 
Then $w_0-z_0\in 2\Z^2$.  Here $w_0-z_0\ne 0$ unless 
$\alpha\cap \wp^{-1} (a\cup c)=\emptyset$. \end{lemma}  

\begin{proof}  Note that the set  $\wp^{-1}(a\cup c)$ consists 
precisely of the lines $L_n\coloneq\{ z\in \C: \text{Im} (z)=n\}$, $n\in \Z$. Moreover, such  a line $L_n$  is mapped to  $a$ or $c$ depending on whether $n$ is even or odd, respectively. 

We consider the second case  first when $\alpha\cap \wp^{-1} (a\cup c)\ne \emptyset$.  We denote by  $0<u_1<\dots< u_k<1$, $k\in \N$,  all the  (finitely many)  $u$-parameter values with $\alpha(u_j)\in \wp^{-1} (a\cup c)$ for $j=1, \dots, k$.  Since  $\alpha$ meets each line $L_n$ transversely and the  points in $ \alpha \cap \wp^{-1} (a)\ne \emptyset $  and $ \alpha \cap \wp^{-1} (c) \ne \emptyset$ alternate on $\alpha$,
it is clear that the values $\text{Im} (\alpha(u_j))$ either strictly increase  by $1$ in each step  $j=1, \dots , k$,  or strictly decrease by $1$  in each step. Here we have strict increase  if 
 $\text{Im} (\alpha(u_1)-\alpha(0))>0$, and  strict decrease if 
$\text{Im} (\alpha(u_1)-\alpha(0))<0$.  This implies that  $k>0$ is precisely the number of lines $L_n$, $n\in \Z$,  that separate $z_0$ and $w_0$.  So $z_0$ and $w_0$ lie in different components of $\C\setminus \wp^{-1}(a\cup c)$ which shows that $z_0\ne w_0$.  

By our hypotheses,   $\# (\alpha \cap\wp^{-1} (a))=\# (\alpha \cap\wp^{-1} (c))$, which implies that the number $k$ of intersection points of $\alpha$ with $\wp^{-1} (a\cup c)$ is even. Since $\wp(z_0)=\wp(w_0)$, by \eqref{eq:wpeq}  we have 
$w_0=\pm z_0+v_0$ with $v_0\in 2\Z^2$.  We have to rule out the minus sign here. 

We argue by contradiction and assume that 
$w_0=-z_0+v_0$. Then $v_1\coloneqq \tfrac12(z_0+w_0)=\frac12v_0\in \Z^2$, and so  the endpoints $z_0$ and $w_0$ of $\alpha$ are in symmetric position to the point $v_1\in \Z^2$. 
This implies that the number $k$  of lines $L_n$, $n\in \Z$, separating 
$z_0$ and $w_0$ is odd, contradicting what we have just seen.
We conclude $w_0=z_0+v_0$ with $v_0\in 2\Z^2$, and the statement follows in this case. Note that the exact same argument leading to 
$w_0-z_0\in 2\Z^2$ also applies if $\alpha\cap \wp^{-1} (a\cup c)=\emptyset$. 
\end{proof}

We call a Jordan curve $\ga $ in $(\PP, V)$ {\em null-homotopic in
$(\PP, V)$} if $\ga$ can be homotoped in $\PP\setminus V$ to a point, i.e.,  
if there exists a homotopy $H\: \partial \D\times \I\ra  \PP\setminus V$ such that $H_0$ is a homeomorphism of $\partial \D$ onto $\ga$  and $H_1$ is a constant map.

\begin{lemma}\label{lem:essnonnull}  Let $\ga$ be an essential Jordan curve in 
$(\PP, V)$. Then $\ga$ is not null-homotopic in $(\PP, V)$.  
\end{lemma} 
This statement sounds somewhat tautological, because in topology ``essential" is often defined as ``not null-homotopic". Recall though that in our context $\ga$ is called essential if each of the two components of $\PP\setminus \ga$ contains precisely two of the points in $V$. In the proof, we will use some standard facts about winding numbers; see \cite[Chapter 4]{Burckel} for the basic definitions and background.

\begin{proof} On an intuitive level, every homotopy contracting $\ga$ to a point must slide over all the points in one of the complementary components of $\ga$. Hence it cannot stay in  $\PP\setminus V$, and so $\ga$ is not null-homotopic in $(\PP, V)$.

To make this more rigorous, we  argue by contradiction. By the Sch\"onflies theorem we may identify $\PP$ with $\CDach$ and 
$\ga$ with $\partial \D$  and assume that $0$ and $\infty$ belong 
to the set $Z\sub \CDach$ of marked points corresponding to the points in $V$.  We now argue by contradiction and assume that there exists a homotopy $H\: \partial \D \times \I\ra \CDach \setminus Z\sub \C\setminus \{0\}$ such that 
$H_0$ is a homeomorphism on $\partial \D$ and $H_1$ is a constant map. Then for each $t\in \I$, the map $u\in \I\mapsto
 \alpha_t(u)\coloneq H_t(e^{2\pi i u})$ is a loop in $\C\setminus \{0\}$. Each loop $\alpha_t$, $t\in \I$,  has the same winding number 
 $\text{ind}_{\alpha_t}(0)$  
around $0$, because this winding number is  invariant under homotopies in $\C\setminus \{0\}$ (see \cite[Theorem 4.12]{Burckel}). On the other hand, 
 $\text{ind}_{\alpha_0}(0)=\pm1$, because $\alpha_0$ is a simple loop (see \cite[Theorem 4.42]{Burckel}), while  $\text{ind}_{\alpha_1}(0)=0$, because $\alpha_1$ is a constant loop. This is a contradiction.  \end{proof}

An element $x$ of a rank-2 lattice $\Gamma$ (such as 
$2\Z^2$)  in $\C$ is called {\em primitive} 
if it cannot be represented in the form $x=ny$ with $y\in \Gamma$ and $n\in \N$, $n\ge 2$. Note that then $x\ne 0$.

\begin{lemma}\label{lem:prim} 
Let $\ga$ be a  Jordan curve in $(\PP,V)$ parametrized as a simple loop $\beta\:\I\ra \PP$, and let $\alpha\: \I\ra \C$ be a lift of $\beta$ under $\wp$. 

\begin{enumerate}[label=\text{(\roman*)},font=\normalfont,leftmargin=*]

\item \label{i:prim1} 
If $\alpha(0)=\alpha(1)$ and $\alpha$ is a path in a convex subset of $\C\setminus \Z^2$, then $\ga$ is null-homotopic in $(\PP, V)$.

 \smallskip 
\item \label{i:prim2} 
If $\ga$ is essential, then  $\alpha(1)-\alpha(0)$ is a primitive element of $2\Z^2$. It is uniquely determined up to sign by the isotopy class 
$[\ga]$ of $\ga$ rel.\ $V$.  

\end{enumerate}
\end{lemma}   

Here  we call  $\al$  a {\em lift} of $\be$ under $\wp$ if  $\be=\wp\circ\al$. 
In this lemma and its proof, we will carefully distinguish between a 
path  and its  image set (unlike elsewhere in the paper). 

\begin{proof} Note that since $\be(\I)=\ga\sub \PP\setminus V$ and $\wp$ is a covering map over 
$ \PP\setminus V$, a lift $\al$ of $\be$  under $\wp$ 
exists. Moreover, for each choice of $z_0\in \wp^{-1}(\be(0))$, there exists a unique lift $\alpha$ of $\beta$ such that $\alpha(0)=z_0$
(for these  standard facts see  \cite[Section 1.3, Proposition 1.30]{Hatcher}).
We will use this uniqueness property of lifts repeatedly in the following.

\smallskip
 \ref{i:prim1} The idea for the first part is very simple. We use a ``straight-line homotopy" between $\alpha $ and the constant path $u\in \I\mapsto \alpha(0)$ and push it to $\PP\setminus V$ by applying $\wp$.

More precisely, we define 
$$H(u,t)\coloneq \wp( (1-t)\alpha(u)+t\alpha(0))$$ 
for $u,t\in \I$. Since the path $\alpha$  lies in a  convex set $K\sub \C \setminus \Z^2$,  we have 
$$\alpha_t(u)\coloneq (1-t)\alpha(u)+t\alpha(0)\in K$$ for all $u,t\in \I$, and so $H(\I\times \I)\sub \wp(\C\setminus \Z^2)=\PP\setminus V$. Hence $H$ is a homotopy in $\PP\setminus V$. Moreover, $H_t(u)=(\wp\circ \alpha_t)(u)$ for all $u\in \I$, and so $H_t=\wp\circ \alpha_t$ for all $t\in \I$. 
In particular, $H_0=\wp\circ \alpha_0=\wp\circ \alpha=\beta$. Moreover, $H_1(u)=\alpha(0)$ for $u\in \I$, and so $H_1$ is a constant path.  

For all $t\in \I$ we have 
 $$\alpha_t(1)-\alpha_t(0)=(1-t) (\alpha(1)-\alpha(0))=0.$$  
Hence  $\wp(\alpha_t(1))=\wp(\alpha_t(0))$ for $t\in \I$, and so 
 $H_t=\wp\circ \alpha_t$ is a loop in $\PP\setminus V$ for all $t\in \I$. 
By identifying  
the points $(0,t)$ and $(1,t)$ for each $t\in \I$, we obtain a homotopy 
$\overline H\: \partial \D\times \I\ra \PP\setminus V$ 
such that 
$$\overline H(e^{2\pi iu}, t)= H(u,t)$$ 
for all $u,t\in \I$. Since $\overline H_0(e^{2\pi iu})=H_0(u)=\beta(u)$
for $u\in \I$, we see that $\overline H_0$ is  a homeomorphism of $\partial \D$ onto $\beta(\I)=\ga$. On the other hand, $\overline H_1$ is a constant map. 
Hence $\ga$ is null-homotopic in $(\PP, V)$.

\smallskip 
 \ref{i:prim2} 
The proof is somewhat tedious as we have to worry about different choices of the curve in $[\ga]$, its different parametrizations as a simple loop, and the different lifts of these parametrizations under $\wp$.

To prove the statement, we first consider a special case, namely we choose a Jordan curve 
$\ga_0$ in $(\PP, V)$ that lies in the same  isotopy class 
rel.\ $V$ as $\ga$ with the additional property that $\ga_0$  is in minimal position with  the set $a\cup c$. Since $\ga_0$ cannot be a subset of $a\cup c$, we can parametrize $\ga_0$ as a simple loop $\beta_0\: \I\ra \PP$ such that $\beta_0(0)=\beta_0(1)\not \in a\cup c$. We now consider a lift $\alpha_0\: \I\ra \C$ of $\beta_0$ under $\wp$. Note that $\alpha_0$ is a simple loop or a homeomorphic parametrization on an arc in $\C$.

Since $\ga$ is essential, the Jordan curve  $\ga_0$ is also essential. Indeed, under an isotopy rel.\ $V$ that deforms $\ga$ into $\ga_0$,
the complementary components of $\ga$ in $\PP$ are deformed into the complementary components of $\ga_0$ while the points in $V$ stay fixed. Therefore, each of the two components of $\PP\setminus \ga_0$ contains precisely two points of $V$.

It  follows from  Lemma~\ref{lem:nobigons} that $\ga_0$ meets  $a$ and $c$ transversely. Moreover, by  Lemma~\ref{lem:alternate8} either   $\ga_0\cap (a\cup c)=\emptyset$, or the sets $a\cap \ga_0$ and $c\cap \ga_0$  are  non-empty and finite,   and the points in these sets alternate  on $\ga_0$. We now define $\overline \alpha_0\coloneq 
\al_0(\I)$. Then $\overline \alpha_0$ is a Jordan curve or an arc in $\C$. Moreover, we either have 
$\overline \al_0\cap \wp^{-1}(a\cup c)=\emptyset$,  or  $\overline \alpha_0$  meets each of the lines 
in $\wp^{-1}(a\cup c)$ transversely, 
$$0<\#(  \overline \alpha_0 \cap \wp^{-1} (a))=
\#(\overline \alpha_0 \cap \wp^{-1} (c))<\infty,$$
and the points in $\overline \alpha_0 \cap \wp^{-1} (a)$ and 
$\overline \alpha_0 \cap \wp^{-1} (c)$ alternate on 
$ \overline  \alpha_0$.  Let $z_0\coloneq \al_0(0)$ and 
$w_0\coloneq \al_0(1)$. Then 
$$ \wp (z_0)=\wp(\al_0(0))=\be_0(0)=\be_0(1)=\wp(\al_0(1))=\wp(w_0),$$
and by  the choice of $\beta_0$ we have $z_0, w_0\not \in 
\wp^{-1}(a\cup c)$. Therefore, we are exactly in the situation of 
Lemma~\ref{lem:difference}.

It follows that $ v_0\coloneq w_0-z_0\in 2\Z^2$. Here $v_0\ne 0$. Indeed, otherwise $z_0=w_0$. Then the second part of Lemma~\ref{lem:difference} implies that the arc  $\overline \al_0$ does not meet $\wp^{-1}(a\cup c)$ and so it  lies in a connected component   of 
$\C\setminus \wp^{-1}(a\cup c)$. This  component is an infinite 
strip, and  hence a convex set, contained in $\C\setminus \Z^2$. 
Now  \ref{i:prim1}  implies that $\ga_0$ is null-homotopic in $(\PP, V)$. By Lemma~\ref{lem:essnonnull} this contradicts the fact that $\ga_0$ is essential. We conclude that indeed
 $v_0=w_0-z_0\ne 0$.

This shows that $\al_0(1)-\al_0(0)=w_0-z_0$
is a non-zero element of the lattice $2\Z^2$. We claim that 
$w_0-z_0$ is actually a primitive element of $2\Z^2$. 
To see this, we argue by contradiction and assume that 
$w_0-z_0=ny_0$ with $y_0\in 2\Z^2\setminus\{0\}$ and $n\in \N$, $n\ge 2$.

We now consider the path $\sigma\: \I\ra \C\setminus \{0\}$ given as 
$\sigma(u)=\exp(\frac{2\pi i}{y_0}\al_0(u))$ for $u\in \I$. 
This is a loop with winding number 
$$\text{ind}_\sigma(0)=\frac{1}{y_0} (\al_0(1)-\al_0(0))=n$$ around $0$. 
Now a simple loop in  $\C\setminus \{0\}$ has winding number 
$0$ or $\pm1$ around $0$ (see  \cite[Theorem 4.42]{Burckel}), and so 
$\sigma$ cannot be simple. This implies that there are numbers 
$0\le u<u'<1$ such that $\sigma(u)=\sigma(u')$. This in turn means that 
 $\al_0(u')-\al_0(u)=ky_0$ for some  $k\in \Z$. 
 Since $y_0\in 2\Z^2$, it follows 
 that 
 $$\beta_0(u') =\wp(\al_0(u'))=\wp(\al_0(u))=\be_0(u). $$ 
 This is impossible, since $\beta_0$ is injective on $[0,1)$. 
 
 We have shown the first part of the statement for a particular Jordan curve $\ga_0$ in $[\ga]$ with a special parametrization $\beta_0$, and a choice of a lift $\al_0$ of $\be_0$ under $\wp$. We now have to show that the number  $v_0 =w_0-z_0=\al_0(1)-\al_0(0)$ obtained in this way only depends on $[\ga]$ up to sign. For this we pick an arbitrary Jordan curve in $[\ga]$ which we will simply call $\ga$.

Since $\ga_0\sim \ga$ rel.\ $V$, there exists an isotopy 
$H\: \PP \times \I\ra \PP$ rel.\ $V$ with $H_0=\id_{\PP}$ and 
$H_1(\ga_0)=\ga$. 
We now define a homotopy $\overline H\: \I\times \I\ra \PP\setminus V$ by setting 
$$ \overline H(u,t) \coloneqq H(\be_0(u),t)$$ 
for $u,t\in \I$. 
Note that $ \overline H$ maps into $ \PP \setminus V$
as follows from the facts that $\ga_0\sub  \PP \setminus V$ and $H$ is an isotopy rel.\ $V$. 

The time-$0$ map  $\overline H(\cdot, 0)=\be_0$ of the homotopy $\overline H$   is the parametrization of the loop $\ga_0$, while 
the time-$1$ map $\beta\coloneq \overline H(\cdot, 1)=H_1\circ \be_0 $ gives some parametrization of $\ga=H_1(\ga_0)$ as a simple loop. 

By the homotopy lifting theorem (see \cite[Proposition 1.30]{Hatcher}), there exists 
a homotopy  $\widetilde H\: \I\times \I\ra \C\setminus \Z^2$ such that 
$\overline H= \wp\circ \widetilde H$ and $\widetilde H_0= 
\widetilde H (\cdot, 0) =\al_0$. Then  $\alpha\coloneqq
\widetilde H (\cdot, 1)$ is a lift of $\beta=\overline H (\cdot, 1)$ under $\wp$. We want to show  that $\alpha(1)-\alpha(0)=v_0= w_0-z_0$.

To see this, we  consider the paths
 $\sigma,\tau\:\I\ra \C\setminus \Z^2$ defined as $\sigma(t)=\widetilde H(0,t)$ 
 and $\tau(t)=\widetilde H(1,t)$ for $t\in \I$.
 Then 
\begin{align}\label{eq:tausig} 
\overline \sigma(t)&\coloneqq  \wp(\sigma(t))=\wp(\widetilde H(0,t))=\overline H(0,t)\\
& =
 H_t(\beta(0))=H_t(\beta(1))=\overline H(1,t)=\wp(\widetilde H(1,t))
 = \wp(\tau(t))\notag 
 \end{align}
 for $t\in \I$. Note also that
 $$\tau(0)=\widetilde H(1,0)=\alpha_0(1)=\alpha_0(0)+v_0=\widetilde H(0,0)+v_0=\sigma(0)+v_0.$$  
 This implies that the paths $\tau$ and $t\in \I\mapsto\sigma(t)+v_0$ have the same initial points. Since $v_0\in 2\Z^2$, it follows from \eqref{eq:wpeq} and  \eqref{eq:tausig} that the map $\wp$ sends them both to  $\overline \sigma=\wp\circ\sigma=\wp\circ \tau$, which is a path in $\PP\setminus V$.  It follows from the uniqueness of lifts under $\wp$ that 
 $\tau(t)=\sigma(t)+v_0$ for all $t\in \I$.

This implies that 
$$ \alpha(1)-\alpha(0)=\widetilde  H(1,1)-\widetilde  H(0,1)=\tau(1)-\sigma(1)=v_0, $$
as desired. 

Note that for a given parametrization $\beta$ of $\ga$, the difference  $\alpha(1)-\alpha(0)$ is independent up to sign of the choice 
of the lift $\al$ of $\beta$. Indeed, suppose $\alpha'$ is another lift of 
$\beta$ under $\wp$.
Then 
$$ \wp(\alpha(0))=\beta(0)= \wp(\alpha'(0)),$$
and so by \eqref{eq:wpeq} we have
$$ \alpha'(0)=\pm \alpha(0)+m_0$$ 
for some (fixed) choice of the sign $\pm$ and $m_0\in 2\Z^2$. Then 
$t\in \I\mapsto  \pm \alpha(t)+m_0$ is a lift of $\beta$ with the same initial point as $\alpha'$ and so we see that 
$$ \alpha'(t)= \pm \alpha(t)+m_0$$ 
for all $t\in \I$. This implies that 
\begin{equation}\label{eq:displace} 
\alpha'(1)-\alpha'(0)=\pm(\alpha(1)-\alpha(0)),
\end{equation} 
as desired. 

It remains to show that up to sign   $\alpha(1)-\alpha(0)$ is independent of the choice 
of the parametrization  $\beta$ of $\ga$. For this we consider 
another parametrization $\beta'$ of $\ga$ as a simple loop. 
We first assume that $\beta'(0)=\beta(0)$. 
Then there exists a homeomorphism $h\: \I\ra \I$ such that 
$\beta'=\beta\circ h$. Here $h$ fixes the endpoints $0$ and $1$ of $\I$ or interchanges them depending on whether $\beta'$ parametrizes $\ga$ with the same or opposite orientation as $\beta$, respectively. 
In any case,  $\alpha'=\alpha\circ h$ is a lift of $\beta'$ under $\wp$.
It follows that 
$$ \alpha'(1)-\alpha'(0) =\alpha(h(1))-\alpha(h(0))= \pm(\alpha(1)-\alpha(0)), $$ 
as desired. As we now know, this relation is independent of the specific choice of the lift $\alpha'$ of $\beta'$. 

Finally, we have to consider the case where $\beta'$ has a possibly different 
initial  point than $\beta$, say $p_0\coloneq\beta'(0)\in \ga$. By what we have seen, in order to establish \eqref{eq:displace}, we can choose any 
parametrization $\beta'$ of $\ga$ as a simple loop with $\beta'(0)=p_0$ and any lift $\alpha'$ of $\beta'$. 

We can extend our parametrization  $\beta$ on $\I$  periodically to a continuous map $\beta\: \R \ra \ga$ such that 
$\beta(u+1)=\beta(u)$ for all $u\in \R$. Then $\beta$ lifts under 
$\wp$ to a continuous map $\alpha \: \R \ra \C$ which agrees 
with the original lift $\alpha$ on $\I$. 
Then $u\in \R \mapsto  \alpha (u+1)-v_0$ is also a lift  of  $\beta$ under 
$\wp$. The initial point of this lift corresponding to $u=0$ is equal to $\alpha(0)$. The uniqueness of lifts implies that this lift and the original lift  $\alpha$ are the same paths and so 
\begin{equation}\label{eq:funct}
\alpha (u+1)=\alpha(u)+v_0 
\end{equation}
for all $u\in \R$.

We can find $u_0\in [0,1)$ such that $\beta(u_0)=p_0$. Then 
$\beta'\:\I\ra \ga$ defined as $\beta'(u)=\beta(u_0+u)$ for $u\in \I$ is a parametrization of $\ga$ as a simple loop with the initial point $p_0$.  Under $\wp$ this path $\beta'$ has the lift $\alpha'\: \I\ra \C$ given by $\alpha'(u)=\alpha (u+u_0)$ for $u\in \I$. 
Then \eqref{eq:funct} implies that 
$$\alpha'(1)-\alpha'(0)=\alpha (u_0+1)-\alpha(u_0)=v_0,$$ 
 as desired. The proof is complete. 
 \end{proof} 
 
We are now almost ready to prove Lemma~\ref{lem:isoclassesP}.
Before we get to this, it is useful to discuss an alternative way to view our pillow $\PP$. 
  
We consider a slope $r/s\in \widehat \Q$. Then we  can choose $p, q\in \Z$ 
 such $pr+qs=1$   and define $\om\coloneq 
s+ir $ and $\widetilde \om \coloneq -p+iq$. The numbers  $\om$ and $\widetilde \om$ form a basis of $\C\cong \R^2$ over $\R$, and so every point $z\in \C$ can be uniquely written in the form 
$z=u \widetilde \om +v \om$ with $u,v\in \R$. Accordingly, the map 
$z=u \widetilde \om +v \om \mapsto R(z)\coloneq 
u \widetilde \om -v \om \mapsto$ for $u,v\in \R$ is a well-defined ``skew-reflection" $R$ on $\C$.  Note also that $\Z^2=\{n\widetilde \om+k\om: n,k\in \Z\}$.

We consider the parallelogram $Q\coloneq
 \{u \widetilde \om +v \om: u\in [0,1],\, v\in [-1,1]\}\sub \C$. 
 Then it follows from \eqref{eq:wpeq} that  
$\wp(Q)=\PP$. Moreover, for $z,w\in Q$, $z\ne w$, we have 
$\wp(z)=\wp(w)$ if and only if $z,w\in \partial Q$ and $w=R(z)$. Intuitively, this means that the pillow $\PP$ can be obtained from 
$Q$ by ``folding" $Q$ in its middle segment $[0, \widetilde \om]\sub Q$ and identifying the points on $\partial Q$ that correspond to each other under the skew-reflection $R$. The map  $\wp$ sends 
the set $\{0, \widetilde \om, \widetilde \om +\om, \om\}$ bijectively onto the set $\{A,B,C,D\}$ of vertices of $\PP$ (but not necessarily in that order). 

From this geometric picture it is clear that for each $t\in (0,1)$ 
the set 
$$\tau_{r/s}\coloneq \wp ([t\widetilde \om -\om, t\widetilde \om + \om])=\wp(\ell_{r/s}(t\widetilde \om))$$ 
is a  simple closed geodesic in $(\PP,V)$. Moreover, the sets 
$$\xi_{r/s}\coloneq 
\wp ([ -\om, + \om])=\wp(\ell_{r/s}(0))  \text{ and } \xi'_{r/s}\coloneq 
\wp ([\widetilde  \om -\om, \widetilde\om+\om ])=\wp(\ell_{r/s}(\widetilde \om))$$ 
are geodesic core arcs of  $\tau_{r/s}$  lying  in different components of $\PP\setminus \tau_{r/s}$. In particular, $\tau_{r/s}$ is an essential Jordan curve in $(\PP, V)$.

 It follows from \eqref{eq:wpeq} that 
 $\wp(\ell_{r/s}(t\widetilde \om))\cap \wp(\ell_{r/s}(t'\widetilde \om))  \ne \emptyset$ for $t,t'\in \R$  if and only if $t'-t \in 2\Z$ or $t'+t\in 2\Z$. In this case, we have $\wp(\ell_{r/s}(t\widetilde \om))=\wp(\ell_{r/s}(t'\widetilde \om))$. Moreover,   $\tau=\wp(\ell_{r/s}(t\widetilde \om))$ is a simple closed geodesic $\tau_{r/s}$ in $(\PP, V)$ if $t\in \R\setminus \Z$, it is 
equal to the geodesic arc $\xi_{r/s}$ if $t$ is an even integer, and is equal to the geodesic arc $\xi'_{r/s}$
if $t$ is an odd integer.  Note that $\ell_{r/s}(t\widetilde \om)$ for $t\in \R$ contains a point in $\Z^2=\{n\widetilde \om+k\om: n,k\in \Z\}$ if and only if $t\in \Z$. 

\begin{lemma}\label{lem:twogeod}  Let $\tau_{r/s}$ and $\tau'_{r/s}$ be two distinct simple closed geodesics in $(\PP, V)$  with  slope $r/s\in \widehat \Q$. Then
$\tau_{r/s}$ and $\tau'_{r/s}$ are isotopic rel.\ $V$. 
\end{lemma} 

\begin{proof} The previous considerations imply that we may assume that the geodesics are represented in the form 
$\tau_{r/s} =\wp(\ell_{r/s}(t\widetilde \om))$ and 
$\tau'_{r/s} =\wp(\ell_{r/s}(t'\widetilde \om))$ with
$t,t'\in (0,1)$, $t\ne t'$. We may assume $t<t'$.  Then 
$$U\coloneq \wp(\{ u\widetilde \om + v \om: u\in (t,t'), \, v\in [-1,1]\})$$ 
is an annulus contained in $\PP\setminus V$ with 
$\partial U=\tau_{r/s}\cup \tau'_{r/s}$. It follows from Lemma~\ref{lem:isocrit} that $\tau_{r/s}$ and
$\tau'_{r/s}$ are isotopic  rel.\ $V$. 
\end{proof}

\begin{proof}[Proof of Lemma~\ref{lem:isoclassesP}] 
 Let $\ga$ be an essential Jordan curve in $(\PP, V)$.
If we parametrize $\ga$ as a simple loop $\beta\: \I\ra \PP$ and 
lift $\beta$ to a path $\alpha\: \I\ra \C$ under $\wp$, then by Lemma~\ref{lem:prim} we know that $\alpha(1)-\alpha(0)$ is a primitive element of $2\Z^2$ uniquely determined by $[\ga]$ up to sign. Hence we can find relatively prime integers $r,s\in \Z$ such that 
\begin{equation}\label{eq:liftdisplacement} 
\alpha(1)-\alpha(0)= 2 (s+ir).
\end{equation}
 By switching signs here, which corresponds to parametrizing $\ga$ with opposite orientation, 
we may assume that 
 $r\in \Z$, $s\in \N_0$, and that  $r=1$ if $s=0$. Note that with  these restrictions on $r$ and $s$, the primitive element $2 (s+ir)$ of $2\Z^2$ corresponds to the unique slope $r/s\in \widehat \Q$, and every slope in $ \widehat \Q$ arises from a unique primitive element of $2\Z^2$ in this form. 
 
As before, define $\om\coloneq 
s+ir $ and $\widetilde \om \coloneq -p+iq$,  where   $p, q\in \Z$ 
are chosen so that  $pr+qs=1$.  
We know that  $\xi\coloneq\wp(\ell_{r/s}(0))$ and 
$\xi'\coloneq \wp(\ell_{r/s}(\widetilde \om))$  are disjoint 
 geodesic arcs in $(\PP,V)$. The sets $\wp^{-1}(\xi)$  and $\wp^{-1}  (\xi')$ consist of parallel lines with slope $r/s$, and these lines alternate in the following sense: each component of   $\C\setminus \wp^{-1}(\xi\cup\xi')$ is an infinite strip whose  boundary contains  one line in $\wp^{-1}(\xi)$ and one line in  $\wp^{-1}  (\xi')$.

We may assume that $\ga$ is in minimal 
 position with  $\xi\cup \xi'$. Moreover, we may assume that the parametrization $\beta$ of $\ga$ as a simple loop was chosen so that $\beta(0)=\beta(1) \not \in \xi \cup \xi'$. 
 
 Then by Lemmas~ \ref{lem:nobigons} and~\ref{lem:alternate8} we know that either $\xi\cap \ga=\emptyset=\xi'\cap \ga$, or the following conditions are true:  $\ga$ meets $\xi$ and $\xi'$ transversely,  the sets  $\xi \cap \ga$ and $\xi'\cap \ga$ are non-empty and finite,  and  the points in these sets alternate on $\ga$.  
 We claim that the latter is not possible. 
 
 Otherwise, we choose a lift  $\alpha $  of $\beta$ under $\wp$. Then $\alpha$ intersects the lines  in $\wp^{-1}(\xi\cup\xi')$ transversely, 
 we have   
 $$k\coloneq \#(\alpha \cap  \wp^{-1}(\xi))= \#(\alpha \cap 
  \wp^{-1}(\xi'))\in \N, $$ and  the points in $\alpha \cap  \wp^{-1}(\xi)\ne \emptyset$ and $\alpha \cap \wp^{-1}(\xi')\ne \emptyset$ alternate on $\alpha$. An argument very similar to the proof of Lemma~\ref{lem:difference}  then shows that $2k>0$ is the number of lines in  the set  $\wp^{-1}(\xi\cup\xi')$ that separate $\alpha (0)$ from 
  $\alpha(1)$. On the other hand,  we know that $\alpha (1)-\alpha (0)=
 2 (s+ir)$ and  so $\alpha (1)$ and $\alpha (0)$ lie on a line with  slope $r/s$ and are not separated by any line in $\wp^{-1}(\xi\cup\xi')$. This is a contradiction.
 
 This shows that $\xi\cap \ga=\emptyset=\xi'\cap \ga$. Now consider a simple closed geodesic $\tau_{r/s}=\wp(\ell_{r/s}(t\widetilde \om))$ with  slope $r/s$ and $0<t<1$. Since $\xi \cap \ga
 =\emptyset$,  we can choose 
  $t$ very  close 
 to $0$ so that $\tau_{r/s}\cap \ga=\emptyset$. 
 Now we can apply considerations very similar to the proof of Corollary~\ref{cor:pull}~\ref{pull2}. The complement of 
 $\tau_{r/s} \cup  \ga$ in $\PP$ is a 
  disjoint union $\PP \setminus (\tau_{r/s} \cup  \ga)=W\cup U\cup W' $,
 where $W,W'\sub \PP$ are open Jordan regions and $U\sub \PP$ is an annulus with $\partial U=\tau_{r/s} \cup  \ga$.
  Since  $\tau_{r/s}$ and $\ga$  are essential,  both $W$ and $W'$ must contain at least two points in $V$. Since $\#V=4$, we have 
   $U\cap V =\emptyset$.   Lemma~\ref{lem:isocrit}
now implies that 
 $\tau_{r/s}$ and $\ga$ are isotopic rel.\ $V$. By Lemma~\ref{lem:twogeod}
 the curve $\ga$ is actually isotopic to {\em each} closed geodesic $\tau_{r/s}$ with slope $r/s$.

The map $[\ga]\mapsto r/s$ that sends each isotopy class $[\ga]$  to a slope $r/s\in \widehat \Q$ obtained from a primitive element in $2\Z^2$ associated with $[\ga]$ according to Lemma~\ref{lem:prim} is well-defined. It is clear that it is surjective, because the isotopy class $[\tau_{r/s}]$ of a geodesic $\tau_{r/s}$ with slope $r/s\in \widehat \Q$  is sent to $r/s$.  To see that it is injective,  suppose two isotopy 
classes $[\ga]$ and $[\ga']$ are sent to the same slope $r/s\in \widehat \Q$ by this map.  If  $\tau_{r/s}$ is a closed geodesic with slope $r/s$, then by our previous discussion we have 
$ \ga \sim   \tau_{r/s}\sim \ga'$ rel.\ $V$. 
Hence $[\ga]=[\ga']$. It follows that  the map $[\ga]\mapsto r/s$ is indeed a bijection. 
\end{proof} 

\begin{rem}\label{rem:rslift} Suppose $\ga$ is an essential Jordan curve in $(\PP,V)$ and its isotopy class $[\ga]$ rel.\ $V$ 
corresponds to slope 
$r/s\in \widehat \Q$ according to Lemma~\ref{lem:isoclassesP}. Let $\beta\: \I \ra \ga$  
be  a parametrization of $\ga$ as a simple loop, and 
 $\alpha\: \I\ra \C$  be a lift of $\beta$ under $\wp$. Then  \eqref{eq:liftdisplacement} in the proof of  Lemma~\ref{lem:prim}  shows that we always have   $\alpha(1)-\alpha(0)=\pm2(s+ir)$. 
\end{rem} 

A  similar statement is true for arcs in $(\PP, V)$. 

\begin{cor}\label{cor:isoarcs}
Let $\xi$ be an arc in $(\PP, V)$. Then $\xi$ is isotopic rel.\ $V$ to a geodesic arc $\xi_{r/s}$ for some slope $r/s\in \widehat \Q$. Moreover, if $\beta\: \I\ra \xi$ is a homeomorphic parametrization of $\xi$ and $\alpha$ is any lift of 
$\beta$ under $\wp$, then $\alpha(1)-\alpha(0)=\pm(s+ir)$.
\end{cor}

\begin{proof} The arc $\xi$ joins two of the points in $V$, while the other two points in $V$ do not lie on $\xi$. Hence we may ``surround" $\xi$ by an essential  Jordan curve $\ga$ in $(\PP, V)$ such that $\xi$ is a core arc of $\ga$.  By Lemma~\ref{lem:isoclassesP} we know that $\ga$ is isotopic rel.\ $V$ to a closed geodesic $\tau_{r/s}$ in $\PP$ with  some slope $r/s\in \widehat \Q$. Hence $\xi$ is isotopic rel.\ $V$ to a core arc $\xi'$ of $\tau_{r/s}$. 

Now any two arcs in the interior of a closed topological disk $D$ with the same endpoints are isotopic by an isotopy that fixes the endpoints of these arcs and the points in $\partial D$ (see \cite[Theorem A.6 (ii)]{BuserGeometry}). This implies that any two core arcs of an essential Jordan curve  in $(\PP, V)$ are isotopic rel.\ $V$  if  the core arcs  have the same endpoints. 
We know that the closed geodesic $\tau_{r/s}$ has precisely two 
geodesic  arcs with slope $r/s$ as core arcs in different components of $\PP\setminus \tau_{r/s}$. Therefore, $\xi'$, and hence also $\xi$,  is isotopic rel.\ $V$ to a geodesic arc $\xi_{r/s}$ with slope $r/s$.  This proves the first part of the statement.

Let $\beta\: \I \ra \xi$ be a homeomorphic parametrization 
of $\xi$ and $\alpha\: \I \ra \C$ be a lift of $\beta$ under $\wp$. 
By what we have seen, we can choose an isotopy $H\: \PP\times \I\ra \PP$ rel.~$V$ with $H_0=\id_{\PP}$ and $H_1(\xi)=\xi_{r/s}$. 
We use this to  define a homotopy $\overline H\: \I\times \I\ra \PP$ by setting 
$$ \overline H(u,t) \coloneqq H(\be(u),t)$$ 
for $u,t\in \I$. 
Note that $ \overline H_t$ for $t\in \I$ gives a homeomorphic parametrization
of an arc in $(\PP, V)$. These arcs have all the same endpoints. In particular $u\in \I\mapsto \beta'(u)\coloneqq \overline H_1(u)=H_1(\beta(u))$ gives a homeomorphic parametrization of $H_1(\xi)=\xi_{r/s}$. 

We can lift  $ \overline H_t$ under $\wp$ to  find a homotopy 
$\widetilde H\: \I\times  \I \ra \C$ such  that $\widetilde H_0=\alpha $ 
and  $\wp\circ \widetilde H_t=\overline H_t$
for all $t\in \I$. To see this,  one first applies the standard homotopy lifting theorem (see \cite[Proposition 1.30]{Hatcher}) to the homotopy $\overline H$ restricted $(0,1)\times \I$ and the covering map $\wp\: \C\setminus \Z^2\ra \PP\setminus V\supset \overline H((0,1)\times \I)$
to obtain a unique homotopy  $\widetilde H\: (0,1)\times  \I\ra \C$ with $\widetilde H_0=\alpha|(0,1)$. Now as $u_0\to 0^+$, the set 
$\overline H((0,u_0]\times \I)$ shrinks to the point $\beta(0)\in V$, and so  the connected set $\widetilde  H((0,u_0]\times \I)$
shrinks to a unique point in $\wp^{-1}(V)=\Z^2$. This point can only be $\alpha(0)$. Hence $\overline H(u,t)\to \alpha(0)$ uniformly for $t\in \I$ as $u\to 0^+$, and similarly  $\overline H(u,t)\to \alpha(1)$
uniformly for $t\in \I$ as $u\to 1^-$.
This implies that we can continuously extend 
$\widetilde H$ to a homotopy on $\I\times \I$ with the desired properties by setting $\widetilde H(0,t)=\alpha(0)$ and 
$\widetilde H(1,t)=\alpha(1)$ for $t\in \I$.

Then $\alpha'\coloneq\widetilde H_1$ is a lift of $\beta'$, because 
$\wp\circ \alpha'= \wp\circ \widetilde H_1 =\overline H_1=\beta'$.
Since $\beta'$ is a homeomorphic parametrization of the geodesic arc $\xi_{r/s}$, the path $\alpha'$  sends $\I$ homeomorphically onto a subsegment of 
 a line $\ell_{r/s}\sub \C$. Since  $\alpha'$  has its endpoints in $\Z^2$ and 
$\alpha'((0,1))$ is disjoint from $\Z^2$, this implies 
$\alpha'(1)-\alpha'(0)=\pm (s+ir)$.
Since $\alpha$ and $\alpha'$ have the same endpoints, the statement follows. \end{proof}

\begin{proof}[Proof of Lemma~\ref{lem:i-properties-curve}] In the proof all isotopies, isotopy classes, intersection numbers, etc.\ are for isotopies on $\PP$ rel. $V$. We will use the facts about the geodesics on $(\PP, V)$ discussed before Lemma~\ref{lem:twogeod} without further reference. 

\smallskip 
\ref{item:i1} Let $\alpha$ and $\beta$ be essential  Jordan curves in $(\PP,V)$ as in the statement. As before, we define $\om=s+ir$ and 
   $\widetilde \om=-p+iq$, where $p,q\in \Z$ and $pr+qs=1$.

First suppose that $r/s=r'/s'$. Then in the isotopy class $[\alpha]=[\beta]$ we can find simple closed geodesics with slope 
$r/s$ that are disjoint, for example the curves 
$\tau_{r/s} =\wp (\ell_{r/s}(\widetilde \om/3))$ and 
$\tau'_{r/s} =\wp (\ell_{r/s}(2\widetilde \om/3))$. It follows that
$\ins(\alpha, \beta)=0$ in this case.

We now assume that $r/s\ne r'/s'$. In order to determine $\ins(\alpha, \beta)$, we have to find  the minimum of all numbers 
$\#(\alpha\cap \beta)$, where $\alpha$ and $\beta$ range over the given isotopy classes.  By applying a suitable isotopy, we  can reduce to the case where $\alpha$ is a fixed curve in  its isotopy class and we only have to take variations over $\beta$.  So by Lemma~\ref{lem:isoclassesP} we may assume that 
$\alpha=\tau_{r/s}=\wp(\ell_{r/s})$ is a simple closed geodesic as in the statement.  Then $\tau_{r/s}=
\wp(\ell_{r/s}(t_0\widetilde \om ))$ for some $t_0\in (0,1)$. The preimage
 $\wp^{-1}(\tau_{r/s})$ of $\tau_{r/s}$ under $\wp$ consists of the two disjoint families
 \begin{equation}\label{eq:linefam} 
 \mathcal{F}_1\coloneq \{\ell_{r/s}((t_0+2j)\widetilde \om ):j \in \Z\}
 \text{ and } 
 \mathcal{F}_2\coloneq  \{\ell_{r/s}((-t_0+2j)\widetilde \om): j\in \Z\}
 \end{equation}  of distinct lines with  slope $r/s$. 
 
 Now let $\widetilde \beta\:\I\ra \C$ be a lift of $\beta$ under $\wp$, 
 where we think of $\beta$ as a simple closed loop in a parametrization with suitable orientation.  
Then if  $z_0\coloneq\widetilde \beta(0)$ and $w_0\coloneq \widetilde \beta(1)$, we have $w_0-z_0=2(s'+ir')$ as follows from 
Remark~\ref{rem:rslift}. By changing the basepoint of $\beta$ if neccesary, we may assume that   
$\widetilde \beta(0),\widetilde \beta(1)\not\in \wp^{-1}(\tau_{r/s})$. If $\om'\coloneq s'+ir'$, then we can write $\om'$ 
uniquely in the form 
\begin{equation}\label{eq:om'} 
 \omega'=k\omega+n\widetilde \omega, 
 \end{equation} 
where $k,n\in \Z$.  Note that then $|n|=N\coloneq  |rs'-sr'|>0$ (to see this, multiply \eqref{eq:om'}  by the complex conjugate of $\om$  and take imaginary parts).

Now each  family $\mathcal{F}_j$, $j=1, 2$,  consists of equally spaced parallel lines with  slope $r/s$ such  that consecutive lines in each family differ by a translation by $2\widetilde \om$. This implies that the points $z_0$ and $w_0=z_0+2\om'$ are separated  by precisely $N$ lines from each of the families $\mathcal{F}_j$, $j=1, 2$. 
So  $\widetilde \beta$ must have at least $2N$ points in common with 
$\wp^{-1}(\tau_{r/s})$. Since $\wp\circ \widetilde \beta $ maps $[0,1)$ injectively onto $\beta$,  we conclude  that $\beta=\wp(\widetilde \beta)$ has at least $2N$ points in common with $\tau_{r/s}$. If $\beta=\tau_{r'/s'}$, then $\widetilde \beta$ is a  parametrization of the line segment $[z_0, w_0]$, and so $\widetilde \beta$ meets 
$\wp^{-1}(\tau_{r/s})$ in precisely $2N$ points. This  means that $\tau_{r/s}$ and $\tau_{r'/s'}$ have exactly $2N$ points in common. 
  
It follows that for all $\beta$ we have 
$$ 
2N \le \#(\wp^{-1}(\tau_{r/s})\cap \widetilde \beta)=\#(\tau_{r/s} \cap \beta), $$ 
and so $ 2N\le \ins(\alpha,\beta)\le  \#(\tau_{r/s}\cap 
\tau_{r'/s'})=2N . $  
Thus we have equality here and the statement follows.

\smallskip 
 \ref{item:i2} This is a variant of the argument  in   \ref{item:i1} and we use the same notation. 
 
Since $\beta \sim \tau_{r'/s'}$, the core arc 
$\xi$ of $\beta$ is isotopic to a core arc of  $ \tau_{r'/s'}$.
Now two core arcs of a given essential Jordan curve in $(\PP, V)$ are isotopic rel.\ $V$ if they have the same endpoints (this was pointed out in the proof of Corollary~\ref{cor:isoarcs}).  It follows that $\xi\sim \xi'_{r'/s'}$, where $\xi'_{r'/s'}$ is one of the two 
geodesic core arcs  of $ \tau_{r'/s'}$. In particular, 
     $\xi'_{r'/s'}=\wp(\ell_{r'/s'})$ for a line  $\ell_{r'/s'}\sub \C$ that contains a point in $\Z^2$.  Note that $\xi'_{r'/s'}$ possibly differs from the geodesic arc $\xi_{r'/s'}$ as in the statement (if $\xi'_{r'/s'}$ and $\xi_{r'/s'}$ lie in different components of $\PP\setminus \tau_{r'/s'}$).

If $\widetilde \xi\: \I\ra \C$ is a lift of $\xi$ under $\wp$ in suitable orientation, and $z_1\coloneq \widetilde\xi(0)\in \Z^2$, $w_1\coloneq \widetilde \xi(1)\in \Z^2$, then it follows from Corollary~\ref{cor:isoarcs} that $w_1-z_1=\om'=s'+ir'$. 
Now  \eqref{eq:om'} implies that there are  exactly $N=|n|=  |rs'-sr'|$ lines 
in $ \mathcal{F}_1\cup  \mathcal{F}_2$ that separate $z_1$ and $w_1$
(essentially, this follows from the fact that the set $[0,n]\cap \{2j\pm t_0:j\in \Z\}$, where $t_0\in (0,1)$,  contains precisely $N=|n|$ points).  

Let  $\widetilde \om'\coloneq -p'+iq'$, where $p',q'\in \Z$ and $p'r'+q's'=1$.  Then,  by the discussion before Lemma~\ref{lem:twogeod}, the map $\wp $ sends one of the segments $[z_1, w_1]$ and  $[z_1+\widetilde\om', w_1+ \widetilde\om']$ 
homeomorphically onto $\xi_{r'/s'}$ depending on whether $\xi'_{r'/s'}= \xi_{r'/s'}$ or  $\xi'_{r'/s'}\ne  \xi_{r'/s'}$, respectively.  In either case,  each segment meets exactly $N$ lines 
in $ \mathcal{F}_1\cup  \mathcal{F}_2$. 

Arguing as before, we see that
$$ \#(\tau_{r/s}\cap \xi_{r'/s'})=N\le 
 \#(\wp^{-1}(\tau_{r/s})\cap\widetilde  \xi)= \#(\tau_{r/s}\cap \xi). $$
 This leads to 
 $ N\le \ins(\alpha,\xi)\le  \#(\tau_{r/s}\cap 
\xi_{r'/s'})=N . $  
So we have equality here and the statement follows.

 \smallskip 
 \ref{item:i3}--\ref{item:i5} These are special cases of 
  \ref{item:i1} and \ref{item:i2}.
  For example,  $a=\wp(\R\times\{0\})=\wp(\ell_0(0))$ is a core arc of $\alpha^h$ corresponding to slope $r'/s'=0/1=0$.  Hence by \ref{item:i2} we have $$\ins(\alpha,a)=|r\cdot 1-s\cdot 0|=|r|= 
   \#(\tau_{r/s}\cap \xi_0 )=
\#(\tau_{r/s}\cap a ).  $$
The other statements follow from similar considerations. 
\end{proof}

%\nocite{*}
%\bibliographystyle{alpha}
%\bibliography{main}

%\begin{thebibliography}{}
   \newcommand{\etalchar}[1]{$^{#1}$}

%\end{thebibliography}

\end{document}